\documentclass[10pt,a4paper]{article}

\usepackage[backgroundcolor=orange!20]{todonotes} 

\usepackage[utf8]{inputenc}
\usepackage{amsmath}
\usepackage{amsfonts}
\usepackage{amssymb}
\usepackage{amsthm}
\usepackage{dsfont}
\usepackage{cases}
\usepackage{comment}

\usepackage{mathtools} 

\usepackage{enumitem} 

\usepackage{floatrow}
\usepackage{hyperref}
\usepackage[
    capitalize,
    nameinlink,
    noabbrev
]{cleveref}

\crefformat{equation}{(#2#1#3)}
\crefmultiformat{equation}{(#2#1#3)}%
{ and~(#2#1#3)}{, (#2#1#3)}{ and~(#2#1#3)}

\hypersetup{colorlinks=true, linkcolor=blue, citecolor=blue, urlcolor=blue,
  pdftitle={}, pdfauthor={}}

\theoremstyle{definition}
\newtheorem{theorem}{Theorem}[section]
\newtheorem{lemma}[theorem]{Lemma}
\newtheorem{proposition}[theorem]{Proposition}
\newtheorem{corollary}[theorem]{Corollary}
\newtheorem{definition}[theorem]{Definition}

\theoremstyle{remark}
\newtheorem{remark}[theorem]{Remark}

\numberwithin{equation}{section}

\usepackage[left=3cm,right=3cm,top=2.5cm,bottom=2.5cm]{geometry}

\newcommand{\Real}{\exponential{Re}}
\newcommand{\Image}{\exponential{Im}}
\newcommand{\smallo}{o}
\newcommand{\bigO}{\mathcal{O}}

\usepackage{xcolor}
\definecolor{lgreen}{rgb}{0.0, 0.48, 0.0}
\definecolor{lpurple}{rgb}{0.48, 0.0, 0.48}

\usepackage{tikz}

\definecolor{bblue}{rgb}{0.2, 0.4, 0.8}
\definecolor{bgreen}{rgb}{0.2, 0.6, 0.4}
\definecolor{bred}{rgb}{0.8, 0.4, 0.2}
\definecolor{bviolet}{rgb}{0.7, 0.2, 0.7}
\definecolor{blackred}{rgb}{0.6, 0.3, 0.3}
\definecolor{blackblue}{rgb}{0.3, 0.3, 0.6}
\definecolor{borange}{rgb}{0.8, 0.4, 0.2}

\usetikzlibrary{fit,arrows,trees,shapes,shapes.geometric,calc,matrix,
decorations.markings}
\tikzset{
  treenode/.style = {align=center, inner sep=0pt, text centered,
    font=\sffamily},
  arnBleuPetit/.style = {treenode, circle, bblue, draw=bblue,
    fill=bblue!10,
    minimum width=0.8em, minimum height=0.5em
  },
  arnRougePetit/.style = {treenode, circle, bred, draw=bred,
    fill=bred!40,
    minimum width=0.8em, minimum height=0.5em
  },
  arnBleuGrande/.style = {treenode, circle, bblue, draw=bblue,
    text width=1.5em, very thick,
    fill=bblue!10
  },
  arnOrangePetit/.style = {treenode, circle, black, draw=borange,
    fill=borange!20,
    minimum width=0.8em, minimum height=0.5em
  },
  arnVioletPetit/.style = {treenode, circle, black, draw=bviolet,
    fill=bviolet!08,
    minimum width=0.8em, minimum height=0.5em
  },
  arnVioletGrande/.style = {treenode, circle, bviolet, draw=bviolet,
    text width=1.5em, very thick,
    fill=bblue!10
  },
  arnOrangeGrande/.style = {treenode, circle, borange, draw=borange,
    text width=1.5em, very thick,
    fill=borange!10
  },
  arnVertPetit/.style = {treenode, circle, black, draw=bgreen,
    fill=bgreen!20,
    minimum width=0.8em, minimum height=0.5em
  },
  arnVertGrande/.style = {treenode, circle, bgreen, draw=bgreen,
    text width=1.5em, very thick,
    fill=bgreen!10}
}

\newcommand{\fromto}[2]{
    \path (#1) edge [black,thick,
    decoration={markings,mark=at position 1 with
    {\arrow[ultra thick,blackblue, rotate=0]{>}}}, postaction={decorate}
    ] node {} (#2);
}
\newcommand{\aprod}[2]{
    \path (#1) edge [blackred,thick,
    decoration={markings,mark=at position 1 with
    {\arrow[ultra thick,blackred, rotate=0]{>}}}, postaction={decorate}
    ] node {} (#2);
}

\def\GrER           {\mathbb{G}}    
\def\GrERMulti      {\mathbb{MG}}   
\def\GrGilb         {\mathbb{G}}    
\def\Di             {\mathbb{D}}    
\def\DiER           {\mathbb{SD}}   
\def\DiERBoth       {\mathbb{D}}    
\def\DiGilb         {\mathbb{SD}}   
\def\DiGilbBoth     {\mathbb{D}}    
\def\DiERMulti      {\mathbb{MD}}   
\def\DiGilbMulti    {\mathbb{MD}}   

\let\geq\geqslant
\let\leq\leqslant

\newcommand{\graphic}[1]{\widehat{\mathbf{#1}}}
\newcommand{\gset}{\graphic{Set}}
\newcommand{\gsetsimple}{\graphic{Set}^{(\text{simple})}}

\usepackage{xspace}
\newcommand*{\eg}{\textit{e.g.}\@\xspace}
\newcommand*{\cf}{\textit{cf.}\@\xspace}
\newcommand*{\ie}{\textit{i.e.}\@\xspace}
\newcommand*{\whp}{\textit{w.h.p.}\@\xspace}
\newcommand*{\resp}{\textit{resp.}\@\xspace}

\newcommand{\exponential}[1]{\operatorname{#1}}
\newcommand{\eComplex}{\exponential{Complex}}
\newcommand{\eStrong}{\exponential{Strong}}
\newcommand{\eStrongSimple}{\exponential{Strong}^{(\text{simple})}}
\newcommand{\eStrongStrict}{\exponential{Strong}^{(\text{strict})}}

\newcommand{\MG}{M\!G}
\newcommand{\MD}{M\!D}

\newcommand{\ai}{\operatorname{Ai}}
\newcommand{\bi}{\operatorname{Bi}}
\newcommand{\hi}{\operatorname{Hi}}
\newcommand{\Delemsimple}{\widehat D_{\mathrm{elem}}^{(\mathrm{simple})}}

\title{The birth of the strong components\footnote{
    This paper is an extended version of several earlier works by the
    authors. Some parts of this paper appeared in
    the 10th European Conference on Combinatorics, Graph Theory and Applications
    (EUROCOMB 2019),
    in the 32nd International Conference on Formal Power Series and Algebraic
    Combinatorics (FPSAC 2020),
    and in
    the 31st International Conference on Probabilistic,  Combinatorial and
    Asymptotic Methods for the Analysis of Algorithms (AofA 2020),
    see~\cite{de2019symbolic,Panafieu2020,Naina2020}.
    Authors are presented in alphabetical order.
  }}
\usepackage[affil-it]{authblk}

\author[$\dag$, $\sharp$, $\natural$]{Sergey \textsc{Dovgal}}
\author[$*$]{\'Elie \textsc{de Panafieu}}
\author[$\ddag$]{Dimbinaina \textsc{Ralaivaosaona}}
\author[$\S$]{Vonjy \textsc{Rasendrahasina}}
\author[$\ddag$, $\P$]{Stephan \textsc{Wagner}}

\affil[$\dag$]{LIPN, Universit\'e  Sorbonne Paris Nord, France}
\affil[$\sharp$]{LaBRI, Universit\'e de Bordeaux, France}
\affil[$\natural$]{IMB, Universit\'e de Bourgogne, France}
\affil[$*$]{Nokia Bell Labs, France}
\affil[$\ddag$]{Department of Mathematical Sciences, Stellenbosch
  University, South Africa}
\affil[$\S$]{\'Ecole Normale Supérieure, Universit\'e d'Antananarivo, Madagascar}
\affil[$\P$]{Department of Mathematics, Uppsala University, Sweden}
\date{}

\begin{document}
\maketitle

\begin{abstract}
    It is known that random directed graphs $D(n,p)$ undergo a phase transition
    around the point $p = 1/n$.
    Earlier, \L{}uczak and Seierstad
    have established that as $n \to \infty$ when $p = (1 + \mu n^{-1/3})/n$,
    the asymptotic probability that
    the strongly connected components of a random directed graph are only cycles
    and single vertices decreases from 1 to 0 as
    $\mu$ goes from $-\infty$ to $\infty$.

    By using techniques from analytic combinatorics, we establish the exact
    limiting value of this probability as a function of $\mu$ and provide
    more statistical insights into the structure of a random digraph around,
    below and above its transition point.
    We obtain the limiting probability that a
    random digraph is acyclic and the probability that it has one strongly
    connected complex component with a given difference between the number of
    edges and vertices (called excess). Our result can be extended to the case of
    several complex components with given excesses as well in the whole range of
    sparse digraphs.

    Our study is based on a general symbolic method which can deal with
    a great variety of possible digraph families, and a version of the
    saddle point method which can be systematically applied to the complex contour
    integrals appearing from the symbolic method.
    While the technically easiest model is the model of random multidigraphs,
    in which
    multiple edges are allowed, and where edge multiplicities are
    sampled independently
    according to a Poisson distribution with a fixed parameter $p$,
    we also show how to systematically approach the family of
    simple digraphs, where multiple edges are forbidden, and where 2-cycles are
    either allowed or not.

    Our theoretical predictions are supported by numerical simulations when
    the number of vertices is finite, and we provide tables of numerical values
    for the integrals of Airy functions that appear in this study.
\end{abstract}


\tableofcontents

\section{Introduction}
\label{section:introduction}

Random \emph{graphs} and \emph{directed graphs} (\emph{digraphs})
are omnipresent in many modern
research areas such as computer science, biology, the study of social and
technological networks, analysis of algorithms, and theoretical physics.
Depending on the problem formulation, one chooses either an undirected or a
directed version.
Many different models of randomness
can be defined on the set of graphs, ranging from the models where all the
edges are sampled independently or almost independently to the models where
edge attachment is more preferential towards vertices with higher degrees.
The current interest in this thriving topic is witnessed by a large number
of books on random graphs, such as the books of Bollob\'as~\cite{Bollobas2001},
Kolchin~\cite{Kolchin1998}, Janson, \L{}uczak and Ruci\'
nski~\cite{Janson2000},  Frieze and Karo\'nski~\cite{Frieze2015} and
van der Hofstad~\cite{Remco2016}.

\subsection{Historical background}

The simplest possible model of a random graph is that of Erd\H{o}s and Rényi,
where all \( 2^{\binom{n}{2}} \) graphs whose vertex set is \( [n] := \{1,
\ldots, n\} \) have the same probability of being chosen.
Following this work, the study of (undirected) random graphs  was launched by
Erd\H{o}s and Rényi in a series of seminal
papers~\cite{Erdos1959,Erdos1960,Erdos1964,Erdos1966} where they set the
foundations of random graph theory while studying the \emph{uniform model}
$\GrER(n,m)$. We will make the precise definitions of the various models in the
following section.
Simultaneously, Gilbert~\cite{Gilbert1959} proposed the \emph{binomial model}
$\GrGilb(n,p)$ which is also studied by
Stepanov~\cite{Stepanov1970a,Stepanov1970b,Stepanov1973}. These two models are
expected to have a similar asymptotic behaviour provided that the expected number
of edges in  $\GrGilb(n,p)$ is near $m$. This equivalence of random graph models
has been established for instance in \cite[Chapter 2]{Bollobas2001}.
Another useful variation (\emph{multigraphs}) allows multiple edges and loops.

When the number of edges  $m$ or, equivalently, the probability $p$  is increased,
the graph  changes from being \emph{sparse} to \emph{dense}.
In the sparse case the number of edges is $m=cn/2$  for $\GrER
(n,m)$   (or equivalently $p=c/n$ for $\GrGilb(n,p)$) for some
constant $c$, while in the dense case \( c \to \infty \) as \( n \to \infty \).
A fascinating transition appears inside the sparse case
when \emph{the density}
$c$ of the graph passes through the \emph{critical value} $c=1$
where the structure of the graph changes abruptly.
This spectacular phenomenon, also called the \emph{phase transition}
or the \emph{double-jump threshold}, has been established by Erd\H{o}s and Rényi.
More precisely, in the \emph{subcritical phase} where $c<1$,
the graph contains
\emph{with high probability}  (\emph{\whp})\footnote{We shall
say that a random (di)graph $G$
 has some property with high probability if the probability that $G$
 has this property tends to $1$ as $n\to\infty$.}
only
acyclic  and unicyclic components  and the largest  component is  of
size of order  $\Theta(\log n)$ while in the \emph{supercritical phase} where
$c>1$ it contains a unique giant component of size
proportional to $n$. Furthermore, on the borders of the critical window defined as
$c = 1+ \mu n^{-1/3}$ with $|\mu|=o(n^{1/3})$, the graph still demonstrates a
subcritical behaviour as $\mu \to -\infty$,
while as \( \mu \to +\infty \) it has a unique giant component
of size $\mu n^{2/3}$ \emph{\whp}
Finally, when $\mu = \Theta(1)$, the largest component has a size of order
$\Theta(n^{2/3})$, where several components of this order may occur.
Another remarkable  phenomenon occurs also when $c = \log
n$. As \( c \) passes \( \log n \), the graph changes from having isolated
vertices to being connected \emph{\whp}

\begin{figure}[hbt]
    \RawFloats
    \begin{minipage}[t]{0.5\textwidth}
    \begin{center}
        \begin{tikzpicture}[>=stealth',thick, node distance=1.0cm]
\draw
node[arnBleuPetit](a) at ( 0, 0)  { }
node[arnBleuPetit](s) at ( 1,-1)  { }
node[arnBleuPetit](d) at ( 2, 0)  { }
node[arnBleuPetit](f) at ( 1, 1)  { }
;
\fromto{a}{d};
\fromto{a}{f};
\fromto{s}{a};
\fromto{d}{s};
\fromto{f}{d};
\draw
node[arnOrangePetit](q)  at (3.8, 0)  { }
node[arnVertPetit](w)  at (3,  -1)  { }
node[arnVioletPetit](e)  at (4.8,-1)  { }
node[arnOrangePetit](r)  at (4.8, 0)  { }
node[arnOrangePetit](t)  at (3.8, 1)  { }
;
\aprod{q}{w};
\fromto{q}{r};
\aprod{e}{r};
\fromto{r}{t};
\fromto{t}{q};
\aprod{d}{t};
\aprod{d}{q};
\aprod{s}{w};
\node[rectangle,dashed,draw,fit=(a)(s)(d)(f), very thick,
      rounded corners=5mm,inner sep= 5pt, bblue] {};
\node[rectangle,dashed,draw,fit=(w), very thick,
      rounded corners=3mm,inner sep= 4pt, bgreen] {};
\node[rectangle,dashed,draw,fit=(t)(q)(r), very thick,
      rounded corners=5mm,inner sep= 5pt, borange] {};
\node[rectangle,dashed,draw,fit=(e), very thick,
      rounded corners=3mm,inner sep= 4pt, bviolet] {};
\end{tikzpicture}
    \end{center}
    \caption{
    \label{fig:strong:comps}
    A directed graph and its strong components. For the sake of clarity, the vertex labels have been omitted.}
    \end{minipage}%
    \begin{minipage}[t]{0.5\textwidth}
    \begin{center}
        \begin{tikzpicture}[>=stealth',thick, node distance=1.0cm]
\draw
node[arnBleuGrande](a) at (-1,-1) {a$\circlearrowright$}
node[arnOrangeGrande](b) at (2,-.5) {b$\circlearrowright$}
node[arnVertGrande](c) at (1,-2) {c$\circlearrowright$}
node[arnVioletGrande](d) at (3,-2) {d$\circlearrowright$}
;
\aprod{a}{b};
\aprod{a}{c};
\aprod{b}{c};
\aprod{d}{b};
\end{tikzpicture}
    \end{center}
    \caption{\label{fig:condensation}Condensation of a directed graph.}
    \end{minipage}
\end{figure}

In \emph{directed graphs}, there are several notions of connected
components. If edge directions are ignored, then the resulting connected
components of the underlying undirected graph are called the
\emph{weakly connected} components of the directed graph. Next,
a component is called \emph{strongly connected}, or
just \emph{strong}, if there is a directed path between any pair of its
vertices (see~\cref{fig:strong:comps}).
If each strong component is replaced
by a single vertex, then the resulting digraph is called a \emph{condensation
digraph}, and it does not contain oriented cycles anymore
(see~\cref{fig:condensation}).

The two random models \( \GrGilb(n,p) \) and \( \GrER(n,m) \) have their
counterparts \( \Di(n,p)\) and \( \Di(n,m) \)
in the world of directed graphs, with the additional possibility
of either allowing or not allowing multiple edges, loops or cycles of length \( 2 \)
(\emph{2-cycles}). A similar
phase transition has been discovered in \( \Di(n,p) \) as well. In the
model where multiple edges and
loops are forbidden, and $2$-cycles are allowed,
Karp~\cite{Karp1990} and \L{}uczak~\cite{Luczak1990} proved that
(i) if $np$ is fixed with \( np < 1 \), then
for any sequence \( \omega(n) \) tending to infinity
arbitrarily slowly every strong component has \emph{\whp} at most \(
\omega(n) \) vertices, and all the strong components are either cycles or single
vertices; (ii) if \( np \) is fixed with \( np > 1 \), then there exists a unique
strong component of size \( \Theta(n) \), while all the other strong components
are of logarithmic size (see also~\cite[Chapter~13]{Frieze2015}).
In the current paper, we call a digraph \emph{elementary} if its strong
components are either single vertices or cycles.
Recently,
\L{}uczak and Seierstad~\cite{Luczak2009} obtained more precise results about
the width of the window in which the phase transition
occurs.
They established that the   scaling window is
given by $np = 1 + \mu n^{-1/3}$, where  $\mu$ is fixed. There, the
largest strongly connected components have sizes of order $n^{1/3}$.
Bounds on  the tail probabilities of the distribution of the size of
the largest component are also given by Coulson~\cite{Coulson2019}.

The structure of the strong components of
a random digraph
has been studied by
many  authors in the dense case, i.e., when $np\to \infty$ as
$n\to\infty$.   The largest strong components in a
random digraph with a given degree sequence were studied by Cooper and
Frieze~\cite{Cooper2004} and the strong connectivity of an
inhomogeneous random digraph was studied by Bloznelis, G\"otze and Jaworski
in~\cite{Bloznelis2012}. The hamiltonicity
was investigated by  Hefetz, Steger and Sudakov~\cite{Hefetz2016} and by
Ferber, Nenadov, Noever, Peter and \v Skori\'c~\cite{Ferber2017}, by
Cooper, Frieze and  Molloy~\cite{Cooper1994}  and by
Ferber, Kronenberg and Long~\cite{Ferber2015}.
Krivelevich, Lubetzky and  Sudakov~\cite{Krivelevich2013} also proved
the existence of  cycles of linear size \whp
when $np$ is large enough.

\subsection{The analytic method}

The \emph{analytic method} used in the current paper manipulates
generating functions and uses complex contour integration to extract their
coefficients. The \emph{symbolic method}, which is a part of this framework,
allows us to construct various families of graphs and
digraphs using a dictionary of admissible operations.
A spectacular example of
such an application is the enumeration of connected graphs \emph{à la} Wright.
Let the \emph{excess} be the difference
between the numbers of edges and vertices of a connected graph.
Any connected graph with given excess can be \emph{reduced}
to one of finitely many multigraphs of the same excess by
recursively removing vertices of degree $1$, and then
by recursively replacing each vertex of degree $2$ by an edge connecting its two
neighbours.
The degree of the resulting multigraph is at least $3$ for each vertex.
A similar reduction procedure can be applied to directed graphs as well. If the
excess of a strongly connected digraph is positive, then such a digraph is called
\emph{complex}.
The reduction procedure can be efficiently expressed on the level of generating
functions using the standard operations. It has been shown
in~\cite{Flajolet89,Janson93} that numerous statistical properties of a random
graph around the point of its phase transition can be obtained in a systematic
manner using the analytic method. More applications of the analytic method can
be found in~\cite{FSBook}.

The enumeration of \emph{directed acyclic graphs} (DAGs)
and strongly connected digraphs on the level of generating
functions or recurrences
has been successfully approached at least since 1969.
Apparently, it was Liskovets~\cite{liskovets1969,liskovets1970number}
who first deduced a recurrence for the number
of strongly connected digraphs and also introduced and
studied the concept of initially connected digraph,
a helpful tool for their enumeration.
Subsequently, Wright~\cite{wright1971number} derived
a simpler recurrence for strongly connected digraphs
and Liskovets~\cite{liskovets1973} extended his techniques
to the unlabelled case.
Stanley counted labelled DAGs in~\cite{stanley1973acyclic}, and Robinson,
in his paper~\cite{Robinson73},
developed a
general framework to enumerate digraphs with given strong components, and
introduced a \emph{special generating function}, which incorporates a special
convolution rule inherent to directed graphs.
In the unlabelled case, his approach is very much related
to species theory~\cite{bergeron1998combinatorial}
which systematises the usage of cycle index series.
Robinson~\cite{robinson1977strong} also provided
a simple combinatorial explanation for the generating function
of strongly connected digraphs
in terms of the cycle index function.

More explicitly, it was established by several authors
that if we let \( a_{n,m} \) denote the
number of acyclic digraphs with \( n \) vertices and \( m \) edges, and set
\[
    A(z, w) = \sum_{n,m \geq 0} a_{n,m}
    \frac{w^m}{(1 + w)^{\binom{n}{2}}} \dfrac{z^n}{n!} \, ,
\]
then this bivariate generating function can be written in the simple form
\[
    A(z, w) = \dfrac{1}{\widetilde \phi(z, w)} \, ,
    \quad \text{where} \quad
    \widetilde \phi(z, w) = \sum_{n = 0}^\infty
    (1 + w)^{-\binom{n}{2}}\dfrac{(-z)^n}{n!} \, .
\]
The function \( \widetilde \phi(\cdot, \cdot) \) is called the \emph{deformed
exponential}, because \( \widetilde \phi(z, 0) = e^{-z} \).

Using this generating function of directed acyclic graphs, Bender, Richmond,
Robinson and Wormald~\cite{Bender86,bender88} analysed the asymptotic number of
DAGs in the \emph{dense} case in the model \( \Di(n, p) \) when \( p \) is a
positive constant, for both labelled and unlabelled cases.
In the current paper, we continue asymptotic enumeration for the sparse
case \( p = c/n \)
and discover a phase transition for the asymptotic number of DAGs around
\( c = 1 \). We present a refined analysis for
the region \( c = 1 + \mu n^{-1/3} \), where \( \mu \) is a bounded real value,
or when \( \mu \to \pm \infty \) as \( n \to \infty \).

\subsection{Our results}

Our study is based on a general symbolic method which can deal with
a great variety of possible digraph families. This method allows us to construct
integral representations of the generating functions of interest.
Then we develop a version of the saddle point method which can be systematically
applied to the complex contour integrals appearing from the symbolic method.
Note that while similar techniques for undirected graphs have been
well developed, the integral representations for directed graphs have not yet
been pushed to their full potential.

When \( p = (1 + \mu n^{-1/3})/n \) and \( n \to \infty \), we establish the
exact limiting values of the probabilities that a random digraph is (i) acyclic,
(ii) elementary, or (iii) has one complex component of given excess, as a
function of \( \mu \), and provide more statistical insights into the structure
of a random digraph around, below and above its transition point \( p = 1/n \).
Our results can be extended to the case of several complex components with given
excesses as well in the whole range of sparse digraphs.
Note that while the technically easiest model is the model of random
\emph{multidigraphs}, in which multiple edges are allowed,
and where edge
multiplicities are sampled independently according to a Poisson distribution
with a fixed parameter $p$, we also show how to systematically approach the
family of simple digraphs, where multiple edges are forbidden, and where
2-cycles are either allowed or not.

As Janson, Knuth, \L{}uczak and Pittel put it in~\cite{Janson93},

\begin{quote}\emph{``%
    \ldots what seems to be the single most
    important related area ripe for investigation at the present time.
    \ldots Random directed multigraphs are of great importance in computer
    applications, and it is shocking that so little attention has been
    given to their study so far.
    \ldots A complete analysis of the random
    directed multigraph process is clearly called for, preferably based on
    generating functions so that extensive quantitative information can be
    derived without difficulty.%
''}\end{quote}

The form of our results suggests that the analytic method is the most
natural approach to the problem, since the numerical constants involve
certain complex integrals. In a similar context, we can quote a fragment of the
paper ``The first cycles in an evolving graph''
by Flajolet, Pittel and Knuth~\cite{Flajolet89}.

\begin{quote}\emph{``%
    \ldots an interesting question posed by Paul Erd\H{o}s and communicated
    by Edgar Palmer to the 1985 Seminar on Random Graphs in Posna\'{n}:
    `What is the expected length of the first cycle in an evolving graph?'
    The answer turns out to be rather surprising: the first cycle has length
    \( K n^{1/6} + \bigO(n^{3/22}) \) on the average, where
    \[
        K = \dfrac{1}{\sqrt{8\pi} i}
        \int_{-\infty}^{+\infty}
        \int_{1 - i \infty}^{1 + i \infty}
        e^{(\mu + 2s)(\mu-s)^2 / 6}
        \dfrac{\mathrm ds}{s}
        \mathrm d \mu
        \approx
        2.0337\, .
    \]
    The form of this result suggests that the expected behavior may be quite
    difficult to derive using techniques that do not use contour
    integration.%
''}\end{quote}

Consider a random digraph in the model \( \Di(n,p) \) whose vertex set
is $[n]$ and where each of the $n(n-1)$ possible edges occurs independently
with probability $p$, $0 \leq p\leq 1$.
The value \( p = 1/n \) is special because it represents the centre of the phase
transition window.  The following two questions and their answers are
paradigmatic for the approach considered in the current paper:
\begin{itemize}
    \item What is the probability that a digraph \( \Di(n,\frac{1}{n}) \)
        is acyclic?
    \item What is the probability that a digraph \( \Di(n,\frac{1}{n}) \)
        is elementary?
\end{itemize}
The asymptotic answers to the two questions are, respectively,
\[
    \mathbb P_{\mathsf{acyclic}}(n, \tfrac{1}{n})
    \sim
    (2n)^{-1/3} \dfrac{e}{2 \pi i}
    \int_{-i \infty}^{i \infty}
    \dfrac{1}{\ai(-2^{1/3} s)} \mathrm ds
    \approx
    1.328524 \cdot n^{-1/3}
\]
and
\[
    \mathbb P_{\mathsf{elementary}}(n, \tfrac{1}{n})
   \sim -\dfrac{2^{-2/3}}{2 \pi i}
    \int_{-i \infty}^{i \infty}
    \dfrac{1}{\ai'(-2^{1/3} s)} \mathrm ds
    \approx
    0.6997 \, ,
\]
where \( \ai(\cdot) \) is the Airy function. Its definition and properties are
discussed in detail in~\cref{section:airy}.
The integral form presented above may seem exotic, but turns out to be quite efficient for numerical computations.

\paragraph*{Semi-assisted computer verification.}
In order to ensure that this work is free from arithmetic mistakes and other types of mistakes to a maximum possible extent, we created a repository with detailed \texttt{ipython} notebooks
whose purpose is to provide a way to formally verify the correctness of the theorems in the form of mathematical assertions inside a computer algebra system.

This repository also includes toolboxes for
    (i) symbolic and arithmetic transformations, including the generalised Airy function and Airy integrals;
    (ii) numerical computations of the probabilities from multivariate generating functions;
    (iii) exhaustive enumeration of directed graphs according to the structure of their strongly connected components;
    (iv) detailed per-section analysis of most of the propositions and proofs from the current paper.
In addition, we provide a way to reproduce the numerical computations of the constants that we are using in our theorems, and show how to reproduce the numerical plots of the probabilities of various digraph families around the point of the phase transition.
The list of available notebooks is collected in~\cref{table:cas}. The repository contains six utility libraries written in \texttt{Python}.

\begin{table}[hbt]
    \centering
    \begin{tabular}{l p{.7\textwidth}}
    \textsf{Filename} & \textsf{Description} \\
        \hline
        \hline
        \texttt{\href{https://nbviewer.jupyter.org/urls/gitlab.com/sergey-dovgal/strong-components-aux/-/raw/master/ipynb/Section\%201.ipynb}{Section 1.ipynb}}
        & High-precision numerical evaluation of the constants
        arising from complex contour integrals,~\cref{section:introduction}.\\
        \hline
        \texttt{\href{https://nbviewer.jupyter.org/urls/gitlab.com/sergey-dovgal/strong-components-aux/-/raw/master/ipynb/Section\%203.ipynb}{Section 3.ipynb}}
        & Validation of the series expansions from~\cref{section:symbolic:method} using an exhaustive enumeration for small digraph sizes.\\
        \hline
        \texttt{\href{https://nbviewer.jupyter.org/urls/gitlab.com/sergey-dovgal/strong-components-aux/-/raw/master/ipynb/Section\%206.x.ipynb}{Section 6.x.ipynb}}
        & Plots and numerical computations for~\cref{section:asymptotics:multidigraphs}.
        \\
        \texttt{\href{https://nbviewer.jupyter.org/urls/gitlab.com/sergey-dovgal/strong-components-aux/-/raw/master/ipynb/Section\%207.x.ipynb}{Section 7.x.ipynb}}
        & Plots for~\cref{section:asymptotics:simple}.
        \\
        \hline
        \texttt{\href{https://nbviewer.jupyter.org/urls/gitlab.com/sergey-dovgal/strong-components-aux/-/raw/master/ipynb/Section\%208.1.ipynb}{Section 8.1.ipynb}}
        & Numerical evaluation of Airy integrals for~\cref{section:numerical:airy}.
        \\
        \texttt{\href{https://nbviewer.jupyter.org/urls/gitlab.com/sergey-dovgal/strong-components-aux/-/raw/master/ipynb/Section\%208.2.ipynb}{Section 8.2.ipynb}}
        & Evaluating empirical probabilities for various digraph families
        based on formal power series for~\cref{section:empirical:probabilities:within}.
        \\
        \texttt{\href{https://nbviewer.jupyter.org/urls/gitlab.com/sergey-dovgal/strong-components-aux/-/raw/master/ipynb/Section\%208.3.ipynb}{Section 8.3.ipynb}}
        & Comparing empirical and theoretical probabilities outside of the
        critical window for~\cref{section:empirical:probabilities:outside}.
        \\
        \hline
        \hline
        & \textit{Validating intermediate computations using Sympy:}\\
        \hline
        \texttt{\href{https://nbviewer.jupyter.org/urls/gitlab.com/sergey-dovgal/strong-components-aux/-/raw/master/ipynb/Section\%204.2.ipynb}{Section 4.2.ipynb}}
        & Intermediate computations for~\cref{section:zeros}.
        \\
        \texttt{\href{https://nbviewer.jupyter.org/urls/gitlab.com/sergey-dovgal/strong-components-aux/-/raw/master/ipynb/Section\%204.3.ipynb}{Section 4.3.ipynb}}
        & Intermediate computations for~\cref{sec:roots:deformed:exponential}.
        \\
        \hline
        \texttt{\href{https://nbviewer.jupyter.org/urls/gitlab.com/sergey-dovgal/strong-components-aux/-/raw/master/ipynb/Section\%205.1.ipynb}{Section 5.1.ipynb}}
        & Intermediate computations for~\cref{section:external:integration:multi}.
        \\
        \texttt{\href{https://nbviewer.jupyter.org/urls/gitlab.com/sergey-dovgal/strong-components-aux/-/raw/master/ipynb/Section\%205.2.ipynb}{Section 5.2.ipynb}}
        & Intermediate computations for~\cref{section:external:integration:simple}.
        \\
        \hline
        \texttt{\href{https://nbviewer.jupyter.org/urls/gitlab.com/sergey-dovgal/strong-components-aux/-/raw/master/ipynb/Section\%206.1.ipynb}{Section 6.1.ipynb}}
        & Intermediate computations for~\cref{section:dag:multi}.
        \\
        \texttt{\href{https://nbviewer.jupyter.org/urls/gitlab.com/sergey-dovgal/strong-components-aux/-/raw/master/ipynb/Section\%206.2.ipynb}{Section 6.2.ipynb}}
        & Intermediate computations for~\cref{section:elementary:multi}.
        \\
        \hline
        \texttt{\href{https://nbviewer.jupyter.org/urls/gitlab.com/sergey-dovgal/strong-components-aux/-/raw/master/ipynb/Section\%207.1.ipynb}{Section 7.1.ipynb}}
        & Intermediate computations for~\cref{section:dag:simple}.
        \\
        \texttt{\href{https://nbviewer.jupyter.org/urls/gitlab.com/sergey-dovgal/strong-components-aux/-/raw/master/ipynb/Section\%207.2.ipynb}{Section 7.2.ipynb}}
        & Intermediate computations for~\cref{section:elementary:simple}.
        \\
        \hline
        \hline
    \end{tabular}
    \caption{
    \label{table:cas}
        Available notebooks in the auxiliary repository.}
\end{table}

This auxiliary material is available under the following URL on \texttt{GitLab}
\begin{center}
    \url{https://gitlab.com/sergey-dovgal/strong-components-aux}
\end{center}
We encourage other researchers to reuse these tools if they wish to.

\paragraph*{Structure of the paper.}
Our main results are summarised in
\cref{sec:summary:exact:results}.
There, two propositions provide exact expressions
for the probability that a random (multi)digraph is acyclic,
elementary, or contain exactly one strong component
satisfying some constraint.
Pointers to the corresponding asymptotic results
are provided as well.
In order to express and derive those results,
we start with~\cref{section:models}, where we recall various models of
digraphs and introduce their generating functions.
In~\cref{section:symbolic:method} we describe the symbolic method for digraphs
that allows us to construct special generating functions of various digraph
families in the form of integral representations.
Next, in~\cref{section:airy}
we carry out an asymptotic analysis of these integrals using the saddle point method
and express them in terms of a generalised Airy function. There, we also point
out the asymptotic properties of the zeros of the deformed exponential
\(
    \sum_{k \geq 0} (1 + y)^{-\binom{n}{2}} \frac{(-x)^k}{k!}
\)
and the generalised deformed exponential.
\cref{section:external:integral} builds on the results from \cref{section:airy}
to express the asymptotics of the coefficients of generating functions
involving generalised deformed exponentials.
This analysis is carried using, again, the saddle point method.
In~\cref{section:asymptotics:multidigraphs} we apply these asymptotic
approximations to obtain the probabilities of various multidigraph families.
The case of multidigraphs is convenient to start with, and the
whole of~\cref{section:asymptotics:multidigraphs} is devoted to this case.
We extend our
tools to simple digraph models
in~\cref{section:asymptotics:simple}. Numerical constants and empirical
probability evaluations are provided in~\cref{section:numerical:results}.
Finally, in~\cref{section:discussion} we discuss some further directions and open
problems related to the current research and also discuss the appearance of Airy
integrals in a totally different context.

            \section{Models of random graphs and digraphs}
            \label{section:models}

Before analysing directed graphs and their asymptotics,
we define various models of graphs and directed graphs.
They model different situations,
and each has its own benefits and drawbacks:
efficiency of random generation,
simplicity of the analysis and asymptotic formulae.
The choice of the right model highlights
the essential tools for digraph analysis
and will guide the analysis for other models.

\cref{sec:models:models} presents the models
and random distributions on them.
Exponential generating functions are introduced in
\cref{sec:models:egf}.
In \cref{sec:multigraph:model},
we motivate the labelling of the edges in our multidigraph model
and compare it to a multigraph model from the literature.
\cref{sec:graphic} introduces another type of generating functions,
called \emph{graphic}, essential for directed graph enumeration.
We also link those graphic generating functions
to probability computations.
Finally, tools translating exponential generating functions
into graphic generating functions are proposed
in \cref{sec:exponential:graphic}.

        \subsection{Models}
        \label{sec:models:models}

We first define four graph-like models.
One model has undirected edges, the \emph{(simple) graphs},
and is widely accepted as canonical in the combinatorial community.
The other three models have directed edges:
\emph{(simple) digraphs}, \emph{strict digraphs} and \emph{multidigraphs}.
They differ on whether loops and multiple edges are allowed,
and whether edges are labelled.
The multidigraph model is not common in the literature.
We introduce it because calculations are easier on it,
and the results can then be extended to the other models.
We will see in \cref{sec:multigraph:model}
its relation with the multigraph models
used by \cite{Flajolet89}, \cite{Janson93}
and \cite{panafieu2019analytic}.

\begin{definition}
A (simple) graph, (simple) digraph, strict digraph
or multidigraph is a pair $(V, E)$
where $V$ represents a finite set of labelled vertices
$V = \{1, 2, \ldots, n\}$
and $E$ represents the set of edges,
which differ depending on the model.
\begin{itemize}
\item[$\GrER$.]
For (simple) graphs, the edges are unlabelled,
unoriented, and loops and multiple edges are forbidden.
Formally, $E$ is then a set of unordered pairs of distinct vertices
\( E \subset \{ \{ u, v \} \mid u, v \in V, \, u \neq v \} \).
\item[$\DiERBoth$.]
For (simple) digraphs, the edges are unlabelled,
oriented, loops and multiple edges are forbidden,
but cycles of length $2$ are allowed.
Formally, $E$ is then a set of ordered pairs of distinct vertices
\( E \subset \{ (u, v) \mid u, v \in V, \, u \neq v \} \).
\item[$\DiER$.]
Strict digraphs are simple digraphs that contain no
cycle of length $2$, \ie each pair of vertices is linked
by at most one edge.
\item[$\DiERMulti$.]
For multidigraphs, the edges are labelled and oriented,
loops and multiple edges are allowed.
Formally, $E$ is then a pair $(L, r)$
where $L = \{1, 2, \ldots, m\}$ for some integer $m$
and $r$ is a mapping from $L$ to
the ordered pairs of vertices $V \times V$.
\end{itemize}
\end{definition}

An equivalent representation of the multidigraph edges
is as an ordered sequence of ordered pairs of
(not necessarily distinct) vertices.
Allowing for loops and multiple edges
simplifies graph enumeration,
as illustrated by \cite{Flajolet89}, \cite{Janson93}
and the configuration model \cite{Bo80}.
This motivates the introduction of multigraphs
and multidigraphs to study graphs and digraphs.
However, multiple edges introduce symmetries in digraphs:
in most natural random generation models,
the probability of a multidigraph with unlabelled edges
varies depending on edge multiplicities.
\cite{Flajolet89} and \cite{Janson93}
handle those symmetries by counting multidigraphs
with a weight.
Following \cite{panafieu2019analytic},
we are able to establish the same distribution on the space
of multidigraphs by labelling the edges of our multidigraphs
instead (see \cref{sec:multigraph:model} for more details).

The numbers of graphs, digraphs, strict digraphs, multidigraphs
with $n$ vertices and $m$ edges are
\begin{equation} \label{eq:counting:graph:like}
    \GrER_{n,m} = \binom{\binom{n}{2}}{m}, \quad
    \DiERBoth_{n,m} = \binom{n (n - 1)}{m}, \quad
    \DiER_{n,m} = 2^m \binom{\binom{n}{2}}{m}, \quad
    \DiERMulti_{n,m} = n^{2m}.
\end{equation}
The last expression is obtained by writing
down the edges of the multidigraph
following the order of their labels
as a sequence of $2m$ vertices.


The uniform distribution on simple graphs
with $n$ vertices and $m$ edges
is denoted by $\GrER(n,m)$
and called the \emph{Erd\H{o}s--R\'enyi model},
following their work \cite{Erdos1959}, \cite{Erdos1960}.
It extends naturally to simple digraphs,
strict digraphs and multidigraphs.

\begin{definition}
The $\GrER(n,m)$, $\DiERBoth(n,m)$, $\DiER(n,m)$
and $\DiERMulti(n,m)$ models
correspond to the uniform distribution on, respectively,
simple graphs, simple digraphs, strict digraphs and multidigraphs
with $n$ vertices and $m$ edges.
\end{definition}

An elegant algorithm to generate a random multidigraph
from $\DiERMulti(n,m)$ is the \emph{multidigraph process},
inspired by the multigraph process from \cite{Flajolet89}.
It applies the following two steps:
\begin{enumerate}
\item
sample independently $2m$ vertices $(x_1, x_2, \ldots, x_{2m})$,
each drawn uniformly in $\{1, 2, \ldots, n\}$,
\item
for $i$ from $1$ to $m$, the edge of label $i$
is the oriented pair of vertices $(x_{2i-1}, x_{2i})$.
\end{enumerate}

Gilbert \cite{Gilbert1959} proposed a different model
for random simple graphs, denoted by $\GrGilb(n,p)$.
It also extends naturally to simple digraphs,
strict digraphs and multidigraphs.

\begin{definition}
Consider a nonnegative integer $n$ and a probability $p$.
\begin{itemize}
\item
$\GrGilb(n,p)$ generates a random simple graph
with $n$ vertices by adding each of the $\binom{n}{2}$
possible edges independently with probability $p$.
\item
$\DiGilbBoth(n,p)$ generates a random simple digraph
with $n$ vertices by adding each of the $n (n-1)$
possible oriented edges independently with probability $p$.
\item
$\DiGilb(n,p)$ generates a random strict digraph
with $n$ vertices by sampling a random simple graph
from $\GrGilb(n,p)$,
then choosing the orientation of each edge
independently uniformly at random.
\item
$\DiGilbMulti(n,p)$ adds, between each of the $n^2$ oriented pairs
of vertices, a multiple edge whose multiplicity follows
a Poisson law of parameter $p$.
Then the edges are labelled uniformly at random.
\end{itemize}
\end{definition}

An equivalent algorithm to generate a random multidigraph
in $\DiGilbMulti(n,p)$ is to first fix the number $m$ of edges
following a Poisson law of parameter $n^2 p$,
and second, draw the multidigraph from $\DiERMulti(n,m)$
(using the multidigraph process, for example).
Indeed, let $m_{x,y}$ denote the number of edges of $D$
from the vertex $x$ to the vertex $y$,
and let $m(D)$ denote the total number of edges of $D$.
Then, the probability to draw the multidigraph $D$
with the first method is
\[
    \mathbb P_{n,p}(D) =
    \bigg(
    \prod_{x=1}^n \prod_{y=1}^n
        e^{-p}
        \frac{p^{m_{x,y}}}{m_{x,y}!}
    \bigg)
    \times
    \frac{\prod_{x=1}^n \prod_{y=1}^n m_{x,y}!}{m(D)!}
    =
    e^{-n^2 p}
    \frac{p^{m(D)}}{m(D)!},
\]
while this probability for the second method is
\[
    e^{-n^2 p}
    \frac{(n^2 p)^{m(D)}}{m(D)!}
    \frac{1}{n^{2 m(D)}}
    =
    e^{-n^2 p}
    \frac{p^{m(D)}}{m(D)!}.    
\]

    \subsection{Exponential generating functions}
    \label{sec:models:egf}

Our analysis in this paper relies on analytic combinatorics.
Its principle is to represent sequences of interest,
typically counting digraphs in some family,
as the coefficients of formal power series,
called \emph{generating functions}.
We introduce the variable $z$ to mark the vertices,
and $w$ to mark the edges.

\begin{definition}
Let $n(G)$, $m(G)$ denote the number of vertices and edges
of a graph-like object $G$.
Consider a simple graph (\resp simple digraph, \resp strict digraph) family $\mathcal F$
containing $\mathcal F_{n,m}$ elements
with $n$ vertices and $m$ edges.
Its \emph{exponential generating function} is then defined as
\[
    F(z,w) =
    \sum_{G \in \mathcal F}
    w^{m(G)}
    \frac{z^{n(G)}}{n(G)!}
    =
    \sum_{n,m}
    \mathcal F_{n,m}
    w^m
    \frac{z^n}{n!}.
\]

If $\mathcal F$ is a family of multidigraphs
and $\mathcal F_{n,m}$ denotes the number of multidigraphs
it contains with $n$ vertices and $m$ edges,
we define the exponential generating function of $\mathcal F$ as
\[
    F(z,w) =
    \sum_{G \in \mathcal F}
    \frac{w^{m(G)}}{m(G)!}
    \frac{z^{n(G)}}{n(G)!}
    =
    \sum_{n,m}
    \mathcal F_{n,m}
    \frac{w^m}{m!}
    \frac{z^n}{n!}.
\]
\end{definition}

In the vocabulary of the \emph{symbolic method} (\cite{FSBook}),
we use for graphs, digraphs and strict digraphs
generating functions that are exponential with respect to $z$
(because vertices are labelled)
and ordinary with respect to $w$ (since edges are unlabelled).
For multidigraphs, since both vertices and edges are labelled,
the generating functions are exponential with respect to both
$z$ and $w$.

As follows from the definition, the exponential generating function
of all simple graphs is
\[
    G(z,w) =
    \sum_{n \geq 0}
    (1+w)^{\binom{n}{2}}
    \frac{z^n}{n!},
\]
and the exponential generating function of all multidigraphs is
\[
    \MD(z,w) =
    \sum_{n \geq 0}
    e^{n^2 w}
    \frac{z^n}{n!}.
\]
Those results can be obtained using the expression
\eqref{eq:counting:graph:like}
for the total number of graphs $\GrER_{n,m}$
and multidigraphs $\DiERMulti_{n,m}$
with $n$ vertices and $m$ edges.
Another way to derive the expression of $G(z,w)$
is to consider that to build a graph on $n$ vertices,
one has to choose for each of the $\binom{n}{2}$ edges
whether it belongs to the graph or not.
For the expression of $\MD(z,w)$,
each multidigraph on $n$ vertices
is a set of labelled edges,
each edge chosen among $n^2$ possibilities.

\begin{lemma} \label{th:egfs:connections}
Consider a multidigraph family $\mathcal F_{\mathrm{MD}}$
stable by edge relabelling.
Assume erasing the edge labels produces
a simple digraph family $\mathcal F_{\mathrm D}$
(\ie no multidigraph from $\mathcal F_{\mathrm{MD}}$
contains a loop or a multiple edge). Then we have
\[
    F_{\mathrm{D}}(z,w) = F_{\mathrm{MD}}(z,w).
\]
Assume furthermore that no multidigraph from $\mathcal F_{\mathrm{MD}}$
contains a cycle of length $2$,
and let $\mathcal F_{\mathrm{SD}}$ denote
the strict digraph family obtained by erasing the edge labels. Then we have
\[
    F_{\mathrm{SD}}(z,w) = F_{\mathrm{MD}}(z,w).
\]
Assume in addition that $\mathcal F_{\mathrm{MD}}$
is stable by change of orientation of the edges.
Then the exponential generating function
of the simple graph family $\mathcal F_{\mathrm{G}}$
obtained by erasing the edge labels and orientations
satisfies
\[
    F_{\mathrm{G}}(z,w) = F_{\mathrm{MD}}(z,w/2).
\]
\end{lemma}

\begin{proof}
Let
$\mathcal F_{\mathrm{D}, n, m}$,
$F_{\mathrm{SD}, n, m}$,
$\mathcal F_{\mathrm{G}, n, m}$
and $\mathcal F_{\mathrm{MD}, n, m}$
denote the number of elements of each family
with $n$ vertices and $m$ edges.
There are $m!$ ways to label the edges
of a simple or strict digraph,
and $2^m m!$ ways to label and orient
the edges of a simple graph to turn them into a multidigraph.
Thus, under the first assumption,
the first two assumptions,
or the three assumptions of the lemma,
we have respectively
\[
    \mathcal F_{\mathrm{D}, n, m} =
    \frac{\mathcal F_{\mathrm{MD}, n, m}}{m!},
    \qquad
    \mathcal F_{\mathrm{SD}, n, m} =
    \frac{\mathcal F_{\mathrm{MD}, n, m}}{m!},
    \qquad
    \mathcal F_{\mathrm{G}, n, m} =
    \frac{\mathcal F_{\mathrm{MD}, n, m}}{2^m m!}.
\]
Multiplying by $w^m z^n / n!$ and summing over $n$ and $m$
provides the claimed results on the exponential generating functions.
\end{proof}

    \subsection{Link between multidigraphs and multigraphs}
    \label{sec:multigraph:model}

The reader discovering our multidigraph definition
might wonder why we chose to label the edges.
To answer, this subsection presents two options
explored in the literature on multigraphs
(instead of multidigraphs):
either assign a weight to multigraphs, or label their edges.
This is also the opportunity to link
the models of (weighted) multigraphs and multidigraphs (with edge labels).

\paragraph{Compensation factors.}
\cite{Flajolet89} and \cite{Janson93}
introduced a model of \emph{multigraphs}
where edges are unlabelled and unoriented,
but loops and multiple edges are allowed.
In this model, each multigraph $G$ is counted with a weight,
called the \emph{compensation factor}, and defined as
\[
    \varkappa(G) =
    \prod_{x = 1}^{n(G)} 2^{-m_{x, x}}
    \prod_{y = x}^{n(G)} \dfrac{1}{m_{x, y}!}
\]
where $n(G)$ denotes the number of vertices of $G$,
and $m_{x,y}$ denotes the number of edges linking
the vertices $x$ and $y$.
The $\GrERMulti(n,m)$ model is then defined
as generating each multigraph
on $n$ vertices and $m$ edges
with a probability proportional to its compensation factor.

\paragraph{Multigraph process.}
To generate such a multigraph, the authors proposed
an algorithm called the \emph{multigraph process}:
\begin{enumerate}
\item
sample independently $2m$ vertices $(x_1, x_2, \ldots, x_{2m})$
each drawn uniformly in $\{1, 2, \ldots, n\}$,
\item
define the set of edges of the multigraph as
$\{\{x_{2i-1}, x_{2i}\} \mid 1 \leq i \leq m\}$.
\end{enumerate}
Let us recall the proof that the distribution on multigraphs
induced by $\GrERMulti(n,m)$ and the multigraph process
are the same.
The multigraph process corresponds to
sampling a multidigraph using the multidigraph process,
then erasing its edge labels and orientations
to turn it into a multigraph.
To conclude, observe that the number of multidigraphs
corresponding to a given multigraph $G$
by labelling and orienting its edges
is $2^{m(G)} m(G)! \varkappa(G)$.

\paragraph{Generating functions.}
\cite{Flajolet89} and \cite{Janson93} associate
to any multigraph family $\mathcal F$
the generating function
\[
    F(z,w) =
    \sum_{G \in \mathcal F}
    \varkappa(G)
    w^{m(G)}
    \frac{z^{n(G)}}{n(G)!}.
\]
Let $\mathcal F_{\mathrm{MD}}$ denote
the multidigraph family obtained from $\mathcal F$
by labelling and orienting the edges
in all possible ways.
Since each multigraph $G$ corresponds to exactly
$2^{m(G)} m(G)! \varkappa(G)$ multidigraphs,
the exponential generating functions are linked by the relation
\[
    F(z,w) = F_{\mathrm{MD}}(z, w/2).
\]
In particular, since the exponential generating function
of all multidigraphs is
\[
    \MD(z,w) =
    \sum_{n \geq 0}
    e^{n^2 w}
    \frac{z^n}{n!},
\]
we recover the exponential generating function of all multigraphs
\[
    \MG(z,w) =
    \sum_{n \geq 0}
    e^{n^2 w / 2}
    \frac{z^n}{n!}
\]
obtained by \cite{Flajolet89} and \cite{Janson93}.
Thus, any result from those papers obtained on $\MG(z,w)$
translates to $\MD(z,w)$ by the simple change of variables
$w \mapsto w/2$.

\paragraph{Conclusion.}

To summarise, two equivalent approaches exist
when working on multigraphs
\begin{enumerate}
\item
consider multigraphs weighted by their compensation factor,
\item
or consider multidigraphs, then erase the edge labels and orientations.
\end{enumerate}
As we saw, both approaches give the same distribution
on multigraphs with $n$ vertices and $m$ edges,
and related generating functions.
The generating functions with respect to the variable $w$
marking the edges are ordinary or exponential,
depending on which option is chosen.
Combinatorial operations are simpler to express
in the exponential setting compared to the ordinary setting,
so we prefer the second option.
To illustrate this point, consider the generating function of all multigraphs
\[
    \MG(z,w) =
    \sum_{n \geq 0}
    e^{n^2 w / 2}
    \frac{z^n}{n!}.
\]
Explaining the factor $e^{n^2 w / 2}$
as a set of edges, each marked by $w/2$ and
chosen among $n^2$ possibilities,
is simpler than summing compensation factors.

Our multidigraph model is identical to
the multigraph model used by \cite{panafieu2019analytic},
except this paper marks edges with $w/2$ instead of $w$,
so that its generating functions agree
with those of \cite{Flajolet89} and \cite{Janson93}.

    \subsection{Graphic generating functions}
    \label{sec:graphic}

Contrary to the case of graphs,
the exponential generating functions in their classical form
are not sufficient to capture the combinatorics of directed graphs.
So \cite{Robinson73}, \cite{liskovets1970number}, \cite{G96}
introduced a different class of generating functions,
referred to as \emph{graphic generating functions},
or \emph{special generating functions}.
The motivation for their definition is provided
in \cref{section:symbolic:method}.
To distinguish between the exponential
and graphic generating functions,
we put a hat on the second type: $\widehat F(z,w)$.
In the following, we sometimes replace
\emph{graphic} with \emph{simple graphic}
to emphasise the distinction with \emph{multi-graphic}
(defined below).

\begin{definition} \label{def:SGGF}
Let \( \mathcal F \) be a family of simple or strict digraphs
and denote by $\mathcal F_{n,m}$ its number of elements
containing $n$ vertices and $m$ edges.
Then its \emph{graphic generating function} \( \widehat F(z,w) \)
is defined as
\[
    \widehat F(z,w) =
    \sum_{D \in \mathcal F}
    \frac{w^{m(D)}}{(1 + w)^{\binom{n(D)}{2}}}
    \dfrac{z^{n(D)}}{n(D)!}
    =
    \sum_{n,m}
    \mathcal F_{n,m}
    \frac{w^m}{(1+w)^{\binom{n}{2}}}
    \frac{z^n}{n!},
\]
where \( n(D) \) and \( m(D) \) denote, respectively,
the numbers of vertices and edges of \( D \).

When \( \mathcal F \) is a multidigraph family,
its \emph{multi-graphic generating function}
\( \widehat F(z,w) \) is defined as
\[
    \widehat F(z,w) =
    \sum_{D \in \mathcal F}
    e^{-n(D)^2 w/2}
    \dfrac{w^{m(D)}}{m(D)!}
    \dfrac{z^{n(D)}}{n(D)!}
    =
    \sum_{n,m}
    \mathcal F_{n,m}
    e^{-n^2 w / 2}
    \frac{w^m}{m!}
    \frac{z^n}{n!}.
\]
\end{definition}

\begin{remark}
The terms \( (1 + w)^{\binom{n(D)}{2}} \)
for digraphs and \( e^{w n(D)^2/2} \) for multidigraphs
are reminiscent of the exponential generating functions
of all graphs and all multigraphs
\[
    G(z,w) =
    \sum_{n \geq 0}
    (1+w)^{\binom{n}{2}}
    \frac{z^n}{n!},
    \qquad
    \MG(z,w) =
    \sum_{n \geq 0}
    e^{n^2 w / 2}
    \frac{z^n}{n!}.
\]
From the analytic viewpoint, when the value of \( w \)
is small enough, the term \( (1 + w)^{\binom{n}{2}} \)
can be approximated by \( e^{w n^2/2} \).
If \( w \) is of order \( 1/n \), this approximation
gains an additional constant factor.
\end{remark}

\begin{lemma} \label{lemma:P:Dnp}
Let \( \mathcal F \) be a simple digraph family
whose graphic generating function is \( \widehat F(z, w) \).
Then, the probability that a random simple digraph
from \(\DiGilbBoth(n,p) \) belongs to $\mathcal F$ is
\[
    \mathbb P_{\mathcal F}(n,p)=
    (1 - p)^{\binom{n}{2}}
    n! [z^n] \widehat F \left(z, \dfrac{p}{1 - p}\right)\, .
\]
Let \( \mathcal F \) be a strict digraph family
whose graphic generating function is \( \widehat F(z, w) \).
Then, the probability that a random strict digraph
from \( \DiGilb(n,p) \) belongs to $\mathcal F$ is
\[
    \mathbb P_{\mathcal F}(n,p)  =
    (1 - p)^{\binom{n}{2}}
    n! [z^n] \widehat F \left(z, \dfrac{p}{1 - 2p}\right)\, .
\]
Let \( \mathcal F \) be a multidigraph family
whose multi-graphic generating function is \( \widehat F(z, w) \).
Then the probability that a random multidigraph
from \( \DiGilbMulti(n,p) \) belongs to $\mathcal F$ is
\[
    \mathbb P_{\mathcal F}(n,p) =
    e^{-n^2 p/2} n! [z^n] \widehat F(z, p)\, .
\]
\end{lemma}

\begin{proof}
Let $\mathcal F_n$ denote the subfamily of $\mathcal F$
containing the elements with $n$ vertices.
In the \( \DiGilbBoth(n,p) \) model, the probability that
a random digraph belongs to \( \mathcal F \)
is given by
\begin{align*}
    \mathbb P_{\mathcal F}(n,p) &=
        \sum_{D \in \mathcal F_n} p^{m(D)}
        (1 - p)^{ n(n-1) - m(D)}
        =
        (1 - p)^{ n(n-1) }
        \sum_{D \in \mathcal F_n}
        \left(\dfrac{p}{1 - p}\right)^{m(D)}
    \\ &=
        \left.
        (1 - p)^{n (n-1)}
            (1 + w)^{\binom{n}{2}}
        \right|_{w = \frac{p}{1 - p}}
        n! [z^n] \widehat F \left(z, \dfrac{p}{1 - p} \right)
    \\ &=
        (1 - p)^{\binom{n}{2}} n! [z^n] \widehat F \left(z, \dfrac{p}{1 - p}\right).
\end{align*}
Similarly, in the \( \DiGilb(n,p) \) model,
\begin{align*}
    \mathbb P_{\mathcal F}(n,p) &=
        \sum_{D \in \mathcal F_n} \dfrac{(2p)^{m(D)}}{2^{m(D)}}
        (1 - 2p)^{ {\binom{n}{2}} - m(D)}
        =
        (1 - 2p)^{ {\binom{n}{2}}}
        \sum_{D \in \mathcal F_n}
        \left(\dfrac{p}{1 - 2p}\right)^{m(D)}
    \\ &=
        \left.
            (1 - 2p)^{\binom{n}{2}}
            (1 + w)^{\binom{n}{2}}
        \right|_{w = \frac{p}{1 - 2p}}
        n! [z^n] \widehat F \left(z, \dfrac{p}{1 - 2p} \right)
    \\ &=
        (1 - p)^{\binom{n}{2}} n! [z^n] \widehat F \left(z, \dfrac{p}{1 - 2p}\right).
\end{align*}
Finally, for the \( \DiGilbMulti(n,p) \) model,
we have
\[
    \mathbb P_{\mathcal F}(n,p) =
    \sum_{D \in \mathcal F_n}
    \frac{(n^2 p)^{m(D)}}{m(D)!}
     e^{-n^2 p}
    \frac{1}{n^{2 m(D)}}
    =
    e^{- n^2 p /2}
    n! [z^n]
    \widehat F(z, p)
\]
because the probability for a given multidigraph
to be generated by $\DiGilbMulti(n,p)$
conditioned on having $m$ edges is $1/n^{2m}$.
\end{proof}

\begin{remark}
The probability that a digraph is acyclic
can be easily converted between the first and the second model.
Indeed, since acyclic digraphs do not contain loops and
2-cycles, the corresponding probabilities can be obtained
from one another using the conditional probability that a digraph obtained
in the model \( \DiGilbBoth(n,p) \) does not have 2-cycles.
\end{remark}

    \subsection{Linking exponential and graphic generating functions}
    \label{sec:exponential:graphic}

In order to convert an exponential generating function
into a graphic generating function,
we can use two strategies:
the \emph{exponential Hadamard product}
and an integral representation using the Fourier transform.
The exponential Hadamard product of two bivariate series
\[
    A(z) = \sum_{n \geq 0} a_n(w) \dfrac{z^n}{n!}
    \qquad \text{and} \qquad
    B(z) = \sum_{n \geq 0} b_n(w) \dfrac{z^n}{n!}
\]
with respect to the variable \( z \) is denoted by and defined as
\[
    A(z, w) \odot_z B(z, w)
    =
    \bigg(
        \sum_{n \geq 0} a_n(w) \dfrac{z^n}{n!}
    \bigg)
    \odot_z
    \bigg(
        \sum_{n \geq 0} b_n(w) \dfrac{z^n}{n!}
    \bigg)
    =
    \sum_{n \geq 0}
    a_n(w) b_n(w) \dfrac{z^n}{n!}.
\]
The following elementary graphic generating functions
are primary building blocks for more complex directed (multi-)graphs,
and also a tool for conversion between exponential
and graphic generating functions.

\begin{definition}
The multi-graphic generating function of sets
(labelled graphs that do not have any edges)
is denoted by and defined as
\[
    \gset(z, w) = \sum_{n \geq 0} e^{-n^2 w/2} \dfrac{z^n}{n!}.
\]
The graphic generating function of sets is given by
\[
    \gsetsimple(z, w)
    =
    \sum_{n \geq 0} (1 + w)^{-\binom{n}{2}}
    \dfrac{z^n}{n!}.
\]
\end{definition}

The generating function \( \gset(z, w) \) and its variations
are ubiquitous in combinatorics and many other areas.
It can be seen immediately, for example,
that the generating functions of multigraphs and simple graphs
are given by
\[
    \MG(z, w) = \gset(z,-w)
    \qquad \text{and} \qquad
    G(z, w) = \gsetsimple\left(z, \frac{-w}{1 + w} \right).
\]
In his paper~\cite{sokal2012leading},
Sokal calls a variant of \( \gset \)
the ``deformed exponential function''.
He also provides a lot of conjectures related
to both combinatorics and analysis
in his review~\cite{sokal2013some},
see in addition the references therein.
Another variant of this function
without the factorials in the denominator,
which can be obtained by applying the Laplace transform,
is called the \emph{partial theta function}
after Andrews and has also been extensively studied,
including in relation with Ramanujan's lost notebook.

The following conversion lemma allows us to switch
from exponential to graphic generating functions.
The integral transform is a classical Fourier-type integral, considered, for example, in~\cite{AiryAndConnected}.
In this paper, the authors consider the function
\( \gset(z, -w) \) as an exponential generating function
of all graphs (after a mapping \( w \to -w \)
in the case of multigraphs,
or \( (1 + w) \mapsto (1 + w)^{-1} \) in the case of simple graphs),
and use it to study the asymptotic properties of connected graphs.
They also obtain a saddle point representation in their paper.
Even earlier, an integral representation of a variant
of \( \gset(z, w)\) was obtained in~\cite{Mahler40} by Mahler in 1940.

\begin{lemma} \label{lemma:conversion:egf:ggf}
Assume $w$ stays in a compact interval of ${\mathds R}_{> 0}$.
Let \( A(z,w) = \sum_{n \geq 0} a_n(w) \frac{z^n}{n!} \)
be an exponential generating function
whose non-negative coefficients \( a_n(w) \)
grow more slowly than \( e^{n^2 w/2} \).
Then, for its corresponding multi-graphic generating function
\( \widehat A(z, w) \)
and graphic generating function
\( \widehat A^{(\text{simple})}(z, w) \) we have
\begin{enumerate}
    \item
    \(
        \widehat A(z, w)
        :=
        \displaystyle\sum\limits_{n \geq 0}
        a_n(w) e^{-n^2w/2} \dfrac{z^n}{n!}
        =
        A(z, w) \odot_z \gset(z, w)
        \,;
    \)
    \item
    \(
    \widehat A^{(\text{simple})}(z, w)
    :=
    \displaystyle\sum\limits_{n \geq 0}
        a_n(w) (1 + w)^{-{\binom{n}{2}}} \dfrac{z^n}{n!}
        =
        A(z, w) \odot_z \gsetsimple(z, w)
        \,;
    \)
    \item
    \(
        \widehat A(z, w) =
        \dfrac{1}{\sqrt{2 \pi w}}
        \displaystyle\int_{-\infty}^{+\infty}
        \exp \left(
            {-\dfrac{x^2}{2 w}}
        \right)
        A\left(z e^{-i x}, w\right)
        \mathrm dx\,;
    \)
    \item
    \(
        \widehat A^{(\text{simple})}(z, w)=
        \dfrac{1}{\sqrt{2 \pi \alpha}}
        \displaystyle\int_{-\infty}^{+\infty}
        \exp \left(
            {-\dfrac{x^2}{2 \alpha}}
        \right)
        A\left(
            z \beta e^{- i x}, w
        \right)
        \mathrm dx\,;
    \)
    \end{enumerate}
where $\alpha = \log(1 + w)$ and $\beta = \sqrt{1 + w}$.
\end{lemma}

\begin{proof}
    The first two identities are directly obtained
    from the definition of the exponential Had\-am\-ard product.
    To prove the third identity, we use the Fourier integral
    \[
        e^{-t^2/2} = \dfrac{1}{\sqrt{2 \pi}}
        \int_{-\infty}^{+\infty}
        e^{ixt} e^{-x^2/2} \mathrm dx
    \]
    with $t = - n \sqrt{w}$ and obtain
    \[
        \widehat A(z, w) =
        \sum_{n \geq 0}
        a_n(w)
        \dfrac{1}{\sqrt{2 \pi}}
        \int_{-\infty}^{+\infty}
        e^{- i n x \sqrt{w}} e^{-x^2/2} \mathrm dx
        \frac{z^n}{n!}.
    \]
    Since the coefficients \( a_n(w) \) are non-negative,
    the integration and summation of the convergent series
    can be interchanged:
    \[
        \widehat A(z, w) =
        \dfrac{1}{\sqrt{2 \pi}}
        \int_{-\infty}^{+\infty}
        \sum_{n \geq 0}
        a_n(w)
        e^{- i n x \sqrt{w}} 
        \frac{z^n}{n!}
        e^{-x^2/2} \mathrm dx
        =
        \dfrac{1}{\sqrt{2 \pi}}
        \int_{-\infty}^{+\infty}
        A(z e^{- i x \sqrt{w}}, w)
        e^{-x^2/2} \mathrm dx.
    \]
    Finally, we apply the change of variables $y = \sqrt{w} x$,
    just to slightly simplify the forthcoming asymptotic analysis.
    The lemma can be generalised to other cases where
    interchanging of the summation and integration is possible---for example,
    when the coefficients \( a_n(w) \) are
    not necessarily positive, but the series converges sufficiently fast.
    
    For the fourth identity,
    we use the fact that
    \[
        \gsetsimple(z, w) =
        \sum_{n \geq 0}
        (1 + w)^{-\binom{n}{2}}
        \frac{z^n}{n!}
        =
        \sum_{n \geq 0}
        e^{- \alpha n^2 / 2}
        \frac{(\beta z)^n}{n!}
    \]
    where $\alpha = \log(1+w)$ and $\beta = \sqrt{1+w}$.
    The rest of the proof is the same as for the third identity.
\end{proof}

It is worth noting  that the integral representation of
$\gsetsimple(-z,w)$ can be obtained by means of Mahler's
transformation. In fact, $\gsetsimple(-z,w)$
is exactly the same as $F(-z(1+y)^{-1})$
where $F(z)$ is defined by \cite[Equation~(6)]{Mahler40}.
Thus using the integral form of $F(z)$ in
\cite[Equation~(4)]{Mahler40} we obtain the following formula:
\begin{equation*}
\gsetsimple(-z,w)=\sqrt{\frac{\alpha}{2\pi}}\int_{-\infty}^{+\infty}
\exp\left(-\frac{\alpha}{2}x^2-z\beta e^{-i\alpha x} \right)\mathrm dx.
\end{equation*}

\begin{remark}\label{rem:simple-digraphs}
    The two functions \( \gset(z, w) \) and \( \gsetsimple(z, w) \)
    are connected by the simple transform
    \[
        \gsetsimple(z, w) = \gset \left(z \sqrt{1 + w}, \log(1 + w) \right),
    \]
    which will allow us to extend the asymptotic analysis of multidigraphs to the case
    of simple digraphs. The inverse transformation takes the form
    \[
        \gset(z, w) = \gsetsimple(z e^{-w/2}, e^w - 1).
    \]
    Using the expressions for
    the generating functions of multigraphs \( \MG(z, w) = \gset(z, -w) \)
    and simple graphs \( G(z, w) = \gsetsimple(z, \tfrac{-w}{1+w}) \),
    we can rewrite this identity as
    \[
        \MG(z, w) = G(z e^{w/2}, e^w - 1),
    \]
    which has the following interpretation.
    The change of variables \( w \mapsto e^w - 1 \) represents the insertion of
    a non-empty multiset of edges in place of each single edge,
    and \( z \mapsto z e^{w/2} \) allows the insertion of a set of loops,
    where the factor \( 1/2 \) accounts for the fact
    that the loops are not directed.
\end{remark}

In this paper, we often encounter exponential generating functions
of the form $(1 - w z)^r e^{-z} F(w z)$.
The following corollary provides integral representations
for their graphic counterparts.

\begin{corollary} \label{theorem:phi}
Consider an entire function $F(\cdot)$
and an integer $r$.
Let $\alpha = \log(1+w)$ and $\beta = \sqrt{1+w}$ and
let $\phi_r(z,w;F(\cdot))$ and $\widetilde \phi_r(z,w;F(\cdot))$
denote the functions
\begin{align}
\label{eq:phi:tilde:phi}
    \phi_r(z,w;F(\cdot)) &=
    \dfrac{1}{\sqrt{2 \pi w}}
    \int_{-\infty}^{+\infty}
    (1 - w z e^{-i x})^r
    F(w z e^{-i x})
    \exp \left(-\dfrac{x^2}{2 w} - ze^{-i x} \right)
    \mathrm dx,
    \\
    \widetilde \phi_r(z,w;F(\cdot)) &=
    \frac{1}{\sqrt{2 \pi \alpha}}
    \int_{-\infty}^{+\infty}
    (1 - w z \beta e^{-i x})^r
    F(w z \beta e^{-i x})
    \exp \left( - \frac{x^2}{2 \alpha} - z \beta e^{-i x} \right)
    \mathrm dx.
\end{align}
Then we have
\begin{align*}
    (1 - w z)^r e^{-z} F(w z) \odot_z \gset(z, w)
    &=
    \phi_r(z,w;F(\cdot)),
    \\
    (1 - w z)^r e^{-z} F(w z) \odot_z \gsetsimple(z, w)
    &=
    \widetilde \phi_r(z,w;F(\cdot)).
\end{align*}
\end{corollary}


\begin{remark}
If \( r \) is a negative integer, then, strictly speaking, the integral
representation of \( \phi_r(z, w; F) \)
is invalid for \(|zw|=1\). Even while we do not use
values beyond $e^{-1}$, one can show that for any fixed \(w>0\),
\( \phi_r(z, w; F(\cdot)) \) has an analytic continuation.
To see this, we consider, for example, an alternative representation
\[
\phi_r(z, w) =
\dfrac{1}{\sqrt{2 \pi w}}
    \int_{i-\infty}^{i+\infty}
    (1 - zw e^{-i x})^r
    F(w z e^{-i x})
    \exp \left(
        -\dfrac{x^2}{2w} - z e^{-i  x}
    \right)
    \mathrm dx,
\]
which is valid for \(|z|=w^{-1}\). Both representations define analytic
functions which agree for $z$ in an open neighbourhood of zero. Hence, we can
use the latter representation to define \( \phi_r(z, w; F) \) for \(|z|=w^{-1}\).
The second branch of the analytic continuation is given by an integral from \( -i-\infty \) to \( -i + \infty \).
\end{remark}


\section{Symbolic method for directed graphs}
\label{section:symbolic:method}

In the previous section we mentioned that the exponential generating
function in its classical form is not suitable for the structural analysis of
directed graphs. In his paper~\cite{Robinson73}, Robinson develops a general
theory for counting digraphs whose strong components belong to a given class,
and introduces the \emph{special generating function}. This method has been
rediscovered later again in~\cite{de2019symbolic}, see also references therein.
The current paper aims to systematise and solidify the methods presented there.
Afterwards, the method has been extended even further to the case of directed
hypergraphs~\cite{Ravelomanana2020} and enumeration of satisfiable 2-SAT
formulae~\cite{dovgal2021exact}.

In \cref{sec:basic:principles}, we recall the definition
of the \emph{arrow product},
a key combinatorial operation on digraph families,
and its connection with graphic generating functions.
As a first application, we express the graphic generating function
of digraphs where all strongly connected components
belong to a given family.
In \cref{section:symbolic:elementary},
we apply this result to count acyclic digraphs and elementary digraphs
(all strongly connected components are isolated vertices or cycles).
\cref{section:repr:alt} presents alternative formulae
based on graph decomposition.
Our asymptotic analysis does not rely on them,
but they are interesting on their own
and might lead to future developments.
We recall Wright's approach for digraph decomposition
in \cref{sec:complex:components},
adapting it to multidigraphs.
Those results are used in \cref{sec:one:complex:component}
to count digraphs with one complex component.
We also summarise our exact enumerative results
in \cref{sec:summary:exact:results}.

    \subsection{Basic principles}
    \label{sec:basic:principles}

Let us recall the \emph{arrow product} of digraph families defined
in~\cite{de2019symbolic}
(see also previous work from \cite{Robinson73} and \cite{G96}),
which is an analogue of the Cartesian product for graphs.

\begin{definition} \label{def:arrow_product}
    An \emph{arrow product} of two digraph families \( \mathcal A \) and \(
    \mathcal B \) is the family \( \mathcal C \) consisting of a pair \( (a, b) \) of digraphs
    \( a \in \mathcal A \) and \( b \in \mathcal B \) and an arbitrary set of
    additional edges oriented from the vertices of \( a \) to the vertices of \( b \).
\end{definition}

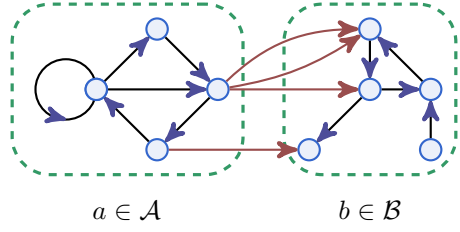
\begin{figure}[hbt!]
    \begin{center}
        \begin{tikzpicture}[>=stealth',thick, scale = 0.8]
\draw
node[arnBleuPetit](a) at ( 0, 0)  { }
node[arnBleuPetit](s) at ( 1,-1)  { }
node[arnBleuPetit](d) at ( 2, 0)  { }
node[arnBleuPetit](f) at ( 1, 1)  { }
node(loop) at ( -1, 01)  { }
;
\draw
node[arnBleuPetit](q)  at (4.5, 0)  { }
node[arnBleuPetit](w)  at (3.5,  -1)  { }
node[arnBleuPetit](e)  at (5.5,-1)  { }
node[arnBleuPetit](r)  at (5.5, 0)  { }
node[arnBleuPetit](t)  at (4.5, 1)  { }
;
\node at (.5, -2) {$a \in \mathcal A$};
\node at (4.5, -2) {$b \in \mathcal B$};
%
%
\node[rectangle,dashed,draw,fit=(a)(s)(d)(f)(loop), very thick,
      rounded corners=5mm,inner sep= 5pt, bgreen] {};
\node[rectangle,dashed,draw,fit=(q)(w)(e)(r)(t), very thick,
      rounded corners=5mm,inner sep= 5pt, bgreen] {};
%
%
\fromto{a}{d};
\fromto{a}{f};
\fromto{s}{a};
\fromto{d}{s};
\fromto{f}{d};
%
\fromto{q}{w};
\fromto{q}{r};
\fromto{e}{r};
\fromto{r}{t};
\fromto{t}{q};
%
\path (d) edge [blackred,thick, bend left=22,
decoration={markings,mark=at position 1 with
{\arrow[ultra thick,blackred, rotate=0]{>}}}, postaction={decorate}
] node {} (t);
\path (d) edge [blackred,thick, bend right=12,
decoration={markings,mark=at position 1 with
{\arrow[ultra thick,blackred, rotate=0]{>}}}, postaction={decorate}
] node {} (t);
\aprod{d}{q};
\aprod{s}{w};
%
%
\draw[
    decoration={markings, mark=at position 0.75 with
    {\arrow[ultra thick, blackblue, rotate=-15]{>}}},
    postaction={decorate}
    ]
    (-.5,0) circle (.5);
\draw node[arnBleuPetit](a0) at ( 0, 0)  { };
\end{tikzpicture}
    \end{center}
    \caption{The arrow product. As usual, labels are omitted.}
    \label{fig:arrow:product}
\end{figure}

For all the digraph models that we consider in~\cref{section:models}, the arrow
product corresponds to a product of (simple or multi-) graphic generating functions.
This is the motivation behind \cref{def:SGGF}.

\begin{lemma}
    \label{lemma:arrow:product}
    Let \( \widehat A(z, w) \) and \( \widehat B(z, w) \)
    be the (simple or multi-) graphic generating functions
    corresponding to two digraph families \( \mathcal A \) and \( \mathcal B \),
    and let \( \mathcal C \) be their arrow product with a corresponding
    graphic generating function \( \widehat C(z, w) \). Then we have
    \[
        \widehat C(z, w) = \widehat A(z, w) \widehat B(z, w).
    \]
\end{lemma}
\begin{proof}
    Let the two sequences associated with the graphic generating functions \( \widehat A(z, w)
    \) and \( \widehat B(z, w) \) be, respectively, \( (a_n(w))_{n \geq 0} \)
    and \( (b_n(w))_{n \geq 0} \). Depending on which kind of generating
    function (multigraphic or graphic) is considered, we have either
    \[
        \widehat A(z, w) =
        \sum_{n \geq 0} a_n(w) e^{-n^2w/2}\dfrac{z^n}{n!}
        \quad
        \text{and}
        \quad
        \widehat B(z, w) =
        \sum_{n \geq 0} b_n(w) e^{-n^2w/2}\dfrac{z^n}{n!},
    \]
    or
    \[
        \widehat A(z, w) = \sum_{n \geq 0} a_n(w)
        (1 + w)^{-{\binom{n}{2}}} \dfrac{z^n}{n!}
        \quad
        \text{and}
        \quad
        \widehat B(z, w) = \sum_{n \geq 0} b_n(w)
        (1 + w)^{-{\binom{n}{2}}} \dfrac{z^n}{n!}.
    \]
    Let \( (c_n(w))_{n \geq 0} \) be the sequence associated with the
    generating function \( \widehat C(z, w) \).

    \textit{The multigraphic case.} If the graphic generating function is of the
    multidigraph type, the resulting sequence \( c_n(w) \) is equal to
    \begin{align*}
        c_n(w) &= e^{n^2w/2} n! [z^n]
        \left(
            \sum_{k \geq 0} a_k(w) e^{-k^2w/2} \dfrac{z^k}{k!}
        \right)
        \left(
            \sum_{\ell \geq 0} b_\ell(w) e^{-\ell^2w/2}\dfrac{z^\ell}{\ell!}
        \right)
        \\ &=
        \sum_{k + \ell = n}
        {\binom{n}{k}}
        e^{k \ell w} a_k(w) b_\ell(w).
    \end{align*}
    This sum has the following interpretation:
    a new digraph is formed by choosing two digraphs \( a \in \mathcal A \) and
    \( b \in \mathcal B \) with the respective sizes \( k \) and \( \ell \);
    the binomial coefficient gives the number of ways to choose the labels of
    the vertices so that they form a partition of the set
    \( \{ 1, \ldots, n\} \) into two sets; for any pair of vertices
    \( u \) in \( a \) and \( v \) in \( b \), an arbitrary collection of edges is
    added, which has generating function \( e^w \) for each of the
    \( k \ell \) edges.

    \textit{The graphic case.} Similarly, the convolution product of the graphic
    type of the two sequences yields
    \begin{align*}
        c_n(w) &= (1 + w)^{\binom{n}{2}} n! [z^n]
        \bigg(
            \sum_{k \geq 0} a_k(w) (1 + w)^{-{\binom{k}{2}}} \dfrac{z^k}{k!}
        \bigg)
        \bigg(
            \sum_{\ell \geq 0} b_\ell(w) (1 + w)^{-{\binom{\ell}{2}}}
            \dfrac{z^\ell}{\ell!}
        \bigg)
        \\ &=
        \sum_{k + \ell = n}
        {\binom{n}{k}}
        (1 + w)^{k \ell} a_k(w) b_\ell(w),
    \end{align*}
    which has a similar interpretation. The only difference is that instead of
    a set of multiple edges between any pair of vertices \( u \in a \) and \(
    v \in b \), it is only allowed to either choose or not choose an edge, which
    results in a generating function \( (1 + w) \) instead of \( e^w \).
\end{proof}

The next theorem, first derived by \cite{RobinsonGeneral},
is the key to the enumeration of digraph families.
All the asymptotic results of this paper
rely on exact generating function expressions
obtained from this theorem.

\begin{theorem}
    \label{theorem:given:scc}
    Let \( \mathcal S \) be a family of strongly connected multidigraphs,
    and let \( S(z, w) \) denote its exponential generating
    function. Then, the multi-graphic generating function \( \widehat D_{\mathcal S}(z, w) \)
    of all the multidigraphs whose strongly connected components belong to
    \( \mathcal S \) is given by the formula
   \[
        \widehat D_{\mathcal S}(z, w)
        = \dfrac{1}{e^{-S(z, w)} \odot_z \gset(z, w)}.
    \]
    Moreover, the multi-graphic generating function of the same family
    where a variable \( u \) marks source-like components
    (components without incoming edges) is
    \[
        \widehat D_{\mathcal S}(z, w, u) =
        \dfrac
        {e^{(u-1)S(z, w)} \odot_z \gset(z, w)}
        {e^{-S(z, w)} \odot_z \gset(z, w)}.
    \]
    A similar statement holds for simple graphic generating functions,
    and is obtained by replacing
    \( \gset(z, w) \) with \( \gsetsimple(z, w) \).
\end{theorem}

\begin{proof}
    Consider the family of digraphs whose
    strongly connected components are from \( \mathcal S \), where each
    source-like component is either marked by a variable \( u \) or left
    unmarked. The multi-graphic generating function of this family is
    \( \widehat D_{\mathcal S}(z, w, u+1) \).
    The family can also be expressed as the arrow product of
    the marked source-like components, whose exponential generating
    function is $e^{u S(z,w)}$, with the rest of the digraph.
    This decomposition translates into a product of graphic generating
    functions
    \[
        \widehat D_{\mathcal S}(z, w, u+1)
        =
        \left(
            e^{u S(z, w)}\odot_z \gset(z, w)
        \right)
        \widehat D_{\mathcal S}(z, w).
    \]
    By definition, the only digraph without source-like component
    is the empty digraph, whose generating function is $1$,
    so $\widehat D_{\mathcal S}(z, w, 0) = 1$.
    By letting \( u = -1 \) we obtain the first part of the proposition.
    The proof is finished by replacing \( u \) with \( u-1 \).
\end{proof}

    \subsection{Directed acyclic graphs and elementary digraphs}
    \label{section:symbolic:elementary}

The simplest possible non-trivial application of the symbolic method for
directed graphs is the case of directed acyclic graphs, as the family of its
allowed strongly connected components only consists of the single-vertex graph.

\begin{lemma}
    \label{lemma:ggf:dags}
    Let $\phi(z,w) = \phi_0(z,w; 1)$ be defined as in \cref{theorem:phi}.
    The multi-graphic generating function \( \widehat D_{\mathrm{DAG}}(z, w) \)
    of acyclic multidigraphs is
    \[
        \widehat D_{\mathrm{DAG}}(z, w)
        =
        \dfrac{1}{\gset(-z, w)}
        =
        \dfrac{1}{\phi(z,w)}.
    \]
\end{lemma}

\begin{proof}
    The first expression of $\widehat D_{\mathrm{DAG}}(z, w)$
    is a direct application of \cref{theorem:given:scc},
    where the family of allowed strongly connected components
    includes only a single vertex.
    In this case, the exponential generating function of this family is \( S(z, w) = z \).
    The integral representation is obtained by \cref{theorem:phi}.
\end{proof}

The following lemma presents the corresponding result
for simple graphic generating functions,
and is obtained by replacing $\gset(-z, w)$ with $\gsetsimple(-z, w)$.

\begin{lemma}
  \label{lemma:sggf:dags}
  Let $\widetilde \phi(z,w) = \widetilde \phi_0(z,w;1)$
  be defined as in \cref{theorem:phi}.
  The graphic generating function
  \( \widehat D_{\mathrm{DAG}}^{(\text{simple})}(z, w) \)
  of simple  directed acyclic graphs is
    \[
        \widehat D_{\mathrm{DAG}}^{(\text{simple})}(z, w)
        =
        \dfrac{1}{\gsetsimple(-z, w)}
        =
        \dfrac{1}{
            \widetilde \phi(z,w)
        }\, .
      \]
\end{lemma}

The next family of directed graphs is the family of \emph{elementary digraphs},
whose strongly connected components can only be single vertices and cycles. Such
digraphs are central in the study of the phase transition, and it has been shown
in~\cite{Luczak2009} that below the critical point of the phase transition, a
directed graph is, with high probability, elementary.

\begin{lemma}
    \label{lemma:ggf:elementary}
    Let $\phi_1(z,w) = \phi_1(z,w;1)$
    be defined as in \cref{theorem:phi}.
    The multi-graphic generating function \( \widehat D_{\mathrm{elem}}(z, w) \)
    of elementary multidigraphs is
    \[
        \widehat D_{\mathrm{elem}}(z, w)
        =
        \dfrac{1}{\gset(-z, w) - zw e^{-w/2} \gset(-z e^{-w}, w)}
        =
        \dfrac{1}{\phi_1(z,w)}.
    \]
\end{lemma}

\begin{proof}
    Consider the family of multidigraphs containing a multidigraph
    consisting of a single vertex and multidigraphs that are cycles. The exponential
    generating function of this family is
    \[
        S_{\mathrm{elem}}(z, w) = z + \log \dfrac{1}{1 - zw}.
    \]
    Plugging the exponential generating function of this family
    into \cref{theorem:given:scc}, we obtain
    \[
        \widehat D_{\mathrm{elem}}(z, w) = \dfrac{1}
        {(1 - zw)e^{-z}\odot_z \gset(z, w)}
    \]
    and the corresponding integral representation by \cref{theorem:phi}.
    The exponential Hadamard product is linear, so the Hadamard product of
    \( (1 - zw)e^{-z} \odot_z \gset(z, w) \) in the denominator
    can be decomposed into a sum of Hadamard products
    \( e^{-z} \odot_z \gset(z, w) = \gset(-z, w) \)
    and
    \( -zw e^{-z} \odot_z \gset(z, w) \). The last summand can be transformed
    using the rule
    \[
        z^k e^{-z} \odot_z \gset(z, w) = (-z)^k
        \dfrac{\mathrm d^k}{\mathrm dz^k} \gset(-z, w).
    \]
    Additionally, the derivative of \( \gset(-z, w) \) satisfies
    \[
        \dfrac{\mathrm d}{\mathrm dz} \gset(-z, w)
        =
        - e^{w/2} \gset(-z e^{-w}, w),
    \]
    which yields another expression for \( \widehat D_{\mathrm{elem}}(z, w) \).
\end{proof}

In order to enumerate elementary digraphs among simple digraphs, we need to
exclude cycles of length 1 and 2, or, depending on the model, we might still
want to preserve the cycles of length 2.
We need to exclude multiple edges as well, and this is treated automatically by
choosing graphic instead of multi-graphic generating functions.
In total, the graphic generating functions
of elementary digraphs, corresponding to the two different models
\( \DiER(n, p) \) and \( \DiERBoth(n, p) \) are denoted by
\( \Delemsimple(z, w; 2) \) and \( \Delemsimple(z, w; 1) \), and defined in the
following lemma.

\begin{lemma}
    \label{lemma:ggf:elementary-simple}
    Let $\widetilde \phi_r(z,w; F(\cdot))$ be defined as in \cref{theorem:phi},
    and set
    \[
        C_k(z)= z + \frac{z^2}{2}+\cdots+\frac{z^k}{k}\, .
    \]
    For any $k \geq 1$, the graphic generating function \(\Delemsimple(z, w;k) \)
    of elementary digraphs where cycles of length $\leq k$
    and multiple edges are forbidden is
    \[
        \Delemsimple(z, w;k) =
        \frac{1}{(1 - w z) e^{-z} e^{C_k(w z)} \odot_z \gsetsimple(z,w)}
        =
        \frac{1}{\widetilde \phi_1(z,w; e^{C_k(\cdot)})}.
    \]
\end{lemma}

\begin{proof}
    The family $\mathcal S_k$ consisting of a single vertex
    and cycles of length at least \( k + 1 \)
    has exponential generating function
    \[
        S_k(z,w) = z + \log \frac{1}{1 - w z} - C_k(w z).
    \]
    A simple (resp.~strict) digraph is elementary if and only if
    all its strongly connected components belong to $\mathcal S_1$
    (resp.~$\mathcal S_2$).
    The first expression of $\Delemsimple(z, w;k)$
    is obtained by application of \cref{theorem:given:scc}.
    The second one comes from \cref{theorem:phi}.
\end{proof}

    \subsection{A graph decomposition approach to digraph enumeration}
    \label{section:repr:alt}

In this subsection, we provide alternative expressions
for the generating functions of acyclic and elementary multidigraphs.
They are based on the product decomposition of (multi-)graphs.
Our asymptotic analysis does not rely on them,
but they are interesting on their own.
A detailed version of their asymptotic analysis is given in \cite{Panafieu2020}.

The \emph{excess} of a (multi-)graph is defined as the
difference between the number of its edges and its vertices. For example, trees
have excess \( -1 \). A connected (multi-)graph with excess \( 0 \) is called a
\emph{unicycle}, or a \emph{unicyclic graph}, and a (multi-)graph whose connected
components have positive excess is called a \emph{complex (multi-)graph}.
Correspondingly, the \emph{complex components} of a graph are the
connected components whose excess is strictly positive.

Let \( U(z) \), \( V(z) \) and \( \eComplex_k(z) \)
denote the exponential generating functions of, respectively,
unrooted trees, unicycles and complex multigraphs of excess \( k \).
We follow the convention that the only complex graph of excess \( 0 \) is empty,
\ie, \( \eComplex_0(z) = 1 \).

The generating function of rooted trees \( T(z) \) satisfies the functional
equation \( T(z) = z e^{T(z)} \). The generating functions of unrooted
trees and unicycles can be expressed in terms of \( T(z) \):
\[
    U(z) = T(z) - \dfrac{T(z)^2}{2},
    \quad \text{and} \quad
    V(z) = \dfrac{1}{2} \log \dfrac{1}{1 - T(z)},
\]
see \eg~\cite{Janson93}. It is also known that a complex
graph of excess \( r \) is reducible to a \emph{kernel} (multigraph of minimal
degree at least $3$) of the same excess, by recursively removing vertices of degree
$0$ and $1$ and fusing edges sharing a degree $2$ vertex.
The total weight of \emph{cubic} kernels (all degrees equal to $3$) of excess $r$
is given by~\eqref{eq:complex:cf}.
They are central in the study of large critical graphs,
because non-cubic kernels do not typically occur.

\begin{lemma}[{\cite[Section 5]{Janson93}} and \cite{W772}]
\label{th:complex}
For each \( r \geq 0 \) there exists a polynomial \( P_r(T) \) such that
\begin{equation}
    \label{eq:complex:cf}
    \eComplex_r(z) =
    e_{r}\cfrac{T(z)^{2r}}{(1-T(z))^{3r}}
    +
    \dfrac{P_r(T(z))}{(1 - T(z))^{3r - 1}},
    \quad
    \text{where}
    \quad
    e_r = \frac{(6r)!}{2^{5r} 3^{2r} (2r)! (3r)!}.
\end{equation}
\end{lemma}

Having the multigraph decomposition into its connected components at hand, we
can use it together with the symbolic method for directed graphs to obtain
several different representations of the multi-graphic generating function of
directed acyclic graphs.

\begin{lemma}
    The multi-graphic generating function of acyclic multidigraphs
    is equal to
    \[
        \widehat D_{\mathrm{DAG}}(z, w) =
        \dfrac{e^{U(zw)/w - V(zw)}}
        { \sum_{k \geq 0} \eComplex_k(zw) (-w)^k}.
    \]
\end{lemma}

\begin{proof}
    From \cref{lemma:ggf:dags}, we have
    \[
        \widehat D_{\mathrm{DAG}}(z, w) =
        \frac{1}{\gset(-z,w)}.
    \]
    The generating function of all multigraphs and
    the graphic generating function of sets
    are linked by the identity \( \gset(-z, w) = \MG(-z,-w) \).
    Next, we represent the exponential generating function of multigraphs in a product form.
    Note that each of the functions \( U(z) \), \( V(z) \) and
    \( \eComplex(z) \) is a function of one variable, and in order to mark both
    vertices and edges, we need to perform the substitution \( z \mapsto zw \) and multiply by
    \( w \) to the power of the excess of the component.
    Since every multigraph can be decomposed into a set of trees, unicycles and
    its complex components, we obtain
    \begin{equation}
        \label{eq:multigraph:decomposition}
        \gset(-z, w) =
        \MG(-z, -w) = e^{-U(zw)/w + V(zw)} \sum_{k \geq 0}
        \eComplex_k(zw) (-w)^k.
    \end{equation}
    The statement of the lemma is obtained by taking the reciprocal.
\end{proof}

\begin{lemma}
    The multi-graphic generating function of elementary multidigraphs is
    \[
        \widehat D_{\mathrm{elem}}(z, w) =
        \dfrac{1}{1 - T(zw)}
        \cdot
        \dfrac{
            \exp \left(
                \dfrac{U(zw)}{w} - V(zw)
            \right)
        }{
            \left(
                1 + \dfrac{w T(zw)}{2(1 - T(zw))^3}
            \right)
            E(z, w)
            +
            zw\, \partial_z
            E(z, w)
        },
    \]
    where \( E(z, w) = \sum_{k \geq 0} \eComplex_k(zw)(-w)^k \).
\end{lemma}

\begin{proof}
    From \cref{lemma:ggf:elementary}, we have
    \[
        \widehat D_{\mathrm{elem}}(z, w) =
        \dfrac{1}{\gset(-z, w) - zw e^{-w/2} \gset(-z e^{-w}, w)}.
    \]
    We use the exponential generating function of multigraphs again,
    and replace \( \gset(-z, w) \) with \( \MG(-z, -w) \). This yields
    \[
        \widehat D_{\mathrm{elem}}(z, w) = \dfrac{1}{
            \MG(-z, -w) + zw \dfrac{\mathrm d}{\mathrm dz} \MG(-z, -w)
        }.
    \]
    Plugging the identity~\eqref{eq:multigraph:decomposition} in, we obtain
    \[
        \widehat D_{\mathrm{elem}}(z, w) =
        \dfrac{
            \exp \left(
                \dfrac{U(zw)}{w} - V(zw)
            \right)
        }{
            \left(
                zw\dfrac{\mathrm d}{\mathrm dz}
                \big(
                {-\frac{U(zw)}{w}} + V(zw)
                \big)
                +1
            \right)
            E(z, w)
            +
            zw\, \partial_z
            E(z, w)
        }.
    \]
    After simplifying the derivatives using the rules
    \[
        z \dfrac{\mathrm d}{\mathrm dz} U(zw) = T(zw)
        \quad \text{and} \quad
        z \dfrac{\mathrm d}{\mathrm dz} V(zw) = \dfrac{T(zw)}{2(1 - T(zw))^2},
    \]
    we finally obtain the product form.
\end{proof}

\begin{remark}
    \cref{th:complex} allows us to approximate \( E(z, w) \) by a series
    in \( w(1 - T(zw))^{-3} \), with coefficient \( (-1)^r e_r \) at the
    \( r \)-th power.
    Consequently, it lets one carry out an asymptotic analysis of the product
    form expression and obtain an asymptotic answer in a different form.
\end{remark}

    \subsection{Enumeration of strongly connected components}
    \label{sec:complex:components}

In this subsection, we derive the exponential generating function
of strongly connected digraphs, strict digraphs and multidigraphs.
Those results were first obtained for digraphs by
\cite{wright1971number}.

The difference between the numbers of edges and vertices
of a digraph is called its \emph{excess}.
For example, the excess of a cycle is equal to \( 0 \),
and the excess of an isolated vertex is equal to \( -1 \).
These are, respectively, the only possible strongly connected digraphs
with excesses \( 0 \) and \( -1 \).
We say that a strongly connected component of a digraph
is \emph{complex} if it has positive excess.
The \emph{deficiency} of a digraph with $n$ vertices
and excess $r$ is defined as $d = 2r - n$.

Similarly in spirit to~\cref{th:complex},
we derive an asymptotic approximation
for strongly connected multidigraphs with given excess
using the exponential generating function
of cubic strongly connected multidigraphs,
where \emph{cubic} means that the sum of in- and out-degrees
of each vertex is equal to $3$.
We say that a vertex has type \( (d_1, d_2) \)
if it has \( d_1 \) ingoing and \( d_2 \) outgoing edges.
If a cubic multidigraph is strongly connected,
then its vertices can only have type \( (1, 2) \) or \( (2, 1) \).

\begin{lemma} \label{th:complex:scc}
Consider a strongly connected kernel \( K \)
of excess \( r \) and deficiency \( d \),
and let \( \mathcal S \) denote the family
of strongly connected multidigraphs whose kernel is $K$.
Then the exponential generating function of
\( \mathcal S \) is
\[
    S_{r,d}(z, w) =
    \frac{1}{(3r-d)!}
    \dfrac{w^r}{(1 - w z)^{3r - d}}
    \dfrac{(w z)^{2r-d}}{(2r - d)!}.
\]
Let \( \eStrong_r(z, w) \),
\( \eStrongSimple_r(z, w) \), and
\( \eStrongStrict_r(z, w) \)
denote the exponential generating functions
of strongly connected multidigraphs, simple digraphs, and
strict digraphs of excess \( r \), respectively,
with the convention that they are equal to $1$ for $r=0$.
Then, for each \( r \geq 0 \),
there exist polynomials
\( A_r(z) \), \( A^{(\text{simple})}_r(z) \),
\( A^{(\text{strict})}_r(z) \)
of respective degrees at most $2r$, $5r$ and $8r$,
such that
\begin{align*}
    \eStrong_r(z, w) &=
    w^r \dfrac{A_r(w z)}{(1 - w z)^{3r}},
    \\
    \eStrongSimple_r(z, w) &=
    w^r \dfrac{A^{(\text{simple})}_r(w z)}{(1 - w z)^{3r}},
    \\
    \eStrongStrict_r(z, w) &=
    w^r \dfrac{A^{(\text{strict})}_r(w z)}{(1 - w z)^{3r}}.
\end{align*}
Furthermore, letting $s_r$ denote the number
of cubic strongly connected multidigraphs of excess $r$
divided by $(2r)! (3r)!$, we have
\[
    A_r(1) =
    A^{(\text{simple})}_r(1) =
    A^{(\text{strict})}_r(1) =
    s_r.
\]
\end{lemma}

\begin{proof}
The proof is similar to the one
of \cite{W772} and \cite[Section 9]{Janson93}.
The in- and out-degree of each vertex
in a strongly connected multidigraph is at least one.
By fusing the vertices of type \( (1, 1) \),
the multidigraph is reduced to its \emph{kernel},
whose vertex degrees are at least \( (1, 2) \) or \( (2, 1) \).
Observe that the multidigraph and its kernel
share the same excess.
Consider a strongly connected kernel $K$
of excess $r$ and deficiency $d$. It contains $2r-d$ vertices and $3r-d$ edges,
so its exponential generating function is
\[
    \frac{w^{3r-d}}{(3r-d)!}
    \frac{z^{2r-d}}{(2r-d)!}.
\]
The multidigraphs whose kernel is $K$
are obtained by replacing each edge
with a sequence (edge, vertex, edge, $\ldots$ , vertex, edge).
Thus their exponential generating function is
\[
    S_{r,d}(z,w) =
    \frac{\left(\frac{w}{1 - w z} \right)^{3r-d}}{(3r-d)!}
    \frac{z^{2r-d}}{(2r-d)!}.
\]

\noindent \textbf{Strong multidigraphs.}
For any given excess $r$, there is only
a finite number of kernel multidigraphs of excess $r$.
Indeed, the sum of the degrees (out- and in-)
is twice the number of vertices, so $3n \leq 2m$
and we deduce $n \leq 2r$ and $m \leq 3r$.
Let $K_{r,d}$ denote the number of strongly connected kernels
of excess $r$ and deficiency $d$.
They contain $2r-d$ vertices and $3r-d$ edges,
and the exponential generating function
of strongly connected kernels of excess $r$ is
\[
    K_r(z,w) =
    \sum_{d=0}^{2r}
    K_{r,d}
    \frac{w^{3r-d}}{(3r-d)!}
    \frac{z^{2r-d}}{(2r-d)!}.
\]
Thus, by the previous result,
the exponential generating function
of strongly connected multidigraphs of excess $r$ is
\[
    \eStrong_r(z,w) =
    w^r
    \sum_{d=0}^{2r}
    \frac{K_{r,d}}{(3r-d)! (2r-d)!}
    \frac{(w z)^{2r - d}}{(1 - w z)^{3r - d}},
\]
so there exists a polynomial $A_r(z)$
of degree at most $2r$
with $A_r(1) = K_{r,0} / ((2r)! (3r)!) = s_r $ such that
\[
    \eStrong_r(z, w) =
    w^r \dfrac{A_r(w z)}{(1 - zw)^{3r}}.
\]

\noindent \textbf{Strong simple digraphs.}
For a multidigraph $D$,
let $\ell_0(D)$ denote the number of loops,
and for each $j \geq 1$,
let $\ell_j(D)$ denote the number of oriented pairs
of distinct vertices $(x,y)$ linked by exactly $j$ edges
oriented from $x$ to $y$.
By definition, $\sum_{j \geq 0} \ell_j(D)$
is equal to the total number of edges of $D$.
Let also $K_{r,d,\ell_0, \ldots, \ell_{3r}}$
denote the number of strongly connected kernels $D$
of excess $r$, deficiency $d$,
and satisfying $\ell_j(D) = \ell_j$ for all $j \geq 1$
(recall that $3r$ is a bound on the total number of edges of $D$).
The exponential generating function
of strongly connected kernels of excess $r$
is refined into the multivariate polynomial
\[
    K_r(z,w,u_0, \ldots, u_{3r}) =
    \sum_{d, \ell_0, \ldots, \ell_{3r}}
    K_{r,d,\ell_0, \ldots, \ell_{3r}}
    u_0^{\ell_0}
    \cdots
    u_{3r}^{\ell_{3r}}
    \frac{w^{3r-d}}{(3r-d)!}
    \frac{z^{2r-d}}{(2r-d)!}.
\]
By \cref{th:egfs:connections},
the exponential generating function
of strongly connected simple digraphs is equal to
the exponential generating function
of strongly connected multidigraphs
that contain neither loops nor multiple edges.
To obtain any simple digraph reducing to a given kernel,
observe that
\begin{itemize}
\item
each loop of the kernel is replaced by
a path of length at least $2$,
\item
each multiple edge of multiplicity $j$
is replaced by $j$ paths,
all of length at least $2$ except possibly one of them.
\end{itemize}
It follows that the exponential generating function
of strongly connected simple digraphs
is obtained from $K_r(z,w,u_0, \ldots, u_{3r})$
by replacing
$u_0$ with $\frac{w z}{1 - w z}$,
and
$u_j$ with
$\left( \frac{w z}{1 - w z}\right)^j + j \left( \frac{w z}{1 - w z} \right)^{j-1}$
for each $j \geq 1$.
We indeed obtain a series
\[
    \eStrongSimple_r(z, w) =
    w^r \dfrac{A^{(\text{simple})}_r(w z)}{(1 - w z)^{3r}},
\]
where $A^{(\text{simple})}_r(z)$ is a polynomial
whose value at $1$ is $s_r = K_{r,0} / ((2r)! (3r)!)$.
Its degree is at most $5r$ because
each edge of the kernel increases
the degree of $A^{(\text{simple})}_r(z)$
by at most $1$,
and there are at most $3r$ such edges.

\noindent \textbf{Strong strict digraphs.}
We follow the same principle as in the previous case.
By \cref{th:egfs:connections},
the exponential generating function
of strongly connected strict digraphs
is equal to the exponential generating function
of strongly connected multidigraphs
that contain neither loops,
nor unordered pairs of vertices $\{x,y\}$
linked by at least $2$ edges (in any orientation).
Any strongly connected strict digraph
reducing to a given kernel is obtained by replacing
\begin{itemize}
\item
each loop with a path of length at least $2$,
\item
each unordered pair of distinct vertices $\{x,y\}$
linked in total by $j \geq 1$ edges
(from $x$ to $y$ or $y$ to $x$)
with $j$ paths, all of length at least $2$
except possibly one of them.
\end{itemize}
Let $\ell_0(D)$ denote the number of loops
of the multidigraph $D$,
and for each $j \geq 1$, let $\ell_j(D)$ denote
the number of unordered pairs of distinct vertices $\{x,y\}$
linked by exactly $j$ edges (from $x$ to $y$ or $y$ to $x$).
Let $K_{r,d,\ell_0, \ldots, \ell_{3r-d}}$
denote the number of strongly connected kernels $D$
of excess $r$, deficiency $d$,
and satisfying $\ell_j(D) = \ell_j$
for all $j \geq 0$.
Let us refine again the exponential generating function
of strongly connected kernels into the polynomial
\[
    K_r(z,w,u_0, \ldots, u_{3r}) =
    \sum_{d, \ell_0, \ldots, \ell_{3r}}
    K_{r,d,\ell_0, \ldots, \ell_{3r}}
    u_0^{\ell_0}
    \cdots
    u_{3r}^{\ell_{3r}}
    \frac{w^{3r-d}}{(3r-d)!}
    \frac{z^{2r-d}}{(2r-d)!}.
\]
The exponential generating function
of strongly connected strict digraphs
is then obtained from this polynomial by replacing
$u_0$ with $\frac{(w z)^2}{1 - w z}$,
and
$u_j$ with $\left( \frac{w z}{1 - w z} \right)^j + j \left( \frac{w z}{1 - w z} \right)^{j-1}$
for each $j \geq 1$.
We obtain
\[
    \eStrongStrict_r(z, w) =
    w^r \dfrac{A^{(\text{strict})}_r(w z)}{(1 - w z)^{3r}}
\]
for some polynomial $A^{(\text{strict})}_r(z)$
satisfying $A^{(\text{strict})}_r(1) = s_r$.
Its degree is at most $8r$ because
each edge of the kernel increases
the degree of $A^{(\text{simple})}_r(z)$
by at most $2$,
and there are at most $3r$ such edges.
\end{proof}

\begin{remark} \label{rem:strong:complex:simple}
\cite{liskovets1970number,robinson1977strong,de2019symbolic}
provide a direct expression for
the exponential generating function
of strongly connected simple digraphs,
which extends naturally to multidigraphs and strict digraphs, namely
\begin{align*}
    \eStrong(z,w) &=
    - \log \left(
        \MG(z,w) \odot_z \dfrac{1}{\MG(z,w)}
    \right),
    \\
    \eStrongSimple(z,w) &=
    - \log \left(
        G(z, w) \odot_z \dfrac{1}{G(z, w)}
    \right),
    \\
    \eStrongStrict(z,w) &=
    - \log \left(
        G(z, w) \odot_z
            \bigg( \sum_{n = 0}^\infty
            \dfrac{(1+2w)^{\binom{n}{2}}}{(1+w)^{\binom{n}{2}}}
            \dfrac{z^n}{n!}
            \bigg)^{-1}
    \right) \, .
\end{align*}
Those expressions allow us to compute the polynomials
mentioned in the previous lemma;
\[
    A_r(w z) =
    \frac{(1 - w z)^{3 r}}{w^r}
    \eStrong_r(z,w) \, ,
\]
since the excess $r$ denotes the difference between the number of edges and
vertices.
Rewriting $\eStrong_r(z,w)$ as $[y^r] \eStrong(z/y,y w)$
and replacing the redundant variable $w$ with $1$,
we obtain
\begin{equation} \label{eq:A:r:z}
    A_r(z) =
    (1 - z)^{3r}
    [y^r] \eStrong(z/y,y),
\end{equation}
and similarly for $A^{(\text{simple})}_r(z)$
and $A^{(\text{strict})}_r(z)$.
Observe that since a bound on the degree of $A_r(z)$ is known,
we need only expand the right-hand side series
up to degree $2r$ (\resp $5r$ and $8r$ in the simple and strict cases)
to obtain the desired polynomial.

While the coefficient \( e_r \) in~\cref{th:complex}
carries a similar meaning to \( s_r \),
namely the sum of compensation factors of cubic multigraphs
divided by \( (2r)! \), it is much easier
to obtain the exact expression for \( e_r \)
than for \( s_r \).
In fact, having a simple recurrence for \( s_r \)
is mentioned as an open problem at the end
of the paper~\cite{Janson93}.
A table of the first several values of \( s_r \)
is provided in that paper as well:
\[
    (s_r)_{r \geq 1} = \left(
        \frac{1}{2},
        \frac{17}{8},
        \frac{275}{12},
        \frac{26141}{64},
        \frac{1630711}{160},
        \ldots
    \right).
\]
For example, there are exactly six
strongly connected cubic multidigraph of excess $1$,
so \( s_1 = \tfrac{6}{2!3!} = \tfrac12 \)
(\cf~\cref{fig:bicycles}, the multidigraph on the left).
The coefficients $s_r$ can be obtained from \cite[(4.2)]{W772}
using the expression for $\eStrong_r(z,w)$.
Let us present a direct approach.
By definition, we have $s_r = A_r(1)$,
and $A_r(z)$ is a polynomial of degree at most $2r$.
Since the value at $1$ of a polynomial $P(z)$
of degree at most $d$ is equal to $[z^d] (1-z)^{-1} P(z)$,
we deduce, from \eqref{eq:A:r:z},
\[
    s_r =
    [z^{2r} y^r]
    (1 - z)^{3r - 1}
    \eStrong(z/y, y).
\]
After the change of variables $(z,w) = (z/y,y)$
and replacing $\eStrong(z,w)$ with its expression,
we obtain
\[
    s_r =
    - [z^{2r} w^{3r}]
    (1 - w z)^{3r - 1}
    \log \left(
        \MG(z,w) \odot_z \dfrac{1}{\MG(z,w)}
    \right).
\]
\end{remark}

\begin{figure}[hbt!]
    \RawFloats
    \begin{minipage}[t]{0.7\textwidth}
    \begin{center}
        \begin{tikzpicture}[>=stealth',thick, scale = 0.8]
\draw
node[arnBleuPetit](a) at (0, 0)  { }
node[arnBleuPetit](b) at (2, 0)  { }
node[arnBleuPetit](c) at (5, 0)  { }
;
%
\path (a) edge [black,thick, bend left=80,
decoration={markings,mark=at position .99 with
{\arrow[ultra thick,blackblue, rotate=16]{>}}}, postaction={decorate}
] node {} (b);
\path (b) edge [black,thick,
decoration={markings,mark=at position .99 with
{\arrow[ultra thick,blackblue, rotate=0]{>}}}, postaction={decorate}
] node {} (a);
\path (a) edge [black,thick, bend left = -80,
decoration={markings,mark=at position .99 with
{\arrow[ultra thick,blackblue, rotate=-16]{>}}}, postaction={decorate}
] node {} (b);
%
%
\draw[
    decoration={markings, mark=at position 0.75 with
    {\arrow[ultra thick, blackblue, rotate=-15]{>}}},
    postaction={decorate}
    ]
    (5-.5,0) circle (.5);
\draw[
    decoration={markings, mark=at position 0.2 with
    {\arrow[ultra thick, blackblue, rotate=-15]{>}}},
    postaction={decorate}
    ]
    (5+.5,0) circle (.5);
\draw node[arnBleuPetit](a0) at ( 5, 0)  { };
\end{tikzpicture}
    \caption{Bicyclic strongly connected kernel multidigraphs}
    \label{fig:bicycles}
    \end{center}
    \end{minipage}%
    \begin{minipage}[t]{0.3\textwidth}
    \begin{center}
        \begin{tikzpicture}[>=stealth',thick, scale = 0.8]
\draw
node[arnBleuPetit](a) at (-.5, 0)  { }
node[arnBleuPetit](b) at (1, 0)  { }
node[arnBleuPetit](c) at (2, 1)  { }
node[arnBleuPetit](d) at (2, -1)  { }
;
%
%
\path (a) edge [black,thick,
decoration={markings,mark=at position .99 with
{\arrow[ultra thick,blackblue, rotate=0]{>}}}, postaction={decorate}
] node {} (b);
\path (b) edge [black,thick,
decoration={markings,mark=at position .99 with
{\arrow[ultra thick,blackblue, rotate=0]{>}}}, postaction={decorate}
] node {} (c);
\path (b) edge [black,thick,
decoration={markings,mark=at position .99 with
{\arrow[ultra thick,blackblue, rotate=0]{>}}}, postaction={decorate}
] node {} (d);
%
\path (a) edge [black,thick, bend left = 40,
decoration={markings,mark=at position .99 with
{\arrow[ultra thick,blackblue, rotate=3]{>}}}, postaction={decorate}
] node {} (c);
\path (c) edge [black,thick, bend left = 20,
decoration={markings,mark=at position .99 with
{\arrow[ultra thick,blackblue, rotate=3]{>}}}, postaction={decorate}
] node {} (d);
\path (d) edge [black,thick, bend left = 40,
decoration={markings,mark=at position .99 with
{\arrow[ultra thick,blackblue, rotate=3]{>}}}, postaction={decorate}
] node {} (a);
\end{tikzpicture}
    \caption{Another example of a cubic strongly connected kernel}
    \label{fig:complex}
    \end{center}
    \end{minipage}%
\end{figure}

    \subsection{Digraphs with one complex component}
    \label{sec:one:complex:component}

A natural application of~\cref{th:complex:scc} in connection
with~\cref{theorem:given:scc} is to enumerate digraphs having exactly one strong
complex component of given excess, while all other components are single vertices
and cycles. A simple application of this lemma for a component of excess \(
1 \) can also be derived.

\begin{lemma}
    \label{lemma:one:marked:scc}
    If \( \mathcal S \) is some family of complex strongly
    connected multidigraphs, and \( S(z, w) \) is its exponential generating
    function, then the multi-graphic generating function
    \( \widehat H_{\mathcal S}(z, w) \) of multidigraphs containing
    exactly one strongly connected component from \( \mathcal S \) while all others
    are single vertices or cycles, is
    \[
        \widehat H_{\mathcal S}(z, w)
        =
        \dfrac{
            (1 - w z) S(z, w)
            e^{-z}
            \odot_z
            \gset(z, w)
        }{
            \big(
                (1 - w z) e^{-z} \odot_z
                \gset(z, w)
            \big)^2
        }.
    \]
    In particular, if $S(z,w)$ is of the form $A(w) (1 - w z)^{r} B(w z)$
    for some integer $r$ and functions $A(\cdot)$ and $B(\cdot)$, then
    \[
        \widehat H_{\mathcal S}(z, w) =
        A(w)
        \frac{\phi_{1 + r}(z,w;B(\cdot))}
            {\phi_1(z,w;1)^2},
    \]
    where $\phi_r(z,w;F(\cdot))$ comes from \cref{theorem:phi}.
\end{lemma}

\begin{proof}
    We introduce an additional variable \( v \) marking the  components from $\mathcal S$
    in a digraph. According to~\cref{theorem:given:scc},
    the multi-graphic generating function \( \widehat H(z, w, v) \)
    of digraphs whose strongly connected components are either single
    vertices or cycles, or from \( \mathcal S \),
    where each component from \( \mathcal S \) is marked by \( v \), is
    \[
        \dfrac{1}{(1 - w z)e^{-z - vS(z, w)} \odot_z \gset(z, w)}.
    \]
    By extracting the coefficient of \( v^1 \), we obtain the first expression.
    The second result is a direct application of \cref{theorem:phi}.
\end{proof}

To illustrate the last lemma, we apply it
to count (multi)digraphs where one strong component is bicyclic,
while the others are isolated vertices and cycles.
Strong components of higher excess
and in higher numbers could be handled in the same way.

\begin{corollary}
    \label{lemma:one:bicycle}
    The multi-graphic generating function \( \widehat H_{\mathrm{bicycle}}(z, w) \) of
    multidigraphs having exactly one bicyclic
    strongly connected component and whose other strongly connected components
    are either single vertices or cycles, is given by
    \[
        \widehat H_{\mathrm{bicycle}}(z, w)
        =
        \frac{w}{2}
        \dfrac{
            w z (1 - w z)^{-2}
            e^{-z}
            \odot_z
            \gset(z, w)
        }{
            \big(
                (1 - w z) e^{-z} \odot_z
                \gset(z, w)
            \big)^2
        }
        =
        \frac{w}{2}
        \frac{\phi_{-2}(z,w; x \mapsto x)}
            {\phi_1(z,w;1)^2}.
    \]
\end{corollary}

\begin{proof}
    The exponential generating function of bicyclic multidigraphs
    (\cf~\cref{fig:bicycles}) is
    \[
        S_{\mathrm{bicycle}}(z, w) =
        \dfrac{1}{2} \dfrac{w^3 z^2}{(1 - w z)^3}
        +
        \dfrac{1}{2} \dfrac{w^2 z}{(1 - w z)^2}
        =
        \frac{1}{2}
        \frac{w^2 z}{(1 - w z)^3}.
    \]
    Plugging this in for \( S(z, w) \), we obtain the statement.
\end{proof}

We now transfer the last two results to simple and strict digraphs.

\begin{lemma}
    \label{lemma:one:marked:scc:simple}
    If \( \mathcal S \) is some family of complex strongly
    connected digraphs where multiple edges and cycles of length $\leq
    k$ are  forbidden, and \( S(z, w;k) \) is its exponential generating
    function, then the simple graphic generating function
    \( \widehat H_{\mathcal S}^{\textrm{(simple)}}(z, w;k) \)
    of digraphs containing
    exactly one strongly connected component from \( \mathcal S \) while all others
    are single vertices or cycles of length $>k$, is
    \[
        \widehat H_{\mathcal S}^{\textrm{(simple)}}(z, w;k)
        =
        \dfrac{
            (1 - zw) S(z, w;k)
            e^{-z+C_k(w z)}
            \odot_z
            \gsetsimple(z, w)
        }{
            \big(
                (1 - zw) e^{-z+C_k(w z)} \odot_z
                \gsetsimple(z, w)
            \big)^2
        }\,,
    \]
    where $C_k(z) = z + z^2/2 + \cdots + z^k / k$.
    In particular, when $S(z,w;k)$ is of the form $A(w;k) (1 - w z)^{r} B(w z; k)$
    for some integer $r$ and functions $A(\cdot;k)$ and $B(\cdot;k)$,
    then
    \[
        \widehat H_{\mathcal S}^{\textrm{(simple)}}(z,w;k) =
        \frac{\widetilde \phi_{1 + r} \left( z,w; x \mapsto A(x; k) e^{C_k(x)} \right)}
            {\widetilde \phi_1 \left(z,w; e^{C_k(\cdot)} \right)^2},
    \]
    where $\widetilde \phi_r(z,w;F(\cdot))$ is defined in \cref{theorem:phi}.
\end{lemma}

\begin{proof}
    According to~\cref{theorem:given:scc},
    the graphic generating function of digraphs whose strongly connected components are
    either single vertices or cycles of length $>k$, or from \( \mathcal S \),
    where each component from \( \mathcal S \) is marked by \( v \), is
    \[
        \dfrac
        {1}
        {
            (1 - zw) e^{-z - v S(z, w; k) + C_k(w z)}
            \odot_z \gsetsimple(z, w)
        } \,.
    \]
    Next, as in \cref{lemma:one:marked:scc}, the formula follows
    by extracting the coefficient of \( v^1 \).
    In the particular case where $S(z,w;k) = A(w;k) (1 - w z)^{r} B(w z; k)$,
    the result is obtained by application of \cref{theorem:phi}.
\end{proof}

\begin{corollary}
    The exponential generating functions of bicyclic simple and strict digraphs
    are respectively
    \[
        S_{\mathrm{bicycle}}(z, w;1) =
       \frac{1}{2}{\frac {{w}^{4}{z}^{3} \left(
       2- wz \right) }{
            \left( 1-zw \right) ^{3}}}
        +
        \frac{1}{2} {\frac {{w}^{4}{z}^{3}}{ \left( 1- w z \right) ^{2}}}\,
    \]
    and
    \[
        S_{\mathrm{bicycle}}(z, w;2)
       =  \frac{1}{2}{\frac {{w}^{5}{z}^{4} \left( 3- 2\,w z \right) }{
            \left( 1- wz\right) ^{3}}} +
       \frac{1}{2} {\frac {{w}^{6}{z}^{5}}{ \left( 1- w z \right)
           ^{2}}}\,,
    \]
    depending on whether \(2\)-cycles are allowed or not.
    The graphic generating functions of simple (for $k=1$) and strict (for $k=2$) digraphs
    where the only complex strongly connected component is bicyclic are
    \[
        \widehat H_{\mathrm{bicycle}}^{\textrm{(simple)}}(z, w;k) =
        \frac{w}{2}
        \frac{\widetilde \phi_{-2}(z,w;F_k(\cdot))}
            {\widetilde \phi_1 \left( z,w;e^{C_k(\cdot)} \right)^2}
    \]
    where $F_1(z) = z^3 (3 - 2 z) e^z$, $F_2(z) = z^4 (3 - z - z^2) e^{x + x^2/2}$
    and $C_k(z) = z + z^2/2 + \cdots + z^k/k$.
\end{corollary}

    \subsection{Summary of the exact results}
    \label{sec:summary:exact:results}

In \cref{section:asymptotics:multidigraphs}
and \cref{section:asymptotics:simple},
we obtain the asymptotic probability
of various (multi)digraphs properties.
We summarise their exact expressions
in the following propositions.

As in \cref{theorem:phi},
we define $\alpha = \log(1+w)$, $\beta = \sqrt{1+w}$ and
\begin{align*}
    \phi_r(z,w;F(\cdot)) &=
    \dfrac{1}{\sqrt{2 \pi w}}
    \int_{-\infty}^{+\infty}
    (1 - w z e^{-i x})^r
    F(w z e^{-i x})
    \exp \left(-\dfrac{x^2}{2 w} - ze^{-i x} \right)
    \mathrm dx,
    \\
    \widetilde \phi_r(z,w;F(\cdot)) &=
    \frac{1}{\sqrt{2 \pi \alpha}}
    \int_{-\infty}^{+\infty}
    (1 - w z \beta e^{-i x})^r
    F(w z \beta e^{-i x})
    \exp \left( - \frac{x^2}{2 \alpha} - z \beta e^{-i x} \right)
    \mathrm dx.
\end{align*}
When $r=0$ or $F(\cdot) = 1$,
those parameters are often omitted.
For example, $\phi(z,w) = \phi_0(z,w; 1)$.

\begin{proposition} \label{th:summary:multidigraph}
For a random multidigraph from $\DiGilbMulti(n,p)$,
the probability
\begin{itemize}
\item
to be acyclic is equal to
\[
    e^{- n^2 p / 2} n ! [z^n] \frac{1}{\phi(z,p)},
\]
see \cref{theo:MD_model_lambda_fixed}
and \cref{th:acyclic-window} for the asymptotics;
\item
to be elementary
(all its strong components are vertices or cycles)
is equal to
\[
    e^{- n^2 p / 2} n ! [z^n] \dfrac{1}{\phi_{1}(z, p)},
\]
see \cref{theo:elementary:multi} for the asymptotics;
\item
to contain only one complex component,
whose kernel is a given kernel of excess $r$ and deficiency $d$,
is equal to
\[
    \dfrac{p^r}{(2 r - d)!}
    e^{- n^2 p / 2} n ! [z^n]
    \frac{\phi_{1 - 3 r + d}(z,p;x \mapsto x^{2r-d})}
    {\phi_1(z,p)^2},
\]
see \cref{theorem:probability:one:strong:component}
for the asymptotics;
\item
to contain only one complex component,
which is bicyclic, is equal to
\[
    \dfrac{p}{2}
    e^{- n^2 p / 2} n ! [z^n]
    \frac{\phi_{-2}(z,p;x \mapsto x^{2})}{\phi_1(z,p)^2},
\]
see \cref{corollary:bicyclic} for the asymptotics.
\end{itemize}
\end{proposition}

\begin{proof}
\cref{lemma:P:Dnp} expresses
the probability of a random multidigraph family
using its multi-graphic generating function.
The four results of the proposition
are obtained by application of this lemma.
\cref{lemma:ggf:dags} provides the expression
of the multi-graphic generating function of acyclic multidigraphs.
\cref{lemma:ggf:elementary} is used for elementary multidigraphs.
\cref{th:complex:scc} provides the exponential generating function
of multidigraphs with a given strongly connected kernel
of excess $r$ and deficiency $d$.
Combined with \cref{lemma:one:marked:scc},
this gives the multi-graphic generating function
of multidigraphs containing exactly one complex component,
whose kernel is a given kernel of excess $r$ and deficiency $d$.
Finally, \cref{lemma:one:bicycle}
gives the multi-graphic generating function
of multidigraphs containing exactly one complex component,
which is bicyclic.
\end{proof}

\begin{proposition} \label{th:summary:digraph}
Consider a random simple digraph from $\DiGilbBoth(n,p)$
or strict digraph from $\DiGilb(n,p)$.
We fix $a=1$ in the first case, and $a=2$ in the second case.
Then the probability for the digraph
\begin{itemize}
\item
to be acyclic is equal to
\[
    (1 - p)^{\binom{n}{2}} n!
    [z^n]
    \frac{1}{\widetilde\phi(z,w)}\Big|_{w=\frac{p}{1 - a p}},
\]
see \cref{theo:SD_model_lambda_fixed}
and \cref{th:digraph-acyclic-window} for the asymptotics;
\item
to be elementary is equal to
\[
    (1 - p)^{\binom{n}{2}} n!
    [z^n]
    \frac{1}{\widetilde \phi_1(z,w; F_a(\cdot))}
    \Big|_{w=\frac{p}{1 - a p}},
\]
where $F_1(z) = e^z$ and $F_2(z) = e^{z + z^2/2}$,
see \cref{theo:D_model_lambda_fixed} for the asymptotics;
\item
to contain exactly one complex component,
whose excess is $r$, is equal to
\[
    (1 - p)^{\binom{n}{2}} n!
    [z^n]
    w^r \dfrac{
      \widetilde \phi_{1-3r} (z,w;F_a(\cdot))
    }{
      {\widetilde \phi}_1(z, w;G_a(\cdot))^2
    }
    \Big|_{w=\frac{p}{1 - a p}},
\]
where  $F_1(z) = A^{(\text{simple})}_r(z) e^{z}$
and $G_1(z) = e^{z}$,
$F_2(z) =  A^{(\text{strict})}_r(z) e^{z+z^2/2}$
and $G_2(z) = e^{z+z^2/2}$,
$A_r^{(\text{simple})}(z)$ and $A_r^{(\text{strict})}(z)$
are the polynomials defined in \cref{th:complex:scc}
and computable using \cref{rem:strong:complex:simple}.
See \cref{theorem:digraph:one:strong:component}
for the asymptotics.
\end{itemize}
\end{proposition}

\begin{proof}
\cref{lemma:P:Dnp} expresses the probability
of a random simple or strict digraph family
using its graphic generating function.
The three results of the proposition
are obtained by application of this lemma.
\cref{lemma:sggf:dags} provides the graphic generating function
of acyclic simple or strict digraphs.
\cref{lemma:ggf:elementary-simple} is applied
for the elementary case.
The exponential generating function
of strongly connected simple or strict digraphs
of a given excess is provided by \cref{th:complex:scc}.
It is plugged into \cref{lemma:one:marked:scc:simple}
to obtain our third result.
\end{proof}

We could obtain other results,
such as the probability for a random digraph
from $\DiGilbBoth(n,p)$ to contain
exactly one complex component,
which has a specific kernel, or which is bicyclic.
We aim to illustrate the scope of the techniques presented rather
than to provide an exhaustive treatment.

The last two propositions show that
one way to obtain various asymptotic probabilities
of interest on random (multi)graphs
is to gain a good understanding
of the functions $\phi_r(z,w;F(\cdot))$
and $\widetilde \phi_r(z,w;F(\cdot))$,
then to extract coefficients in products of powers
of those functions.
Those two steps are achieved
respectively in \cref{section:airy}
and \cref{section:external:integral}.
The asymptotics corresponding to the exact results
from the last two propositions are finally derived
in \cref{section:asymptotics:multidigraphs}
and \cref{section:asymptotics:simple}.

\section{Analysis of the inner integrals}
\label{section:airy}

In the previous section, we expressed the generating functions of various
digraph families using a reciprocal or a quotient of different
functions \( \phi_r(z, w; F(\cdot)) \), \( r \in \mathbb Z \).
In this section, we develop the tools required
to analyse the asymptotic behaviour of these functions.
The results will then be used in
\cref{%
section:external:integral,%
section:asymptotics:multidigraphs,%
section:asymptotics:simple}
to obtain the asymptotic probabilities for a random digraph
to belong to one of these families.
The target probabilities are then extracted by
relying on~\cref{lemma:P:Dnp},
which is further transformed into an integral using
Cauchy's integral formula:
\[
    [z^n] H(z) = \dfrac{1}{2 \pi i} \oint_{|z| = R} \dfrac{H(z)}{z^{n+1}} \mathrm dz.
\]

In \cref{section:airy:trees} we recall some properties of the Airy function
and its variants. Then, in \cref{section:zeros}, we turn to the analysis
of asymptotics of \( \phi_r(z, w; F(\cdot)) \) and
\( \widetilde\phi_r(z, w; F(\cdot)) \).
\Cref{prop:complex:general,corollary:complex:general} are designed specifically for this purpose and
generalise the analytic lemma from~\cite{Naina2020}.
Finally, in~\cref{sec:roots:deformed:exponential},
we investigate the roots of the generalised deformed exponential function.

Let us present a general
scheme of the analysis when \( w \to 0^+ \).
This scheme will guide our choices of the crucial ranges
in which we want to obtain the asymptotic approximations.

\begin{itemize}
    \item
The \emph{Laplace method} is designed to estimate
asymptotics of integrals of the form
\begin{equation} \label{eq:laplace:method}
    \int A(x) e^{\lambda B(x)} d x
\end{equation}
as $\lambda$ tends to infinity.
It states that, under certain technical conditions,
the main contribution is concentrated around
the \emph{saddle points}, which are values $x$ such that $B'(x) = 0$.
In the following, we assume for simplicity
that $B(x)$ has a unique saddle point $\zeta$.
The \emph{saddle point method} is applicable
when $A(x)$ and $B(x)$ are complex-analytic.
The value of the integral is then independent of the integration path.
A desirable property of the contour is
to pass through $\zeta$,
so that the Laplace method becomes applicable.
The geometry of the path around $\zeta$ is important as well.
A common rule-of-thumb is to choose a contour
that follows the \emph{steepest descent},
meaning a contour on which $|e^{B(x)}|$ decreases fast
as $x$ moves away from $\zeta$.
For more details on the Laplace and saddle point methods,
see \cite{Bruijn81} and \cite[Chapter VIII]{FSBook}.
    \item The graphic generating functions of families of interest
        are given in~\cref{section:symbolic:method}.
        While \( w \) is tending to zero as \( n \to \infty \), we can still
        consider that it is fixed, from the viewpoint of variable \( z \).
        The coefficient extraction $[z^n]$
        is accomplished by
        taking a Cauchy integral over a circle of some
        radius \( |z| = R \), 
        provided that the only singularity inside this circle is \( z = 0 \).
        We distinguish three regimes.
        In the subcritical regime, when \( wn \sim \lambda < 1 \),
        the rescaled radius of integration \( wR \) is a positive real
        number in the interval \( [0, e^{-1}) \).
        In the critical phase of the phase transition,
        \ie, when \( w = n^{-1}(1 + \mu n^{-1/3}) \) and \( \mu \) is fixed,
        the value \( R \) needs to be chosen in such a way that \( zw \)
        is close to the singular point of the generating function of trees, \ie,
        \( Rw \approx e^{-1} \).
        For the supercritical case,
        \( wn > 1 \), we will need information on the roots of the
        generating function and its derivatives at the roots.
    \item
        The radius \( R \) is chosen in such a way that
        the asymptotic value of the integral over a circle \( z = R e^{iu} \)
        is concentrated near the saddle point \( z = R \), \ie, when \( u = 0 \).
        After having zoomed in around this point, the contour locally becomes a
        vertical line. Consider a fixed complex
        point \( z \) on this rescaled line.
        The function inside the Cauchy integral is
        represented as a product of functions of the form
        \( \phi_r(z, w; F(\cdot)) \) or \( \widetilde \phi_r(z, w; F(\cdot)) \)
        (or the inverses thereof), which are
        themselves integrals over \( \mathbb R \).
        Therefore, the goal is  to find an approximation for each of these integrals
        as a function of \( z \), while paying particular attention to the vicinity of the point $z = (ew)^{-1}$ dictated by the outer integration procedure.
        Again, we use a saddle point method to
        approximate these inner
        integrals by deforming their respective
        integration contours, since the inner integrands are complex analytic.
        The choice of a deformation is
        dictated by Taylor expansion of the inner integrand around the
        point on the imaginary line
        where its first derivative vanishes (the point of stationary
        phase), and motivates the contours in~\cref{fig:path}.
\end{itemize}

In order to estimate the asymptotics of various digraph families in
the supercritical phase, when \( w = \lambda/n \) and \( \lambda > 1 \),
we need to modify the above scheme.
\begin{itemize}
    \item Assuming that
        we have an integer power of
        \( \phi_r(z, w; F) \)
        (or of \( \widetilde \phi_r(z, w; F) \))
        in the denominator of
        the inner graphic generating function,
        we need to use the refined
        asymptotic approximation of \( \phi_r(z, w; F) \)
        (\resp \( \widetilde \phi_r(z, w; f) \))
        up to \emph{two terms} in order to
        establish an asymptotic approximation of the \emph{roots} of
        \( \phi_r(z, w; F) \)
        (or \resp \( \widetilde \phi_r(z, w; F) \)).
        Then, we shall need \emph{three terms} of the asymptotic expansion of
        the roots.
    \item Using the fine asymptotic approximations of the roots
        and of the function \( \phi_r(z, w; F) \)
        (\resp \( \widetilde \phi_r(z, w; F) \)),
        we can obtain the main asymptotic term of the \emph{derivative} of
        \( \phi_r(z, w; F) \) with respect to \( z \) at its roots.
        Also, in one of the cases
        that we consider, the dominant root of \( \phi_1(z, w; F) \) is not simple,
        and therefore, we will need its \emph{second derivative} as well.
    \item The radius of the contour of integration is increased
        in a way that enables us to capture the dominant root of \( \phi_r(z, w;
        F) \).
        In this case, the absolute value \( |zw| \) exceeds \( e^{-1} \).
        Then, since the target probability is expressed using Cauchy's
        integral formula,
        it can be asymptotically expressed as a complex residue at the point
        corresponding to the first root of \( \phi_r(z, w; F) \).
\end{itemize}

\subsection{Generalised Airy functions and trees}
\label{section:airy:trees}

In this section, we introduce a special function which generalises the
Airy function and recall some of its asymptotic properties.
Furthermore, we recall some of the
asymptotic properties of the generating functions of labelled trees that are
involved in our results and use them to establish the asymptotic properties of
the principal generating functions involved: the deformed exponential function
\( \phi(z, w) \) and its generalisations arising in the symbolic method
for digraphs.

The \emph{Airy function} is defined by
\begin{equation}\label{eq:def-airy-function}
    \ai(z) = \dfrac{1}{2 \pi i} \int_{-i\infty}^{+i\infty}
    \exp\left(
        -zt + t^3/3
    \right) \mathrm dt
    =   \dfrac{1}{\pi} \int_0^{+\infty} \cos \left(
          \dfrac{t^3}{3} + zt
      \right)\mathrm dt\, ,
\end{equation}
and solves  the linear differential equation that is also called
\emph{Airy equation}, namely
\begin{equation}\label{eq:airy-equation}
\ai''(z) - z \ai(z) = 0\, .
\end{equation}
Let us  recall the following classical properties
of the Airy function and its derivatives.

\begin{lemma} \label{lem:airy_function}
For any non-negative integer $k$,
\begin{equation}
    \label{eq:airy:recurrence}
    \ai^{(k+3)}(z)=(k+1) \ai^{(k)}(z)+z \ai^{(k+1)}(z) \, .
\end{equation}
For any angle $\pi/6 \leq \varphi \leq \pi/2$
and complex value $z$, when $x$ tends to infinity, we have
\[
    \ai^{(k)}(z) \sim
    \frac{(-1)^k}{2 i \pi}
    \int_{x e^{- i \varphi}}^{x e^{i \varphi}}
    t^k
    e^{- z t + t^3/3}
\mathrm  d t.
\]
The convergence is exponentially fast if $\pi/6 < \varphi < \pi/2$.
Furthermore,
\begin{equation}\label{eq:ai_asymp}
    \ai(z)
    \sim
    \frac
    {e^{-\frac{2}{3}z^{3/2}}}
    {2\sqrt{\pi}z^{1/4}}
    \  \text{ as } z\to \infty,
    \ \text{and } |\arg (z)| \leq \pi - \epsilon
\, ,
\end{equation}
see for example \cite[(9.7.5)]{NIST}.
\end{lemma}

It is natural to consider a biparametric special function capturing the Airy
function and its derivatives as well as antiderivatives. This more general
function will enter the final approximations for the probabilities of interest.
Note that the contour of integration can be deformed provided that the starting
and the finishing points at infinity remain unchanged.
\begin{definition}
    \label{def:special:airy}
    The generalised Airy function \( \ai(k; z) \) is defined for all
    integer \( k \) and complex \( z \) and is given by
    \begin{equation}
        \label{eq:special:airy}
        \ai(k; z)
        :=
        \dfrac{(-1)^k}{2 \pi i}
        \int_{t \in \Pi(\varphi)}
        t^k e^{-zt + t^3/3}
        \mathrm dt,
    \end{equation}
    where \( \Pi(\varphi) \) is composed of three line segments
    (see~\cref{fig:contour:Pi})
    \[
        t(s) =
        \begin{cases}
            -e^{-i \varphi} s ,
            & \text{for}  \quad {-\infty} < s \leq -1; \\
            \cos \varphi + is \sin \varphi,
                    & \text{for} \quad
            -1 \leq s \leq 1 ; \\
            e^{i \varphi} s ,
                & \text{for}   \quad 1 \leq s < +\infty,
        \end{cases}
    \]
    with \( \varphi \in [\pi/6, \pi/2] \).
    If \( k \) is non-negative, then \( \ai(k;z)=\partial_z^k \ai(z) \).
    When \( k \) is negative, then \( \ai(k;z) \) is a \( (-k) \)-fold
    antiderivative of the Airy function, \ie,
    \( \partial_z^{-k} \ai(k;z) = \ai(z) \). Moreover, the
    recurrence~\eqref{eq:airy:recurrence} is satisfied for negative \( k \) as
    well, which can be easily checked by integration by parts.
\end{definition}

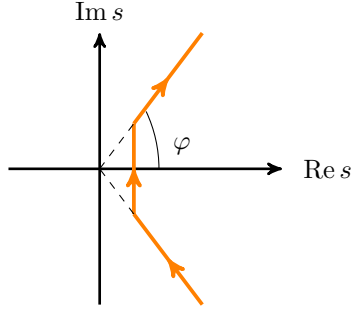
\begin{figure}[hbt]
    \begin{center}
        \pgfdeclarelayer{edgelayer}
\pgfsetlayers{edgelayer,main}
\tikzstyle{axe}=[->,draw=black,line width=1]
\tikzstyle{arc}=[-,draw=orange,postaction={decorate},decoration={markings,mark=at position .5 with {\arrow{>}}},line width=1.400]
\tikzstyle{edge}=[-,draw=black,line width=1.4]
\begin{tikzpicture}[>=stealth',scale=0.6]
	\begin{pgfonlayer}{edgelayer}
		\draw [style=axe] (-15, 0) to (-15, 6);
		\draw [style=axe] (-17, 3) to (-11, 3);

		\draw [style=arc] (-12.75, 0) to (-14.25, 2);
        \draw [style=arc] (-14.25, 2) to (-14.25, 4);
		\draw [style=arc] (-14.25, 4) to (-12.75, 6);

		\draw [style=dashed] (-15, 3) to (-14.25, 4);
		\draw [style=dashed] (-14.25, 2) to (-15, 3);

        \draw (-13.7, 3) arc (0:25:3) ;
        \draw (-13.2, 3.5)  node {$\varphi$};

        \draw (-10, 3) node {$\Real s$};
        \draw (-15, 6.5) node {$\Image s$};
	\end{pgfonlayer}
\end{tikzpicture}
    \end{center}
    \caption{The contour \( \Pi(\varphi) \).}
    \label{fig:contour:Pi}
\end{figure}

The following lemma proves the convergence
of the generalised Airy function
and will be used in the proof of \cref{prop:complex:general}.

\begin{lemma} \label{th:airy:bigO}
For any $\varphi \in (\pi/6, \pi/2)$, integer $k$, and any function $f$ satisfying
$f(t) = \bigO(t^k)$ for $t \in \Pi(\varphi)$,
we have
\[
    \int_{\Pi(\varphi)} f(t) e^{-z t + t^3/3} \mathrm dt = \bigO(1)
\]
uniformly for $z$ in any compact set.
\end{lemma}

\begin{proof}
Since $f(t) = \bigO(t^k)$,
there is a constant $C$ such that
for any $t \in \Pi(\varphi)$, we have
\[
    |f(t) e^{-z t + t^3/3}|
    \leq
    C |t|^k e^{\Real(-z t + t^3/3)}.
\]
Consider $t$ in the third part of the path $\Pi(\varphi)$,
so $t = e^{i \varphi} s$ for $1 \leq s < +\infty$.
Then
\[
    \Real(-z t + t^3/3) =
    - \Real(z s e^{i \varphi})
    + \cos(3 \varphi) \frac{s^3}{3}
    =
    \left(1 - \frac{3 \Real(z e^{i \varphi})}{\cos(3 \varphi) s^2} \right)
    \cos(3 \varphi) \frac{s^3}{3}.
\]
Since $\varphi \in (\pi/6, \pi/2)$, we have $\cos(3 \varphi) < 0$.
We assumed that $z$ belongs to a compact set,
so there is a constant $S \geq 1$ such that for all $s \geq S$, we have
\[
    1 - \frac{3 \Real(z e^{i \varphi})}{\cos(3 \varphi) s^2} > \frac{1}{2}.
\]
This implies for $s \geq S$
\[
    \Real(-z t + t^3/3) < \cos(3 \varphi) \frac{s^3}{6}
\]
and
\[
    |t^k| e^{\Real(-z t + t^3/3)} < s^k e^{\cos(3 \varphi) s^3/6}.
\]
The right-hand side is integrable on $s \in [S, +\infty)$,
so the left-hand side is integrable as well.

By symmetry, when $t$ belongs to the first part of the $\Pi(\varphi)$ path,
so $t = - e^{- i \varphi} s$ for $-\infty < s \leq -1$,
we also have for any $s \leq -S$
\[
    |t^k| e^{\Real(-z t + t^3/3)} < |s|^k e^{- \cos(3 \varphi) s^3/6},
\]
which is integrable for $-\infty < s \leq -S$.

Finally, for any non-negative $k$,
\[
    |t^k| e^{\Real(-z t + t^3/3)}
\]
is bounded when $t$ belongs to $\Pi(\varphi)$ and $|t| \leq S$.
This also holds for negative $k$.
Indeed, on the second part of $\Pi(\varphi)$,
$t$ has absolute value at least $\cos(\varphi)$,
which is positive.
Thus, this part is integrable as well.

Combining those three parts, we conclude that
\[
    |f(t) e^{-z t + t^3/3}|
\]
is integrable on $\Pi(\varphi)$,
and the upper bound we derived is uniform with respect to $z$
in any compact set.
\end{proof}

The generalised Airy function was considered first in a paper of Grohne in
1954~\cite{grohne1954spektrum} in the context of hydrodynamics, and somewhat
independently (only for the case of the first antiderivative), in 1966 by
Aspnes~\cite{aspnes1966} in the context of electric optics.
Further asymptotic results appear in a paper where the authors study
approximations of the Orr--Sommerfeld equation~\cite{hughes1968},
originating in hydrodynamics.
The zeros of generalised Airy functions have been
studied by Baldwin in~\cite{baldwin1985}.
The numerous properties of Airy functions reach far
beyond the scope of the current paper, but the interested reader can find more
details in~\cite{vallee2004airy}.

The function \( \ai(k; z) \) is also defined in
\texttt{SageMath} as \texttt{airy\_ai\_general(k,z)}
and in \texttt{Maple} as \texttt{AiryAi(k, z)} for non-negative integers $k$.
It can also be computed numerically through its
hypergeometric series representations, which we discuss below, or using its
integral representation as well. We provide an independent numerical
implementation of \( \ai(k; z) \) for arbitrary integers $k$
in one of the utility libraries located in the auxiliary repository
mentioned in~\cref{section:introduction}.
Another special function related to \( \ai(k;z) \), has been considered
in~\cite[Lemma 3]{Janson93}, \cite{banderier2001random} and~\cite[Theorem
2]{de2015phase}. This function arises in situations where one needs to extract
  the
coefficients of a large power of a generating function \( g(z) \) with a known
Newton--Puiseux expansion, or when this function is multiplied by another
function \( h(z) \) with a coalescent singularity. That is, the coefficient extraction from
expressions of the type \( [z^n] h(x) g(z)^m \), where both \( n \) and \( m \) tend
to infinity, and where the singularity of \( h(x) \) is also a saddle point of
\( g(z) \). The results of~\cite{banderier2001random}
and \cite{Janson93} generalise the Airy function in two different directions,
and the most general version is given in~\cite[Theorem 2]{de2015phase}
and can be defined as
\[
    G(\lambda, r; x) := \dfrac{-1}{2 \pi i \lambda}
    \int_{-\infty}^{(0)}
    e^{u - x u^{1/\lambda}} u^{(1 - r)/\lambda} \dfrac{\mathrm du}{u},
\]
where the integration contour is a classical Hankel contour which goes
counter-clockwise from \( -\infty-0i \) to \( 0 \) and back to \( -\infty +
0i \), and \( u^{1/\lambda} \)
represents the principal branch of the power function. This function can be
expressed as an infinite sum involving the Gamma function by expanding
the Taylor series of \( e^{-x u^{1/\lambda}} \), see~\cite{de2015phase} for an
explicit expression.

\begin{remark}
    If \( r \) is a non-negative integer, then
    \[
        \ai(r; z) = (-1)^{r-1} 3^{(r+1)/3} G(3, -r; 3^{1/3}z).
    \]
    Indeed, since, according to~\cref{lem:airy_function}, the \( r \)-th derivative of the Airy
    function can be expressed as
    \[
        \ai^{(r)}(z) = \dfrac{(-1)^r}{2 \pi i}
        \int_{\infty e^{-i \varphi}}^{\infty e^{i \varphi}}
        t^r e^{-zt + t^3/3} \mathrm d t,
    \]
    for any angle \( \varphi \in (\pi / 6,  \pi/2 ) \), we can perform the change of
    variables \( t^3/3 = s \), so that the new variable \( s \) is running
    along a new contour starting at \( \infty e^{3i \varphi} \), winding around
    zero, and going back to \( \infty e^{-3i \varphi} \). By choosing
    \( \varphi = \pi/3 \), we obtain the classical Hankel contour. This yields
    \begin{align*}
        \ai^{(r)}(z)
        & = \dfrac{(-1)^r}{2 \pi i}
        3^{r/3}
        \int_{-\infty}^{(0)}
        s^{r/3} e^{s -z 3^{1/3} s^{1/3}} \dfrac{\mathrm d s}{(3s)^{2/3}}
        \\
        & =
        (-1)^{r-1} 3^{(r+1)/3} \cdot
        \dfrac{-1}{2 \pi i} \cdot \dfrac{1}{3}
        \int_{-\infty}^{(0)}
        s^{(r-2)/3} e^{s - (3^{1/3}z) s^{1/3}} \mathrm ds.
    \end{align*}
\end{remark}

\begin{remark}
    Similarly to~\cref{lem:airy_function}, it is possible to obtain the asymptotics
    of \( \ai(r, z) \) when \( z \to \infty \),
    \( |\arg(z)| \leq \pi - \epsilon \), using the classical saddle point method
    (\cf \cite[(9.7.5)]{NIST}):
    \begin{equation}
        \label{eq:airy:asymptotics}
        \ai(r; z)
        \sim
        (-1)^{r}
        \dfrac{
            z^{r/2-1/4}
            \exp(-\tfrac{2}{3} z^{3/2})
        }{
            2
            \pi^{1/2}
        }.
    \end{equation}
 \end{remark}


Finally, let us recall some basic properties of the generating functions associated with labelled trees,
as they will be needed for the asymptotic analysis of
the deformed exponential functions and their roots.
Let $U(z)$ be the
exponential generating function of labelled unrooted trees and $T
(z)$  be the  exponential generating function of rooted labelled trees,
see~\cref{section:symbolic:elementary}.
Using the Newton--Puiseux expansions at a
singularity~\cite[Theorem~VII.7]{FSBook} (see
also~\cite{Knuth-Pittel89}),  we have
\begin{equation}\label{eq:expansion-T}
    T(z) = 1 - S(\sqrt{2(1- ez)})\,
\end{equation}
where the analytic function
\[
  S(z) = z - \frac{1}{3}z^2 + \frac{11}{72}z^3 -\frac{43}{540} z^4
  +\cdots\,
\]
satisfies the functional relation $(1-S(z))e^{S(z)} = 1 - z^2/2$.
This yields asymptotic expansions around the singular point \( z \sim e^{-1} \)
for both \( T(z) \) and \( U(z) \):
\begin{align} \label{eq:tree-expansion}
\begin{split}
    U(z) &= \dfrac{1}{2} - (1 - ez) + \dfrac{2^{3/2}}{3} (1 - ez)^{3/2}
    + \bigO( (1 - ez)^2),\\
    T(z) &= 1 - \sqrt{2} \sqrt{1 - ez} + \dfrac{2}{3} (1 - ez)
    + \bigO( (1 - ez)^{3/2}).
\end{split}
\end{align}

It is useful to establish some properties of the behaviour of these functions in the complex plane. The generating function \( T(z) \) has non-negative coefficients, and therefore, satisfies the property
\( |T(z)| \leq T(|z|) \) for any \( z \) inside the circle of convergence.
Moreover, \( T(z) = 1 \) at the singular point \( z = e^{-1} \).
We would like to prove that if \( |T(z)| \approx 1 \) and \( z \) is near the circle of convergence, then \( z \) is restrained to be around the singular point.

\begin{lemma}
\label{lemma:complex:plane:trees}
Let \( |ze| \leq 1 + o(1) \) with \(z \notin \mathbb{R}_{> 1/e} \)
and \( |T(z)| = 1 + o(1) \).
Then, \( ze = 1 + o(1) \) and
\[
    T(z) = 1 - \sqrt{2} \sqrt{1 - ez} + O(|1 - ez|).
\]
\end{lemma}
\begin{proof}
Let \( T(z) = |T(z)| e^{i \varphi(z)} \),
where \( \varphi(z) \in [-\pi, \pi] \).
Since \( z = T(z) e^{-T(z)} \), we have
\[
    |z| = |T(z)| e^{-\Real T(z)}
        = |T(z)| e^{-|T(z)| \cos \varphi(z)}
         \geq |T(z)| e^{-|T(z)|}.
\]
Since \( |T(z)| = 1 + o(1) \),
we conclude \( |z| \geq e^{-1} (1 + o(1)) \). Moreover,
due to the condition \( |ze| \leq 1 + o(1) \) of the lemma,
we obtain the asymptotic equivalence \( |z| = e^{-1} (1 + o(1)) \).

It follows that \( |T(z)| e^{-|T(z)| \cos \varphi(z)} \sim e^{-1} \), which implies
\( \cos \varphi(z) = 1 + o(1) \), which further implies \( \varphi(z) = o(1) \).
Consequently, \( z = T(z) e^{-T(z)} = e^{-1}(1 + o(1)) \), and we use the Puiseux series approximation for \( T(z) \) near its singular point.
\end{proof}

\begin{remark}
An alternative proof of the last lemma is as follows.
Since $T(z)$ satisfies the conditions of the
Daffodil Lemma \cite[p.266]{FSBook},
for $|z| \leq 1/e$, the function $|T(z)|$
reaches its unique maximum at $z = 1/e$.
This maximum is $T(1/e) = 1$ according to~\eqref{eq:tree-expansion}.
Thus, by continuity of $T(z)$,
if $z$ stays in a small enough vicinity of the circle of radius $1/e$
and $|T(z)|$ converges to $1$, then $z$ converges to $1/e$
and~\eqref{eq:tree-expansion} is applicable.
\end{remark}

\subsection{The generalised deformed exponential and its asymptotics}
\label{section:zeros} 

The simplest possible integral representation arising
in~\cref{section:symbolic:method} corresponds to the case of directed acyclic
(multi-)digraphs. Let us denote by \( \phi(z, w) \) the denominator of the
multi-graphic generating function \( \widehat D_{\mathrm{DAG}}(z, w) \) given
in~\cref{lemma:ggf:dags}; after  the change of variable
$x \mapsto x / \sqrt{w}$ in the integrand expression, we have
\begin{equation}\label{eq:phi}
    \phi(z, w) = \gset(-z, w)
    = \dfrac{1}{\sqrt{2 \pi w}}
        \int_{-\infty}^{+\infty}
        e^{f(z)}
        \mathrm dx
        \quad  \mbox{where} \quad
        f(x) = -\dfrac{x^2}{2w} - ze^{-ix}
        \,.
\end{equation}
Since the zeros of \( \phi(z, w) \) are the poles of \( \widehat D_{\mathrm{DAG}}(z, w) \), it is clear that the behaviour of \( \phi(z, w) \) around its zeros plays an
important role in the enumeration of acyclic (multi)-digraphs. Here, by a zero of \( \phi(z, w) \) we mean a function
\( z = z(w) \) that satisfies \( \phi(z(w), w) = 0 \).
Similarly,  let us denote by \( \widetilde\phi(z, w) \) the
denominator of the graphic generating function
\( \widehat D_{\mathrm{DAG}}^{(\text{simple})}(z, w) \) given
in~\cref{lemma:sggf:dags}, that is,
\( \widetilde \phi(z, w) := \phi \left(z \sqrt{1 + w}, \log(1 + w) \right) \), or
\begin{equation}\label{eq:phi-simple}
    \widetilde \phi(z, w)
    = \gsetsimple(-z, w)
    = \dfrac{1}{\sqrt{2 \pi \alpha}}
        \int_{-\infty}^{+\infty} e^{g(x)}
        \mathrm dx
        \quad  \mbox{where} \quad
        g(x) = -\frac{x^2}{2\alpha}-z \beta e^{-ix}
        \, ,
\end{equation}
with \( \alpha = \log(1 + w) \) and \( \beta = \sqrt{1 + w} \).
The properties of these two functions and their
zeros are certainly interesting in their own right.



%
In order to find the asymptotic expressions for their zeros and derivatives at
these zeros, we first need to find an asymptotic
approximation for \( \phi_r(z, w; F) \) itself as \( w \to 0^+ \), and around
its roots.
The integral in the definition of $\phi_r(z,w; F)$
in \eqref{eq:phi:tilde:phi} is
an integral over the real line. However,
we can deform this path of integration
in the complex plane, which allows us
to  apply the saddle point method
to $\phi_r$. The objective is to find a path that goes through
a saddle point $x_0$ on the imaginary line, \ie a point $x_0$
such that $f'(x_0)=0$.
By using the definition of \(f\) in \eqref{eq:phi}, $f'(x_0)=0$ implies $-i x_0 = z w e^{-i x_0}$.
We consider the solution defined in the vicinity of $z w = 0$ that is given by
\begin{equation}\label{eq:saddle-point}
x_0 = i T(zw)\, .
\end{equation}
The fact that the generating function $T$  has a singularity at $z=
1/e$   suggests that we should
consider $z$ to be a function of $w$  such that $zw$ is close to
$1/e$.

Then, the Taylor series  of $f$ around $x_0$ is
\begin{equation}\label{eq:f-expansion}
  f(x)=f(x_0)+f''(x_0)\frac{(x-x_0)^2}{2!}+
  f'''(x_0)\frac{(x-x_0)^3}{3!}+\cdots  \, ,
\end{equation}
where
\begin{equation}\label{eq:first-derivatives-f}
f(x_0)  =-\frac{U(zw)}{w}, \quad
f''(x_0)  = \frac{T(zw)-1}{w},
\mbox{ and }
f^{(k)}(x_0) = -(-i)^k\frac{T(zw)}{w}
\mbox{ for all $k\geq$  3}\, .
\end{equation}
Looking at the first few terms in the Taylor series above, note that
the quadratic term $f''(x_0)(x-x_0)^2$ also vanishes when $x_0=i$.
So we proceed as follows: if $x_0$ is sufficiently far from $i$ (in our
case, this means $|x_0-i|\gg w^{2/3}$), then we shift the path of
integration to the horizontal line  passing through $x_0$,
and if  $|x_0-i|\ll w^{2/3}$, then we take the path $\Gamma$ on the
right-hand side  in~\cref{fig:path}. In the latter case, the angle between the two middle
segments of the path meeting at \( i \) will be \( 2 \varphi = 2 \pi/3 \).
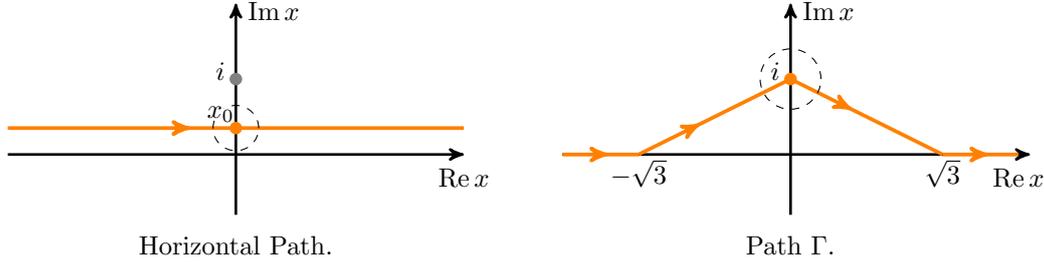
\begin{figure}[h]
\tikzstyle{arc}=[-,draw=orange,postaction={decorate},decoration={markings,mark=at
position .4 with {\arrow{>}}},line width=1.400]
\tikzstyle{invarc}=[-,draw=orange,postaction={decorate},decoration={markings,mark=at
position .6 with {\arrow{>}}},line width=1.400]
\begin{tikzpicture}[>=stealth',scale=1]
\draw[->,draw=black,line width=1] (-3,0)--(3,0) node[right]{};
\draw[->,draw=black,line width=1] (0,-.8)--(0,2) node[above]{};
\draw[style=arc] (-3,0.35)--(3,0.35);
\draw[dashed] (0,.35) circle (.3cm);
\node[circle,fill=orange,scale=.5] at (90:3.5mm) (0,0) {};
\node[circle,fill=gray,scale=.5] at (90:10mm) (0,1) {};
\node[] at (-0.2,1.1) {$i$};
\node[circle,fill=white,scale=.5] at (-0.2,.54) (0,1) {};
\node[] at (-0.2,.54) {$x_0$};
\node[] at (0,-1.2) {\text{Horizontal Path.}};
\draw (3, -.3) node {$\Real x$};
\draw (.5, 1.9) node {$\Image x$};
\end{tikzpicture}
\qquad
\begin{tikzpicture}[>=stealth',scale=1]
\draw[->,draw=black,line width=1] (-3,0)--(3.15,0) node[right]{};
\draw[->,draw=black,line width=1] (0,-.8)--(0,2) node[above]{};
\draw[invarc] (-3,0)--(-2,0);
\draw[invarc] (2,0)--(3,0);
\draw[arc] (-2,0)--(0,1);
\draw[arc] (0,1)--(2,0);
\draw[dashed] (0,1) circle (.4cm);
\node[circle,fill=orange,scale=.5] at (90:10mm) (0,1) {};
\node[] at (-2,-0.25) {$-\sqrt{3}$};
\node[] at (2,-0.25) {$\sqrt{3}$};
\node[] at (-0.2,1.1) {$i$};
\node[] at (0,-1.2) {\text{Path} $\Gamma.$};
\draw (3, -.3) node {$\Real x$};
\draw (.5, 1.9) node {$\Image x$};
\end{tikzpicture}
\caption{Two Paths.}\label{fig:path}
\end{figure}
We can use the path $\Gamma$ in the integral \eqref{eq:phi} since
$f(x)$ is an entire function (as a function of \(x\)).  But we can also shift the path of
integration to any horizontal line since \(|e^{f(x)}|\) tends to zero exponentially fast, as \(|\mathrm{Re}(x)|\to \infty\), in any fixed horizontal strip for \(w>0\). The dashed circles in the two
graphs in~\cref{fig:path} represent circles of radius $w^c$,
where  $c$ is a positive constant. The exponent $c$ will be chosen separately
for each  range of $z$. We can split the integral into two parts such that
the part of the integral stemming from the path within the circle
will be called the \textit{local integral} and the rest will be called
the \textit{tail}.

Now we are ready to formulate and prove the central analytic lemma of the current paper in full generality. We also need to compute refined asymptotic approximations of \( \phi_r(z, w; F(\cdot)) \) in order to give a refined asymptotic estimate of the \emph{zeros} of these functions. They will be helpful for studying the supercritical behaviour of families of random digraphs.

\begin{remark}
The forthcoming~\cref{prop:complex:general} partitions the disc \( \{ z \mid
|ezw| \leq 1 + Kw^{2/3} \} \) into three regions depending on the distance to
the branching point \( z = (ew)^{-1} \). The completeness of this partitioning
is justified by~\cref{lemma:complex:plane:trees}. An analogous analytic
lemma (Proposition 8) from~\cite{Naina2020} provides an identical
partitioning, but in their case, \( \varepsilon \) is fixed to be \( 1/12 \),
and the regions are defined in terms of the argument of the generating
function, and not in terms of \( T(zw) \). Doing so in parts (a) and (b) seems
more natural to us from the analytic perspective. Note that when \( \tau \) is
in a bounded closed subset of \( \mathbb C \), the approximations in parts (b)
and (c) overlap, but the technique to obtain them is slightly different:
firstly, \( \tau \) cannot be in a vicinity of the branch cut in part (b), and
secondly, in part~(c) we develop an additional asymptotic term which yields \(
K_r(\tau) \). This is required to get a finer asymptotic estimate of the roots
of \( \phi_r \) in~\cref{sec:roots:deformed:exponential} for \( r \in \{0, 1\}
\).
\end{remark}

\begin{lemma}
    \label{prop:complex:general}
    Let \( \phi_{r}(z,w; F(\cdot)) \) denote the function
    \[
    \phi_{r}(z, w; F(\cdot)) =
        \dfrac{1}{\sqrt{2 \pi w}}
        \int_{-\infty}^{+\infty}
        (1 - zw^+ e^{-i x})^r
        \exp \left(
            -\dfrac{x^2}{2w} - z e^{-i x}
        \right)
        F(zw^+ e^{-ix})
        \mathrm dx
        \enspace ,
    \]
    where \( w^+ = w + \mathcal O(w^2) \)
    and $F(\cdot)$ is an entire function that does not vanish on $\mathbb{C} \setminus \{0\}$.
    Then, we have the following asymptotic
    formulas for \( \phi_{r}(z, w; F(\cdot)) \) as \( w \to 0^+ \):
    \begin{enumerate}
        \item[(a)] If there are fixed positive numbers $K$ and $\varepsilon$ such that
        $1-|T(zw)|\geq w^{1/3-\varepsilon}$
        and
        \( |ezw|\leq 1+Kw^{2/3} \), then
        \[
            \phi_r(z,w; F(\cdot)) \sim e^{-U(zw)/w} (1-T(zw))^{r-1/2} F(T(zw))\, ;
        \]
    \item[(b)]If there are fixed positive numbers $K$ and $\varepsilon$
    such that $\varepsilon \leq \tfrac{1}{12}$,
    $w^{1/3}\leq 1-|T(zw)|\leq w^{1/3-\varepsilon}$ and
    \( |ezw|\leq 1+Kw^{2/3} \), then
        \[
            \phi_r(z, w; F(\cdot)) \sim
        (-1)^r\sqrt{2\pi} \cdot 2^{r/3+1/3}w^{r/3-1/6}\ai(r; 2^{-2/3}\theta^2)e^{\theta^3/3-U(zw)/w}F(1)\, ,
        \]
        where $\theta=w^{-1/3}(1-T(zw))$;
    \item[(c)] If $1 - ezw = \tau w^{2/3}$, then the estimate
        \begin{equation} \label{eq:partc}
            \phi_r(z,w; F(\cdot))\sim
            (-1)^r
            \sqrt{2 \pi} \cdot
            2^{r/3 + 1/3}
            w^{r/3 - 1/6}
            \ai(r; 2^{1/3}\tau)
            \exp \left(
                -\dfrac{1}{2 w} + \dfrac{\tau}{w^{1/3}}
            \right) F(1)
            \,
        \end{equation}
        holds uniformly for $\tau$ in any bounded closed subset of $\mathbb{C}$. Moreover, when \( r \in\{0,1\} \), the estimate is refined to
        \begin{equation} \label{eq:refined:partc}
            \phi_r(z,w; F(\cdot))=
            (-1)^r
            \sqrt{2 \pi} \cdot
            2^{r/3 + 1/3}
            w^{r/3 - 1/6}
            K_r(\tau)
            \exp \left(
                -\dfrac{1}{2 w} + \dfrac{\tau}{w^{1/3}}
            \right)
            \,,
        \end{equation}
        where
        \begin{align*}
        K_0(\tau) & =
            F(1) \ai(2^{1/3} \tau)
            + \\
            & \quad w^{1/3}
            \left(
                \tfrac{5}{6} \tau^2 F(1) \ai(2^{1/3}\tau)
                - \tfrac{1}{6} 2^{1/3} F(1) \ai'(2^{1/3}\tau)
                + 2^{1/3} F'(1) \ai'(2^{1/3}\tau)
            \right) + \bigO(w^{2/3}) \, ,
              \\
        K_1(\tau) & =
            F(1) \ai'(2^{1/3} \tau)
            + \\
            & \quad w^{1/3}
            \left(
                \tfrac{1}{3}( F(1) + 6 F'(1)) 2^{-1/3} \tau
                \ai(2^{1/3} \tau)
                +
                \tfrac{5}{6} F(1) \tau^2 \ai'(2^{1/3}\tau)
            \right) + \bigO(w^{2/3}) \, .
        \end{align*}
    \end{enumerate}
\end{lemma}

\begin{proof}
    Let us recall the definition of the function \( f(x) \) and let us introduce
    a new function \( h(x) \):
    \[
        f(x) = -\dfrac{x^2}{2 w} - z e^{-ix},
        \quad \text{and} \quad
        h(x) = 1 - zw^+ e^{-ix}.
    \]
    The target function \( \phi_r(z, w; F(\cdot)) \) can then be expressed as
    \[
        \phi_r(z, w; F(\cdot)) = \dfrac{1}{\sqrt{2 \pi w}} \int_{-\infty}^{+\infty}
        h(x)^{r} e^{f(x)} F(1-h(x)) \mathrm dx\, .
    \]
	For a complex number $z$ and $w>0$, the function $e^{f(x)}$ decays exponentially fast as $x\to \infty$ in a fixed horizontal strip, while the term $h(x)^rF(1-h(x))$ remains bounded. Hence, we can shift path of integration above to any horizontal line as long as the path transformation does not go through any singularity of $h(x)^rF(1-h(x))$.

    \textbf{Part~(a):} If $z=0$, then
    \[
        \phi_r(0, w; F(\cdot)) =
        \frac{1}{\sqrt{2 \pi w}} \int_{-\infty}^{+\infty}
        \exp \left( - \frac{x^2}{2 w} \right) F(0) \mathrm dx
        =
        F(0),
    \]
    which matches the claimed asymptotics.
    We now assume $F(T(z w)) \neq 0$.
    For this case, let $c$ be a fixed number in the interval $(1/3,1/3+\varepsilon/2)$. We shift the line of integration in the definition of $\phi_r(z, w; F(\cdot))$ to the horizontal line passing through the point $x_0=iT(zw)$.  Define the local and tail integrals as follows:
    \begin{align*}
    I& =\int_{-w^c}^{w^c}
        h(x_0+t)^{r} e^{f(x_0+t)} F(1-h(x_0+t)) \mathrm dt\, ,\text{ and } \\[.5em]
    J & = \int_{|t|\geq w^c}
        h(x_0+t)^{r} e^{f(x_0+t)} F(1-h(x_0+t)) \mathrm dt\, .
    \end{align*}

Let us begin with the local integral. By expanding \( h(x) \) around the saddle point
\( x_0 = i T(zw) \) for $|t| \leq w^c$, we have
    \[
        h(x_0 + t)
        = 1 - T(zw) e^{-it} (1 + \mathcal O(w) )
        = 1 - T(zw) + \mathcal{O}(w^c)\, .
    \]
Similarly, we also have \( F(1-h(x_0+t)) \sim F(T(zw)) \), and the latter term is non-zero. On the other hand, the Taylor approximation of $f(x_0+t)$ is
\begin{align*}
    f(x_0+t)
    & =f(x_0)+f''(x_0)\frac{t^2}{2}+\mathcal{O}(w^{3c-1}) \\[.5em]
    & =-\frac{U(zw)}{w}-\frac{1-T(zw)}{w}\frac{t^2}{2}+\mathcal{O}(w^{3c-1})\, .
\end{align*}
Hence, the local integral satisfies the asymptotic formula
\[
    I \sim
    F(T(zw))
    (1 - T(zw))^r
    e^{-U(zw)/w}
    \int_{-w^c}^{w^c}
    e^{-w^{-1} (1-T(zw)) t^2/2} \mathrm dt\, .
\]
The range of integration of the integral on the right-hand side can be extended to
$(-\infty, +\infty)$ at the expense of a small error term.
To see this, we apply the change of variables $t \mapsto t \sqrt{w}$ and bound the integral
\[
    \left|
        \int_{w^c}^{+\infty}
        e^{-w^{-1}(1-T(zw))t^2/2}
        \mathrm dt
    \right|
    \leq
    \int_{w^{c-1/2}}^{+\infty}
    e^{-\Real(1-T(zw))t^2/2} \mathrm dt \, .
\]
Since $\Real(1-T(zw))\geq 1-|T(zw)|\geq w^{1/3-\varepsilon}$, the right-hand side is a $\mathcal{O}(e^{-w^{2c-2/3-\varepsilon}/2})$, which (by the choice of $c$) tends to zero faster than any power of $w$. Thus, we deduce
\[
    I \sim
    F(T(zw))(1 - T(zw))^r e^{-U(zw)/w}
    \int_{-\infty}^{+\infty}
    e^{-w^{-1}(1-T(zw))t^2/2} \mathrm dt\, .
\]

Next, we consider the tail integral $J$. From the definition of $f(x)$ in~\eqref{eq:phi} and $x_0$, we have
\[
    f(x_0 +t) - f(x_0) =
    -\frac{1}{2w}\left(
        t^2 + 2(e^{-it} + it - 1) T(zw)
    \right)\,.
\]
The first derivative with respect to $t$ of this function at $t=0$ is zero and its second derivative with respect to $t$ is $-w^{-1}(1-T(zw)e^{-it})$.
Using the mean-value
form of  the Taylor  approximation around zero of
$\Real (f(x_0 + t) - f(x_0))$,  there exists $\eta(t)\in [-\pi, \pi]$ such that
\[
    \Real (f(x_0 + t) - f(x_0))
    =
    - w^{-1} \Real \left(
        1 - T(zw) e^{-i\eta(t)}
    \right)
    \frac{t^2}{2}.
\]
On the other hand, we have
\[
    \Real \left(
        1 - T(zw) e^{-i\eta(t)}
    \right) = 1-\Real
    \left(
        T(zw) e^{-i\eta(t)}
    \right) \geq 1 - |T(zw)|\, .
\]
Hence,
\[
    \Real (f(x_0 + t) - f(x_0))
    \leq
    -w^{-1} \left(
        1 - |T(zw)|
    \right)
    \frac{t^2}{2}.
\]
Since the term \( h(x_0 + t)^r F(1 - h(x_0 + t)) \) remains bounded in any horizontal strip, this yields
\[
    e^{U(zw)/w} J
    =\mathcal{O}\left(
        \int_{|t|\geq w^c}
        e^{-w^{-1}(1 - |T(zw)|)t^2/2}
        \mathrm dt
    \right).
\]
Again as before, the integral on the right-hand side tends to zero faster than any power of $w$ as $w\to0^+$. Hence, the contribution from the tail integral is negligible. Therefore,
\[
    \phi_r(z, w; F(\cdot)) \sim
    \dfrac{F(T(zw))}{\sqrt{2 \pi w}}
    (1 - T(zw))^r
    e^{-U(zw)/w}
    \int_{-\infty}^{+\infty}
    \exp \left(- \frac{1-T(z w)}{w} \frac{t^2}{2} \right)
    \mathrm dt.
\]
The expression in the statement of part (a) in the lemma is obtained by evaluating the Gaussian integral above.

    \textbf{Part~(b):} Here, we choose $c$ to be a fixed number in the interval $(1/4,1/3)$, and we still consider the integral over the horizontal line passing through $x_0=iT(zw)$. The local and  the tail integrals are defined in the same way as in the previous case:
\begin{align*}
    I& =\int_{-w^c}^{w^c}
        h(x_0+t)^{r} e^{f(x_0+t)} F(1-h(x_0+t)) \mathrm dt\, , \\[.5em]
    J & = \int_{|t|\geq w^c}
        h(x_0+t)^{r} e^{f(x_0+t)} F(1-h(x_0+t)) \mathrm dt\, .
\end{align*}
Let us begin by estimating the local integral. Since the term $1 - T(zw)$ can be small, we consider further terms in various Taylor approximations.
According to~\cref{lemma:complex:plane:trees}, \( ezw \) is now in the
vicinity of \( 1 \), and \( T(zw) \sim 1 \), so \( T(zw) \sim |T(zw)|
\). Therefore, according to the condition of part (b),
\[
    |1 - T(zw)| = O(w^{1/3 - \varepsilon}).
\]
For $|t|\leq w^c$, we have
\begin{align*}
    h(x_0 + t)
    &=
    1 - T(zw) + it T(zw) + \bigO(t^2 + w T(zw))
    \\
    &=
    1 - T(zw) + it + it(T(zw) - 1) + \bigO(w^{2c})
    \\
    &=
    1-T(zw)+it+\mathcal{O}(w^{c+1/3-\varepsilon}+w^{2c})\, .
\end{align*}
Observe that
\[
    |1 - T(z w) + i t| \geq \Real(1-T(zw)+it)\geq 1-|T(zw)|\geq w^{1/3}.
\]
Thus we have $h(x_0 + t) =(1+o(1)) (1-T(zw)+it)$, uniformly for $|t|\leq w^c$. On the other hand, we still have \( F(1-h(x_0+t)) \sim F(1) \). The Taylor approximation of $f(x_0+t)$ gives
\[
    f(x_0+t)
    =
    - \frac{U(zw)}{w}
    - \frac{1 - T(zw)}{w} \frac{t^2}{2}
    - \frac{i}{w} \frac{t^3}{6}
    + \mathcal{O}(
        w^{3c-\varepsilon-2/3} + w^{4c-1}
    ).
\]
Since \( \varepsilon \leq \tfrac{1}{12} \), we also have
\( 3c - \varepsilon - 2/3 > 0 \), and \( w^{3c - \varepsilon - 2/3} \to 0 \).
Therefore,
\[
    I \sim F(1) e^{-U(zw)/w}
    \int_{-w^c}^{w^c}
        (1 - T(zw)+it)^r
        e^{-w^{-1}(1 - T(zw))t^2/2 - iw^{-1}t^3/6}
    \mathrm dt\, .
\]
In order to extend the range of integration as in the previous case, we need to estimate the integral
\[
    \int_{w^c}^{+\infty}|1 - T(zw)+it|^re^{-w^{-1}\Real(1 - T(zw))t^2/2} \mathrm dt.
\]
Since we have $\Real(1 - T(zw))\geq w^{1/3}$, the above integral is of the form $e^{-\Omega(w^{2c-2/3})}$, which tends to zero faster than any power of $w$.

To estimate the tail integral, the argument we used in part (a) is still valid:
again, since \( h(x_0 + t)^r F(1 - h(x_0 + t)) \) is bounded in any horizontal strip, we have
\[
    e^{U(zw)/w}J =\mathcal{O}\left(  \int_{|t|\geq w^c}
    e^{ - w^{-1}(1-|T(zw)|)t^2/2}\mathrm dt\right).
\]
Since $1-|T(zw)|\geq w^{1/3}$, the integral on the right-hand side tends to zero faster than any power of $w$.  Therefore, we deduce that
    \[
        \phi_r(z, w; F(\cdot)) \sim
        \dfrac{F(1)}{\sqrt{2 \pi w}}
        e^{-U(zw)/w}
        \int_{-\infty}^{+\infty}(1 - T(zw)+it)^re^{-w^{-1}(1 - T(zw))t^2/2-iw^{-1}t^3/6} 	    \mathrm dt\, .
    \]
    Since \( \Real( 1 - T(zw)) > 0 \), the term \( 1 - T(zw) + it \) follows the trajectory \( \Pi(\varphi) \) with \( \varphi = \pi/2 \) required to define the generalised Airy function as in~\eqref{eq:special:airy}. Therefore, after a change of variables
    \[
        t \mapsto 1 - T(zw) + it,
    \]
    the quadratic term in the exponent disappears, and the exponent becomes
    \[
        \frac{t^3}{6w}
        -
        \dfrac{t}{2w}(1-T(zw))^2
        +
        \dfrac{1}{3w}(1 - T(zw))^3.
    \]
    The remaining change of variables \( t \mapsto 2^{1/3} w^{1/3} t \) leads to the form given in the statement of the current lemma.

    \textbf{Part (c):} We take \( c \in (1/4, 1/3) \).
    The integration path is slightly different from the previous two. The
    saddle point $x_0$ is too close to $i$ so we choose the path of
    integration $\Gamma$ shown in~\cref{fig:path}.
    In this case
    the approximation also becomes different from the classical Airy function
    due to the fact that \( h(x) \) around the local part is not
    approximated by a constant any longer.

    Let us denote the part of $\Gamma$ that lies in the disk $|x-i|\leq  w^c$
    by $\Gamma_c$ and call it the \emph{central part}.
    The \emph{first tail} denotes the part of $\Gamma$
    equal to $(-\infty, -\sqrt{3}) \cup (\sqrt{3}, +\infty)$.
    Finally, the \emph{second tail} corresponds to the rest of $\Gamma$:
    two straight segments linking, respectively,
    $-\sqrt{3}$ to $i + w^c e^{-5 i \pi / 6}$,
    and $i + w^c e^{-i\pi/6}$ to $\sqrt{3}$.

    In the rest of the proof, we first consider the central part
    and prove that its asymptotic behaviour is given by~\eqref{eq:partc},
    then refine the estimate, which corresponds to~\eqref{eq:refined:partc},
    and finally prove that the first and second tails are negligible.

    \textsf{Central part.}
    For $x = i$ and $\tau = 0$, we have $h(i) = 1$.
    Thus, $x = i$ could be a singular point of the integrand
    $h(x)^r e^{f(x)} F(1 - h(x))$ when $r$ is negative.
    To avoid this case, we define an integration path $\Gamma_c^*$
    obtained from $\Gamma_c$ by circumventing $i$ from below.
    The local integral is
    \[
        I_r :=  \int_{\Gamma_c^*} h(x)^r e^{f(x)} F(1 - h(x))
        \mathrm dx.
    \]
    The substitution
    \[
        x-i = -i 2^{1/3} w^{1/3} t
    \]
    is applied and we deform the integration path,
    so that the new integration variable \( t \) follows a trajectory \( \Pi_c(\varphi) \)
    composed of three line segments (see~\cref{fig:contour:Pi}):
    \[
        t(s) =
        \begin{cases}
            -e^{-i \varphi} s ,
                & \text{for}  \quad {-w^{c-1/3}} \leq s \leq -1; \\
            \cos \varphi + is \sin \varphi,
                    & \text{for} \quad
            -1 \leq s \leq 1 ; \\
            e^{i \varphi} s ,
                & \text{for}   \quad 1 \leq s \leq w^{c-1/3},
        \end{cases}
    \]
    with \( \varphi = \pi/3 \).
    Note that this path is similar to the one considered
    in the proof of~\cite[Lemma 3]{Janson93}.
    Henceforth, the local part becomes
    \[
        I_r = \int_{t \in \Pi_c(\varphi)} h(x(t))^r e^{f(x(t))} F(1-h(x(t))) \mathrm d x(t),
    \]
    where \( t = t(s) \) and \( s \in [-w^{c-1/3}, w^{c-1/3}] \),
    which implies \( t = \bigO(w^{c-1/3}) \).
    In order to extract the asymptotics of $I_r$,
    we provide approximations for the terms
    $h(x(t))^r$,
    $F(1 - h(x(t)))$
    and $e^{f(x(t))}$.
    Since $h(x) = 1 - z w^+ e^{-i x}$,
    plugging
    $x(t) = i(1 - 2^{1/3} w^{1/3} t)$ and $e z w = 1 - \tau w^{2/3}$
    into the expression of $h(x(t))$, we obtain
    \begin{equation} \label{eq:hxt}
        h(x(t)) =
        1 - z w^+ e^{1 - 2^{1/3} w^{1/3} t} =
        1 - (1 - \tau w^{2/3}) e^{- 2^{1/3} w^{1/3} t} (1 + \mathcal O(w)).
    \end{equation}
    Expanding the exponential series and
    using $t = \bigO(w^{c-1/3})$
    provides the approximation
    \begin{equation} \label{eq:hxt:taylor}
        h(x(t)) =
        1 - (1 - \tau w^{2/3})
        \left(1 - 2^{1/3} w^{1/3} t + \bigO(w^{2 c}) \right)
        =
        2^{1/3} w^{1/3} t (1 + \bigO(w^c)).
    \end{equation}
    For any non-negative $r$, it follows that
    \[
        h(x(t))^r =
        2^{r/3} w^{r/3} t^r (1 + \bigO(w^c)).
    \]
    Thus, the main asymptotic term of \( h(x(t)) \) only depends on \( t \), but not on \( \tau \).
    Since $h(x(t)) = \bigO(w^c)$, we also deduce
    \[
        F(1 - h(x(t))) = F(1) (1 + \bigO(w^c)).
    \]
    Plugging $x(t) = i(1 - 2^{1/3} w^{1/3} t)$
    into $f(x) = -\frac{x^2}{2 w} - z e^{- i x}$, we obtain
    \[
        f(x(t)) =
        \frac{1}{2 w} - 2^{1/3} w^{-2/3} t + 2^{-1/3} w^{-1/3} t^2
        - z e^{1 - 2^{1/3} w^{1/3} t}.
    \]
    Now $z$ is replaced using the relation $e z w = 1 - \tau w^{2/3}$, yielding
    \[
        f(x(t)) =
        \frac{1}{2 w} - 2^{1/3} w^{-2/3} t + 2^{-1/3} w^{-1/3} t^2
        - (w^{-1} - \tau w^{-1/3}) e^{- 2^{1/3} w^{1/3} t}.
    \]
    Expanding the exponential series and using the relation $t = \bigO(w^{c-1/3})$
    gives the approximation
    \begin{align*}
        f(x(t)) &=
        \frac{1}{2 w} - 2^{1/3} w^{-2/3} t + 2^{-1/3} w^{-1/3} t^2
        \\& \quad
        - (w^{-1} - \tau w^{-1/3})
        \left(1 - 2^{1/3} w^{1/3} t + 2^{2/3} w^{2/3} \frac{t^2}{2} - 2 w \frac{t^3}{6} + \bigO(w^{4c}) \right).
    \end{align*}
    Since $1/4 < c < 1/3$, we deduce
    \[
        f(x(t)) =
        -\frac{1}{2 w} + \frac{\tau}{w^{1/3}}
        - 2^{1/3} \tau t + \frac{t^3}{3} + \bigO(w^{4 c - 1})
    \]
    and
    \[
        e^{f(x(t))} =
        \exp \left(-\frac{1}{2 w} + \frac{\tau}{w^{1/3}} \right)
        e^{- 2^{1/3} \tau t + t^3 /3}
        (1 + \bigO(w^{4 c - 1})).
    \]
    We finally plug
    the approximations of $h(x(t))^r$, $F(1 - h(x(t)))$,
    $e^{f(x(t))}$ and $\mathrm d x(t) = (2 w)^{1/3} \mathrm d t / i$
    into the integral $I_r$
    and apply \cref{th:airy:bigO} to interchange integral and big $\bigO$, to get
    \[
        I_r =
        (1 + \bigO(w^{4c-1}))
        F(1)
        \exp \left(-\frac{1}{2 w} + \frac{\tau}{w^{1/3}} \right)
        \dfrac{(2w)^{(1+r)/3}}{i}
        \int_{\Pi_c(\varphi)}
        t^{r} e^{-2^{1/3} \tau t + t^3/3} \mathrm dt\, .
    \]
    As in part (b), after adding the negligible tails,
    the integral yields the generalised Airy function
        \[
            I \sim
            2 \pi F(1) \exp \left(
                -\dfrac{1}{2 w} + \dfrac{\tau}{w^{1/3}}
            \right)
            (2w)^{(1 + r)/3}
            (-1)^r \ai(r; 2^{1/3}\tau),
        \]
    which is the desired asymptotic formula \eqref{eq:partc}.

    \textsf{Refined central part.}
    From now on let us abbreviate \(x(t)\) as just \(x\).
    In order to obtain the refined estimate~\eqref{eq:refined:partc}
    for the asymptotics of \( \phi_r(z, w; F(\cdot)) \) for $r \in \{0, 1\}$,
    we need better approximations for $h(x)$, \(F(1-h(x))\)
    and \( e^{f(x)} \).
    Pushing the Taylor expansion from~\eqref{eq:hxt} and~\eqref{eq:hxt:taylor} further,
    we obtain
    \[
        h(x) =
        Q(t, w^{1/3})
        (1 + \bigO(w^{2/3})),
    \]
    where
    \[
        Q(t, w^{1/3}) =
        2^{1/3} w^{1/3} t + (\tau - 2^{-1/3} t^2) w^{2/3}.
    \]
    We also have by Taylor expansion and using $F(1) \neq 0$
    \[
        F(1 - h(x)) =
        (F(1) - Q(t, w^{1/3}) F'(1))
        (1 + \bigO(w^{2/3})).
    \]
    Finally,
    considering only the relevant terms up to \( \bigO(w^{2/3}) \),
    we obtain
    \begin{align*}
        f(x) &=
        \frac{1}{2 w} - 2^{1/3} w^{-2/3} t + 2^{-1/3} w^{-1/3} t^2
        - (w^{-1} - \tau w^{-1/3})
        \bigg( \sum_{k=0}^4 \frac{(-2^{1/3} w^{1/3} t)^k}{k!}
        + \bigO(w^{5/3} t^5) \bigg)
        \\&=
        - \frac{1}{2 w}
        + \frac{\tau}{w^{1/3}}
        - 2^{1/3} \tau t
        + \frac{t^3}{3}
        + \tau 2^{-1/3} w^{1/3} t^2
        - 2^{-2/3} w^{1/3} \frac{t^4}{6}
        + \bigO(w^{2/3} t^5).
    \end{align*}
    Thus, we have
    \[
        e^{f(x)} =
        \exp\left(- \frac{1}{2 w} + \frac{\tau}{w^{1/3}} \right)
        e^{- 2^{1/3} \tau t + t^3/3}
        P(t, w^{1/3})
    \]
    where
    \[
        P(t, w^{1/3}) =
        1 +
        \left(
            \tau 2^{-1/3} t^2
            - 2^{-2/3} \frac{t^4}{6}
        \right) w^{1/3}
        + \bigO(w^{2/3} t^5).
    \]
    Plugging in the approximations of $h(x)$, $F(1 - h(x))$, $e^{f(x)}$
    and $\mathrm d x(t) = (2 w)^{1/3} \mathrm d t / i$
    in the integral $I_r$, we obtain
    \begin{align*}
        I_r &=
        \exp \left(-\frac{1}{2 w} + \frac{\tau}{w^{1/3}} \right)
        \dfrac{(2w)^{1/3}}{i}
        \int_{\Pi_c(\varphi)}
        R_r(t, w^{1/3})
        e^{-2^{1/3} \tau t + t^3/3} \mathrm dt,
    \end{align*}
    where
    \[
        R_r(t, w^{1/3}) =
        Q(t, w^{1/3})^r
        \left(F(1) - Q(t, w^{1/3}) F'(1) \right)
        P(t, w^{1/3})
        (1 + \bigO(w^{2/3})).
    \]
    Plugging in the values of $P(t, w^{1/3})$ and $Q(t, w^{1/3})$, we deduce
    \begin{align*}
        R_0(t, w^{1/3}) &=
        F(1)
        + \left(
            - F'(1) 2^{1/3}t
            + F(1) \tau 2^{-1/3} t^2
            - F(1) 2^{-5/3} \frac{t^4}{3}
        \right) w^{1/3}
        + \bigO(w^{2/3} t^5),
        \\
        R_1(t, w^{1/3}) &=
        F(1) 2^{1/3} t w^{1/3}
        \\ & \quad
        + \left(
            F(1) \tau
            - (F(1) 2^{-1/3} + F'(1) 2^{2/3}) t^2
            + F(1) \tau t^3
            - F(1) 2^{-1/3} \frac{t^5}{6}
        \right) w^{2/3}
        + \bigO(w t^5).
    \end{align*}
    Applying \cref{th:airy:bigO} to bound the error terms,
    we obtain for the integrals $I_0$ and $I_1$
    \begin{align*}
        I_0 &=  \exp\left(- \frac{1}{2 w} + \frac{\tau}{w^{1/3}} \right)
            \Big(
            C_{01} w^{1/3} + C_{02} w^{2/3} + \bigO(w)
        \Big),
        \\
        I_1 &=  \exp\left(- \frac{1}{2 w} + \frac{\tau}{w^{1/3}} \right)
        \Big(
            C_{12} w^{2/3} + C_{13} w + \bigO(w^{4/3})
        \Big),
    \end{align*}
    where

    \begin{align*}
        C_{01} &=
        \frac{2^{1/3}}{i} \int_{\Pi(\varphi)} F(1) e^{-2^{1/3} \tau t + t^3/3} dt =
        2^{4/3} \pi \ai(2^{1/3} \tau) F(1) \, , \\
        C_{02} &=
        \frac{2^{1/3}}{i} \int_{\Pi(\varphi)}
        \Big(
            - F'(1) 2^{1/3} t
            + F(1) \tau 2^{-1/3} t^2
            - F(1) 2^{-5/3} \frac{t^4}{3}
        \Big)
        e^{-2^{1/3} \tau t + t^3/3} dt
        \\&=
        2 \pi F(1) \Big(
            \tau \ai(2; 2^{1/3} \tau)
            - \frac{2^{2/3}}{12} \ai(4; 2^{1/3} \tau)
        \Big)
        +
        2^{5/3} \pi F'(1) \ai(1; 2^{1/3} \tau)
        \, , \\
        C_{12} &=
        \frac{2^{1/3}}{i} \int_{\Pi(\varphi)} 2^{1/3} t F(1) e^{-2^{1/3} \tau t + t^3/3} dt =
        -2^{5/3} \pi \ai(1; 2^{1/3} \tau) F(1) \, , \\
        C_{13} &=
        \frac{2^{1/3}}{i} \int_{\Pi(\varphi)}
        \Big(
            F(1) \tau
            - (F(1) 2^{-1/3} + F'(1) 2^{2/3}) t^2
            + F(1) \tau t^3
            - F(1) 2^{-1/3} \frac{t^5}{6}
        \Big)
        e^{-2^{1/3} \tau t + t^3/3} dt
        \\
        &=
        2 \pi F(1)
        \Big(
            2^{1/3} \tau \ai(2^{1/3} \tau)
            - 2^{1/3} \tau \ai(3; 2^{1/3} \tau)
            - \ai(2; 2^{1/3} \tau)
            + \tfrac16 \ai(5; 2^{1/3} \tau)
        \Big)
        \\&
        \quad - 4 \pi F'(1) \ai(2; 2^{1/3} \tau)
        \, . \\
    \end{align*}
    To simplify \( C_{02} \) and \( C_{13} \) further, we use the expressions for the derivatives of the Airy
    functions:
    \[
        \ai(2; z) = z\ai(z)
        \quad \text{and} \quad
        \ai(3; z) = \ai(z) + z \ai'(z),
    \]
    \[
        \ai(4; z) = 2 \ai'(z) + z^2 \ai(z)
        \quad \text{and} \quad
        \ai(5; z) = 4z \ai(z) + z^2 \ai'(z).
    \]
    Hence, we obtain
    \begin{align*}
    C_{02} &=
    2 \pi \left(
        \frac{5}{6} \tau^2 2^{1/3} F(1) \ai(2^{1/3}\tau)
        - \frac{1}{6} 2^{2/3} F(1) \ai'(2^{1/3}\tau)
        + 2^{2/3} F'(1) \ai'(2^{1/3}\tau)
    \right) \, , \\
    C_{13} &= - 2 \pi \left(
        \frac{1}{3}( F(1) + 6 F'(1)) 2^{1/3} \tau
        \ai(2^{1/3} \tau)
        +
        \frac{5}{6} F(1) 2^{2/3} \tau^2 \ai'(2^{1/3}\tau)
    \right) \, ,
    \end{align*}
    which gives the last estimate of the theorem:
    \[
        K_0(\tau) = \dfrac{2^{-1/3}}{2 \pi} \left(
            C_{01} + C_{02} w^{1/3}
        \right)
        \quad \text{and} \quad
        K_1(\tau) = \dfrac{-2^{-2/3}}{2 \pi} \left(
            C_{12} + C_{13} w^{1/3}
        \right).
    \]

    \textsf{First tail.}
    Now, we can complete the tail estimate.
    For $t$ real and $|t|\geq \sqrt{3}$, \( h(t)^rF(1-h(t)) \) is bounded by a constant, and we also have
    \[
    f(t)=-\frac{1}{2w}\left(t^2+2e^{-it-1}+\mathcal{O}(w^{2/3})\right).
    \]
This yields
\[
\Real (f(t)-f(i))=-\frac{1}{2w}\left(t^2-1+2e^{-1}\cos(t)+\mathcal{O}(w^{2/3})\right)\leq -\frac{1}{2w}\left(t^2-1-2e^{-1}+\mathcal{O}(w^{2/3})\right).
\]
Hence
\[
e^{-f(i)}\int_{|t|\geq \sqrt{3}}h(t)^rF(1-h(t))e^{f(t)} \mathrm dt
 =\mathcal{O}\left(\int_{|t|\geq \sqrt{3}}e^{\Real (f(t)-f(i))} \mathrm dt\right),
 \]
  and
\[
\int_{|t|\geq \sqrt{3}}e^{\Real (f(t)-f(i))} \mathrm dt\leq e^{(e^{-1}+1/2)w^{-1}+\mathcal{O}(w^{-1/3})}\int_{|t|\geq \sqrt{3}}e^{-t^2/(2w)} \mathrm dt=e^{(e^{-1}-1)w^{-1}+\mathcal{O}(w^{-1/3})},
\]
which tends to zero faster than any power of $w$. Hence, this contribution is negligible.

\textsf{Second tail.}
Another part of the tail corresponds to
$x = i + t e^{- i \pi / 6}$
with $t \in [w^c, 2]$. For this case, we write $\tau =a+ib$. So, by our assumption on $\tau$, $a$ and $b$ are contained in fixed bounded real intervals. We have
\[
\Real(f(i + t e^{- i \pi / 6}) - f(i))  =  w^{-1}\left(h_1(t)-a w^{2/3}h_2(t)-bw^{2/3}h_3(t)\right)\, ,
\]
where
\begin{align*}
  h_1(t) & =
    1 - \frac{t^2}{4} - \frac{t}{2}- \cos \left( \sqrt{3}\, t / 2 \right) e^{-t/2} \, , \\
    h_2(t)& =  1 - \cos \left( \sqrt{3}\, t / 2 \right) e^{-t/2} \, , \\
    h_3(t) & = -\sin \left( \sqrt{3}\, t / 2 \right) e^{-t/2} \, .
\end{align*}
Note that the function $h_1'(t)$ is negative on the interval $(0,2]$, and since $a$ and $b$ are bounded, the function $h'_1(t)-a w^{2/3}h'_2(t)-bw^{2/3}h'_3(t)$ is also negative if $t$ is bounded away from zero and for sufficiently small $w$. The first term in the Taylor approximation of $h'_1(t)-a w^{2/3}h'_2(t)-bw^{2/3}h'_3(t)$ with respect to $t$  is
\[
    h'_1(t) - a w^{2/3} h'_2(t) - bw^{2/3} h'_3(t)
    =
    -t^2 / 2
    - (a/2 - \sqrt{3} b / 2 ) w^{2/3} + \mathcal{O}(|t|^3 + w^{2/3}|t|).
\]
This implies that $h'_1(t)-a w^{2/3}h'_2(t)-bw^{2/3}h'_3(t)$ can only be positive or zero for $t$ below $\mathcal{O}(w^{1/3})$. But, since we chose $c<1/3$,
$w^{1/3} = o(w^c)$ as $w\to 0^+$. Hence, for sufficiently small $w$, the function $h_1(t)-a w^{2/3}h_2(t)-bw^{2/3}h_3(t)$ is decreasing on the interval  $[w^c,2]$. This implies that
\[
  \Real(f(i + t e^{- i \pi / 6}) - f(i)) \leq
  w^{-1}\left(h_1(w^c)-a w^{2/3}h_2(w^c)-bw^{2/3}h_3(w^c)\right).
\]
The Taylor expansions of $h_1(t)$, $h_2(t)$ and $h_3(t)$ at $0$ start with $-t^3/6 + \bigO(t^4)$,  $t/2+\mathcal{O}(t^2)$, and $\sqrt{3}t/2+\mathcal{O}(t^2)$ respectively. Thus
\[
  \Real(f(i + t e^{- i \pi / 6}) - f(i)) \leq
  - \frac{1}{6} w^{3 c - 1} \left(1 + \bigO(w^c+w^{-2(3c-1)/3})\right)
\]
holds uniformly for $t\in [w^c,2]$. Therefore
\[
  \left| \int_{w^c}^2 e^{f(i + t e^{-i \pi/6}) - f(i)}
    \mathrm d t \right|
  \leq
  2 \inf_{w^c \leq t \leq 2}
  e^{\Real(f(i + t e^{-i \pi/6}) - f(i))}
  \leq
  2 e^{-w^{3 c - 1} (1 + \bigO(w^c+w^{-2(3c-1)/3})) / 6}.
\]
Once again, since $3 c - 1 < 0$, the last term tends to zero as $w\to0^{+}$ faster than any power of $w$, which implies that this part of the tail is negligible. A similar proof establishes that the last part of the tail, corresponding to $x = i + t e^{i \pi / 6}$ for $t \in [-2, -w^c]$, is negligible as well.
\end{proof}


\begin{corollary}
    \label{corollary:complex:general}
    Let \( \widetilde \phi_r(z, w; F(\cdot)) \)
    be as defined in \eqref{eq:phi:tilde:phi},
    and assume $F(z)$ does not vanish on $\mathbb{C} \setminus\{0\}$.
    Then, we have the following asymptotic formulas for
    \(\widetilde \phi_r(z, w; F(\cdot))\) as \(w\to 0^+\):
    \begin{enumerate}
        \item[(a)] If there are fixed positive numbers $K$ and $\varepsilon$ such that
        $1-|T(zw)|\geq w^{1/3-\varepsilon}$
        and
        \( |ezw|\leq 1+Kw^{2/3} \), then
        \[
            \widetilde \phi_r(z,w; F(\cdot)) \sim e^{-U(zw)/w-U(zw)/2} (1-T(zw))^{r-1/2} F(T(zw))\, ;
        \]
    \item[(b)] If there are fixed positive numbers $K$ and $\varepsilon$ such that $\varepsilon \leq \tfrac{1}{12}$,
        $w^{1/3}\leq 1-|T(zw)|\leq w^{1/3-\varepsilon}$ and
        \( |ezw|\leq 1+Kw^{2/3} \), then
        \[
            \widetilde \phi_r(z, w; F(\cdot)) \sim
            (-1)^r \sqrt{2\pi}
            \cdot 2^{r/3+1/3} w^{r/3-1/6} e^{-1/4}
            \ai(r; 2^{-2/3}\theta^2)
            e^{\theta^3/3-U(zw)/w} F(1),
        \]
        where $\theta=w^{-1/3}(1-T(zw))$;
        \item[(c)] If $1 - ezw = \tau w^{2/3}$, then the estimate
        \[
            \widetilde \phi_r(z,w; F(\cdot))\sim
            (-1)^r
            \sqrt{2 \pi} \cdot
            2^{r/3 + 1/3}
            w^{r/3 - 1/6} e^{-1/4}
            \ai(r; 2^{1/3}\tau)
            \exp \left(
                -\dfrac{1}{2 w} + \dfrac{\tau}{w^{1/3}}
            \right) F(1)
            \,
        \]
        holds uniformly for $\tau$ in any bounded closed subset of $\mathbb{C}$. Moreover, if \( r \in \{0, 1\} \), then this estimate can be refined to
        \[
            \widetilde \phi_r(z,w; F(\cdot))=
            (-1)^r
            \sqrt{2 \pi} \cdot
            2^{r/3 + 1/3}
            w^{r/3 - 1/6} e^{-1/4}
            K_r(\tau)
            \exp \left(
                -\dfrac{1}{2 w} + \dfrac{\tau}{w^{1/3}}
            \right)
            \,,
        \]
        where $K_r(\tau)$ is defined in \cref{prop:complex:general}.
    \end{enumerate}
\end{corollary}

\begin{proof}
    Note that according to the way \( \phi \) is defined
    in~\cref{prop:complex:general}, we have in the given range
    \[
        \widetilde \phi_r(z, w; F(\cdot))
        = \phi_r(z\sqrt{1+w}, \log(1+w); F(\cdot))
        \times (1 + \mathcal O(w^{2/3})) \, ,
    \]
    and so,
    the corresponding asymptotic approximations can be obtained using a change
    of variables. In particular, if we let
    \[
        \widetilde z = z \sqrt{1+w}
        \quad \text{and} \quad
        \widetilde w = \log(1 + w),
    \]
    then \( \widetilde \phi_r(z, w; F) =
    \phi_r(\widetilde z, \widetilde w; F)
    \cdot (1 + \mathcal O(w^{2/3}))
    \), and,
    as \( \widetilde w \to 0^+ \), in part~(a), we have an
    asymptotic equivalence: \( zw = \widetilde z \widetilde w (1 + \mathcal
    O(w^2)) \), so the
    expressions \( U(zw) \) and \( T(zw) \) in part~(a) remain unchanged.
    However, by expanding further, we see that \( \widetilde w = w - w^2/2+\bigO(w^3) \), and therefore, the term \( -U(\widetilde z \widetilde w)/\widetilde w \) in the exponent of the first part
    becomes \( -U(zw)/w -U(zw)/2 + o(1) \). In Parts (b) and (c), if we let
    \( zw = e^{-1}(1 - \tau w^{2/3}) \),
    then
    \( \widetilde z \widetilde w
    =
    e^{-1}(1 - \widetilde \tau \widetilde w^{2/3}) \),
    where
    \[
        \widetilde \tau = \tau (1 + w/3) + \bigO(w^{4/3}).
    \]
    After substituting this linear shift into the asymptotic expressions,
    we see that it does not affect the validity of
    the asymptotic approximations: only a change by \( w^{1/3} \) or \( w^{2/3}
    \) in \( \tau \)
    would affect the final expression, in the exponent or in the refined part,
    but not a smaller change \( \mathcal O(w) \).
    Therefore, the only change, again, happens in the exponent, where by replacing \( w \)
    with \( \widetilde w = w - w^2/2 + \bigO(w^3) \),
    we obtain
    \[
        -\frac{U(\widetilde z \widetilde w)}{\widetilde w} =
        - \frac{U(z w)}{w} - \frac{1}{2} U(z w) + \mathcal O(w).
    \]
    Since $T(z w) \sim 1$ and $U(z w) = T(z w) - T(z w)^2/2$,
    we obtain
    \[
        -\frac{U(\widetilde z \widetilde w)}{\widetilde w} =
        - \frac{U(z w)}{w} - \frac{1}{4} + \mathcal O(w)
    \]
    giving the additional factor $e^{-1/4}$.
\end{proof}

\subsection{Roots and derivatives of the deformed exponential}
\label{sec:roots:deformed:exponential}
The roots of the deformed exponential function have already been examined
before in the range where \( w \) is a positive constant.
It is known, for example, that all zeros of $\widetilde\phi$ and \(
\phi \) are real, positive and  distinct when $w > 0$,
see~\cite{Polya14,Laguerre83}. For a given $w>0$
and $j\in \mathbb{N}$, let $\widetilde\varrho_j(w)$ be the $j$-th smallest
solution to the equation $\widetilde\phi(z,w)=0$. As mentioned
in~\cite{Bender86},
we have $\widetilde \varrho_1(1)\approx 1.488079$.  Grabner and
Steinsky~\cite{Grabner05} studied the behaviour of the other zeros of
$\widetilde \phi(z,1)$, extending  the work of Robinson~\cite{Robinson73}.

Before we estimate the asymptotics of the roots and derivatives of \( \phi_r \),
it is useful to provide an explicit expression for its derivative first.

\begin{lemma}
\label{lemma:derivative:explicit}
Let \( \phi_r(z, w; F(\cdot)) \) be defined, as in~\eqref{eq:phi:tilde:phi}, by
\[
    \phi_{r}(z, w; F(\cdot)) =
    \dfrac{1}{\sqrt{2 \pi w}}
    \int_{-\infty}^{+\infty}
    (1 - zw e^{-i x})^r
    \exp \left(
        -\dfrac{x^2}{2w} - z e^{-i x}
    \right)
    F(zwe^{-ix})
    \mathrm dx.
\]
Then, locally uniformly
for \( |z| < w^{-1} \), we have
\begin{multline}\label{eq:partial:next}
    \partial_z \phi_r(z, w; F(\cdot))
    =\\
    \frac{1}{z}\left(
        - r \phi_{r-1}\big(z, w; x \mapsto xF(x)\big)
        - \frac{1}{w}\phi_{r}\big(z, w; x \mapsto xF(x)\big)
        + \phi_{r}\big(z, w; x \mapsto xF'(x)\big)
    \right).
\end{multline}
\end{lemma}

\begin{proof}
Since the integral representation of \( \phi_r(z, w; F(\cdot)) \)
converges locally uniformly for \(|z|<w^{-1}\), we have
\[
\partial_z \phi_r(z, w; F(x))=\dfrac{1}{\sqrt{2 \pi w}}
        \int_{-\infty}^{+\infty}
        \partial_z \left((1 - zw e^{-i x})^r
        \exp \left(
            -\dfrac{x^2}{2w} - z e^{-i x}
        \right)
        F(zwe^{-ix})\right)
        \mathrm dx.
\]
By calculating the partial derivative in the integrand, we obtain  the statement of the lemma.
\end{proof}

Our first result in the direction of analysis of the integral representation
provides asymptotic formulas for the zeros of \( \phi_0 \) and \(\widetilde
\phi_0 \) as \( w \to 0^+ \).

\begin{theorem}\label{theo:smallest_zero}
Let \( \phi_0(z, w; F(\cdot)) \) be given by
\[
    \phi_{0}(z, w; F(\cdot)) =
    \dfrac{1}{\sqrt{2 \pi w}}
    \int_{-\infty}^{+\infty}
    \exp \left(
        -\dfrac{x^2}{2w} - z e^{-i x}
    \right)
    F(zw^+e^{-ix})
    \mathrm dx \,
\]
where \( w^+ = w + \mathcal O(w^2) \)
and assume \( F(z) \) analytic, \( F(z) \neq 0 \) for \(0 < |z| \leq 1 \)
and \( \overline{F(z)}  = F(\overline z)\).
For a given $w$, let $\varrho_{j}(w)$ be the solution to the equation
\( \phi_0(z,w; F(\cdot))=0 \) that is the $j$-th closest to zero.
If $j\in \mathbb{N}$ is fixed, then we have
\begin{equation}\label{eq:varrho}
\varrho_{j}(w)
     =
    \dfrac{1}{ew}
    \left(
        1
        - \frac{a_j}{2^{1/3}} \, w^{2/3}
        - w \left( \frac{1}{6} - \dfrac{F'(1)}{F(1)} \right)
        + \bigO(w^{4/3})
    \right),
\end{equation}
as \(w\to 0^+\), where $a_j$ is the zero of the Airy function that is the $j$-th closest to
$0$. Moreover, we have the following estimates for the partial derivatives of
$\phi_0(z,w; F(\cdot))$ at these zeros:
\begin{equation}\label{eq:derivative_PHI}
    \partial_z \phi_0(\varrho_j(w),w; F(\cdot))
    \sim
    -\kappa_j \, w^{1/6}\,
    \exp\left(
        -\frac{1}{2w} + \frac{2^{-1/3}a_j}{w^{1/3}}
        - \dfrac{F'(1)}{F(1)}
    \right) F(1),
  \end{equation}
as \(w\to 0^+\), where
\(
    \kappa_j =
    \sqrt{2 \pi} \cdot 2^{2/3}
    e^{7/6}\ai'(a_j)
 \).
\end{theorem}

\begin{proof}
Observe that the main terms of $\phi_0(z,w; F)$ in
part (a) and part (b) of~\cref{prop:complex:general} cannot vanish.
However for part (c), the term
$K_0(\tau)$ can be zero, and this happens
when $\tau = 2^{-1/3}a_j + o(1),$ where $a_j$ is one of the
zeros of $\ai(z)$.
If we make $\tau$ vary in a
small interval around $2^{-1/3}a_j$, then the main term of
$\phi_0(z,w; F)$
changes its sign. So, by the intermediate value theorem, there
must be a zero \( z \) of \( \phi_0(z, w; F) \)
close to  $(1-2^{-1/3}a_jw^{2/3})/(ew)$.
In order to obtain the asymptotic formula of such a zero, let
\( \tau = 2^{-1/3} a_j + \varepsilon \) in part (c)
of~\cref{prop:complex:general}.
Then, if we let \( \varepsilon \to 0 \),
\[
    K_0(\tau) =
    \ai'(a_j) 2^{1/3} F(1) \varepsilon
    + 2^{1/3} w^{1/3} \ai'(a_j)
    \Big(
        F'(1) - \dfrac{F(1)}{6}
        + \bigO(\varepsilon)
    \Big)
    + \bigO(\varepsilon^2)
    ,
\]
from which we deduce
\( \varepsilon = w^{1/3} \left( \dfrac{1}{6} - \dfrac{F'(1)}{F(1)} + o(1)\right)  \).
This gives an asymptotic formula for the roots of the form
\begin{equation}\label{eq:zero_asymp}
    \varrho =
    \dfrac{1}{ew}
    \left(
        1
        - \frac{a_j}{2^{1/3}} \, w^{2/3}
        - w \left( \frac{1}{6} - \dfrac{F'(1)}{F(1)} \right)
        + \bigO(w^{4/3})
    \right),
\, \, \, \text{ as } \, w\to 0^+.
\end{equation}
To show that there is only one zero that satisfies this asymptotic
formula, we use the derivative from~\cref{lemma:derivative:explicit} with $r=0$:
\begin{align}
\label{eq:partial}
\begin{split}
    \partial_z \phi_0(z, w; F(\cdot))
    &= \dfrac{1}{z} \left(
        \phi_0\big(z, w; x \mapsto xF'(x)\big)
        - \dfrac{1}{w} \phi_0\big(z, w; x \mapsto xF(x)\big)
    \right)\\
    &=  - \dfrac{1}{zw} \phi_0\big(z, w; x \mapsto x F(x)\big)
    (1 + \bigO(w^{1/3})).
\end{split}
\end{align}
Now, suppose that there are two different zeros
$\varrho'$ and $\varrho''$ that both satisfy \eqref{eq:zero_asymp} for
the same $j$. Then by Rolle's theorem, there exists $C$ between
$\varrho'$ and $\varrho''$ such that
\( C \) is a zero of \( \partial_z \phi_0 \).
Consequently, \( C \) should have the form
\[
    C
    =
    \dfrac{1}{ew}
    \left(
        1
        - \frac{a_j}{2^{1/3}} \, w^{2/3}
        - w \left( \frac{1}{6} - \dfrac{F(1) + F'(1)}{F(1)} \right)
        + \bigO(w^{4/3})
    \right),
\]
as \( w \to 0^+ \).
Since $C$ stays between $\varrho'$ and $\varrho''$,
it satisfies the asymptotic formula \eqref{eq:zero_asymp}).
Those two asymptotic formulae that $C$ must satisfy
are different, because $F(1) = F(T(e^{-1})) \neq 0$.
Therefore, \( C \) cannot be a zero of \( \phi_0(z, w; F) \).

Now that we have established that there is only one zero of $\phi_0(z,w;F)$
satisfying \eqref{eq:zero_asymp} for each fixed $j\in
\mathbb{N}$ and sufficiently small $w$, we name it $\varrho_j(w).$
To estimate $\partial _z \phi_0(\varrho_j(w),w;F)$
as $w\to 0^+$, we make
use of the above relation~\eqref{eq:partial}
again with
\[
    z = \dfrac{1}{ew} (1 - \tau w^{2/3})
    \quad \text{and} \quad
    \tau = 2^{-1/3} a_j + w^{1/3} \left(
        \dfrac16 - \dfrac{F'(1)}{F(1)}
    \right).
\]
By expanding the term at \( w^{1/3} \) in \( K_0(\tau) \) in the asymptotic approximation of
\( \phi_0\big(z, w; x \mapsto x F(x)\big) \) from part (c) of~\cref{prop:complex:general},
we obtain the statement of the theorem.
\end{proof}

\begin{corollary}\label{theo:smallest_zero_simple}
For a given $w$, let $\widetilde \varrho_{j}(w)$ be the solution to the equation
$\widetilde\phi(z,w)=0$ that is the $j$-th closest to zero.  If $j\in \mathbb{N}$ is
fixed, then we have
\begin{equation}\label{eq:tilde_varrho}
\widetilde \varrho_{j}(w)
     =
    \dfrac{1}{ew}
    \left(
        1
        - \frac{a_j}{2^{1/3}} \, w^{2/3}
        - \frac{1}{6} w
        + \bigO(w^{4/3})
    \right),
\end{equation}
as \(w\to 0^+\), where $a_j$ is the zero of the Airy function that is $j$-th closest to
$0$. Moreover, we have the following estimates for the partial derivative of  $\widetilde\phi(z,w)$ at its zeros:
\begin{equation}\label{eq:derivative_PHI_simple}
    \partial_z \widetilde\phi(\widetilde \varrho_j(w),w)
     \sim
    -\widetilde\kappa_j \, w^{1/6}\,
    \exp\left(
        -\frac{1}{2w} + \frac{2^{-1/3}a_j}{w^{1/3}}
    \right)
\end{equation}

as \(w\to 0^+\), where
\(
    \widetilde\kappa_j =
    \sqrt{2 \pi} \cdot 2^{2/3}
    e^{11/12}\ai'(a_j).
\)
\end{corollary}

\begin{proof}
Once again, we make use of  the relation
\(
    \widetilde \phi(z, w)
    =
    \phi(z \sqrt{1 + w}, \log(1 + w)).
\)
This equation and the estimate \eqref{eq:varrho} yield
\[
    \widetilde \varrho_j(w)
    =
    \dfrac{\varrho_j(\log(1 + w))}{\sqrt{1 + w}}
    =
    \dfrac{1}{ew} \left(
        1 - \dfrac{a_j}{2^{1/3}} w^{2/3}
        - \dfrac{1}{6}w
        + \bigO(w^{4/3})
    \right),
    \quad
    \text{as }
    w \to 0^+.
\]
By plugging \( \widetilde \varrho_j(w) \) into \( \partial_z \widetilde \phi(z,
w)\), we obtain
\[
    \partial_z \widetilde \phi(\widetilde \varrho_j(w), w)
    =
    -
    \widetilde \phi\left(
        \frac{\widetilde \varrho_j(w)}{1+w}, w
    \right).
\]
Since \( e^{-w} = 1 - w + \bigO(w^2) \) and
\( (1 + w)^{-1} = 1 - w + \bigO(w^2) \), we can use the same
adjustment of variables
\( z (1+w)^{-1} = (ew)^{-1}(1 - \tau_1 w^{2/3}) \) as in~\cref{corollary:complex:general}, therefore we obtain the expression for the
derivative of \( \widetilde \phi \), which differs only by a factor \(
e^{-1/4} \).
\end{proof}

In addition to~\cref{theo:smallest_zero} and~\cref{theo:smallest_zero_simple} ,
we can also deduce asymptotic estimates of the zeros of the functions \(
\phi_1(z, w; F(\cdot)) \) and \(\widetilde \phi_1(z, w; F(\cdot)) \) as well as their derivatives using~\cref{prop:complex:general} and~\cref{corollary:complex:general}.
\begin{theorem}\label{theorem:zeros:phi:r}
	Let
    \[
        \phi_1(z, w; F(\cdot)) =
        \dfrac{1}{\sqrt{2 \pi w}}
        \int_{-\infty}^{+\infty}
        (1 - zw^+ e^{-i x})
        \exp \left(
            -\dfrac{x^2}{2w} - z e^{-i x}
        \right)
        F(zw^+ e^{-ix})
        \mathrm dx \, ,
    \]
    where \( w^+ = w + \mathcal O(w^2) \)
    and assume \(F(z)\neq 0\) for \( 0 < |z| \leq 1\)
    and \(\overline{F(z)}=F(\overline{z})\).
    For a given \( w \), let \( \varsigma_j(w) \) be the solution to the
    equation \( \phi_1(z, w; F(\cdot)) = 0 \) that is the \( j \)-th closest to zero.
    If \( j \in \mathbb N \) is fixed, then we have, as \( w \to 0^+ \),
    \[
        \varsigma_j(w) = \dfrac{1}{ew} \left(
            1 - \dfrac{a_j'}{2^{1/3}} w^{2/3} + \left(\dfrac{1}{6}+\frac{F'(1)}{F(1)}\right)w
            + \bigO(w^{4/3})
        \right),
    \]
    where \( a_j' \) is the zero of the derivative of the Airy function
    \( \ai'(\cdot) \) that is \( j \)-th closest to zero.

    The \( k \)-fold partial derivatives of \( \phi_1(z, w; F(\cdot)) \) with
    respect to the first variable for \( k > 0 \) at \(z=\varsigma_j(w)\) admit the following estimate
    as \( w \to 0^+ \):
    \begin{align*}
        \partial_z^k \phi_1(\varsigma_j(w), w; F(\cdot))
        & \sim
        (-1)^{k+1}
        k w^{1/2} \varkappa_{j}F(1)
        \exp \left(
            - \dfrac{1}{2 w} + \dfrac{2^{-1/3} a_j'}{w^{1/3}}
            - \left(\dfrac{1}{6}+\frac{F'(1)}{F(1)}\right) + k
        \right),
    \end{align*}
    where
    \(
        \varkappa_j = 2 \sqrt{2 \pi} a_j' \ai(a_j')
    \).
    More generally, when $z = (e w)^{-1} \left(1 + \bigO(w^{2/3})\right)$
    and \( \phi_r(z, w; F(\cdot)) \neq 0 \),
    we have, as $w \to 0^+$,
    \[
        \partial_z \phi_r(z, w; F(\cdot)) \sim - \dfrac{1}{zw} \phi_r(z, w; x \mapsto x F(x)).
    \]
\end{theorem}

\begin{remark}
The extra condition \(\overline{F(z)}=F(\overline{z})\) guarantees that the
function \(\phi_1(z, w; F(\cdot)) \) is real for real \(z\), and that the first few zeros
are also real. However, the asymptotic formulas in the theorem may still hold without this condition.
\end{remark}

\begin{proof}
The arguments we used to prove existence of the zeros in~\cref{theo:smallest_zero}
are still valid since  \( \phi_1(z, w) \) is real for real \(z\). The
approximations in parts (a) and (b) of~\cref{prop:complex:general} cannot be
zero, so the  roots can be found in part (c). The asymptotic formula for
\(\phi_1(z, w; F(\cdot)) \) can be obtained in the same manner as in the proof of~\cref{theo:smallest_zero}.

By using~\cref{lemma:derivative:explicit}, we can obtain the derivative of \( \phi_r \) for any $r$, which can be inductively expanded to the \( k \)-fold derivative, by taking into account only the dominant term. More specifically,
using the estimates in~\cref{prop:complex:general}, we find that for \(ezw=1+\mathcal{O}(w^{2/3}),\) the dominant term comes from the middle term of~\eqref{eq:partial:next} in brackets. Hence, firstly,
\[
    \partial_z \phi_r(z, w; F(\cdot)) =
    -\frac{1}{zw}\phi_{r}(z, w;x \mapsto xF(x))(1 +\mathcal{O} (w^{1/3})).
\]
More generally, one can inductively prove from \eqref{eq:partial:next} and \cref{prop:complex:general}, skipping the tedious calculations, that
\begin{equation}\label{eq:derivatives}
    \partial_z^k \phi_r(z, w; F(\cdot)) =
    \frac{(-1)^k}{(zw)^k}
    \phi_{r}(z, w; x \mapsto x^kF(x))(1 +\mathcal{O} (w^{1/3})).
\end{equation}

Now, to  deduce the asymptotic formula for \(\partial_z^k
\phi_1(\varsigma_j(w), w; F(\cdot))\), we just need to estimate
\(\phi_1(\varsigma_j(w), w; x \mapsto x^kF(x))\)
using~\cref{prop:complex:general}, part (c).
\end{proof}

Once again, the version of~\cref{theorem:zeros:phi:r} for \(\widetilde
\phi_1(z,w; F(\cdot))\) is given as a corollary below.

\begin{corollary}\label{corollary:zeros:phi:r}
Let
    \(
    \widetilde \phi_1(z, w; F(\cdot))
    \) be given by \eqref{eq:phi:tilde:phi} when $r=1$, and assume that \(F(\cdot)\) satisfies the same conditions as in~\cref{theorem:zeros:phi:r}.
    For a given \( w \), let \( \widetilde \varsigma_j(w) \) be the solution to the
    equation \( \widetilde \phi_1(z, w; F(\cdot)) = 0 \) that is the \( j \)-th closest to zero.
    If \( j \in \mathbb N \) is fixed, then we have, as \( w \to 0^+ \),
    \[
        \widetilde \varsigma_j(w) = \dfrac{1}{ew} \left(
            1 - \dfrac{a_j'}{2^{1/3}} w^{2/3} + \left(\dfrac{1}{6}+\frac{F'(1)}{F(1)}\right)w
            + \bigO(w^{4/3})
        \right),
    \]
    where \( a_j' \) is the zero of the derivative of the Airy function
    \( \ai'(\cdot) \) that is \( j \)-th closest to zero.

    The \( k \)-fold partial derivatives of \( \widetilde \phi_1(z, w; F(\cdot)) \) with
    respect to the first variable for \( k > 0 \) at \(z=\widetilde \varsigma_j(w)\) admit the following estimate
    as \( w \to 0^+ \):
    \begin{align*}
        \partial_z^k \widetilde \phi_1(\widetilde \varsigma_j(w), w; F(\cdot))
        & \sim
        (-1)^{k+1}
        k w^{1/2} \varkappa_{j}F(1)
        \exp \left(
            - \dfrac{1}{2 w} + \dfrac{2^{-1/3} a_j'}{w^{1/3}}
            - \left(\dfrac{5}{12}+\frac{F'(1)}{F(1)}\right) + k
        \right),
    \end{align*}
    where
    \(
        \varkappa_j
    \) is as defined in~\cref{theorem:zeros:phi:r}.
  \end{corollary}

  \begin{proof}
    Let us remark that in the given range of asymptotic approximation,
    $\widetilde \phi_1$ and $\phi_1$ are related by
      \[
          \widetilde \phi_1(z, w; F(\cdot))
          =
          \phi_1(z \sqrt{1 + w}, \log(1 + w); F(\cdot))
          \times (1 + \mathcal O(w^{2/3}))
          \, ,
      \]
      which implies that the value \( \tau \) in~\cref{prop:complex:general}
      changes at most by \( w^{2/3} \) and therefore,
    \[
      \widetilde \varsigma_j(w)
    =
    \dfrac{\varsigma_j(\log(1 + w))}{\sqrt{1 + w}}
    (1 + \mathcal O(w^{4/3})) \, .
  \]
  Thus,
\begin{align*}
    \partial_z^k \widetilde \phi_1(\widetilde \varsigma_j(w), w; F(\cdot))
    & =
    (1+w)^{k/2}\,\partial_z^k
    \phi_1\left(\sqrt{1+w} \,\widetilde \varsigma_j(w),
    \log(1+w); F(\cdot)\right)
    (1 + \mathcal O(w^{2/3}))
    \\
    & =
    (1+w)^{k/2}\,\partial_z^k
    \phi_1\left(\varsigma_j(\log(1+w)),
    \log(1+w); F(\cdot)\right)
    (1 + \mathcal O(w^{2/3})) \, .
\end{align*}
 The results can now be deduced from~\cref{theorem:zeros:phi:r}.
  \end{proof}
\begin{remark}
    The asymptotic formula in \eqref{eq:derivatives} in the proof of~\cref{theorem:zeros:phi:r}
    also applies for \(ezw=1+\mathcal{O}(w^{2/3})\) such that \( \phi_r(z, w;
    F(x)) \neq 0 \).
    This will be useful in further analysis of the
    structure of a random digraph near the point of its phase transition in the
    case where we would like to capture several complex components at the same
    time.
\end{remark}

        \section{Analysis of the external integrals}
        \label{section:external:integral}

As pointed out in~\cref{section:models}, the probability
that a random multidigraph \( D \in \DiGilbMulti(n, p) \)
belongs to a family \( \mathcal H \) is
\[
    \mathbb P_{\mathcal{H}}(n,p)=
    e^{-pn^2/2} n! [z^n] \widehat H(z,p),
\]
where \( \widehat H(z,p) \) is the multi-graphic generating function
of the required family \( \mathcal H \).
Similar statements hold for simple digraphs and strict digraphs.
Since the coefficient extraction operation can be expressed
as a complex contour integral using Cauchy's integral formula,
we establish the asymptotic values for these integrals.
As usual, we start with the multidigraph model,
where the analysis is simplest,
before moving to simple and strict digraphs in \cref{section:external:integration:simple}.

        \subsection{The case of multidigraphs}
        \label{section:external:integration:multi}

One of the basic expressions arising in
\cref{section:repr:alt} and
in part (a) of~\cref{prop:complex:general}
when we finally approach the coefficient extraction
corresponding to the subcritical range of
the phase transition,
is proportional to
\( {e^{U(zw)/w} (1 - T(zw))^{k}} \) for various half-integer values of \( k \).
It is, therefore, useful to have the asymptotics of the coefficient extraction
operation for this generating function in this range.
We also need to explicitly specify the region with respect to the variable \( z \)
in which the approximation is valid in order to extract the coefficient asymptotics.

In the following lemma, we want to remain in the subcritical case, but approach to the critical case as tightly as possible, while keeping the same technique. A variant of this lemma has been considered for the case when \( \lambda \) is not necessarily separated from $1$, see~\cite[Theorem 3.2]{herve2011random}, and also~\cite[Lemma 3, Case $\mu \to - \infty$]{Janson93}. For completeness, we provide a detailed proof with a careful treatment of error terms and tails of the integrals.

\begin{lemma}
    \label{lemma:analytic:first}
    Let \( r \) be a fixed real value, and let \( wn = \lambda \), where \(
    \lambda \in (0, 1) \).
    Let $F(\cdot)$ be an entire function
    satisfying the conditions in Lemma~\ref{prop:complex:general}.
    Suppose that a function \( H(z, w) \) satisfies an
    approximation
    \[
        H(z, w) \sim e^{U(zw)/w} (1 - T(zw))^{1/2 - r}F(T(zw))
    \]
    uniformly for \( |zw| \leq \lambda e^{-\lambda} \).
    Then,
    when \( 0 < \lambda < 1 \) and \( (1 - \lambda)n^{1/3} \to \infty \)
    as \( n \to \infty \), we have
    \[
        e^{-wn^2/2} n!
        [z^n] H(z, w)
        \sim
        (1 - \lambda)^{1 - r}F(\lambda)
        .
    \]
\end{lemma}
\begin{proof}
We start by representing the coefficient extraction using Cauchy's integral formula:
\[
    [z^n] H(z, w)
    =
    \dfrac{1}{2 \pi i}
    \oint_{|z| = R}
    H(z, w)
    \dfrac{\mathrm dz}{z^{n+1}},
\]
where the value \( R \) is chosen to be $\lambda e^{-\lambda} / w$
so that $H(z, w)$ can be replaced by its uniform approximation
\[
    [z^n] H(z, w)
    \sim
    \dfrac{1}{2 \pi i}
    \oint_{|z| = R}
    e^{U(zw)/w} (1 - T(zw))^{1/2 - r}F(T(zw))
    \dfrac{\mathrm dz}{z^{n+1}}.
\]
We apply the change of variables $t = T(z w)$,
using the relations
$z T'(z) = T(z) / (1 - T(z))$,
$z = t e^{-t} / w$
and $U(z) = T(z) - T(z)^2/2$,
and deform the integration path to be a circle of radius $\lambda$ to get
\begin{equation}
  \label{eq:asymp-equiv-path-lambda}
    [z^n] H(z, w)
    \sim
    \dfrac{w^n}{2 \pi i}
    \oint_{|t| = \lambda}
    e^{(t - t^2/2)/w} e^{n t} (1 - t)^{3/2 - r}F(t)
    \dfrac{\mathrm dt}{t^{n+1}}.
\end{equation}
We will see in a moment that this choice of integration path
is motivated by the fact that $\lambda$
is an approximation of the saddle point.
Replacing $w$ with $\lambda / n$, the expression becomes
\[
    [z^n] H(z, w)
    \sim
    \dfrac{(\lambda / n)^n}{2 \pi i}
    \oint_{|t| = \lambda}
    e^{n f(t, \lambda)}
    (1 - t)^{3/2 - r}F(t)
    \dfrac{\mathrm dt}{t^{n+1}},
\]
where
\begin{equation}
\label{eq:f:of:t:lambda}
    f(t, \lambda) = (t - t^2/2)/\lambda + t\, .
\end{equation}
We parametrise the integration path so that $t = \lambda e^{i \theta}$:
\[
    [z^n] H(z, w)
    \sim
    \dfrac{n^{-n}}{2 \pi}
    \int_{-\pi}^{\pi}
    e^{n (f(\lambda e^{i \theta}, \lambda) - i \theta)}
    (1 - \lambda e^{i \theta})^{3/2 - r}
    F(\lambda e^{i \theta})
    \mathrm d \theta.
\]
We have the following Taylor expansions at $\theta = 0$:
\begin{align*}
    f(\lambda e^{i \theta}, \lambda) - i \theta
    &=
    \frac{\lambda e^{i \theta} - \lambda^2 e^{2 i \theta}/2}{\lambda}
    + \lambda e^{i \theta}
    =
    1 + \frac{\lambda}{2}
    - (1 - \lambda) \frac{\theta^2}{2}
    + i (3 \lambda - 1) \frac{\theta^3}{6}
    + \bigO(\theta^4),
    \\
    1 - \lambda e^{i \theta}
    &=
    (1 - \lambda)
    \left(
        1
        - i \frac{\lambda}{1 - \lambda} \theta
        + \frac{\bigO(\theta^2)}{1 - \lambda}
    \right).
\end{align*}
The path of integration has been chosen to cross the saddle point,
which explains why the first derivative of
$f(\lambda e^{i \theta}, \lambda) - i \theta$
vanishes at $0$.
We now cut the integral into the central part and the tail.
We gather requirements on a function
$\epsilon$ of $n$ and $\lambda$ such that
the central part $\theta \in (-\epsilon, \epsilon)$
contains the dominant contribution of the integral,
while the tail $\theta \in (-\pi, \pi] \setminus (-\epsilon, \epsilon)$
is negligible.
This $\epsilon$ will be chosen later in the proof.

Let us start with the central part estimate.
In typical applications of the saddle point method,
one chooses the central part large enough so that
the second term in the Taylor expansion of $n f(\lambda e^{i \theta}, \lambda)$
tends to infinity, but small enough so that the third term goes to $0$.
In the present case, we have to be more cautious
and choose $\epsilon$ so that the fourth term goes to $0$.
The two conditions on $\epsilon$ are then
\begin{equation} \label{eq:analytic:first:epsilon}
    n (1 - \lambda) \epsilon^2 \to +\infty
    \quad \text{and} \quad
    n \epsilon^4 \to 0.
\end{equation}
Taking the quotient, they imply that $(1 - \lambda)^{-1} \epsilon^2$
tends to $0$.
Thus, for $\theta$ in the central part $(-\epsilon, \epsilon)$,
the error terms of the Taylor expansions of $f(\lambda e^{i \theta}, \lambda)$
and $1 - \lambda e^{i \theta}$ tend to $0$,
so we obtain the following central part estimate
\begin{align*}
    \int_{-\epsilon}^{\epsilon}
    e^{n (f(\lambda e^{i \theta}, \lambda) - i \theta)}
    (1 - \lambda e^{i \theta})^{3/2 - r}     F(\lambda e^{i \theta})
    \mathrm d \theta
    \sim\ &
    e^{n(1 + \lambda / 2)}
    (1 - \lambda)^{3/2 - r} F(\lambda)
    \\&
    \times \int_{-\epsilon}^{\epsilon}
    e^{- n (1 - \lambda) \theta^2 / 2
       + i n (3 \lambda - 1) \theta^3 / 6}
    \left( 1 - i \frac{\lambda}{1 - \lambda} \theta \right)
    \mathrm d \theta.
\end{align*}
After the change of variables $x = \sqrt{n (1 - \lambda)} \theta$,
the integral of the right-hand side becomes
\[
    \frac{1}{\sqrt{n (1 - \lambda)}}
    \int_{-\sqrt{n (1 - \lambda)} \epsilon}^{\sqrt{n (1 - \lambda)} \epsilon}
    e^{- x^2 / 2}
    \exp \left(
        \frac{i}{\sqrt{n}}
        \frac{3 \lambda - 1}{(1 - \lambda)^{3/2}}
        \frac{x^3}{6}
    \right)
    \left(
        1 - \frac{i \lambda x}{\sqrt{n} (1 - \lambda)^{3/2}}
    \right)
    \mathrm d x.
\]
The domain of integration tends to infinity
according to the conditions~\eqref{eq:analytic:first:epsilon},
and the integrand converges pointwise to $e^{-x^2/2}$
because $n^{1/3} (1 - \lambda)$ tends to infinity.
Furthermore, the integrand is dominated by $e^{-x^2/2} (1 + |x|)$
for $n$ large enough so that $\sqrt{n} (1 - \lambda)^{3/2} > 1$.
Thus, by the dominated convergence theorem,
the last expression is asymptotically equal to
\[
    \frac{1}{\sqrt{n (1 - \lambda)}}
    \int_{-\infty}^{+\infty}
    e^{-x^2/2}
    \mathrm d x
    =
    \sqrt{\frac{2 \pi}{n (1 - \lambda)}}.
\]
Thus, the central part is asymptotically equal to
\[
    \int_{-\epsilon}^{\epsilon}
    e^{n (f(\lambda e^{i \theta}, \lambda) - i \theta)}
    (1 - \lambda e^{i \theta})^{3/2 - r}F(\lambda e^{i \theta})
    \mathrm d \theta
    \sim
    \sqrt{2 \pi / n}
    e^{n(1 + \lambda / 2)}
    (1 - \lambda)^{1 - r}
    F(\lambda).
\]
We now bound the tail. We only present the proof
for $\theta \in (\epsilon, \pi]$,
since the other half is symmetrical. Firstly,
\[
    \left|
    \int_{\epsilon}^{\pi}
    e^{n (f(\lambda e^{i \theta}, \lambda) - i \theta)}
    (1 - \lambda e^{i \theta})^{3/2 - r}
    \mathrm d \theta
    \right|
    \leq
    \int_{\epsilon}^{\pi}
    e^{n \Real(f(\lambda e^{i \theta}, \lambda) - i \theta)}
    |1 - \lambda e^{i \theta}|^{3/2 - r}
    \mathrm d \theta.
\]
The real part of the exponent is
\[
    \Real(f(\lambda e^{i \theta}, \lambda) - i \theta) =
    (1 + \lambda) \cos(\theta) - \frac{\lambda}{2} \cos(2 \theta).
\]
Its derivative
\[
    (2 \lambda \cos(\theta) - 1 - \lambda) \sin(\theta)
\]
is negative on $(\epsilon, \pi)$,
so the real part is strictly decreasing
and its maximum is reached only at $\theta = \epsilon$:
\[
    \Real(f(\lambda e^{i \theta}, \lambda) - i \theta) \leq
    (1 + \lambda) \cos(\epsilon) - \frac{\lambda}{2} \cos(2 \epsilon)
    =
    1 + \frac{\lambda}{2}
    - (1 - \lambda) \frac{\epsilon^2}{2}
    + \bigO(\epsilon^4).
\]
We also have the bound
\[
    \sup_{\theta \in (\epsilon, \pi)}
    |1 - \lambda e^{i \theta}|^{3/2 - r}|F( \lambda e^{i \theta})|
    \leq
    \begin{cases}
        (1 + 2 \lambda)^{3/2-r}F(\lambda) & \text{if } 3/2 - r \geq 0,\\
        (1 - \lambda)^{3/2-r} F(\lambda) & \text{otherwise,}
    \end{cases}
\]
which in both cases is a $\bigO((1 - \lambda)^{3/2-r})$.
Applying those two bounds to the integral, we obtain
\begin{align*}
    \left|
    \int_{\epsilon}^{\pi}
    e^{n (f(\lambda e^{i \theta}, \lambda) - i \theta)}
    (1 - \lambda e^{i \theta})^{3/2 - r}F(\lambda e^{i\theta})
    \mathrm d \theta
    \right|
    &=
    e^{n \left( 1 + \frac{\lambda}{2}
    - (1 - \lambda) \frac{\epsilon^2}{2}
    + \bigO(\epsilon^4) \right)}
    \bigO \left( (1 - \lambda)^{3/2-r} F(\lambda)\right)
    \\&=
    \frac{e^{n(1 + \frac{\lambda}{2})}}{\sqrt{n}}
    (1 - \lambda)^{1 - r}F(\lambda)
    \sqrt{n (1 - \lambda)}
    e^{-n (1 - \lambda) \frac{\epsilon^2}{2}}
    \bigO(1).
\end{align*}
Thus, the tail is negligible compared to the central part if
\[
    n (1 - \lambda) e^{-n (1 - \lambda) \epsilon^2}
\]
tends to $0$.
The two conditions of~\eqref{eq:analytic:first:epsilon}
and this third condition are satisfied,
when $(1 - \lambda) n^{1/3}$ tends to infinity,
by choosing
\[
    \epsilon = \sqrt{\frac{2 \log(n (1 - \lambda))}{n (1 - \lambda)}}.
\]
Indeed, we then have
\begin{align*}
    n (1 - \lambda) \epsilon^2 &=
    2 \log(n (1 - \lambda))
    \to +\infty,
    \\
    n \epsilon^4 &=
    \frac{\log(n (1 - \lambda))}{\sqrt{n (1 - \lambda)}}
    \frac{4}{(n^{1/3} (1 - \lambda))^{3/2}}
    \to 0,
    \\
    n (1 - \lambda) e^{- n (1 - \lambda) \epsilon^2} &=
    \frac{1}{n (1 - \lambda)}
    = \bigO(n^{-2/3})
    \to 0.
\end{align*}
To conclude, using Stirling's formula, we have
\[
  \begin{split}
    e^{-w n^2/2} n! [z^n] H(z,w) &\sim
    e^{-w n^2/2} n^n e^{-n} \sqrt{2 \pi n}
    \frac{n^{-n}}{2 \pi}
    \sqrt{\frac{2 \pi}{n}}
    e^{n (1 + \lambda / 2)}
    (1 - \lambda)^{1 - r}F(\lambda)\\
    &\sim (1 - \lambda)^{1 - r}F(\lambda)\, .
    \end{split}
\]
\end{proof}

\begin{remark}
The above lemma has a probabilistic interpretation.
Let us call a random multigraph from $\GrERMulti(n, \lambda/n)$
\emph{subcritical} if $w := \lambda/n$ stays in a closed subinterval of $(0,1)$
or $w$ tends to $1$ as $1 + \mu n^{-1/3}$ with $\mu \to -\infty$.
The total weight of multigraphs on $n$ vertices
that contain only trees and unicycles,
with a variable $u$ marking the unicycles, is
\[
    n! [z^n] e^{U(z w)/w}
    \exp \left(\frac{u}{2} \log \left( \frac{1}{1 - T(z w)} \right) \right).
\]
Therefore we observe that the underlying generating function of~\cref{lemma:analytic:first} is related to the function
\[
    t \mapsto e^{-w n^2/2} n! [z^n] e^{U(z w)/w}
    \exp \left(\frac{e^t}{2} \log \left( \frac{1}{1 - T(z w)} \right) \right),
\]
which is the moment generating function of the number of unicycles
in multigraphs from $\GrERMulti(n, \lambda/n)$
conditioned on containing only trees and unicycles.
The conclusion of the last lemma is that
this function tends to $\exp \left( (e^t - 1) \frac{1}{2} \log \frac{1}{1 - \lambda} \right)$ (the value $r$ has been replaced by $(1 + e^t)/2$),
which implies that the limit law of the number of unicycles is Poisson
of parameter $\frac{1}{2} \log \frac{1}{1 - \lambda}$.
This is no surprise, as we know since \cite{Erdos1959}
that subcritical (multi)graphs contain typically only trees and unicycles,
and that the number of unicycles follows a Poisson limit law.
\end{remark}

In order to treat the critical and the supercritical phases of the phase
transition, we need to show that
for various
different expressions involving generating functions of digraph families,
when \( w \to 0^+ \),
the contribution of the external integral
over the circle \( z = \rho e^{iu} \) can be approximated by its contribution
around only the local part \( z = \rho \pm 0i \), even if \( \rho \) is not a
saddle point. Generally speaking, the (multi-)graphic generating function of the multidigraph
families of interest is expressed as a product of various functions of the form
\( \phi_r(z, w) \) (see~\cref{section:symbolic:method}), and therefore, we need to establish
the asymptotic behaviour of the complex contour integral of such a product.

\begin{lemma} \label{lem:cauchy}
    Let \( \phi_r \) be defined as in~\eqref{eq:phi:tilde:phi},
    \[
        \phi_r(z, w; F(\cdot)) = \dfrac{1}{\sqrt{2 \pi w}}
        \int_{-\infty}^{+\infty}
        (1 - zw e^{-ix})^{r}
        \exp \left(
            - \dfrac{x^2}{2 w}
            -  e^{-ix}
        \right)
        F(z w e^{- i x})
        \mathrm dx.
    \]
    Let \( (r_i)_{i=1}^k \) denote integer values,
    \( (p_i)_{i=1}^k \) be integer values summing to $1$,
    and
    \[
        R = \sum_{i=1}^k r_i p_i
        \quad \text{and} \quad
        \xi(z, w) = \prod_{i=1}^k \phi_{r_i}(z, w; F_i(\cdot))^{p_i},
    \]
    where $F_i(z)$ does not vanish on $\mathbb{C} \setminus \{0\}$
    for $i \in \{1, \ldots, k\}$.

    \noindent \textbf{Part (a)}
    If \( w = \lambda / n \) with $\lambda \in (0, 1 - n^{\epsilon - 1/3})$
    for some $\epsilon > 0$, then
    \[
        e^{-w n^2/2} n ! [z^n] \frac{1}{\xi(z,w)}
        \sim
        (1 - \lambda)^{1 - R} \prod_{i=1}^k F_i(\lambda)^{-p_i}.
    \]

    \noindent \textbf{Part (b)}
    If \( w = n^{-1}(1 + \mu n^{-1/3}) \)
    with \( \mu  = \smallo(n^{1/12})\),
    then
    \[
        e^{-wn^2/2} n!
        [z^n] \dfrac{1}{\xi(z, w)}
        \sim
        \frac{(-1)^R
        n^{(R-1)/3}
        2^{-(R+1)/3}}{ \prod_{i=1}^k F_i(1)^{p_i}}
        \dfrac{1}{2 \pi i}
        \int_{-i \infty}^{i \infty}
        \dfrac{
            \exp (-\tfrac{\mu^3}{6} - \mu s)
        }
        {
            \prod_{i=1}^k \ai(r_i; -2^{1/3}s)^{p_i}
        }
        \mathrm ds.
    \]

    \noindent \textbf{Part (c)}
    Consider a fixed real value \( \theta \) such that
    \( \ai(r_i; 2^{1/3}\theta) \neq 0\) for $i \in \{1, \ldots, k\}$,
    and define \( \rho = (1 - \theta w^{2/3}) (ew)^{-1} \).
    Then
    \begin{equation}
    \label{eq:cauchy_up}
        \dfrac{1}{2 \pi i}
        \oint_{|z| = \rho}
        \dfrac{1}{
            \xi(z, w)
        }
        \dfrac{\mathrm dz}{z^{n+1}}
        =
        \mathcal O \left(
            \dfrac{w^{2/3}}{\xi(\rho, w) \rho^n}
        \right).
    \end{equation}
\end{lemma}

\begin{proof}
    \textbf{Part (a)}
    Consider $|z w| \leq \lambda e^{-\lambda}$
    and assume $\lambda \in (0, 1 - n^{\epsilon - 1/3})$.
    Then
    \[
        |e z w| \leq \lambda e^{1 - \lambda} < 1
    \]
    and
    \[
        1 - |T(zw)|
        \geq
        1 - T(|z w|)
        \geq
        1 - T(\lambda e^{-\lambda})
        \geq
        1 - \lambda
        \geq
        n^{\epsilon - 1/3}
        \geq
        \lambda^{\epsilon - 1/3} w^{1/3 - \epsilon}
        \geq
        w^{1/3 - \epsilon}.
    \]
    Thus, Part~$(a)$ of \cref{prop:complex:general}
    is applicable and implies
    \[
        \phi_r(z,w; F(\cdot))
        \sim
        e^{-U(z w) / w}
        (1 - T(z w))^{r - 1/2}
        F(T(z w)).
    \]
    Taking the product, using the fact that $\sum_{i=1}^k p_i = 1$ and the definition of $R$,
    we obtain
    \[
        \xi(z,w) \sim
        e^{- U(z w) / w}
        (1 - T(z w))^{R - 1/2}
        \prod_{i=1}^k F_i(T(z w))^{p_i}.
    \]
    We then apply \cref{lemma:analytic:first} to conclude
    \[
        e^{-w n^2/2} n ! [z^n] \frac{1}{\xi(z,w)}
        \sim
        (1 - \lambda)^{1 - R} \prod_{i=1}^k F_i(\lambda)^{-p_i}.
    \]

    \textbf{Part (c)}
    Using the parametrisation \( z = \rho e^{iu} \) with \( u \in [-\pi, \pi] \),
    the integral expression from the statement is rewritten as
    \[
        \dfrac{1}{2 \pi i} \oint_{|z| = \rho}
        \dfrac{1}{\xi(z, w)} \frac{\mathrm dz}{z^{n+1}}
        =
        \dfrac{1}{2 \pi \xi (\rho, w) \rho^n}
        \int_{-\pi}^{\pi}
        \dfrac{\xi(\rho, w) e^{-iun}}{\xi(\rho e^{iu}, w)}
        \mathrm du.
    \]
    Let $\alpha$ denote a positive constant fixed later in the proof
    and depending only on $\theta$,
    and \( \varepsilon \) a small enough positive constant.
    The integration path is split into three regions:
    \begin{enumerate}[label=(\roman*)]
    \item \( u \in [-\pi, -w^{2/3-\varepsilon}] \cup [w^{2/3-\varepsilon},\pi] \),
    \item \( u \in [-w^{2/3-\varepsilon}, -\alpha w^{2/3}] \cup [\alpha w^{2/3}, w^{2/3-\varepsilon}] \)
     and
    \item \( u \in [-\alpha w^{2/3}, \alpha w^{2/3}] \).
    \end{enumerate}
    This partitioning corresponds to parts (a), (b) and (c) of~\cref{prop:complex:general},
    where \( \varepsilon \) in parts (a) and (b) is chosen to be
    smaller than \( 1/ 6 \).
    We shall also use the fact that the asymptotic expressions of parts (b) and (c) overlap
    when \( w^{-1/3}u \) is in a bounded closed interval of \( \mathbb R \).

    \textbf{First region.}
    The series $T(z)$ has nonnegative coefficients at the origin and is aperiodic.
    According to the Daffodil Lemma \cite[p.266]{FSBook},
    the maximum of its absolute value on any circle centred at the origin
    is reached at the crossing with the positive real line.
    Thus, in the first case, writing \( z = \rho e^{iu} \), we have
    \[
        |T(z w)| \leq
        \left|
        T \left(
            \frac{1 - \theta w^{2/3}}{e} e^{i w^{2/3-\varepsilon}}
        \right)
        \right|.
    \]
    The argument is converging to the singularity $e^{-1}$ of $T(z)$,
    so we apply the Newton--Puiseux expansion from \cref{lemma:complex:plane:trees}
    to deduce
    \begin{equation} \label{eq:T:zw:first:part}
        \left|
            T \left(
            \frac{1 - \theta w^{2/3}}{e} e^{i w^{2/3-\varepsilon}}
            \right)
        \right|
        =
        1 - w^{1/3 - \varepsilon/2} + \bigO(w^{1/3 + \varepsilon/2}).
    \end{equation}
    Thus, the conditions of part (a) of \cref{prop:complex:general} are satisfied
    and it provides an estimate for the denominator of the integrand.
    The numerator is approximated using part (c),
    whose conditions are satisfied as well.
    We obtain
    \[
        \left|
        \dfrac
        {\phi_r(\rho, w; F(\cdot))}
        {\phi_r(\rho e^{iu}, w; F(\cdot))}
        \right|
        \sim
        \sqrt{2 \pi}
        C_r
        \frac{w^{r/3-1/6} F(1)}{F(T(\rho w))}
        \exp\left(
            \dfrac{1}{w}
            \Real\left(U(zw) - \frac12 \right)
            + \theta w^{-1/3}
        \right)
        |1 - T(zw)|^{1/2 - r}\, ,
    \]
    where \( C_r = |\ai(r; 2^{1/3} \theta)| 2^{r/3+1/3} \),
    \( z = \rho e^{iu} \). By multiplying the absolute values of the ratios, we
    obtain, since \( \sum_{i=1}^k p_i = 1 \),
    \[
        \left|
        \dfrac
        {\xi(\rho, w)e^{-iun}}
        {\xi(\rho e^{iu}, w)}
        \right|
        =
        \mathcal O\left(
            w^{R/3-1/6}
            \exp\left(
                \dfrac{1}{w}
                \Real\left(U(zw) - \frac12 \right)
                + \theta w^{-1/3}
            \right)
            |1 - T(zw)|^{1/2 - R}
        \right),
    \]
    where \( R = \sum_{i=1}^k r_i p_i \),
    and the implicit constant is independent of \( n \), but can be expressed
    using \( (r_i)_{i=1}^k \) and \( (p_i)_{i=1}^k \).

    By \( \Omega_+(f(w)) \), we denote some function that is bounded from below by \( f(w) \) times a positive constant, as \( w \to 0^+ \).
    As we saw, when \( |u| \geq w^{2/3 - \varepsilon} \), \( |1 - T(zw)| \) is bounded away from zero, and since $T(zw)$ is in a disk of radius less than $1$ centred at the origin, we also have
    \[
        |1 - T(zw)| \geq 1 - |T(w \rho e^{i w^{2/3-\varepsilon}})| = \Omega_+(w^{1/3 - \varepsilon/2}).
    \]
    Since
    \( U(z) = \tfrac12 (1 - (1 - T(z))^2) \),
    the absolute value of the ratio is of order
    \[
        \left|
        \dfrac
        {\xi(\rho, w)e^{-iun}}
        {\xi(\rho e^{iu}, w)}
        \right|
        =
        \mathcal O\left(
            w^{R/3 - 1/6}
            \exp\left(
                -
                \dfrac{1}{2w}
                \Real\left( (1 - T(\rho e^{iu}w))^2 \right)
                + \theta w^{-1/3}
            \right)
            |1-T(zw)|^{1/2-R}
        \right).
    \]
    Furthermore, depending on the sign of $1/2 - R$,
    the multiple $|1 - T(zw)|^{1/2-R}$ is bounded above by
    $\mathcal O(1)$ if $R \leq 1/2$, or by
    $\mathcal O(w^{(1/3-\varepsilon/2)(1/2-R)})$ if $R > 1/2$.
    Since
    \( \Real\big( (1 - T(\rho e^{iu} w))^2\big)
    \geq (1 - |T(\rho e^{iu} w)|)^2
     = \Omega_+(w^{2/3 - \varepsilon}) \),
    we have
    \[
        -\frac{1}{2 w} \Real((1- T(\rho e^{iu} w))^2)+\theta w^{-1/3} \leq
        - \dfrac{\Omega_+(w^{2/3-\varepsilon})}{2w} + \mathcal O(w^{-1/3})
        =
        -\Omega_+(w^{-1/3-\varepsilon})
    \]
    uniformly for \( u \in [-\pi, -w^{2/3-\varepsilon}]\cup [w^{2/3-\varepsilon}, \pi] \).
    Therefore, the above integral tends to zero at exponential speed.

    \textbf{Second region.}
    The function $x \mapsto |T(z e^{i x})|$ is strictly increasing
    on $(-\pi, 0)$ and strictly decreasing on $(0, \pi)$.
    Furthermore, we have $T(\bar{z}) = \overline{T(z)}$.
    Thus, in the second region $u \in [-w^{2/3-\varepsilon}, - \alpha w^{2/3}] \cup [\alpha w^{2/3}, w^{2/3-\varepsilon}]$, we have
    \[
        \left|
            T \left(
                \frac{1 - \theta w^{2/3}}{e} e^{i w^{2/3 - \varepsilon}}
            \right)
        \right|
        \leq
        |T(z w)|
        \leq
        \left|
            T \left(
                \frac{1 - \theta w^{2/3}}{e} e^{i \alpha w^{2/3}}
            \right)
        \right|.
    \]
    We already derived an approximation of the lower bound in \eqref{eq:T:zw:first:part}.
    For the upper bound, we apply again the Newton--Puiseux expansion
    from \cref{lemma:complex:plane:trees} and obtain
    \begin{equation} \label{eq:T:part:two}
        T \left(
           \frac{1 - \theta w^{2/3}}{e} e^{i w^{2/3 - \varepsilon}}
        \right)
        =
        1 - \sqrt{2} \sqrt{\theta - i \alpha} w^{1/3}
        + \bigO(w^{2/3}),
    \end{equation}
    so
    \[
        \left| T \left(
                \frac{1 - \theta w^{2/3}}{e} e^{i \alpha w^{2/3}}
        \right) \right|
        =
        1 - \sqrt{2} \Real\left(\sqrt{\theta - i \alpha}\right) w^{1/3}
        + \bigO(w^{2/3}).
    \]
    When $\alpha$ tends to infinity,
    then the absolute value of $\sqrt{\theta - i \alpha}$ tends to infinity,
    while its argument tends to $-\pi/4$.
    Therefore, its real part tends to infinity,
    and we can choose $\alpha$ large enough so that
    $\sqrt{2} \Real\left(\sqrt{\theta - i \alpha}\right)$
    is greater than $1$.
    Thus, to estimate the denominator and numerator,
    the hypothesis of parts (b) and (c) of \cref{prop:complex:general}
    are satisfied when choosing $0 < \varepsilon < 1/6$ and we obtain
    \[
        \left|
        \dfrac
        {\phi_r(\rho, w; F(\cdot))}
        {\phi_r(\rho e^{iu}, w; F(\cdot))}
        \right|
        \sim
        \left|
        \dfrac{
            \ai(r; 2^{1/3} \theta)
        }{
            {\ai}\big(r; \frac{(1 - T(zw))^2}{2^{2/3} w^{2/3}}\big)
        }
        \exp \left(
            \dfrac{1}{w} \left(
                U(zw) - \dfrac12
            \right)
            + \theta w^{-1/3}
            - \dfrac{1}{3w} (1 - T(zw))^3
        \right)
        \right|.
    \]
    Let us rewrite $u$ as $s w^{2/3}$.
    We plug instead the approximation from \eqref{eq:T:part:two}
    for the value of $T(z w)$ (with $\alpha$ replaced by $s$)
    and $U(z w) = \frac{1}{2} (1 - (1 - T(w z))^2)$ to obtain
    \begin{align*}
        \left|
        \dfrac
        {\phi_r(\rho, w; F(\cdot))}
        {\phi_r(\rho e^{iu}, w; F(\cdot))}
        \right|
        &\sim
        \left|
        \dfrac{
            \ai(r; 2^{1/3} \theta)
        }{
            \ai(r; 2^{1/3}(\theta - is))
        }
        \exp\left(
            i s w^{-1/3}
            - \frac{2^{3/2}}{3} (\theta - i s)^{3/2}
            + \mathcal O(s^2 w^{1/3})
        \right)
        \right|
        \\&\sim
        \left|
        \dfrac{
            \ai(r; 2^{1/3} \theta)
        }{
            \ai(r; 2^{1/3}(\theta - is))
        }
        e^{2^{3/2} (\theta - i s)^{3/2} / 3}
        \right|
    \end{align*}
    because \( \bigO(s^2 w^{1/3}) = \bigO(w^{1/3 - 2 \varepsilon}) \)
    tends to zero if \( \varepsilon \) is smaller than \( 1/6 \).
    Next, we express \( \xi(z, w) \) as a product and obtain
    \[
        \left|
        \dfrac
        {\xi(\rho, w) e^{-iun}}
        {\xi(\rho e^{iu}, w)}
        \right|
        \sim
        \left|
        \dfrac{
            \prod_{i=1}^k \ai(r_i; 2^{1/3} \theta)^{p_i}
        }{
            \prod_{i=1}^k \ai(r_i; 2^{1/3} (\theta - is))^{p_i}
        }
        e^{2^{3/2} (\theta - i s)^{3/2} / 3}
        \right|.
    \]
    By applying the change of variables \( u = s w^{2/3} \) to the integral,
    we obtain
    \[
        \int_{\alpha w^{2/3}}^{w^{2/3-\varepsilon}}
        \left|
        \dfrac
        {\xi(\rho, w) e^{-iun}}
        {\xi(\rho e^{iu}, w)}
        \right|
        \mathrm du
        \sim
        w^{2/3}
        \int_{\alpha}^{w^{-\varepsilon}}
        \left|
        \dfrac{
            e^{2^{3/2} (\theta - i s)^{3/2} / 3}
            \prod_{i=1}^k \ai(r_i; 2^{1/3} \theta)^{p_i}
        }{
            \prod_{i=1}^k \ai(r_i; 2^{1/3} (\theta - is))^{p_i}
        }
        \right|
        \mathrm ds.
    \]
    The integral is bounded by the integral over $\mathbb{R}$.
    Let us now prove that it is convergent.
    We consider $s$ going to $+\infty$, the case $s \to -\infty$ being symmetrical.
    Since $\theta - i s$ becomes equivalent with $-i s$,
    using \eqref{eq:airy:asymptotics} and $\Real((-i)^{3/2}) = -1/\sqrt{2}$,
    we obtain
    \[
        |\ai(r;2^{1/3}(\theta - i s))| \sim
        \frac{(2^{1/3} s)^{r/2-1/4}}{2 \sqrt{\pi}}
        \exp \left( - \frac{2}{3} \Real \left( (-2^{1/3} i s)^{3/2} \right) \right)
        \sim
        \frac{(2^{1/3} s)^{r/2-1/4}}{2 \sqrt{\pi}}
        \exp \left( \frac{2}{3} s^{3/2} \right),
    \]
    so
    \[
        \left|
        \dfrac{
            e^{2^{3/2} (\theta - i s)^{3/2} / 3}
        }{
            \prod_{i=1}^k \ai(r_i; 2^{1/3} (\theta - is))^{p_i}
        }
        \right|
        \sim
        \frac{(2^{1/3} s)^{R/2-1/4}}{2 \sqrt{\pi}}
        \exp \left( -\frac{4}{3} s^{3/2} \right),
    \]
    which is integrable over $\mathbb{R}$.
    We deduce
    \begin{equation} \label{eq:second:region:bound}
        \int_{\alpha w^{2/3}}^{w^{2/3-\varepsilon}}
        \left|
        \dfrac
        {\xi(\rho, w) e^{-iun}}
        {\xi(\rho e^{iu}, w)}
        \right|
        \mathrm du
        =
        \bigO(w^{2/3}).
    \end{equation}

    \textbf{Third region.}
    For the third region, we use part (c) of \cref{prop:complex:general}
    to estimate both the numerator and denominator.
    We use again $u = s w^{2/3}$ to get
    \[
        \left|
        \dfrac
        {\phi_r(\rho, w; F(\cdot))}
        {\phi_r(\rho e^{iu}, w; F(\cdot))}
        \right|
        \sim
        \left|
        \dfrac{
            \ai(r; 2^{1/3} \theta)
        }{
            \ai(r; 2^{1/3} (\theta - is))
        }
        \right|.
    \]
    Next, we express \( \xi(z, w) \) as a product and obtain
    \[
        \left|
        \dfrac
        {\xi(\rho, w) e^{-iun}}
        {\xi(\rho e^{iu}, w)}
        \right|
        \sim
        \left|
        \dfrac{
            \prod_{i=1}^k \ai(r_i; 2^{1/3} \theta)^{p_i}
        }{
            \prod_{i=1}^k \ai(r_i; 2^{1/3} (\theta - is))^{p_i}
        }
        \right|.
    \]
    By applying the change of variables \( u = s w^{2/3} \) to the integral,
    we finally obtain
    \[
        \int_{-\alpha w^{2/3}}^{\alpha w^{2/3}}
        \left|
        \dfrac
        {\xi(\rho, w) e^{-iun}}
        {\xi(\rho e^{iu}, w)}
        \right|
        \mathrm du
        =
        w^{2/3}
        \int_{-\alpha}^{\alpha}
        \left|
        \dfrac{
            \prod_{i=1}^k \ai(r_i; 2^{1/3} \theta)^{p_i}
        }{
            \prod_{i=1}^k \ai(r_i; 2^{1/3} (\theta - is))^{p_i}
        }
        \right|
        \mathrm ds
        =
        \Theta(w^{2/3})
    \]
    because the integrand is integrable over $\mathbb{R}$,
    as explained in the part that is dealing with the second region.

    \textbf{Part (b).}
    In order to settle the case when \( w \) is in the critical region of the
    phase transition, we assume \( w = n^{-1}(1 + \mu n^{-1/3}) \) and $\theta = 0$,
    so $\rho = (e w)^{-1}$.
    We use a Cauchy integral on a circle of radius $\rho$
    to represent the coefficient extraction
    \[
        e^{-w n^2/2} n! [z^n] \frac{1}{\xi(z,w)}
        =
        e^{-w n^2/2} n! \frac{\rho^{-n}}{2 \pi}
        \int_{-\pi}^{\pi}
        \frac{e^{- i u n}}{\xi(\rho e^{i u}, w)}
        \mathrm d u.
    \]
    The integral is divided into the usual three regions.
    As we saw previously, the contribution of the first region
    is exponentially small and therefore negligible.
    Consider a small positive value $\delta$.
    We now choose $\alpha$ large enough so that
    the integrals
    \[
        \int_{\mathbb{R} \setminus [-\alpha, \alpha]}
        \left|
            \frac{e^{2^{3/2} (- i s)^{3/2} / 3}}
                {\prod_{i=1}^k \ai(r_i; -2^{1/3} i s)^{p_i}}
        \right|
        \mathrm d s
        \quad \text{and} \quad
        \left|
        \int_{\mathbb{R} \setminus [-\alpha, \alpha]}
        e^{- i \mu s}
        \prod_{i=1}^k
        \ai(r_i;-2^{1/3} i s)^{-p_i}
        \mathrm d s
        \right|
    \]
    are both smaller than $\delta$.
    This is possible, because both integrands are integrable over $\mathbb{R}$,
    as explained in the part dealing with the second region.
    We use the previous result~\eqref{eq:second:region:bound} and conclude
    that the contribution from the second region is bounded by
    \[
    \frac{e^{-w n^2/2}}{|\xi(\rho, w)|}
    \frac{n! \rho^{-n}}{2 \pi}
    w^{2/3}
    \int_{\mathbb{R} \setminus [-\alpha, \alpha]}
        \left|
            \frac{e^{2^{3/2} (- i s)^{3/2} / 3} \prod_{i=1}^k \ai(r_i; 0)^{p_i}}
                {\prod_{i=1}^k \ai(r_i; -2^{1/3} i s)^{p_i}}
        \right|
        \mathrm d s.
    \]
    As previously, $\xi(\rho, w)$ is approximated using
    part (c) of \cref{prop:complex:general}:
    \[
        \xi(\rho, w) \sim
        (-1)^R \sqrt{2 \pi}
        2^{R/3+1/3} w^{R/3-1/6}
        \bigg( \prod_{i=1}^k F_i(1)^{p_i} \bigg)
        \exp\left(- \frac{1}{2 w}\right)
        \prod_{i=1}^k
        \ai(r_i; 0)^{p_i}.
    \]
    Thus, the contribution from the second region is bounded by
    \[
        \frac{2^{-R/3 - 1/3}
        w^{-R/3 + 5/6}}{ \prod_{i=1}^k F(1)^{p_i}}
        \exp \left( \frac{1}{2 w} - \frac{w n^2}{2} \right)
        \frac{n! e^n w^n}{(2 \pi)^{3/2}}
        \int_{\mathbb{R} \setminus [-\alpha, \alpha]}
        \left|
            \frac{e^{2^{3/2} (- i s)^{3/2} / 3}}
                {\prod_{i=1}^k \ai(r_i; -2^{1/3} i s)^{p_i}}
        \right|
        \mathrm d s.
    \]
    With our choice of $\alpha$, the integral is bounded by $\delta$.
    Since $w = n^{-1} (1 + \mu n^{-1/3})$, we have, using Stirling's formula,
    \begin{equation} \label{eq:region:stirling}
        w^{-R/3 + 5/6}
        \exp \left( \frac{1}{2 w} - \frac{w n^2}{2} \right)
        \frac{n! e^n w^n}{\sqrt{2 \pi}}
        \sim
        n^{(R-1)/3}
        \exp \left(
        - \frac{\mu^3}{6} + \bigO(\mu^4 n^{-1/3})
        \right).
    \end{equation}
    The error term is negligible when $\mu = \smallo(n^{1/12})$.
    The asymptotic bound on the contribution from the second region
    becomes
    \[
        n^{(R-1)/3}
        2^{-R/3 - 1/3}
        \frac{e^{-\mu^3/6}}{2 \pi}
        \delta
        \prod_{i=1}^k F(1)^{-p_i}.
    \]
    Finally, the contribution of the third region is
    \[
        e^{-w n^2/2} n! \frac{\rho^{-n}}{2 \pi}
        \int_{-\alpha w^{2/3}}^{\alpha w^{2/3}}
        \frac{e^{- i u n}}{\xi(\rho e^{i u}, w)}
        \mathrm d u.
    \]
    We apply part (c) of \cref{prop:complex:general}
    to approximate $\xi(\rho e^{i u}, w)$,
    replace $\rho$ with $(e w)^{-1}$
    and apply the change of variables $u = s w^{2/3}$ to obtain
    \[
        \frac{(-1)^R
        2^{-R/3 - 1/3}
        w^{-R/3 + 5/6}}{\prod_{i=1}^k F(1)^{p_i}}
        \exp \left( \frac{1}{2 w} - \frac{w n^2}{2} \right)
        \frac{n! e^n w^n}{(2 \pi)^{3/2}}
        \int_{-\alpha}^{\alpha}
        \frac{\exp \left( - i s w^{2/3} n + i s w^{-1/3} \right)}
            {\prod_{i=1}^k \ai(r_i; - i 2^{1/3} s)^{p_i}}
        \mathrm d s.
    \]
    We apply~\eqref{eq:region:stirling}, the relation
    \[
        - i s w^{2/3} n + i s w^{-1/3}
        =
        - i \mu s + \bigO(n^{-1/3})
    \]
    and a change of variables $s \mapsto -i s$
    to obtain
    \[
        e^{-w n^2/2} \frac{n! \rho^{-n}}{2 \pi}
        \int_{-\alpha w^{2/3}}^{\alpha w^{2/3}}
        \frac{e^{- i u n}}{\xi(\rho e^{i u}, w)}
        \mathrm d u
        \sim
        \frac{(-1)^R
        2^{-(R+1)/3}}{ \prod_{i=1}^k F(1)^{p_i} }
        \frac{n^{(R-1)/3}}{2 \pi i}
        \int_{-i \alpha}^{i \alpha}
        \frac{e^{- \mu^3 / 6 - \mu s}}
            {\prod_{i=1}^k \ai(r_i; - 2^{1/3} s)^{p_i}}
        \mathrm d s.
    \]
    Because $\alpha$ is large enough, the absolute value
    of the tail of the integral is smaller than $\delta$.
    Gathering the contributions of the three regions,
    we obtain that
    \[
    \frac{1}{n^{(R-1)/3}}
    \left|
        e^{-w n^2/2} n! [z^n] \frac{1}{\xi(z, w)} -
        \frac{(-1)^R n^{(R-1)/3} 2^{-(R+1)/3}}{ \prod_{i=1}^k F(1)^{p_i} }
        \frac{1}{2 \pi i}
        \int_{-i \infty}^{i \infty}
        \frac{e^{-\mu^3/6 - \mu s}}
            {\prod_{i=1}^k \ai(r_i; -2^{1/3} s)^{p_i}}
    \mathrm ds \right |
    \]
    is bounded by a constant multiplied by $\delta$.
    Since this is the case for any positive $\delta$,
    we deduce
    \[
        e^{-w n^2/2} n! [z^n] \frac{1}{\xi(z, w)}
        \sim
        \frac{(-1)^R n^{(R-1)/3} 2^{-(R+1)/3}}{ \prod_{i=1}^k F(1)^{p_i} }
        \frac{1}{2 \pi i}
        \int_{-i \infty}^{i \infty}
        \frac{e^{-\mu^3/6 - \mu s}}
            {\prod_{i=1}^k \ai(r_i; -2^{1/3} s)^{p_i}}\mathrm ds.
    \]
\end{proof}

        \subsection{The case of simple digraphs}
        \label{section:external:integration:simple}

The tools we developed to analyse the \( \DiGilbMulti(n, p) \) model
will need some adaptation for the \( \DiGilb(n, p) \)
and \( \DiGilbBoth(n, p) \) models,
but the driving ideas stay the same.

As we have seen in \cref{lemma:P:Dnp}, the probability $\mathbb
P(n,p)$ that a (simple) digraph $D$ belongs in a family $\mathcal F$
of digraphs whose graphic generating function is $\widehat F(z,w)$ is
\[
  \mathbb P_{\mathcal{F}}(n,p) = (1-p)^{\binom{n}{2}} n! [z^n] \widehat F(z,w) \quad
  \mbox{where}
  \begin{cases}
    w = \frac{p}{1-2p} &\mbox{ if } D \in \DiGilb(n,p)\,,  \\
    w = \frac{p}{1-p} &\mbox{ if } D \in \DiGilbBoth(n,p)\,  .
    \end{cases}
\]
Set $p = \lambda / n$.  Again, we will describe the probability
$\mathbb P_{\mathcal{F}}(n,p)$ when $\lambda$ ranges over $[0, +\infty)$.

In the subcritical case, \ie, when $(1-\lambda)n^{1/3} \to \infty$ as
$n\to \infty$, the integrals are  amenable  to the basic saddle point
method. However, the choice of $w$ involves
the saddle point so that we need a careful study of each case.
In order to cover both  models $\DiGilb(n,p)$ and
$\DiGilbBoth(n,p)$, we set
$w = \frac{p}{1-ap} = \frac{\lambda}{n-a\lambda}$
where $a=2$ corresponds to $\DiGilb(n,p)$ and
$a=1$ to $\DiGilbBoth(n,p)$.

Technical results for the subcritical range are summarised in the
following lemma.

\begin{lemma}
    \label{corollary:lemma:analytic:first2}
    Let \( r \) be a fixed real value, and let
    \( w = \frac{\lambda'}{n} \)
    with $ \lambda' = \lambda \left(1-a\frac{\lambda}{n}\right)^{-1} $ for fixed
    $a$ and \(\lambda \in (0, 1) \).
    Let $F(\cdot)$ be an entire function that does not vanish on $\mathbb{C} \setminus \{0\}$.
    Suppose that a function \( H(z, w) \) satisfies an
    approximation
    \[
        H(z, w) \sim e^{U(zw)/w+U(zw)/2} (1 - T(zw))^{1/2 - r}F(T(zw))
    \]
    uniformly for \( |zw| \leq \lambda' e^{-\lambda'} \).
    Then,  when  \( (1 - \lambda)n^{1/3} \to \infty \)
    as \( n \to \infty \), we have
    \[
        \left( 1 - \frac{\lambda}{n}\right)^{\binom{n}{2}}  n!
        [z^n] H(z, w)\Big|_{w=\frac{\lambda'}{n}}
        \sim
        e^{\lambda+(a-1)\lambda ^2/2 }
        (1 - \lambda)^{1 - r} F(\lambda)
        .
    \]
\end{lemma}

\begin{proof}
Recall that $U(z) = T(z) - T(z)^2/2$, so,
defining $G(x) = F(x) e^{x/2 - x^2/4}$, we have
\[
    e^{U(zw)/w+U(zw)/2} (1 - T(zw))^{1/2 - r}F(T(zw)) =
    e^{U(zw)/w} (1 - T(zw))^{1/2 - r}G(T(zw)).
\]
Observe also that $\lambda' = \lambda + \bigO(n^{-1})$.
We apply \cref{lemma:analytic:first} with $\lambda$ replaced by $\lambda'$
and $F$ replaced by $G$.
Its hypotheses are satisfied, because $w = \lambda' / n$,
we have the expected uniform approximation for $|z w| \leq \lambda' e^{- \lambda'}$,
$0 < \lambda' < 1$ for $n$ large enough,
and $(1 - \lambda') n^{1/3} = (1 - \lambda) n^{1/3} + \bigO(n^{-2/3})$ tends to infinity.
We obtain
\[
    e^{-w n^2 / 2} [z^n] H(z, w) \sim
    (1 - \lambda')^{1 - r} G(\lambda')
    \sim
    (1 - \lambda)^{1 - r} e^{\lambda/2 - \lambda^2/4} F(\lambda).
\]
Taking the exponential of the logarithm, we obtain
\[
    \left( 1 - \frac{\lambda}{n} \right)^{\binom{n}{2}} =
    \exp \left(
        - \left( \frac{n^2}{2} - \frac{n}{2} \right)
        \left(\frac{\lambda}{n} + \frac{1}{2} \frac{\lambda^2}{n^2} + \bigO(n^{-3}) \right)
    \right)
    \sim
    e^{ - \lambda n / 2 - \lambda^2/4 + \lambda / 2}.
\]
Combined with
\[
    e^{-w n^2 / 2} = e^{- \lambda' n / 2} \sim e^{- \lambda n / 2 - a \lambda^2 / 2},
\]
we deduce
\[
    \left( 1 - \frac{\lambda}{n} \right)^{\binom{n}{2}} \sim
    e^{- w n ^2 / 2 - \lambda^2/4 + \lambda / 2 + a \lambda^2/2}
\]
and finally
\[
    \left( 1 - \frac{\lambda}{n} \right)^{\binom{n}{2}} [z^n] H(z, w)
    \sim
    e^{- w n ^2 / 2 - \lambda^2/4 + \lambda / 2 + a \lambda^2/2} [z^n] H(z, w)
    \sim
    (1 - \lambda)^{1 - r} e^{\lambda + (a-1) \lambda^2/2} F(\lambda).
\]
\end{proof}

The following lemma will be used to analyse critical
and super-critical simple digraphs.

\begin{lemma} \label{cor:lem:cauchy}
    Let $(\widetilde \phi_{r_i}(z, w,F_i(\cdot)))_{i=1}^k$
    be defined as in~\eqref{eq:phi:tilde:phi} by
    \[
      \widetilde  \phi_r(z, w; F(\cdot)) =
        \dfrac{1}{\sqrt{2 \pi \alpha}}
        \int_{-\infty}^{+\infty}
        (1 - z\alpha \beta e^{-i x})^r
        \exp \left(
            -\dfrac{x^2}{2\alpha} - z\beta e^{-i x}
        \right)
        F(z\alpha\beta e^{-i x})
        \mathrm dx,
  \]
    where \(\alpha=\log(1+w)\) and \(\beta=\sqrt{1+w}\).
    Let \( (r_i)_{i=1}^k \) be integers,
    \( (p_i)_{i=1}^k \) be integers summing to $1$,
    and set
    \[
        R = \sum_{i=1}^k r_i p_i
        \quad \text{and} \quad
        \widetilde \xi(z, w) = \prod_{i=1}^k \widetilde
        \phi_{r_i}(z, w; F_i(\cdot))^{p_i}\, ,
    \]
    where $F_i(\cdot)$ does not vanish on $\mathbb{C} \setminus \{0\}$
    for $i \in \{1, \ldots, k\}$.

    \noindent \textbf{Part (a).}
    If $p = \lambda / n$ with $\lambda \in (0, 1 - n^{\epsilon - 1/3})$
    for some $\epsilon > 0$ and $w = \frac{p}{1 - a p}$ for a fixed positive $a$,
    then
    \[
        (1 - p)^{\binom{n}{2}}
        n! [z^n] \frac{1}{\widetilde \xi(z, w)} \sim
        e^{\lambda + (a-1) \lambda^2/2}
        (1 - \lambda)^{1-R}
        \prod_{i=1}^k F_i(\lambda)^{-p_i}.
    \]

    \noindent \textbf{Part (b).}
    If \( w \) is of the form \( w = \frac{p}
    {1-ap }\) with $p= n^{-1}(1 + \mu
    n^{-1/3})$, $a$ fixed   and \( \mu = \smallo(n^{1/12})\), then
    \[
        (1-p)^ {\binom{n}{2}} n! [z^n] \dfrac{1}{\widetilde \xi(z,
          w)}    =
        n^{(R-1)/3}
        \frac{(-1)^R
        e^{(a+1)/2}
        2^{-(R+1)/3}}{\prod_{i=1}^k F_i(1)^{p_i}}
        \dfrac{1}{2 \pi i}
        \int_{-i \infty}^{i \infty}
        \dfrac{
            \exp (-\tfrac{\mu^3}{6} - \mu s)
        }
        {
            \prod_{i=1}^k \ai(r_i; -2^{1/3}s)^{p_i}
        }
        \mathrm ds\, .
    \]

    \noindent \textbf{Part (c).}
    Consider a fixed real value \( \theta \) such that
    \( \ai(r_i, 2^{1/3}\theta) \neq 0\) for $i \in \{1, \ldots, k\}$,
    and define  \( \rho = (1 - \theta w^{2/3}) (ew)^{-1} \).
    Then we have
    \[
        \dfrac{1}{2 \pi i}
        \oint_{|z| = \rho}
        \dfrac{1}{\widetilde
            \xi(z, w)
        }
        \dfrac{\mathrm dz}{z^{n+1}}
        =
        \mathcal O \left(
            \dfrac{w^{2/3}}{\widetilde\xi(\rho, w) \rho^n}
        \right).
    \]
    \end{lemma}

    \begin{proof}
    The proof is established along similar lines as in \cref{lem:cauchy}.

    \textbf{Part (a).}
    If $|z w| \leq \lambda'e^{- \lambda'}$, then
    \[
        |e z w| \leq \lambda'e^{1 - \lambda'} < 1.
    \]
    If furthermore $\lambda \in (0, 1 - n^{\epsilon - 1/3})$ for some $\epsilon > 0$
    (and thus, $\epsilon < 1/3$ so $\lambda$ exists), then
    $\lambda'$ is in $(0, 1 - n^{\epsilon / 2 - 1/3})$ for $n$ large enough, so
    \[
        1 - |T(zw)| \geq 1 - T(|z w|) \geq
        1 - T(\lambda'e^{-\lambda'}) \geq 1 - \lambda' \geq n^{\epsilon/2 - 1/3} \geq
        \lambda'^{\epsilon/2 - 1/3} w^{1/3 - \epsilon/2}
        \geq w^{1/3 - \epsilon/2}.
    \]
    Thus, Part (a) of \cref{corollary:complex:general} is applicable
    and provides the approximation
    \[
        \widetilde \phi_r(z,w; F(\cdot)) \sim
        e^{-U(z w) / w - U(zw)/2}
        (1 - T(zw))^{r - 1/2}
        F(T(zw))
    \]
    uniformly for $|z w| \leq \lambda' e^{-\lambda'}$.
    Taking the product, using $\sum_{i=1}^k p_i = 1$ and the definition of $R$, we obtain
    \[
        \frac{1}{\widetilde \xi(z,w)} \sim
        e^{U(zw)/w + U(zw)/2}
        (1 - T(zw))^{1/2 - R}
        \prod_{i=1}^k F_i(T(zw))^{-p_i}.
    \]
    We now apply \cref{corollary:lemma:analytic:first2} to conclude
    \[
        (1-p)^{\binom{n}{2}} n! [z^n] \frac{1}{\widetilde \xi(z,w)}
        \sim
        e^{\lambda + (a-1) \lambda^2/2}
        (1 - \lambda)^{1-R}
        \prod_{i=1}^k F_i(\lambda)^{-p_i}.
    \]

      \textbf{Part (b)}
      Now, assume in addition that $ w =  \frac{p}{1-ap}$ and $\theta
      = 0$, so $\rho = (ew)^{-1}$. We follow the proof of Part (b)
      of \cref{lem:cauchy} to see
      how the substitution of $w$ into  $\frac{p}{1-ap}$ instead of
      $p$  changes  the final results, besides the factor
      $e^{-1/4}$. We want to compute
      \begin{align*}
          (1-p)^ {\binom{n}{2}}  n! [z^n] \frac{1}{\widetilde\xi(z,w)}
        &=
        (1-p)^ {\binom{n}{2}}  n! \frac{\rho^{-n}}{2 \pi}
        \int_{-\pi}^{\pi}
        \frac{e^{- i u n}}{\widetilde \xi(\rho e^{i u}, w)}
        \mathrm d u\\
        &\sim    e^{1/4}
        \left(\prod_{i=1}^k F_i(1)^{-p_i}\right)
        (1-p)^ {\binom{n}{2}} n!
        \frac{\rho^{-n}}{2 \pi}
        \int_{-\pi}^{\pi}
        \frac{e^{- i u n}}{\xi(\rho e^{i u}, w)}
        \mathrm d u\, .
      \end{align*}
    As shown is the proof  of \cref{lem:cauchy}, the contributions of the
    first and second regions can be neglected since $\mu =
    o(n^{1/12})$ and the contribution of the third region becomes
    \[
                  e^{1/4}\left(\prod_{i=1}^k F_i(1)^{-p_i} \right)
                  (1-p)^ {\binom{n}{2}}  n!
                  \frac{\rho^{-n}}{2 \pi}
        \int_{-\alpha w^{2/3}}^{\alpha w^{2/3}}
        \frac{e^{- i u n}}{\xi(\rho e^{i u}, w)}
        \mathrm d u\, ,
    \]
    which is,
    by means of asymptotic approximations of
    $\xi(\rho e^{i u}, w)$ from part (c) of
    \cref{prop:complex:general}
    and under a change of variable $u = s w^{2/3}$,
    asymptotically equivalent to
      \begin{equation}
        \label{eq:third-region-simple}
      \frac{e^{1/4} (1-p)^ {\binom{n}{2}} (-1)^R
        2^{-R/3 - 1/3}
        w^{-R/3 + 5/6}
        e^{\frac{1}{2 w} }}
        {\prod_{i=1}^k F_i(1)^{p_i}}
         \frac{n! e^n w^n}{(2 \pi)^{3/2}}
        \int_{-\alpha}^{\alpha}
        \frac{\exp \left( - i s w^{2/3} n + i s w^{-1/3} \right)}
            {\prod_{i=1}^k \ai(r_i; - i 2^{1/3} s)^{p_i}}
            \mathrm d s,
      \end{equation}
      for some large enough $\alpha$. Then, we get
      \[
        \begin{split}
          (1-p)^ {\binom{n}{2}}  \cdot
          e^{1/(2w)} \cdot
          n! \cdot e^{n}\cdot
          w^n
          & \sim
          e^{1/4} e^{-n/2}e^{-\mu n ^{2/3}/2}
          \cdot e^{-\mu^3/2}e^{-a/2}e^{n/2}e^{-\mu n^{2/3}/2}e^{\mu^2
            n^{1/3}} \\
          &\qquad \cdot \sqrt{2\pi n}n^{n}e^{-n}
          \cdot e^{n} \cdot e^{\mu^3/2}e^{a}
          e^{\mu n^{2/3}}e^{-\mu^{2}n^{1/3}}\\
         & = \sqrt{2\pi n} e^{1/4+a/2}\, .
        \end{split}
      \]
      Combining this last result with \eqref{eq:third-region-simple}
      leads to the desired result since  the relation  $
        - i s w^{2/3} n + i s w^{-1/3}
        =
        - i \mu s + \bigO(n^{-1/3})
    $ is still valid and $w^{5/6} \sqrt{n}\sim n^{-1/3}$.

      \textbf{Part (c)} We follow the proof of \cref{lem:cauchy}
      and define $\xi(z,w)$ and the three regions in the same way.
      By comparing the asymptotics of
      $\phi_{r_i}(z,w; F_i(\cdot))$ in \cref{prop:complex:general}  and  $\widetilde
      \phi_{r_i}(z, w,F_i(\cdot))$ in \cref{corollary:complex:general}, and
      applying the fact that $\sum_{i=1}^k p_i = 1$, we
      observe that
      \[
        \widetilde \xi(z, w) \sim
        \xi(z, w) e^{-U(zw)/2}
      \]
      for the first region and
      \[
        \widetilde \xi(z, w) \sim
        \xi(z, w)
        e^{-1/4}
      \]
      for the second and third regions. Bringing the factor
      $e^{-U(zw)/2}
      $ or  $e^{-1/4}$ with $\xi(z,w)$
      does not affect the
      process of extracting asymptotic coefficients in the three regions.
  \end{proof}

\section{Asymptotics of multidigraph families}
\label{section:asymptotics:multidigraphs}

Having established the asymptotic approximations from~\cref{section:external:integral},
we can compute the asymptotic probabilities of various digraph families.
We start with the easiest digraph model, which is the model of multidigraphs,
and we will extend the technique to the case of simple digraphs
in~\cref{section:asymptotics:simple}.

\subsection{Asymptotics of directed acyclic multigraphs}
\label{section:dag:multi}

As we have said, the probability that a (simple) digraph is acyclic has been
computed
when \( p \) was constant in~\cite{Bender86}.
Their results can be extended to bound the probability
that a  multigraph is acyclic in the stated range since
acyclicity is a monotone decreasing graph property.
Hereby, we compute this probability for the remaining
part when \( p \to 0 \), distinguishing three regimes.

\begin{theorem}\label{theo:MD_model_lambda_fixed}
  Let \( p = \lambda/n \), \( \lambda \geq 0 \) fixed.
  Then, the probability $\mathbb P_{\mathsf{acyclic}}(n,p)$  that a random multidigraph
  $\DiERMulti(n,p)$ is   acyclic satisfies the following asymptotic
  formulas as $n\to \infty$:
    \[
        \mathbb P_{\mathsf{acyclic}}(n,p) \sim
        \begin{cases}
            1 - \lambda, & \lambda \in [0,1) \, ; \\
            \gamma_1 n^{-1/3}, & \lambda = 1 \, ; \\
            \gamma_2(\lambda) n^{-1/3}
            \exp \Big(
            {-\alpha(\lambda) n} + a_1 \beta(\lambda) n^{1/3}
            \Big), &
            \lambda > 1 \, ,
        \end{cases}
    \]
    where
    \[
    \arraycolsep=1.4pt\def\arraystretch{2.2}
    \begin{array}{rlcrl}
        \gamma_1 &= \dfrac{2^{-1/3}}{2\pi i}
                    \displaystyle\int_{-i\infty}^{i\infty}
                    \frac{1}{\ai(-2^{1/3}t)}
                    \mathrm dt \approx 0.488736706,
        %
        %
        %
        %
    &\quad&
        \gamma_2(\lambda)
            &= \dfrac{2^{-2/3}}{\mathrm{Ai'}(a_1)}
              \lambda^{5/6}e^{(\lambda-1)/6},
              \\[.5em]
        \alpha(\lambda)
            &= \dfrac{\lambda^2-1}{2\lambda} - \log \lambda,
    &\quad&
        \beta(\lambda)
            &= 2^{-1/3}\lambda^{-1/3} (\lambda - 1),
    \end{array}
    \]
and $a_1 \approx -2.338107$ is the zero of the Airy function $\ai(z)$ with the
smallest modulus, and where the numerical value of \( \ai'(a_1) \) is given by
\( \ai'(a_1) \approx 0.701211 \).
\end{theorem}

\begin{proof}
    Using \cref{th:summary:multidigraph}, the probability $\mathbb
    P_{\mathsf{acyclic}}(n,p)$ that a random multidigraph $\DiERMulti(n,p)$ is acyclic
    is given by
    \[
        \mathbb P_{\mathsf{acyclic}}(n,p) =
        e^{- n^2 p / 2} n ! [z^n] \frac{1}{\phi(z,p)}\,.
    \]
    We therefore let \( w = p \).
    Recall that $\phi(z,w) = \phi_0(z,w; 1)$.
    The cases $\lambda \in [0, 1)$ and $\lambda = 1$
    are direct applications of Parts (a) and (b) of \cref{lem:cauchy}
    with $k = 1$, $r = 0$, $p = 1$ and $F(z) = 1$.

    \textbf{Case $\lambda >1$:}
    Let \( a_1 > a_2 > a_3 > \cdots \) denote the roots of the Airy function
    \( \ai(z) \) (which are all negative real numbers).
    Let  $\rho=(1-\theta w^{2/3})(ew)^{-1}$, where $\theta$ is fixed and
    satisfies $a_2< 2^{1/3}\theta < a_1$. This implies
    $\ai(2^{1/3}\theta)\neq 0$ and, by \cref{theo:smallest_zero},
    $\varrho_1(w)<\rho<\varrho_2(w)$ when $n$ is large enough.
    Hence, by the residue theorem, we have
    \[
        [z^n] \frac{1}{\phi(z,w)} =
        -\frac{1}{\varrho_1(w)^{n+1}
        \partial_z\phi (\varrho_1(w),w)}+
        \frac{1}{2\pi i}\oint_{|z|=\rho}\frac{1}{\phi(z,w)z^{n+1}}
        \mathrm dz.
    \]
    Then, we use \eqref{eq:cauchy_up} of \cref{lem:cauchy} to
    estimate the integral on the right-hand side, so we get
    \[
        [z^n] \frac{1}{\phi(z,w)} =
        -\frac{1}{\varrho_1(w)^{n+1}\partial_z\phi(\varrho_1(w),w)}
        +\mathcal O\left(\frac{w^{2/3}}{|\phi(\rho,w)|\rho^n}\right).
    \]
    Next, using~\cref{theo:smallest_zero}
    to estimate $\varrho_1(w)$ and $ \partial_z \phi (\varrho_1(w),w)$
    and Part~(c) of~\cref{prop:complex:general} to estimate
    $\phi(\rho,w)$, we obtain,  with the help of a computer algebra system,
    \begin{align}
        -\frac{n!e^{-wn^2/2}}{\varrho_1(w)^{n+1}
        \partial_z\phi(\varrho_1(w),w)}
        & \sim \gamma_2(\lambda) n^{-1/3}
        \exp \left(
          -\alpha(\lambda)n+2^{-1/3}a_1\lambda^{-1/3}(\lambda-1)n^{1/3}
        \right)
        \label{eq:sup_m},\\[.5em]
        \frac{n!e^{-wn^2/2} w^{2/3}}{|\phi(\rho,w)|\rho^n}
        &=
        \mathcal O \left( n^{-1/3}
            \exp\left(
                -\alpha(\lambda)n+\theta\lambda^{-1/3}(\lambda-1)n^{1/3}
            \right)
        \right),\label{eq:sup_e}
    \end{align}
    where the constants $\gamma_2(\lambda)$ and $\alpha(\lambda)$
    are precisely as defined in
    \cref{theo:MD_model_lambda_fixed}. Since we chose $\theta$ in such a way
    that $\theta<2^{-1/3}a_1$, the left-hand side of \eqref{eq:sup_e} is
    exponentially (in $n^{1/3}$) smaller than that of  \eqref{eq:sup_m} as
    $n\to\infty$. Therefore,  the right-hand side of \eqref{eq:sup_m} is
    indeed the main term of $\mathbb P_{\mathsf{acyclic}}(n,p).$
\end{proof}

\begin{remark}
\label{remark:sum:residues}
    We can further compare the asymptotics in the case when $\lambda = 1$ and in the supercritical case when $\lambda > 1$, but with taking the limit $\lambda \to 1^+$.
    By analysing the third expression in~\cref{theo:MD_model_lambda_fixed}, we note that $\alpha(1) = \beta(1) = 0$, and furthermore, direct computation shows that $\gamma_2(1) \approx 0.898389$, while $\gamma_1 \approx 0.488736$. This discontinuity is explained by the fact that taking the limit $\lambda \to 1^+$ does not lead to the correct asymptotic probability, because according to the proof of~\cref{theo:MD_model_lambda_fixed},
    $\lambda$ needs to be separated from $1^+$ by a certain quantity that cannot go to zero too fast: the scaling parameter $\mu$ should go to infinity.

    The same phenomenon can be observed from a different viewpoint.
    For any fixed \( \mu > 0 \) we can expand the contour integral for \( \lambda = 1 \) as a sum of residues, which yields
    \[
        \mathbb P_{\mathsf{acyclic}}(n, \tfrac{1}{n}(1 + \mu n^{1/3}))
        \sim
        2^{-2/3} n^{-1/3} e^{-\mu^3/6}
        \sum_{k \geq 1}
        \dfrac{e^{2^{-1/3} \mu a_k}}{\ai'(a_k)},
    \]
    where \( (a_k)_{k = 1}^\infty \) are the roots of the Airy function sorted
    by their absolute values. Indeed, an integral on a line segment $[-iR, iR]$ can be completed by an integral over a semicircle $z = R e^{i \theta}$, where $\theta \in [\pi/2, 3\pi/2]$, which is negligible because the inverse Airy function decays at exponential speed when $\theta \neq \pi$, and the integrand is dominated by $e^{2^{-1/3} \mu z}$ when $\theta$ is in the vicinity of $\pi$. Moreover, it can be proven that the above series is convergent because, due to~\cite[9.9]{NIST}, the sequence \( a_k \) is negative and grows like $k^{2/3}$, and the sequence $\ai'(a_k)$ is alternating and is growing like $k^{1/6}$ as $k \to \infty$.

    However, in the supercritical
    case, when $\lambda > 1$,  we are considering only the first summand of this sum of residues, because the other residues are negligible when $\mu \to \infty$.
    Therefore, taking the limit $\lambda \to 1^+$ does not give the desired sum, which explains the observed discontinuity.
\end{remark}

\begin{figure}[ht]
    \RawFloats
    \begin{minipage}{.5\textwidth}
    \includegraphics[width=\linewidth]{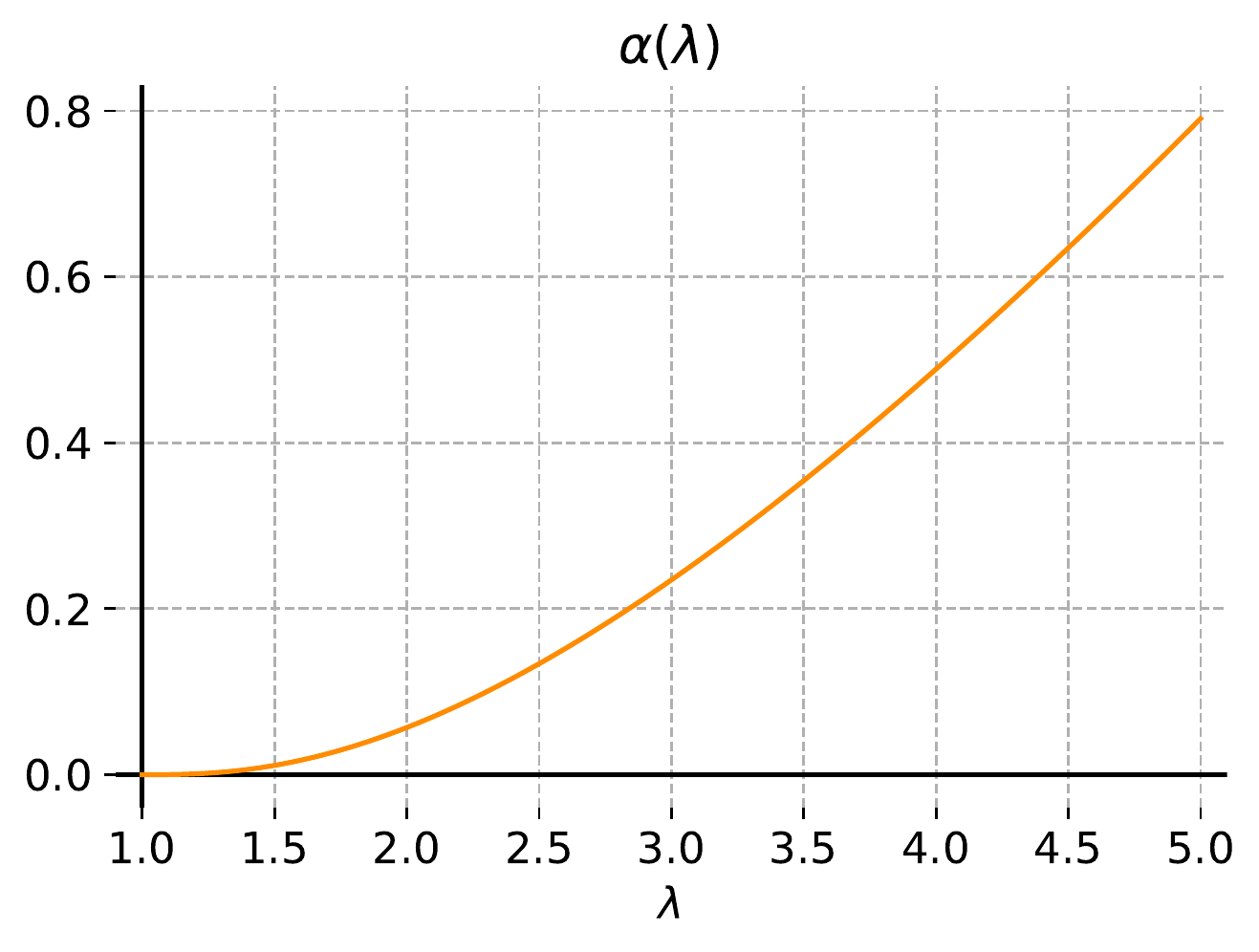}
    \caption{Numerical plot of $\alpha(\lambda)$}
    \end{minipage}%
    \begin{minipage}{.5\textwidth}
    \includegraphics[width=\linewidth]{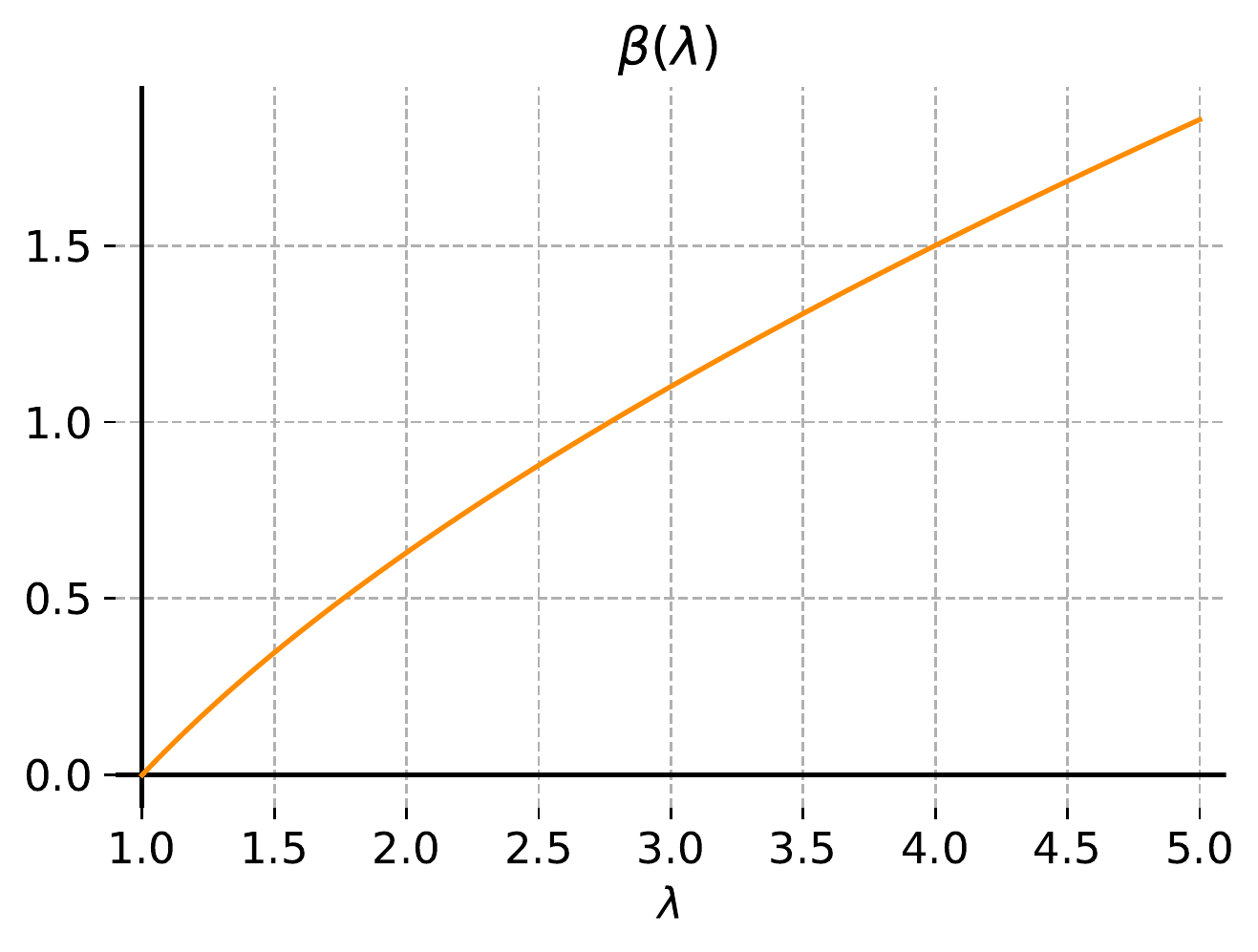}%
    \caption{Numerical plot of $\beta(\lambda)$}
    \end{minipage}
\end{figure}

The asymptotic result from \cref{theo:MD_model_lambda_fixed}
can be also refined and explicitly stated inside the
critical window.

\begin{theorem}
  \label{th:acyclic-window}
  Let \( p = (1+\mu n^{-1/3})/n \),
  $a_1 \approx -2.338107$ denote the zero of the Airy function of smallest absolute value,
  and let $\varphi(\mu)$ denote the function
  \[
      \varphi(\mu) =
      2^{-1/3} e^{-\mu^3/6} \frac{1}{2 \pi i}
      \int_{-i \infty}^{i \infty}
      \frac{e^{-\mu s}}{\mathrm{Ai}(-2^{1/3} s)}
      \mathrm{d} s.
  \]
  Then, the probability $\mathbb P_{\mathsf{acyclic}}(n,p)$  that a random multidigraph
  $\DiERMulti(n,p)$ is   acyclic satisfies the following asymptotic
  formula as $n\to \infty$:
    \[
      \mathbb P_{\mathsf{acyclic}}(n,p) \sim
      \begin{cases}
        |\mu|n^{-1/3} & \mbox{ if $\mu \to - \infty$ with $|\mu| = \smallo(n^{1/3})$}\, ;\\
      \varphi(\mu) n^{-1/3}
      & \mbox{ if $\mu = \smallo(n^{1/12})$} \, ;\\
      \dfrac{2^{-2/3}}{\mathrm{Ai'}(a_1)}
      \exp\left( - \dfrac{\mu^3}{6}+2^{-1/3}a_1\mu
        \right) n^{-1/3}&
      \mbox{ if $\mu \to +\infty$ with $\mu = \smallo(n^{1/12})$}\, .
      \end{cases}
    \]
\end{theorem}

\begin{proof}
    We follow the same steps as in the proof of \cref{theo:MD_model_lambda_fixed}.
    Recall that, according to~\cref{th:summary:multidigraph},
    the probability we are interested in is
    \[
        \mathbb P_{\mathsf{acyclic}}(n,p) =
        e^{-pn^2/2} n ! [z^n] \frac{1}{\phi(z, p)},
    \]
    and we set
    \[
        w = p = \frac{\lambda}{n}
        =
        \frac{1}{n} \left(1 + \mu n^{-1/3} \right).
    \]
    \textbf{First result.}
    We start by proving a weaker version of the first result,
    adding the assumption $|\mu| \geq n^{\epsilon}$
    for some small enough positive $\epsilon$.
    During the proof of the second part, we will be able to complete
    the full range of \( \mu \to -\infty \) as \( n \to \infty \).
    We apply Part (a) of \cref{lem:cauchy}
    with \( \xi(z, w) = \phi(z, w) = \phi_0(z,w; 1)\),
    $k = 1$, $r = 0$, $p = 1$ and $F(z) = 1$ to conclude.

    \noindent \textbf{Second result.}
    The result follows directly from Part (b) of \cref{lem:cauchy}.
    In the first result, we could assume $\epsilon$ as small as wanted,
    so let us fix $\epsilon$ in $(0, 1/12)$.
    Observe that the range of $\mu$ of this second result
    overlaps with the first result
    (take, for example, $\mu = - n^{\alpha}$ for any $\alpha \in [\epsilon, 1/12)$).
    We deduce, for $\mu$ in this intersection,
    \[
        \mathbb P_{\mathsf{acyclic}}(n,p) \sim \varphi(\mu) n^{-1/3} \sim |\mu| n^{-1/3}.
    \]
    Since the function \( \varphi(\mu) \) does not depend on \( n \), we can conclude that
    $\varphi(\mu) \sim |\mu|$ as $\mu \to -\infty$.
    This completes the proof of the first result of the theorem.

    \noindent \textbf{Third result.}
    We follow the same proof as for the case $\lambda > 1$
    of \cref{theo:MD_model_lambda_fixed},
    arriving at the conclusion that
    $\mathbb P_{\mathsf{acyclic}}(n,p)$
    is the sum of the two terms given
    in~\cref{eq:sup_m,eq:sup_e},
    and recalled here for convenience:
    \begin{align*}
        -\frac{n!e^{-wn^2/2}}{\varrho_1(w)^{n+1}
        \partial_z\phi(\varrho_1(w),w)}
        & \sim \gamma_2(\lambda) n^{-1/3}
        \exp \left(
            -\alpha(\lambda)n+2^{-1/3}a_1\lambda^{-1/3}(\lambda-1)n^{1/3}
        \right)\,,\\[.5em]
        \frac{n!e^{-wn^2/2} w^{2/3}}{|\phi(\rho,w)|\rho^n}
        &=
        \mathcal O \left( n^{-1/3}
            \exp\left(
                -\alpha(\lambda)n+\theta\lambda^{-1/3}(\lambda-1)n^{1/3}
            \right)
        \right) \, .
    \end{align*}
    Since $\lambda$ converges to $1$, the term $\gamma_2(\lambda)$
    converges to
    \[
        \gamma_2(1) = \frac{2^{-2/3}}{\mathrm{Ai'}'(a_1)}.
    \]
    The two exponential terms differ
    in that $\theta < 2^{-1/3} a_1$.
    Since $(\lambda - 1)n^{1/3} = \mu$ tends to infinity,
    the second term is negligible compared to the first one.
    We conclude
    \[
        \mathbb P_{\mathsf{acyclic}}(n,p) \sim
        \frac{2^{-2/3}}{\mathrm{Ai'}(a_1)} n^{-1/3}
        \exp \left(
            -\alpha(\lambda)n+2^{-1/3}a_1 \lambda^{-1/3}\mu
        \right).
    \]
    Using $\mu = \smallo(n^{1/12})$, we obtain
    \[
        \lambda^{-1/3} \mu = \mu + \smallo(1)
    \]
    and
    \[
        \alpha(\lambda) = \frac{\lambda^2 - 1}{2 \lambda} - \log(\lambda)
        =
        \frac{\mu^3 n^{-1}}{6} + \smallo(n^{-1}).
    \]
    Plugging in those two approximations in the previous asymptotics,
    we conclude that
    \[
        \mathbb P_{\mathsf{acyclic}}(n,p) \sim
        \frac{2^{-2/3}}{\mathrm{Ai'}(a_1)} n^{-1/3}
        \exp \left(
            -\mu^3/6+2^{-1/3}a_1 \mu
        \right).
    \]
\end{proof}

\begin{remark}
Note that if \( \mu = -C n^{1/3} \), where \( C \in (0, 1) \), then by
applying~\cref{theo:MD_model_lambda_fixed} one obtains again
\( \mathbb P_{\mathsf{acyclic}}(n,p) \sim |\mu| n^{-1/3} \) as \( n \to \infty \).
In the case when \( \mu \to +\infty \) faster than \( n^{1/12} \), the supercritical
case of~\cref{theo:MD_model_lambda_fixed} is again applicable, although the
asymptotic formula takes a different form if expressed in terms of \( \mu \) due to
appearance of additional terms in the Taylor expansion in the exponent.
\cref{fig:critical_acyclic} illustrates the result
by plotting the limiting probabilities inside the critical window
for \( \mu \in [-3, 3] \).
\end{remark}

    \subsection{Asymptotics of elementary digraphs and complex components}
    \label{section:elementary:multi}

The next step after having the asymptotics of directed acyclic graphs is to
express the probability that a random digraph or a multidigraph is
\emph{elementary} (all strong components are single vertices or cycles).

\begin{theorem}
  \label{theo:elementary:multi}
    Let \( p = \lambda / n \).
    The probability that a random multidigraph
    \( D \in \DiERMulti(n, p) \) is elementary
    is
    \[
        \mathbb P_{\mathsf{elementary}}(n,p) \sim
        \begin{cases}
            1, & \lambda < 1;
            \\
            - 2^{-2/3}\dfrac{1}{2 \pi i}
            \displaystyle\int_{-i \infty}^{i\infty}
            \dfrac{\exp(-\mu s - \mu^3/6)}{\ai'(-2^{1/3} s)}
            \mathrm d s, & \lambda = 1 + \mu n^{-1/3};
            \\[1em]
            \sigma_2(\lambda) \exp \Big(
                {-\alpha(\lambda) n} + a_1' \beta(\lambda) n^{1/3}
            \Big)
            , & \lambda > 1,
        \end{cases}
    \]
    where \( a_1' \approx -1.018793 \)
    is the zero of the derivative of the Airy function
    \( \ai'(z) \) with the smallest modulus,
    \( \ai(a_1') \approx 0.53565666 \),
    the constants \( \alpha(\lambda) \)
    and \( \beta(\lambda) \) are as in~\cref{theo:MD_model_lambda_fixed}, and
    \[
        \sigma_2(\lambda)
        =
        - \dfrac{\lambda^{1/2} e^{-(\lambda-1)/6}}{2 a_1' \ai(a_1')}.
    \]
\end{theorem}

\begin{figure}[ht]
    \centering
    \includegraphics[width=.7\linewidth]{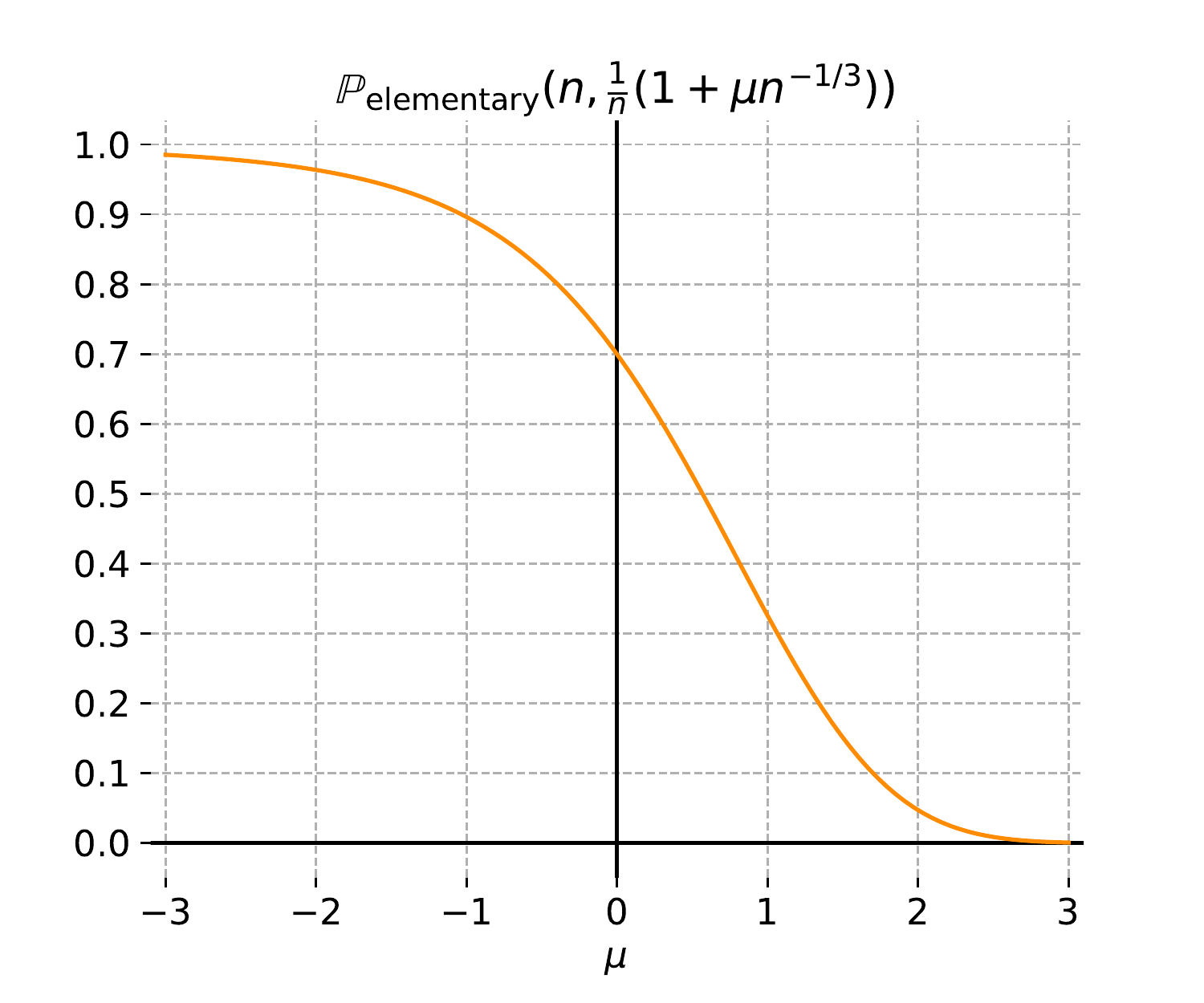}\\
    \caption{
    \label{fig:elementary}
      Numerical plot of the theoretical limit of the
      probability $\mathbb P_{\mathsf{elementary}}(n,p)$ that a random multidigraph is elementary
      inside the critical window for $\mu \in[-3,3]$.}
\end{figure}
\begin{proof}
    According to~\cref{th:summary:multidigraph},
    the probability that a random multidigraph $D$ is elementary is given by
    \[
        \mathbb P_{\mathsf{elementary}}(n, p)
        =
        e^{- n^2 p / 2} n!
        [z^n]
        \dfrac{1}{\phi_{1}(z, p)}.
    \]
    We let \( w = p \) from now on.

    \noindent \textbf{Case $\lambda < 1$.}
    The result is a direct application of Part (a) of \cref{lem:cauchy}
    with \( \xi(z, w) := \phi_1(z, w) = \phi_1(z,w;1) \),
    $k=1$, $r = 1$, $p = 1$ so $R = 1$.

    \noindent \textbf{Case $\lambda = 1 + \mu n^{-1/3}$.}
    The result is a direct application of Part (b) of \cref{lem:cauchy}.

    \noindent \textbf{Case $\lambda > 1$.}
    Let \( a_1' > a_2' > a_3' > \cdots \) denote the roots of \( \ai'(z)
    \). It is known that they are all negative and distinct.
    Consider \( \theta \) such that \( a_2' < 2^{1/3} \theta < a_1' \)
    and define \( \rho = (1 - \theta w^{2/3}) (ew)^{-1} \).
    Let \( \varsigma_j(w) \) denote the solution of the equation
    \( \phi_1(z, w) = 0 \) that is the $j$-th closest to $0$.
    Then, according to \cref{theorem:zeros:phi:r},
    for any small enough $w$, we have
    \( \varsigma_1(w) < \rho < \varsigma_2(w) \).
    Thus, by taking for $z$ the contour of integration \( |z| = \rho \),
    the residue theorem implies
    \[
        [z^n] \dfrac{1}{\phi_{1}(z, w)} =
        -\mathrm{Res}_{z = \varsigma_1(w)}
        \dfrac{1}{z^{n+1} \phi_1(z, w)}
        +
        \dfrac{1}{2 \pi i}
        \oint_{|z| = \rho} \dfrac{1}{\phi_1(z, w) z^{n+1}}
        \mathrm dz.
    \]
    From~\cref{theorem:zeros:phi:r} we conclude that as \( w \to 0^+ \),
    the first root of \( \phi_1(z, w) \) is simple.
    This allows for a simpler expression for the residue.
    We also use \eqref{eq:cauchy_up} of \cref{lem:cauchy} to
    estimate the integral on the right-hand side
    \[
        [z^n] \dfrac{1}{\phi_{1}(z, w)} =
        - \dfrac{1}{\varsigma_1(w)^{n+1} \partial_z \phi_1(\varsigma_1(w), w)}
        +
        \bigO \left( \frac{w^{2/3}}{|\phi_1(\rho, w)| \rho^n} \right).
    \]
    Next, using~\cref{theorem:zeros:phi:r}
    to estimate $\varsigma_1(w)$ and $ \partial_z \phi_1 (\varsigma_1(w),w)$
    and Part~(c) of~\cref{prop:complex:general} to estimate
    $\phi_1(\rho,w)$, we obtain
    \begin{align*}
        \mathbb P_{\mathsf{elementary}}(n,p) & \sim
        {
            - \dfrac
            {n! e^{-wn^2/2}}
            {
                \varsigma_1(w)^{n+1}
                \partial_z \phi_1(\varsigma_1(w), w)
            }
        }
        + \mathcal O \left(
            e^{
                -\alpha(\lambda) n
                + \theta \lambda^{-1/3}(\lambda-1)n^{1/3}
            }
        \right)
        \\ &
        \sim
        \sigma_2(\lambda) \exp \left(
            -\alpha(\lambda) n + a_1' \beta(\lambda) n^{1/3}
        \right).
    \end{align*}
\end{proof}

\begin{figure}[ht]
    \RawFloats
    \begin{minipage}{.5\textwidth}
    \includegraphics[width=\linewidth]{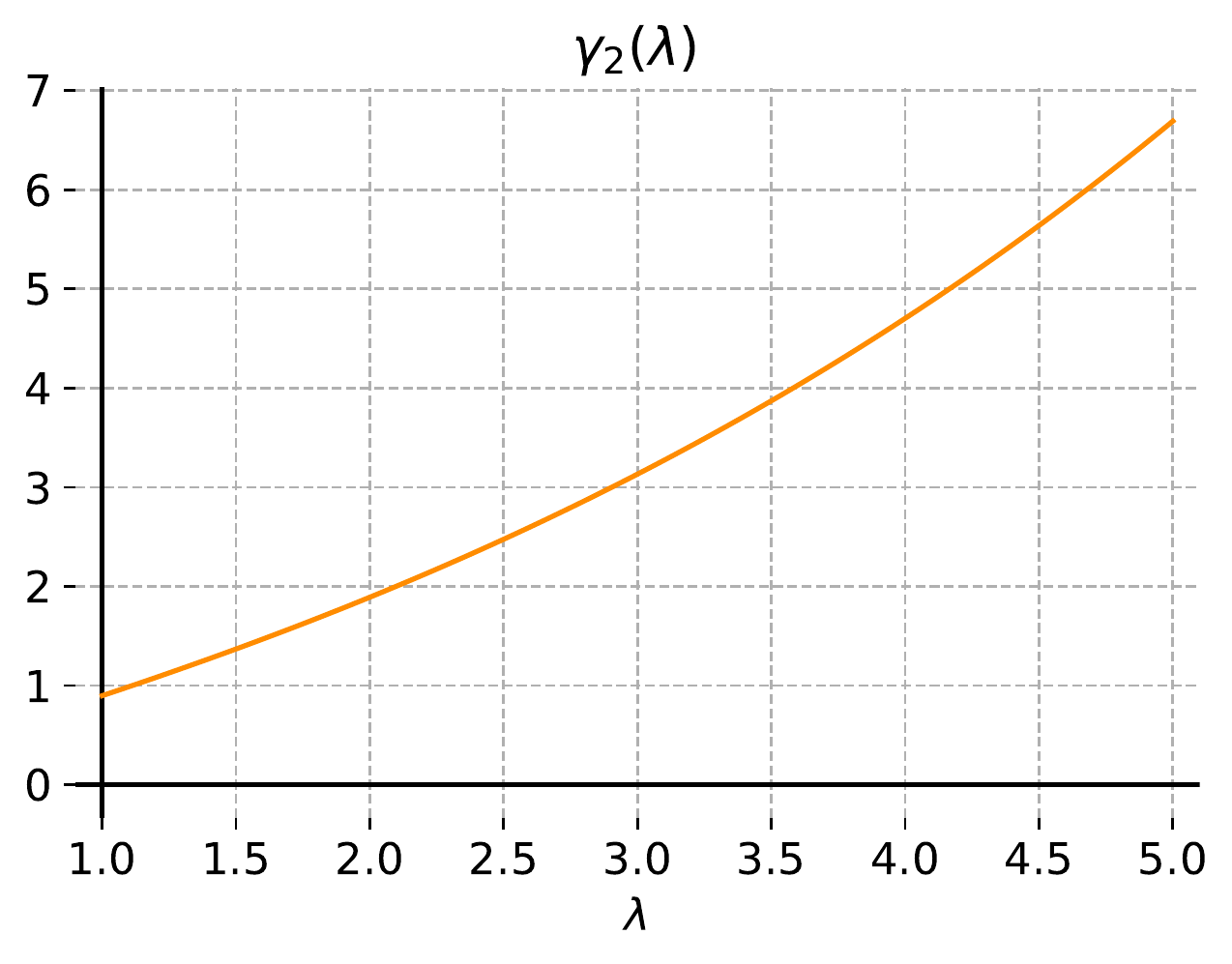}
    \caption{Numerical plot of $\gamma_2(\lambda)$}
    \end{minipage}%
    \begin{minipage}{.5\textwidth}
    \includegraphics[width=\linewidth]{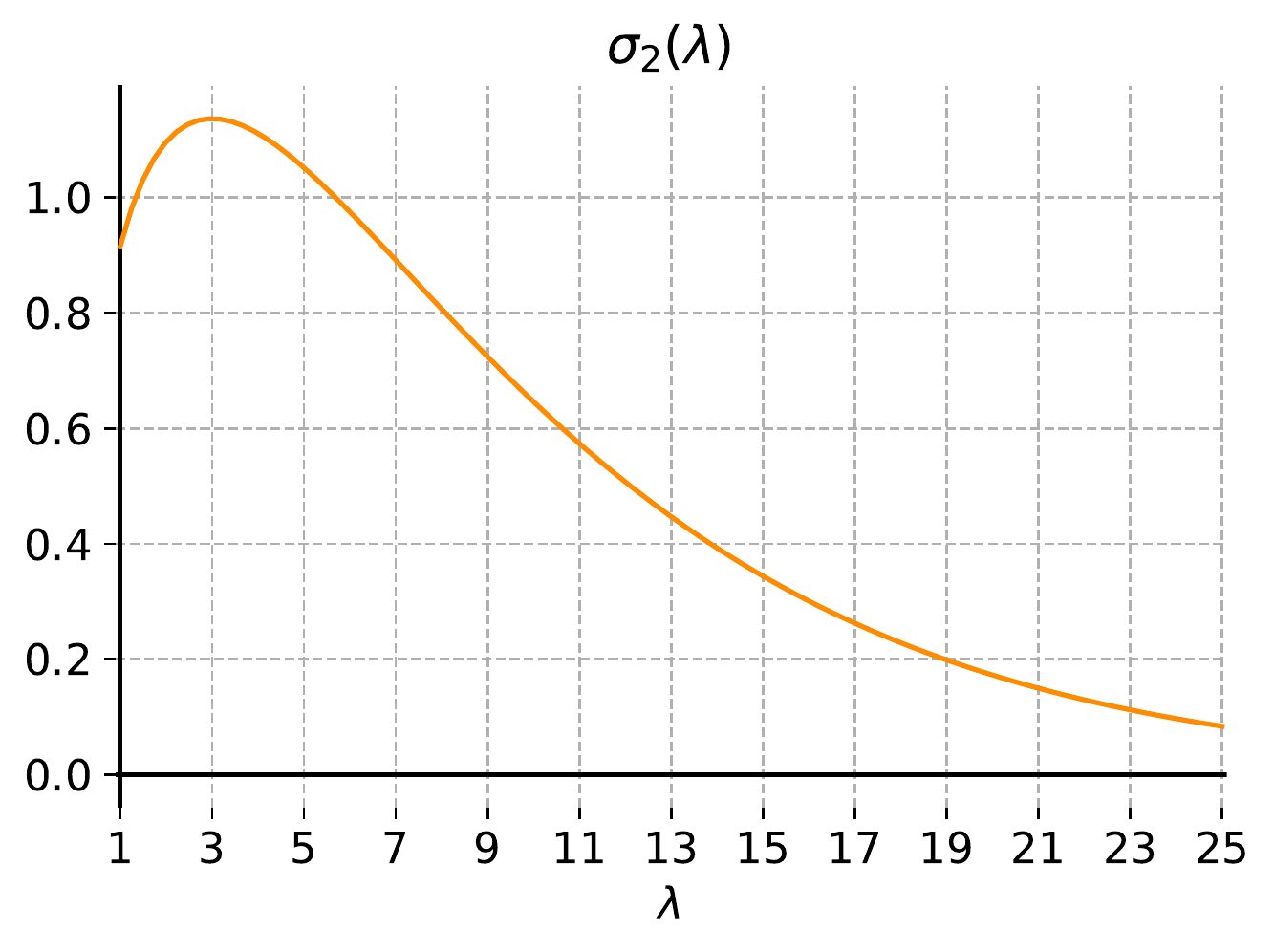}
    \caption{Numerical plot of $\sigma_2(\lambda)$}
    \end{minipage}
\end{figure}

\begin{remark}
    Again, by comparing the critical case with \( \mu = 0 \) and taking the limit
    of the supercritical case as \( \lambda \to 1^+ \), we observe that the
    answers have the same growth order \( \Theta(1) \), but there is
    again a discontinuity between $\sigma_2(1) \approx 0.916215$,
    while $\mathbb P_{\mathsf{elementary}}(n, \tfrac{1}{n}) \approx 0.699687$. Again, when $\mu > 0$, the contour integral
    could be expressed as a sum of residues and we obtain for
    \( \lambda = 1 + \mu n^{-1/3} \),
    \[
        \mathbb P_{\mathsf{elementary}}(n, p) \sim
        - \dfrac{e^{-\mu^3/6}}{2}
        \sum_{k \geq 1}
        \dfrac{\exp(2^{-1/3} \mu a_k')}{\ai''(a_k')}
        =
        - \dfrac{e^{-\mu^3/6}}{2}
        \sum_{k \geq 1}
        \dfrac{\exp(2^{-1/3} \mu a_k')}{a_k' \ai(a_k')}
        \, ,
    \]
    where \( (a_k')_{k=1}^\infty \) are the roots of \( \ai'(\cdot) \) sorted by
    their absolute values. As in~\cref{remark:sum:residues}, the first summand
    becomes dominant in the supercritical phase.
\end{remark}


A similar analysis can reveal the probabilities that a digraph has a given
ensemble of complex components, for example when there is exactly one complex
component of  excess \( r \) whose kernel is cubic or has a deficiency \( d \)
(see~\cref{th:complex:scc}).

\begin{theorem}
\label{theorem:probability:one:strong:component}
Let \( p = \lambda / n \), \( \lambda \geq 0 \).
The probability \( \mathbb P(n, p) \)
that a random multidigraph sampled from \( \DiGilbMulti(n, p) \)
has exactly one complex strong component
and that this component has
a given kernel of excess \( r \) and deficiency \( d \)
(hence \( 2r - d \) vertices and \( 3r - d \) edges),
satisfies
\[
    \mathbb P(n, p) \sim
    \dfrac{1}{(2r - d)!}
    \times
    \begin{cases}
        n^{-r} \lambda^{3r-d} (1 - \lambda)^{-3r + d},
        & \lambda < 1; \\
        n^{-d/3} \varphi_{r,d}(\mu)
        ,
        & \lambda = 1 + \mu n^{-1/3}; \\
        n^{1/3 - d/3} \sigma_{r,d}(\lambda)
        \exp \Big(
            {-\alpha(\lambda)n} + a_1' \beta(\lambda) n^{1/3}
        \Big)
        ,
        & \lambda > 1,
    \end{cases}
\]
where the functions \( \alpha(\cdot) \), \( \beta(\cdot) \) are
from~\cref{theo:MD_model_lambda_fixed}, \( a_1' \) is the zero of \(
\ai'(\cdot) \) with the smallest modulus,
\[
    \varphi_{r,d}(\mu)
    =
    (-1)^{1 + 3r - d}
    2^{- 2/3 - r + d/3}
    \dfrac{1}{2 \pi i}
    \int_{-i \infty}^{i \infty}
    \dfrac{
    \ai(1 - 3r + d; -2^{1/3}s)}
    {\ai'(-2^{1/3}s)^2}
    e^{-\mu s - \mu^3/6}
    \mathrm d s,
\]
and
\[
    \sigma_{r,d}(\lambda)
    =
    (-1)^{1 + 3r - d}
    2^{- 4/3 - r + d/3}
    (\lambda - 1)
    \lambda^{1/6 + d/3}
    e^{1/6-\lambda/6}
    \dfrac{\ai(1 - 3r + d; a_1')}
    { (a_1' \ai(a_1'))^2}
    \, .
\]
\end{theorem}

\begin{proof}
    According to \cref{th:summary:multidigraph},
    the probability we are seeking is
    \[
        \mathbb P(n, p) =
        \dfrac{p^r}{(2 r - d)!}
        e^{- n^2 p / 2} n ! [z^n]
        \frac{\phi_{1 - 3 r + d}(z,p;x \mapsto x^{2r-d})}
        {\phi_1(z,p)^2}.
    \]
    As usual, we let $w = p$ in the following.
    We also denote by
    \[
        \widehat H_{\mathcal S}(z, w) =
        \dfrac{w^r}{(2 r - d)!}
        \frac{\phi_{1 - 3 r + d}(z,w;x \mapsto x^{2r-d})}
        {\phi_1(z,w)^2}
    \]
    the multi-graphic generating function
    of the multidigraph family under investigation.

    \noindent \textbf{Case $\lambda < 1$.}
    Direct application of Part (a) of \cref{lem:cauchy}
    with $p_1 = 2$, $r_1 = 1$, $p_2 = -1$, $r_2 = 1 - 3r + d$,
    $F_1(z) = z^{2r - d}$, $F_2(z) = 1$,
    and therefore with \( R = 2 - (1 - 3r + d) = 1 + 3r - d \).

    \noindent \textbf{Case $\lambda = 1 + \mu n^{-1/3}$.}
    Direct application of Part (b) of \cref{lem:cauchy}.

    \noindent \textbf{Case $\lambda > 1$.}
    Let \( a_1' > a_2' > \cdots \) denote the roots of \( \ai'(\cdot) \),
    fix \( \theta \in (2^{-1/3} a_2', 2^{-1/3} a_1') \)
    and set \( \rho = (1 - \theta w^{2/3}) (ew)^{-1} \).
    Let us also write \( \phi_\ast(z, w) \) for \(\phi_{1 - 3r + d}(z, w; F(\cdot)) \),
    and denote by \( \varsigma_j(w) \) the $j$th zero of \( \phi_1(z, w) = 0 \),
    ordered by increasing modulus.
    Then \cref{theorem:zeros:phi:r} implies that,
    as $w$ tends to $0$, $\varsigma_1(w) < \rho < \varsigma_2(w)$.
    Thus, expressing the coefficient extraction as a Cauchy integral
    and applying the residue theorem, we obtain
    \begin{equation}
        \label{eq:residue}
        [z^n] \widehat H_{\mathcal S}(z, w) =
        - \dfrac{w^r}{(2r-d)!}
            \mathrm{Res}_{z = \varsigma_1(w)}
            \dfrac{\phi_{\ast}(z, w)}{z^{n+1} \phi_1(z, w)^2}
            +
            \dfrac{w^r}{(2r-d)!}
            \dfrac{1}{2 \pi i}
            \oint_{|z| = \rho}
            \dfrac{\phi_{\ast}(z, w)}{z^{n+1} \phi_1(z, w)^2}
            \mathrm dz.
    \end{equation}
    The second term is bounded using~\eqref{eq:cauchy_up}:
    from \cref{lem:cauchy}
    \[
        \dfrac{w^r}{(2r-d)!}
            \dfrac{1}{2 \pi i}
            \oint_{|z| = \rho}
            \dfrac{\phi_{\ast}(z, w)}{z^{n+1} \phi_1(z, w)^2}
            \mathrm dz
        =
        \bigO \left(
            \frac{w^{r+2/3}}{\rho^n}
            \frac{\phi_\ast(\rho, w)}{\phi_1(\rho, w)^2}
        \right).
    \]
    According to Taylor's theorem, as
    \( z \to \varsigma_1(w) \),
    \[
        \phi_1(z, w)
        =
        (z - \varsigma_1(w)) \partial_z \phi_1(\varsigma_1(w), w)
        +
        \dfrac{(z - \varsigma_1(w))^2}{2}
        \partial_z^2 \phi_1(\varsigma_1(w), w)
        + O((z - \varsigma_1(w))^3),
    \]
    and
    \[
        \phi_\ast(z, w)
        =
        \phi_\ast(\varsigma_1(w))
        +
        (z - \varsigma_1(w)) \partial_z \phi_\ast(\varsigma_1(w), w)
        + O((z - \varsigma_1(w))^2).
    \]
    By constructing the Laurent series of the fraction and extracting the
    coefficient at \( (z - \varsigma_1(w))^{-1} \) of
    \( \frac{\phi_\ast(z, w)}{z^{n+1} \phi_1(z, w)^2} \), we obtain
    \[
        \mathrm{Res}_{z = \varsigma_1(w)}
        \dfrac{\phi_{\ast}(z, w)}{z^{n+1}\phi_1(z, w)^2}
        =
        \left.
        \dfrac{
            \partial_z \phi_\ast(z, w) \partial_z \phi_1(z, w)
            -
            \phi_\ast(z, w)
            \big(
            \partial_z^2 \phi_1(z, w)
            +
            \tfrac{n+1}{z} \partial_z \phi_1(z, w)
            \big)
        }
        {
            z^{n+1}
            (\partial_z \phi_1(z, w))^3
        }
        \right|_{z = \varsigma_1(w)}.
    \]
    Let us gather some identities to simplify the numerator.
    We consider the point $z = \varsigma_1(w)$.
    \cref{theorem:zeros:phi:r} implies
    $\partial_z^2 \phi_1(z, w) \sim - 2 e \partial_z \phi_1(z, w)$
    and, using its very last statement,
    $\partial_z \phi_\ast(z, w) \sim - \frac{1}{z w} \phi_{1-3r+d}(z, w; x \mapsto x F(x))$.
    The asymptotics of $\varsigma_1(w)$ implies $z w \sim 1/e$, so
    $\partial_z \phi_\ast(z, w) \sim - e \phi_\ast(z, w)$
    and $(n+1) / z \sim \lambda e$.
    Part $(c)$ of \cref{prop:complex:general} is applied
    to approximate $\phi_\ast(\varsigma_1(w), w)$,
    $\phi_{1-3r+d}(\rho, w; x \mapsto x F(x))$ and $\phi_1(\rho, w)$.
    \cref{theorem:zeros:phi:r} is applied to approximate
    $\partial_z \phi_1(\varsigma_1(w), w)$. Note that under this approximation,
    we also have
    $\phi_{1-3r+d}(\rho, w; x \mapsto x F(x)) \sim \phi_\ast(\varsigma_1(w), w)$.
    The numerator of the residue simplifies into
    \[
        \partial_z \phi_\ast(z, w) \partial_z \phi_1(z, w)
        -
        \phi_\ast(z, w)
        \partial_z^2 \phi_1(z, w)
        - \tfrac{n+1}{z}\phi_\ast(z, w) \partial_z \phi_1(z, w)
        \sim
        e (1 - \lambda)\phi_\ast(z, w) \partial_z \phi_1(z, w)
    \]
    and~\cref{eq:residue} becomes
    \[
        [z^n] \widehat H_{\mathcal S}(z, w) \sim
        \dfrac{w^r}{(2r-d)!}
        \frac{e (\lambda - 1) \phi_\ast(\varsigma_1(w), w)}
            {\varsigma_1(w)^{n+1} \left( \partial_z \phi_1(\varsigma_1(w), w) \right)^2}
            +
            \bigO \left(
            \frac{w^{r+2/3}}{\rho^n}
            \frac{\phi_\ast(\rho, w)}{\phi_1(\rho, w)^2}
        \right).
    \]
    The error term is negligible compared to the first summand,
    and we obtain
    \[
    \mathbb P(n,p) =
    e^{-w n^2 / 2} n! [z^n] \widehat H_{\mathcal S}(z, w)
        \sim
        \dfrac{1}{(2r-d)!}
        n^{\tfrac{1-d}{3}} \sigma_{r,d}(\lambda)
        e^{
            {-\alpha(\lambda)n} + a_1' \beta(\lambda) n^{1/3}
        }.
    \]
\end{proof}

\begin{remark}
    It may seem at first glance that the growth order of the supercritical
    probability receives an additional factor \( n^{1/3} \) compared to the
    respective critical phase, when $\lambda \to 1^+$. However, if one develops the contour integral
    representing \( \varphi_{r,d}(\mu) \) as a sum of residues, then the
    formula obtained by taking the first residue has the same growth rate as the limit
    of supercritical probability as \( \lambda \to 1 \). Indeed, as the contour
    integral contains a square of the function \( \ai'(\cdot) \) in the
    denominator, its zero is not a simple zero any longer, but a double one.
    Therefore, in order to extract the dominant residue, one has to take the
    cross-product of the derivatives of the numerator and the denominator, using
    the formula \( \mathrm{Res}_{z = \rho} f(z) / g(z)^2 = (f(\rho)g''(\rho) -
    f'(\rho) g'(\rho)) / g'(\rho)^3 \). It turns out that the derivative of \(
    e^{-\mu \tau} \) with respect to \( \tau \) gives an additional factor \(
    \mu = (\lambda - 1) n^{1/3} \) which then dominates the two other summands.
    This proves that the growth rates corresponding to the critical and the
    supercritical phases are identical.
\end{remark}

\begin{corollary}
    \label{corollary:bicyclic}
    Let \( p = n^{-1}(1 + \mu n^{-1/3}) \), \( \mu \in \mathbb R \).
    The probability \( \mathbb P(n, p) \) that one component of a random
    multidigraph sampled from \( \DiGilbMulti(n,p) \) is complex and
    bicyclic is
    \begin{equation}
    \label{eq:probability:bicyclic}
        \mathbb P(n,p) \sim
        \dfrac{1}{8}
        \cdot
        \dfrac{1}{2 \pi i}
        \int_{-i \infty}^{i \infty}
        \dfrac{
            \ai(-2,\tau)
        }{
            (
                \ai'(\tau)
            )^2
        }
        e^{2^{-1/3}\mu \tau - \mu^3/6}
        \mathrm d \tau.
    \end{equation}
    Moreover, when \( \mu \to -\infty \), this probability is asymptotically
    equal to \( |\mu|^{-3} / 2 \). The probability that the kernel of the
    complex component is not cubic is also at most \( O(n^{-1/3}) \).
\end{corollary}

\begin{figure}[ht]
    \centering
    \includegraphics[width=.7\linewidth]{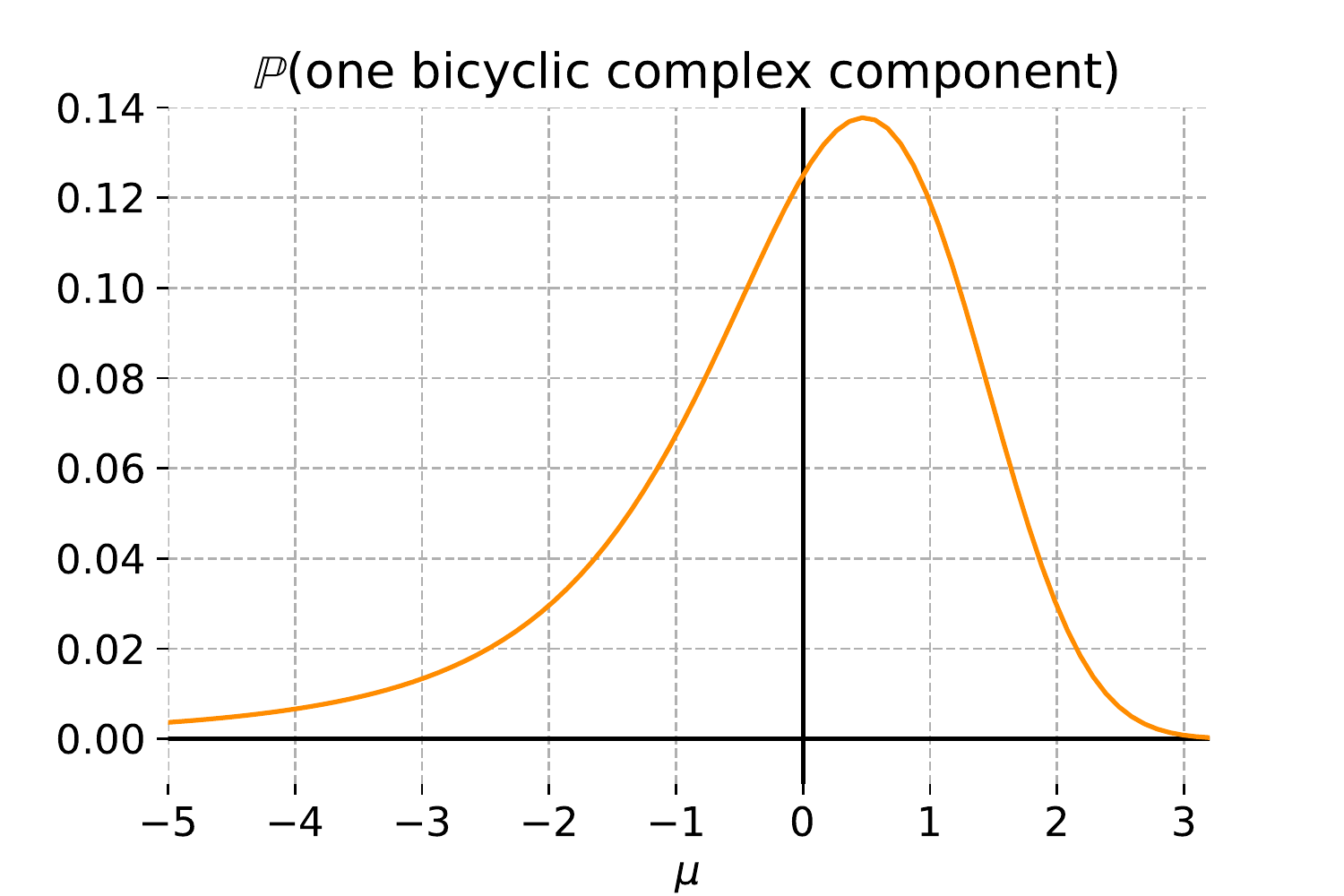}\\
    \caption{
    \label{fig:probability:bicyclic}
      Numerical plot of the theoretical limit of the
      probability $\mathbb P(n,p)$ that a random multidigraph has one 
      bicyclic complex component inside the critical window for
      $\mu \in[-5,3]$.} 
\end{figure}

Another interesting phenomenon happens when \( \mu = 0 \) for the case with one
bicyclic complex component. If we expand the contour integral into an infinite
sum of residues, \textit{each of the residues will be zero}!
However, by considering the residue at infinity, we can show that the integral
takes a finite rational value.

\begin{theorem}
    The probability that a random multidigraph \( \DiGilbMulti(n, \tfrac1n) \)
    has one  bicyclic complex component is asymptotically equal to \( 1/8 \).
\end{theorem}
\begin{proof}
    According to~\cref{corollary:bicyclic}, we need to evaluate the contour
    integral at \( \mu = 0 \).
    First of all, we can replace the arguments \( (-\tau) \) of the generalised
    Airy functions in~\eqref{eq:probability:bicyclic} by \( \tau \), because the
    generalised Airy function is even on the imaginary axis. Then, the zeroes of
    the denominator are \( (a_i')_{i=1}^\infty \)---the zeroes of the
    derivative of the Airy function. Let us calculate the residues at each of
    these zeros.

    By expanding the Taylor series of the numerator and denominator, and taking
    the cross-product, we obtain
    \[
        \left.
        \mathrm{Res}\,
        \dfrac{
            \ai(-2,\tau)
        }{
            (
                \ai'(\tau)
            )^2
        }
        \right|_{\tau = a_i'}
        =
        \dfrac{1}{(a_i')^2 \ai(a_i')^2}
        \left(
            \ai(-1, a_i')
            -
            \dfrac{\ai(-2, a_i')}{a_i'}
        \right).
    \]
    Next, we can use the linear relation~\eqref{eq:airy:recurrence} to obtain
    \begin{equation}
    \label{eq:ai:minus:two}
        \ai(1, z) = -\ai(-2, z) + z \ai(-1, z).
    \end{equation}
    By putting \( z = a_i' \), the term \( \ai(1, z) \) vanishes and it turns
    out that each of the residues becomes zero.

    In order to evaluate the integral, we compute the indefinite integral
    explicitly, and then substitute the endpoints. As pointed out
    in~\cite{hughes1968},
    \( \ai(-1; z) = \int_{-\infty}^z \ai(t) \mathrm dt - 1 \),
    and by comparing it with~\cite[(9.10.2)]{NIST}, we obtain
    \[
        \ai(z) \hi'(z) - \ai'(z) \hi(z)
        =
        \dfrac{\ai(-1; z) + 1}{\pi},
    \]
    where \( \hi(z) \) is the \emph{Scorer function} satisfying
    \begin{equation}
    \label{eq:scorer}
        \hi''(z) = x \hi(z) + \dfrac{1}{\pi}
        \quad \text{and} \quad
        \hi(z) = \dfrac{1}{\pi} \int_{0}^{+\infty}
        \exp \left(
            - \dfrac{1}{3} t^3 + zt
        \right)
        \mathrm dt.
    \end{equation}
    If \( \bi(\cdot) \) is the second solution of the Airy equation, we can use
    the Wronskian formula~\cite[(9.2.7)]{NIST}
    \[
        \ai(x) \bi'(x) - \ai'(x) \bi(x) = \dfrac{1}{\pi}
    \]
    and therefore deduce, using~\eqref{eq:ai:minus:two}, that
    \begin{align*}
        \dfrac{\mathrm d}{\mathrm dz}
        \dfrac{\pi(\bi'(z) - \hi'(z))}{\ai'(z)}
        &=
        \pi \left(
            \dfrac{
                (\bi''(z) - \hi''(z))\ai'(z)
                -
                (\bi'(z) - \hi'(z)) \ai''(z)
            }{
                (\ai'(z))^2
            }
        \right)
        \\
        &=
        \dfrac{
            -z - \ai'(z) + z \ai(-1; z) + z
        }{
            (\ai'(z))^2
        }
        =
        \dfrac{\ai(-2; z)}{(\ai'(z))^2}.
    \end{align*}
    This gives us the explicit antiderivative of the integrand expression.

    In order to compute the corresponding definite integral, we need to find the
    limiting values of the antiderivative at the imaginary infinity.
    Note that, according to the derivative of the integral
    representation~\eqref{eq:scorer}, the function \( \hi'(z) \) is bounded on the
    imaginary axis by \( \pi^{-1} | \int_0^\infty \exp(-t^3/3) \mathrm
    dt| \). At the same
    time, the limits of \( \ai'(\cdot) \) and \( \bi'(\cdot) \)
    at the imaginary infinity are (see~\cite[(9.7.10),(9.7.12)]{NIST})
    \[
        \ai'(z) \sim
        \dfrac{(-z)^{1/4}}{\sqrt \pi}
        \sin \left(
            \dfrac{2 (-z)^{3/2}}{3} - \dfrac{\pi}{4}
        \right)
        \quad \text{and} \quad
        \bi'(z) \sim
        \dfrac{(-z)^{1/4}}{\sqrt \pi}
        \cos \left(
            \dfrac{2 (-z)^{3/2}}{3} - \dfrac{\pi}{4}
        \right).
    \]
    Consequently,
    \[
        \lim_{x \to i \infty} \frac{\pi(\bi'(x) - \hi'(x))}{\ai'(x)}
        = \pi i
        \quad \text{and} \quad
        \lim_{x \to -i \infty} \frac{\pi(\bi'(x) - \hi'(x))}{\ai'(x)}
        = -\pi i.
    \]
    This implies
    \[
        \dfrac{1}{2 \pi i}
        \int_{- i \infty}^{i \infty}
        \dfrac{\ai(-2; \tau)}{(\ai'(\tau))^2}
        \mathrm d \tau = 1
    \]
    and finishes the proof.
\end{proof}

\begin{remark}
\label{remark:knessl}
    Definite integrals involving Airy functions have been studied by many
    authors since at least 1966~\cite{aspnes1966}.
    Using the Wronskian of the Airy function \( \ai(\cdot) \) and the second
    solution of the Airy equation \( \mathrm{Bi}(\cdot) \),
    Albright and Gavathas~\cite{albright1986} presented a
    systematic way to compute indefinite integrals of the form
    \[
        \int \dfrac{1}{\ai^2(x)} f \left(
            \dfrac{\mathrm{Bi}(x)}{\ai(x)}
        \right)
        \mathrm dx
        \, .
    \]
    This includes, as a particular case, the
    integral \( (2 \pi i)^{-1} \int_{-i\infty}^{i \infty} \ai^{-2}(x) \mathrm dx
    = 1 \), which is also mentioned in a paper of Knessl~\cite{knessl2000exact}
    in the context of reflected Brownian motions with drift.

    We expect that such integrals of generalised Airy functions could be
    probably of interest to a broader community of researchers in the context of
    mathematical physics, stochastic processes and queueing theory.
    As an example, in the paper of Knessl~\cite{knessl2000exact} the parameters of
    interest are expressed in terms of integrals of Airy functions. In later
    papers by Janson~\cite{janson2012} and others~\cite{janson2010}, a Brownian
    motion with parabolic drift is considered as well. The expected location of its maximum (and
    the moments) turn out to be expressible in terms of integrals of Airy
    functions. Interestingly, Brownian motion with parabolic drift appears in a
    different context, namely the study of critical digraphs, in a recent paper by
    Goldschmidt and Stephenson~\cite{Christina2019}. Taking into account that
    the expression provided by Knessl~\cite[(3.8), Theorem 3]{knessl2000exact} is
    very much reminiscent of the expressions we obtain for critical
    probabilities, we expect that the Airy integrals might provide an additional
    link reinforcing the quantitative study of the scaling limits of directed
    graphs from the Brownian viewpoint.

    Below, in~\cref{section:numerical:results}, we provide accurate numerical
    values of the first several Airy integrals, the necessary constants for the
    supercritical case for the first several terms, and the figures representing
    the graphs of the Airy integrals proportional to the critical probability
    distribution functions \( \varphi_{r,d}(\cdot) \).
\end{remark}

\section{Asymptotics of simple digraph families}
\label{section:asymptotics:simple}

We started our analysis with multidigraphs,
in the previous section,
because the calculations were slightly simpler compared to digraphs.
We follow the same proofs as in \cref{section:asymptotics:multidigraphs},
applying the results from \cref{section:external:integration:simple}
to replace those of \cref{section:external:integration:multi}.
It is instructive to see how the model choice affects the asymptotics.

\subsection{Asymptotics of acyclic digraphs}
\label{section:dag:simple}

Now, we can state the  results on the probability that a
random digraph is acyclic when $\lambda < 1$, $\lambda \sim 1$ and
$\lambda >1$.  These theorems are handled with the same techniques as
in the previous section by means of
\cref{corollary:lemma:analytic:first2,cor:lem:cauchy}. As we can see below,
results are roughly the same as in the case of multidigraphs except
that models of digraphs involve some additional factor due to
forbidden configurations.

\begin{theorem}\label{theo:SD_model_lambda_fixed}
  Let \( p = \lambda/n \), \( \lambda \geq 0 \) fixed.
  Then, the probability $\mathbb P_{\mathsf{acyclic}}(n,p)$  that a random
  simple digraph
  $D$ is acyclic satisfies the following asymptotic
  formulas as $n\to \infty$:
    \[
        \mathbb P_{\mathsf{acyclic}}(n,p) \sim
        \begin{cases}
            \delta_1(\lambda) (1 - \lambda),
                & \lambda \in [0,1) \, ; \\
            \delta_1(1) \varphi(\mu) n^{-1/3},
                & \lambda = 1 + \mu n^{-1/3} \, ; \\
            \delta_2(\lambda) n^{-1/3}
            \gamma_2(\lambda)
            \exp \left(
                - \alpha(\lambda) n + a_1 \beta(\lambda) n^{1/3}
            \right), &
            \lambda > 1,
        \end{cases}
    \]
    where
    \begin{align*}
        \delta_1(\lambda) =
        \begin{cases}
            e^{\lambda + \lambda^2/2} & \mbox{ if } D \in   \DiER(n,p)  \\
            e^{\lambda } & \mbox{ if } D \in   \DiERBoth(n,p)  \\
        \end{cases},
        \quad
        \delta_2(\lambda) =
        \begin{cases}
            e^{-{\lambda^2/4+5\lambda/2-3/4}}
            & \mbox{ if } D \in   \DiER(n,p) \\
            e^{-{\lambda^2/4+3\lambda/2-1/4}}
            & \mbox{ if } D \in   \DiERBoth(n,p)
        \end{cases}
    \end{align*}
    and the functions $\alpha(\cdot), \beta(\cdot)$, $\gamma_2(\cdot)$
    are the same as given
in~\cref{theo:MD_model_lambda_fixed}, and $\varphi(\cdot)$ is the same as
given in \cref{th:acyclic-window}.
\end{theorem}

\begin{proof}
  The proof is mainly the same as in
\cref{theo:MD_model_lambda_fixed}.
  According to \cref{th:summary:digraph}, the probability
  $\mathbb P_{\mathsf{acyclic}}(n,p)$ that a random simple digraph is acyclic
  is given by
  \[
      \mathbb P_{\mathsf{acyclic}}(n,p) =
      (1 - p)^{\binom{n}{2}} n! [z^n]
      \widehat D_{\mathrm{DAG}}^{(\text{simple})}
      \left(
        z, \dfrac{p}{1 -ap}
      \right)
       = (1 - p)^{\binom{n}{2}} n!
      [z^n]
      \frac{1}{\widetilde\phi(z,w)}\Big|_{w=\frac{p}{1-ap}}\, ,
\]
where $a=1$ corresponds to the statement for the $\DiERBoth(n,p)$ model,
and $a=2$ for the $\DiER(n,p)$ model.

\textbf{Case $\lambda\in [0,1)$.}
We apply Part (a) of \cref{cor:lem:cauchy} with
$\widetilde \xi(z,w) = \widetilde \phi_0(z,w;1) = \widetilde \phi(z,w)$,
so that $R=0$, and obtain
\[
  (1 - p)^{\binom{n}{2}} n!
  [z^n] \frac{1}{\widetilde\phi(z, w)}
  \Big|_{w=\frac{p}{1-ap}}
  \sim
  e^{\lambda+(a-1)\lambda ^2/2 }(1 - \lambda)\, .
\]

\textbf{Case $\lambda = 1 + \mu n^{-1/3}$.}
We apply Part (b) of \cref{cor:lem:cauchy} to obtain
\[
    (1 - p)^{\binom{n}{2}} n!
    [z^n] \frac{1}{\widetilde\phi(z, w)} \sim
    e^{(a+1)/2}\varphi(\mu) n^{-1/3}\, .
\]

\textbf{Case $\lambda >1$.}
  By taking the same contour of integration as in the
  proof of \cref{theo:MD_model_lambda_fixed}, the residue theorem gives
  \[
  [z^n] \widehat D_{\mathrm{DAG}}^{(\text{simple})}(z,w)=-\frac{1}{\widetilde \varrho_1(w)^{n+1}
    \partial_z\widetilde \phi (\widetilde\varrho_1(w),w)}+
  \frac{1}{2\pi i}\oint_{|z|=\rho}\frac{1}{
\widetilde \phi(z,w)z^{n+1}}
  \mathrm dx.
\]
Then \cref{cor:lem:cauchy} implies
\[
  [z^n] \widehat D_{\mathrm{DAG}}^{(\text{simple})}(z,w)=
  -\frac{1}{\widetilde \varrho_1(w)^{n+1}\partial_z
\widetilde\phi(\widetilde\varrho_1(w),w)}
  +O\left(\frac{w^{2/3}}{|\widetilde \phi(\rho,w)|\rho^n}\right).
\]
Next, by means of asymptotic formulas for
$\widetilde\varrho_1(w)$ and $ \partial_z \widetilde
\phi (\widetilde \varrho_1(w),w)$ in \cref{theo:smallest_zero_simple}  and
$\widetilde \phi(\rho,w)=\widetilde\phi_0(\rho,w;1)$
in~\cref{corollary:complex:general}, and by letting
$w=\lambda/(n-a\lambda)$, we obtain,
with the help of a computer algebra system,
\begin{align}
-\frac{n!(1-p)^{\binom{n}{2}}}{\widetilde \varrho_1(w)^{n+1}
\partial_z \widetilde \phi(\widetilde \varrho_1(w),w)}
  & \sim  e^{\lambda/2-{\lambda}^{2}/4 -a/2+1/4+a\lambda}\gamma_2(\lambda)
    n^{-1/3} e^{-\alpha(\lambda)n+a_1\beta(\lambda)n^{1/3}},
    \label{eq:sup_m_simple} \\[.5em]
\frac{n!(1-p)^{\binom{n}{2}} w^{2/3}}
{|\widetilde\phi(\rho,w)|\rho^n} &=
 O\left( n^{-1/3}e^{-\alpha(\lambda)n+\theta\beta(\lambda)n^{1/3}}\right),
\label{eq:sup_e_simple}
\end{align}
where the constant
$\gamma_2(\cdot)$ is precisely as given in
\cref{theo:SD_model_lambda_fixed},
 $\alpha(\cdot)$ and $\beta(\cdot)$
are precisely as given in
\cref{theo:MD_model_lambda_fixed}.
Choosing $\theta$
in the same way as in \cref{theo:MD_model_lambda_fixed},
the right-hand side of \eqref{eq:sup_m_simple} is
indeed the main term of $\mathbb P_{\mathsf{acyclic}}(n,p).$
The statement of the theorem follows by setting $a=1$
for the  $\DiERBoth(n,p)$ model and $a=2$ for $\DiER(n,p)$.
\end{proof}

\begin{figure}[ht]
    \RawFloats
    \begin{minipage}[t]{.47\textwidth}
      \begin{center}
        \includegraphics[width=\linewidth]{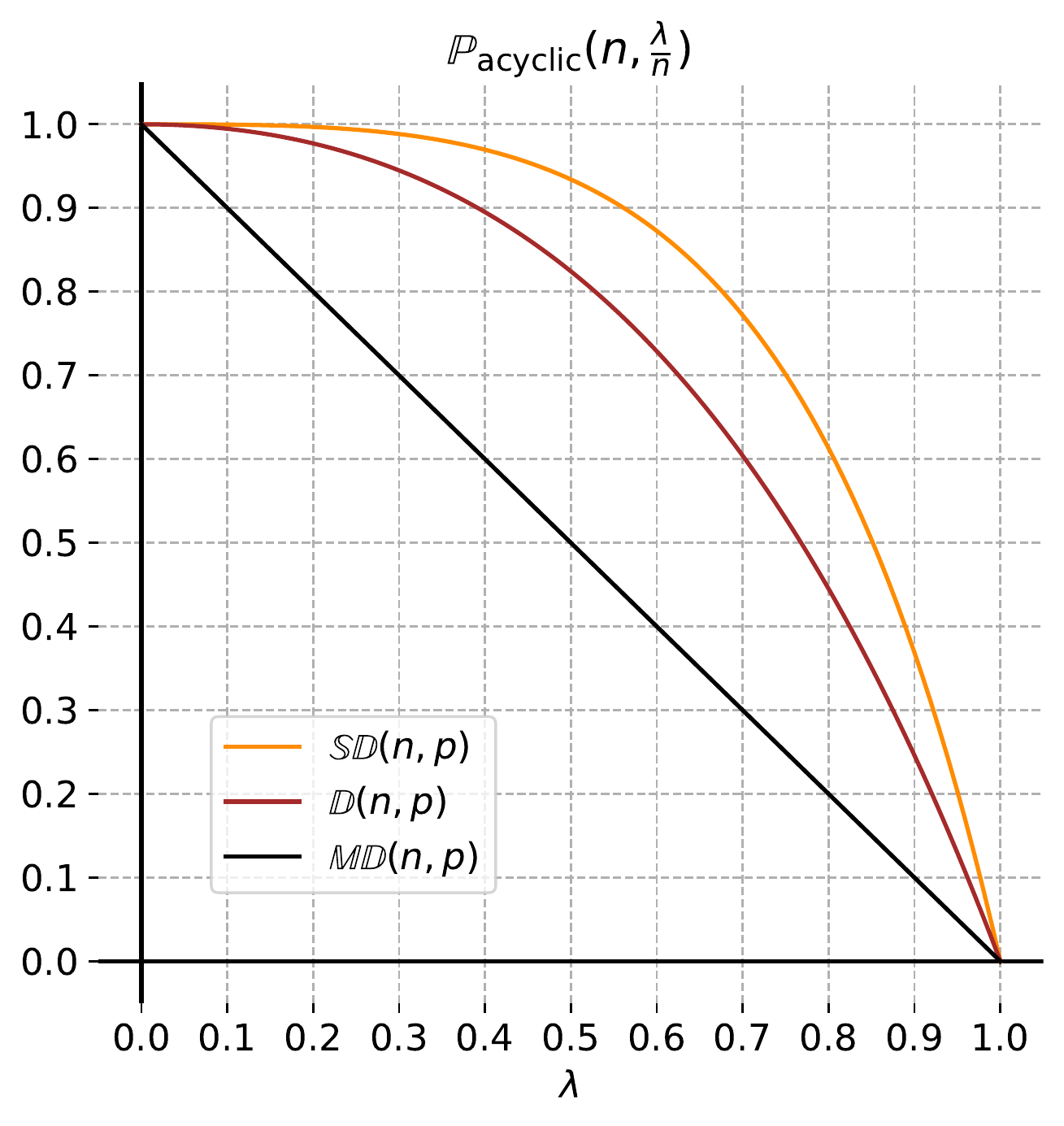}
      \end{center}
      \caption{
      \label{fig:subcritical_acyclic}
      Numerical plots of the theoretical limit of the
      probability $\mathbb P_{\mathsf{acyclic}}(n,p)$ that a random
      digraph is acyclic
      for $np = \lambda <1$.
      }
    \end{minipage}\hfill
    \begin{minipage}[t]{.47\textwidth}
      \begin{center}
        \includegraphics[width=\linewidth]{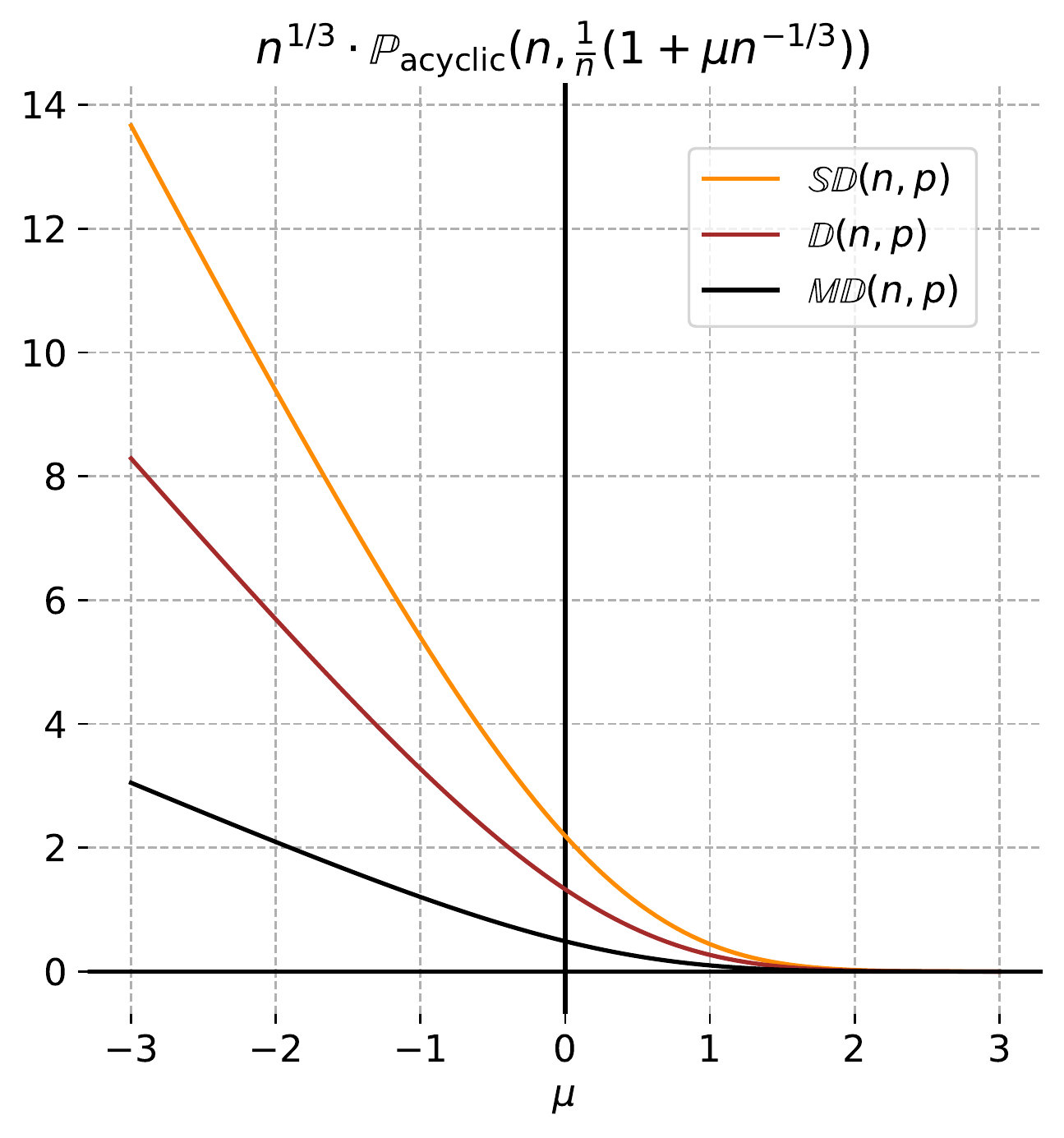}
      \end{center}
      \caption{\label{fig:plt:varphi}\label{fig:critical_acyclic}
      Numerical plots of the theoretical limit of the rescaled
  probability  $n^{1/3}\mathbb P_{\mathsf{acyclic}}(n,p)$ that a
  random digraph is acyclic
  inside the critical window for $\mu \in[-3,3]$.}
    \end{minipage}
\end{figure}

\begin{corollary} \label{th:digraph-acyclic-window}
  Let \( p = (1+\mu n^{-1/3})/n \) with \( |\mu| \) going to infinity with $n$.
  Then, the probability $\mathbb P_{\mathsf{acyclic}}(n,p)$  that a random digraph
  $D$ is acyclic satisfies the following asymptotic
  formula as $n\to \infty$:
    \[
      \mathbb P_{\mathsf{acyclic}}(n,p) \sim
      \begin{cases}
          \delta |\mu|n^{-1/3} & \mbox{ if $\mu \to - \infty$ and $|\mu| = \smallo(n^{1/3})$};\\
        \delta
        \dfrac{2^{-2/3}}{\ai'(a_1)}
        n^{-1/3}  \exp\left( - \dfrac{\mu^3}{6}+2^{-1/3}a_1\mu\right)
        & \mbox{ if $\mu \to +\infty$ and $\mu = \smallo(n^{1/12})$},
      \end{cases}
    \]
    where $a_1$ is the zero of the Airy function of smallest modulus $a_1 \approx -2.338107$ and
        \[
          \delta=
          \begin{cases}
              e^{3/2}& \mbox{ if } D \in   \DiER(n,p);  \\
              e & \mbox{ if } D \in   \DiERBoth(n,p).
          \end{cases}
    \]
\end{corollary}
\begin{proof}
    We follow the same proof as for \cref{th:acyclic-window}.
    As in the proof of the previous theorem, we start with the expression
    \[
      \mathbb P_{\mathsf{acyclic}}(n,p) =
      (1 - p)^{\binom{n}{2}} n!
      [z^n]
      \frac{1}{\widetilde\phi(z,w)}\Big|_{w=\frac{p}{1-ap}}\, .
    \]

    \noindent \textbf{First step.}
    For $\mu$ going to $-\infty$ with $|\mu| \geq n^{\epsilon}$
    for some small enough positive $\epsilon$,
    we apply Part (a) of \cref{cor:lem:cauchy} with
    $\widetilde \xi(z,w) = \widetilde \phi_0(z,w;1) = \widetilde \phi(z,w)$, so $R = 0$,
    and conclude
    \begin{equation} \label{eq:simple:acyclic:lower:critical}
        \mathbb P_{\mathsf{acyclic}}(n,p) \sim
        \delta |\mu| n^{-1/3}.
    \end{equation}

    \noindent \textbf{Second step.}
    For $p = n^{-1} (1 + \mu n^{-1/3})$ with $|\mu| = \smallo(n^{1/12})$,
    we now apply Part (b) of \cref{cor:lem:cauchy} to conclude,
    as in the previous theorem,
    \[
        \mathbb P_{\mathsf{acyclic}}(n,p) \sim
        \delta_1(1) \varphi(\mu) n^{-1/3}.
    \]
    From Part (a) and (b) of \cref{th:acyclic-window},
    since $\varphi(\mu)$ is independent of $n$,
    we deduce that as $\mu$ tends to $-\infty$, we have
    $\varphi(\mu) \sim |\mu|$.
    Thus, for $\mu$ going to $-\infty$ with $|\mu| = \smallo(n^{1/12})$, we have
    \[
        \mathbb P_{\mathsf{acyclic}}(n,p) \sim
        \delta_1(1) \varphi(\mu) n^{-1/3}.
    \]
    Since $\delta = \delta_1(1)$ and $1/12 < 1/3$,
    the asymptotic formula \eqref{eq:simple:acyclic:lower:critical}
    obtained in the first step of the proof
    holds for any $\mu$ going to $-\infty$ with $|\mu| = \smallo(n^{1/3})$.

    \noindent \textbf{Third step.}
    Combining Part (b) and (c) of \cref{th:acyclic-window},
    we deduce that as $\mu$ tends to infinity, we have
    \[
        \varphi(\mu) \sim
        \frac{2^{-2/3}}{\ai'(a_1)}
        \exp \left( - \frac{\mu^3}{6} + 2^{-1/3} a_1 \mu \right).
    \]
    Combining this asymptotic formula with the second step, we conclude
    that when $\mu \to \infty$ while $\mu = \smallo(n^{1/12})$, we have
    \[
        \mathbb P_{\mathsf{acyclic}}(n,p) \sim
        \delta \frac{2^{-2/3}}{\ai'(a_1)}
        \exp \left( - \frac{\mu^3}{6} + 2^{-1/3} a_1 \mu \right)
        n^{-1/3}.
    \]
\end{proof}

\begin{remark}
  As we have seen in \cref{remark:sum:residues}
  when $\lambda = 1+ \mu n^{-1/3}$ and $\mu > 0$,
  the probability $\mathbb P_{\mathsf{acyclic}}(n,p)$  that $\DiGilb(n,p)$ or
  $\DiGilbBoth(n,p)$ is acyclic  can be expressed as a
  summation by expanding the contour integral.  The results
  differ by a factor  $e^{3/2}$ for $\DiGilb(n,p)$ and by $e^1$ for
  $\DiGilbBoth(n,p)$.
\end{remark}

\subsection{Asymptotics of elementary digraphs and complex components}
\label{section:elementary:simple}

The next step after having the asymptotics of directed
acyclic graphs, is to express the probability that a random
digraph is \emph{elementary}
or contains \emph{one complex component}
of a given excess.  The following theorem expresses the probability that
a random digraph is elementary when $\lambda<1$,  $\lambda \sim 1 $
and  $\lambda >1$.
\begin{theorem}
  \label{theo:D_model_lambda_fixed}
  Let \( p = \lambda / n \).
  The probability that a random simple digraph
  $D\in \DiGilb(n,p)$ (resp. $D \in \DiGilbBoth(n,p)$)   is elementary
    is
    \[
        \mathbb P_{\mathsf{elementary}}(n,p) \sim
        \begin{cases}
            1, & \lambda < 1;
            \\
            - 2^{-2/3}\dfrac{1}{2 \pi i}
            \displaystyle\int_{-i \infty}^{i\infty}
            \dfrac{\exp(-\mu \tau - \mu^3/6)}{\ai'(-2^{1/3} \tau)}
            \mathrm d \tau, & \lambda = 1 + \mu n^{-1/3};
            \\[1em]
            \omega_{\mathsf{elem}}(\lambda)
            \exp \Big({-\alpha(\lambda) n} + a_1' \beta(\lambda)
            n^{1/3}  \Big)
            , & \lambda > 1\, ,
        \end{cases}
    \]
    where \( a_1' \approx -1.018793 \) is the zero of $\ai'(\cdot)$ of smallest modulus,
    and
    \begin{align*}
        \alpha(\lambda) &= \frac{\lambda^2-1}{2 \lambda} - \log \lambda,
        \\
        \beta(\lambda) &= 2^{-1/3} \lambda^{-1/3} (\lambda - 1),
        \\
      \omega_{\mathsf{elem}}(\lambda)
      &=
      -\frac{\sqrt{\lambda}}{2a_1'\ai(a_1')} \exp\left(
        - \frac{\lambda^2}{4}
        +\frac{\lambda}{3}
        - \frac{1}{12}
      \right) \, .
    \end{align*}
  \end{theorem}

  \begin{proof}
    According to \cref{th:summary:digraph}, the probability
    $\mathbb P_{\mathsf{elementary}}(n, p)$ that
    a random digraph is elementary is given by
    \[
     \left(1-\frac{\lambda}{n}\right)^{\binom{n}{2}} n!
      [z^n]\frac{1}{\widetilde \phi_1(z,w; F_a(\cdot))}\Big|_{w =
        \frac{\lambda}{n-a\lambda}}\, .
    \]
    We notice that the analytic function  $F_a(\cdot) = e^{C_a(\cdot)}$ satisfies the
    conditions of \cref{theorem:zeros:phi:r},
    as it is entire and does not vanish on $\mathds{C} \setminus \{0\}$.

    \noindent \textbf{Case $\lambda \in [0, 1)$.}
    We apply Part~(a) of \cref{cor:lem:cauchy} with
    $\widetilde \xi(z,w) = \widetilde \phi_1(z,w;F_a(\cdot))$
    so $k$, $r_1$, $p_1$ and $R$ are all equal to $1$:
    \[
        \left(1-\frac{\lambda}{n}\right)^{\binom{n}{2}} n! [z^n]
        \frac{1}{\widetilde \phi_1(z,w; F_a(\cdot))}
        \Big|_{w = \frac{\lambda}{n-a\lambda}}
        \sim 1\, .
    \]

    \noindent \textbf{Case $\lambda = 1 + \mu n^{-1/3}$.}
    We assume $\mu$ to be fixed or in a bounded real interval,
    and apply Part~(b) of \cref{cor:lem:cauchy} with
    $w=\lambda/(n-a\lambda)$.
    We obtain
    \[
        \left( 1-\frac{\lambda}{n} \right)^{\binom{n}{2}} n! [z^n]
        \frac{1}{\widetilde \phi_1(z,w; F_a(\cdot))}
        \Big|_{w = \frac{\lambda}{n-a\lambda}}
        \sim
        - 2^{-2/3}
        \frac{e^{(a+1)/2}}{F_a(1)}
        \dfrac{1}{2 \pi i}
        \int_{-i \infty}^{i\infty}
        \dfrac{\exp(-\mu \tau - \mu^3/6)}{\ai'(-2^{1/3} \tau)}
        \mathrm d \tau\, .
    \]
    We observe that $e^{(a+1)/2}F_a(1)^{-1} = 1$ for $a\in \{1,2\}$.

    \noindent \textbf{Case $\lambda >1$.}
    We follow the same argument as in
    \cref{theo:elementary:multi} by using the residue theorem.
    The probability
    $\mathbb P_{\mathsf{elementary}}(n, p)$ is asymptotically
    equivalent to
    \[
      -\frac{n!(1-p)^{\binom{n}{2}}}{\widetilde\varsigma_1(w)^{n+1}
        \partial_z\widetilde\phi_1(\widetilde\varsigma_1(w),w;F_a(\cdot))}
      \Big|_{w=\frac{\lambda}{n-a \lambda}} \, ,
    \]
    where $\widetilde \varsigma_{1}(w)$ is the solution to the equation
    \( \widetilde \phi_1(z, w; F_a(\cdot)) = 0 \) that is the closest to zero.
    By plugging in the expressions for $\widetilde\varsigma_1(w)$ and
    $\widetilde \phi_1(\widetilde\varsigma_1(w),w;F_a(\cdot))$ from
    \cref{corollary:zeros:phi:r}, we obtain
    \[
        \mathbb P_{\mathsf{elementary}}(n, p)
        \sim
        - \frac{1}{2}\frac{\sqrt{\lambda}}
        {a_1'\ai_1(a_1')}
        e^{ - \lambda^2/4+\lambda/3+ 5/12-C_a(1)+a/2}
        \exp \Big(
            {-\alpha(\lambda) n}
            + a_1' \beta(\lambda) n^{1/3}
        \Big)
        \, .
    \]
\end{proof}

\begin{remark}
    In~\cite{Panafieu2020} it was proven that as \( \mu \to -\infty \), the
    probability that a random digraph is elementary can be further refined as
    \[
        \mathbb P_{\mathsf{elementary}}(n, p) \sim 1 - \dfrac{1}{2 |\mu|^3}.
    \]
    This result was derived for the model of simple digraphs in the model \(
    \Di(n, m) \), and can be deduced similarly for digraphs in the model
    \( \Di(n, p) \) as well. In principle, with an even more refined
    analysis, a complete asymptotic expansion in powers of \( |\mu|^{-3} \)
    could be obtained as well, and it can be shown that the coefficients of this
    expansion are identical for all the models considered
    in~\cref{section:models}.

    By further analysis of the asymptotic probability of the presence of a complex
    component with given excess, it can be seen that when \( \mu \to -\infty \),
    the asymptotic probability that a digraph has one  bicyclic complex component
    is \( \dfrac{1}{2 |\mu|^3} \). These two probabilities sum up
    to one if two terms of the asymptotic expansion are taken into account:
    \( \left(1 - \dfrac{1}{2 |\mu|^3}\right) + \dfrac{1}{2 |\mu|^3} = 1 \),
    leaving an error of order \( O(|\mu|^{-6}) \).
  \end{remark}

\begin{theorem} \label{theorem:digraph:one:strong:component}
Let \( p = \lambda / n \), \( \lambda \geq 0 \).
The probability \( \mathbb P(n, p) \)
that a random  digraph  $D$ in the model \(\DiGilb(n,p)\) or \(\DiGilbBoth(n,p)\)
has exactly one complex strong component
and that the excess of this component is \( r \) satisfies
\[
\mathbb P(n, p) \sim
    \begin{cases}
        c_r(\lambda) n^{-r} \lambda^r (1 - \lambda)^{-3r },
        & \lambda < 1; \\
        c_r(1) \widetilde \varphi_{r}(\mu)
        ,
        & \lambda = 1 + \mu n^{-1/3}; \\
        c_r(1)
        n^{1/3}
        \omega_r(\lambda)
        \exp \Big(
            {-\alpha(\lambda)n} + a_1' \beta(\lambda) n^{1/3}
        \Big)
        ,
        & \lambda > 1,
    \end{cases}
\]
where the functions \( \alpha(\cdot) \), \( \beta(\cdot) \) are
from~\cref{theo:MD_model_lambda_fixed},
\( a_1' \) is the zero of \(
\ai'(\cdot) \) with the smallest modulus,
\(
    \widetilde  \varphi_{r}(\mu)
    =
    \varphi_{r,0}(\mu)
\) from \cref{theorem:probability:one:strong:component},
\[
    \omega_r (\lambda)
    =
    2^{-4/3-r} (\lambda-1) \lambda^{1/6}
    \dfrac
        {(-1)^{1 - 3r} \ai(1-3r; a_1')}
        {(a_1' \ai(a_1'))^2}
    \exp \left(
        - \frac{\lambda^2}{4}
        + \frac{\lambda}{3}
        - \frac{1}{12}
    \right)
\,
\]
and
\[
    c_r(\lambda)  =
  \begin{cases}
A^{(\text{simple})}_r(\lambda) &  \mbox{ if }   D\in \DiGilbBoth(n,p);\\
A^{(\text{strict})}_r(\lambda)  & \mbox{ if }  D \in \DiGilb(n,p).
 \end{cases}
\]
$A^{(\text{simple})}_r(\cdot)$ and $A^{(\text{strict})}_r(\cdot)$ are defined in
\cref{th:complex:scc} and computable using \cref{rem:strong:complex:simple}.
In particular, a singularity analysis for fixed $r$
gives $c_r(1) = A^{(\text{simple})}_r(1) = A^{(\text{strict})}_r(1)=s_{r,0} =s_r$, where $s_r$ is given in \cref{th:complex:scc}.
\end{theorem}

\begin{proof}
  We follow the same structure as in the proof
  of \cref{theorem:probability:one:strong:component}.
  We defined $A^{(\text{simple})}_r(\cdot)$ and $A^{(\text{strict})}_r(\cdot)$
  in \cref{th:complex:scc} so that
  the exponential generating functions of strongly connected simple digraphs of excess $r$
  and strongly connected strict digraphs of excess $r$
  are, respectively
  \[
    S_r(z,w;1) = \frac{w^r A^{(\text{simple})}_r(w z)}{(1 - w z)^{3 r}},
    \quad
    S_r(z,w;2) = \frac{w^r A^{(\text{strict})}_r(w z)}{(1 - w z)^{3 r}}.
  \]
  Let \( \widehat H_{S_r}^{\textrm{(simple)}}(z, w;k) \)
  denote the graphic generating function of digraphs containing
  exactly one strongly connected component of excess $r$,
  while the others are single vertices or cycles of length $>k$.
  Let also $C_k(z) = z + z^2/2 + \cdots + z^k/k$.
  According to \cref{lemma:one:marked:scc:simple},
  we have
  \begin{equation}
    \label{eq:ggf:complex:simple}
    \widehat H_{\mathcal S}^{\textrm{(simple)}}(z, w;1) =
    w^r \dfrac{
      \widetilde \phi_{1-3r} (z,w;F_1(\cdot))
    }{
      {\widetilde \phi}_1(z, w;G_1(\cdot))^2
    }\, ,
    \quad
    \mbox{and}
    \quad
    \widehat H_{\mathcal S}^{\textrm{(simple)}}(z, w;2)   = w^r \dfrac{
      \widetilde \phi_{1-3r} (z,w;F_2(\cdot))
    }{
      {\widetilde \phi}_1(z, w;G_2(\cdot))^2
    }\, ,
  \end{equation}
  where  $F_1(z) =  A^{(\text{simple})}_r(z)  e^{z} $  and   $G_1(z) = e^{z}$,
  $F_2(z) =  A^{(\text{strict})}_r(z)  e^{z+z^2/2} $  and   $G_2(z) = e^{z+z^2/2}$.

  In the rest of the proof, we restrict our attention to simple digraphs.
  The case of strict digraphs can be obtained in a straightforward manner.
  We want to compute the probability
  \begin{equation}
    \label{eq:proba-one-strong-supercritical}
    \mathbb P(n,p)=\left( 1 - \frac{\lambda}{n}\right)^{\binom{n}{2}}
    n![z^n] \widehat H_{\mathcal S}^{\textrm{(simple)}}(z, w;1) \,.
  \end{equation}

  \noindent \textbf{Case $\lambda \in [0, 1)$.}
  We apply Part~(a) of \cref{cor:lem:cauchy} by setting
  \[
      \frac{1}{\widetilde \xi(z,w)} = \dfrac{
      \widetilde \phi_{1-3r} (z,w;F_1(\cdot))
    }{
      {\widetilde \phi}_1(z, w;G_1(\cdot))^2
    }
  \]
  so, in the notation of this lemma,
  $r_1 = 1 - 3 r$, $p_1 = -1$, $r_2 = 1$, $p_2 = 2$ and $a=1$.
  This implies $R=2-(1-3r)=3r+1$ and we obtain
  \[
     \left( 1 - \frac{\lambda}{n}\right)^{\binom{n}{2}}  n![z^n] \widehat H_{\mathcal
       S}^{\textrm{(simple)}}(z, w;1)
    \Big|_{w=\frac{\lambda}{n-\lambda}}  \sim
    \frac{\lambda^r}{n^r}\cdot e^{\lambda} (1-\lambda)^{-3r}
    A^{(\text{simple})}_r(\lambda)e^{-\lambda}\, .
  \]

  \noindent \textbf{Case $\lambda = 1 + \mu n^{-1/3}$.}
  Part~(b) of \cref{cor:lem:cauchy} is applied with the previous parameters.

  \noindent \textbf{Case $\lambda > 1$.}
    Let $a_1' > a_2' > \cdots$ denote the roots of $\ai'(\cdot)$,
    fix $\theta \in (2^{-1/3}a_2',\, 2^{-1/3} a_1')$
    and set $\rho = (1 - \theta w^{2/3}) (e w)^{-1}$.
    Let us also write \( \widetilde \phi_\ast(z, w) \)
    for \( \widetilde \phi_{1 - 3r}(z, w;F_1(\cdot)) \),
    and denote by \( \widetilde \varsigma_j(w) \) the $j$th solution of
    \( \widetilde \phi_1(z, w, G_1(\cdot)) = 0 \),
    ordered by increasing modulus.
    Then \cref{corollary:zeros:phi:r} implies that,
    as $w$ tends to $0$,
    $\widetilde \varsigma_1(w) < \rho < \widetilde \varsigma_2(w)$.
    Thus, expressing the coefficient extraction as a Cauchy integral
    and applying the residue theorem, we obtain
    \[
        [z^n] \widehat H_{\mathcal S}^{(\textrm{simple})}(z, w;1) =
          - w^r
          \mathrm{Res}_{z = \widetilde \varsigma_1(w)}
          \dfrac{\widetilde \phi_{\ast}(z, w)}{z^{n+1} \widetilde
            \phi_1(z, w;G_1(\cdot))^2}
          +
          \dfrac{w^r}{2 \pi i}
          \oint_{|z| = \rho}
          \dfrac{\widetilde \phi_{\ast}(z, w)}{z^{n+1}
            \widetilde \phi_1(z, w;G_1(\cdot))^2}
            \mathrm dz.
    \]
    The second term is bounded using Part~(c) of \cref{cor:lem:cauchy}:
    \begin{equation} \label{eq:residue:simple:error}
        \dfrac{w^r}{2 \pi i}
          \oint_{|z| = \rho}
          \dfrac{\widetilde \phi_{\ast}(z, w)}{z^{n+1}
            \widetilde \phi_1(z, w;G_1(\cdot))^2}
            \mathrm dz
        =
        \bigO \left(
            \frac{w^r}{\rho^n}
            \frac{\widetilde \phi_{\ast}(\rho, w)}
                {\widetilde \phi_1(\rho, w;G_1(\cdot))^2}
        \right).
    \end{equation}
    Taylor's theorem is applied to $\widetilde\phi_1(z,w;G_1(\cdot))$
    at $z = \widetilde\varsigma_1(w)$,
    then the Laurent series of
    \[
        \dfrac{\widetilde \phi_{\ast}(z, w)}
            {z^{n+1} \widetilde \phi_1(z, w;G_1(\cdot))^2}
    \]
    is constructed and the residue at $z = \widetilde \varsigma_1(w)$ is computed as
    \[
        \left.
        \dfrac{
          \partial_z \widetilde \phi_\ast(z, w)
          \partial_z \widetilde \phi_1(z, w;G_1(\cdot))
            -
            \widetilde\phi_\ast(z, w)
            \big(
            \partial_z^2 \widetilde \phi_1(z, w; G_1(\cdot))
            +
            \tfrac{n+1}{z} \partial_z \widetilde \phi_1(z, w;G_1(\cdot))
            \big)
        }
        {
            z^{n+1}
            (\partial_z \widetilde \phi_1(z, w;G_1(\cdot)))^3
        }
        \right|_{z = \widetilde\varsigma_1(w)}\, .
    \]
    We apply \cref{corollary:zeros:phi:r} to approximate
    $\widetilde\varsigma_1(w)$
    and the derivatives of $\widetilde \phi_\ast$ and $\widetilde \phi_1$
    at $z = \widetilde\varsigma_1(w)$.
    We apply Part~(c) of \cref{corollary:complex:general} to approximate
    $\widetilde \phi_{\ast}$ and $\widetilde \phi_1$
    at $z = \rho$ and $z = \widetilde\varsigma_1(w)$.
    The asymptotics of the residue and the error term
    are simplified using a computer algebra system.
    The error term from \eqref{eq:residue:simple:error} is negligible
    and the asymptotics of the residue provide the statement of the theorem.
\end{proof}

\begin{remark} According to the exponential generating function of bicycles from
    \cref{lemma:one:marked:scc:simple}, \cref{corollary:bicyclic} holds
    for both $\DiGilb(n,p)$ and
    $\DiGilbBoth(n,p)$.  In these cases, the constant $c_1(1)$
    is $s_1=A_1(1)=B_1(1)=\frac{1}{2}$.
    Hence, the probability
    that $\DiGilb(n,\frac{1}{n})$ or $\DiGilbBoth(n,\frac{1}{n})$ has
    one  bicyclic complex component is asymptotically equal to $1/8$.
\end{remark}

\section{Numerical results}
\label{section:numerical:results}

In this section we provide numerical values for the probability that a
random digraph (simple or multi-) with \( n \) vertices and parameter \( p \)
belongs to a given family for different values of \( n \) and \( p \).
This is accomplished by computing the coefficients of the generating functions.
We explore empirically how fast they tend to their respective limiting values.
We also provide numerical values for the new special functions that we introduced
in the form of Airy integrals.

\subsection{Airy integrals}
\label{section:numerical:airy}

We start by providing a table of numerical values of Airy integrals of the
form appearing in~\cref{theorem:probability:one:strong:component}.
Let
\[
    \mathcal I(n,\mu) := \dfrac{(-1)^n}{2 \pi i}
    \int_{-i \infty}^{i \infty}
    \dfrac{\ai(-n, \tau)}{\ai'(\tau)^2}
    e^{2^{-1/3} \mu \tau - \mu^3/6}
    \mathrm d \tau
    \, .
\]
The first few values of \( \mathcal I(n, \mu) \) corresponding to
different negative values of \( n \) and integer values of \( \mu \), and also
several values of \( \ai(-n, a_1') \), where \( a_1' \) is the dominant root of
the derivative of the Airy function, are given
in~\cref{table:complex:airy,table:complex:airy:cont}.

\begin{table}[hbt]
    \centering
    \begin{tabular}{c|c|cccc}
        \hline
        \hline
        $n$
        & \( (-1)^n \ai(-n, a_1') \)
        & \( \mathcal I(n,-4) \)
        & \( \mathcal I(n,-3) \)
        & \( \mathcal I(n,-2) \)
        & \( \mathcal I(n,-1) \)
        \\
        \hline
0
& \textsf{0.53566\,66560}  
& \textsf{0.60833\,46108}  
& \textsf{0.78025\,12697}  
& \textsf{1.04871\,55827}  
& \textsf{1.43382\,48189}  
\\
1
& \textsf{0.80907\,32963}  
& \textsf{0.18173\,76482}  
& \textsf{0.29442\,48000}  
& \textsf{0.51700\,82249}  
& \textsf{0.93651\,41410}  
\\
2
& \textsf{0.82428\,81878}  
& \textsf{5.30647\,07\,$\cdot\,\mathsf{10}^{\mathsf{-2}}$}  
& \textsf{0.10665\,46071}  
& \textsf{0.23667\,55204}  
& \textsf{0.53903\,08295}  
\\
3
& \textsf{0.68771\,27402}  
& \textsf{1.51748\,85\,$\cdot\,\mathsf{10}^{\mathsf{-2}}$}  
& \textsf{3.72902\,31\,$\cdot\,\mathsf{10}^{\mathsf{-2}}$}  
& \textsf{0.10211\,30475}  
& \textsf{0.28327\,71366}  
\\
4
& \textsf{0.50324\,67342}  
& \textsf{4.25661\,62\,$\cdot\,\mathsf{10}^{\mathsf{-3}}$}  
& \textsf{1.26360\,41\,$\cdot\,\mathsf{10}^{\mathsf{-2}}$}  
& \textsf{4.19423\,86\,$\cdot\,\mathsf{10}^{\mathsf{-2}}$}  
& \textsf{0.13868\,39970}  
\\
5
& \textsf{0.33424\,30589}  
& \textsf{1.17273\,42\,$\cdot\,\mathsf{10}^{\mathsf{-3}}$}  
& \textsf{4.16331\,71\,$\cdot\,\mathsf{10}^{\mathsf{-3}}$}  
& \textsf{1.65168\,25\,$\cdot\,\mathsf{10}^{\mathsf{-2}}$}  
& \textsf{6.40797\,49\,$\cdot\,\mathsf{10}^{\mathsf{-2}}$}  
\\
6
& \textsf{0.20565\,74439}  
& \textsf{3.17750\,22\,$\cdot\,\mathsf{10}^{\mathsf{-4}}$}  
& \textsf{1.33725\,97\,$\cdot\,\mathsf{10}^{\mathsf{-3}}$}  
& \textsf{6.26910\,09\,$\cdot\,\mathsf{10}^{\mathsf{-3}}$}  
& \textsf{2.81950\,38\,$\cdot\,\mathsf{10}^{\mathsf{-2}}$}  
\\
7
& \textsf{0.11879\,14841}  
& \textsf{8.47523\,11\,$\cdot\,\mathsf{10}^{\mathsf{-5}}$}  
& \textsf{4.19653\,25\,$\cdot\,\mathsf{10}^{\mathsf{-4}}$}  
& \textsf{2.30288\,09\,$\cdot\,\mathsf{10}^{\mathsf{-3}}$}  
& \textsf{1.18914\,35\,$\cdot\,\mathsf{10}^{\mathsf{-2}}$}  
\\
8
& \textsf{6.50381\,41\,$\cdot\,\mathsf{10}^{\mathsf{-2}}$}  
& \textsf{2.22706\,27\,$\cdot\,\mathsf{10}^{\mathsf{-5}}$}  
& \textsf{1.28898\,06\,$\cdot\,\mathsf{10}^{\mathsf{-4}}$}  
& \textsf{8.21386\,92\,$\cdot\,\mathsf{10}^{\mathsf{-4}}$}  
& \textsf{4.83151\,22\,$\cdot\,\mathsf{10}^{\mathsf{-3}}$}  
\\
9
& \textsf{3.39885\,81\,$\cdot\,\mathsf{10}^{\mathsf{-2}}$}  
& \textsf{5.77144\,16\,$\cdot\,\mathsf{10}^{\mathsf{-6}}$}  
& \textsf{3.88113\,77\,$\cdot\,\mathsf{10}^{\mathsf{-5}}$}  
& \textsf{2.85229\,69\,$\cdot\,\mathsf{10}^{\mathsf{-4}}$}  
& \textsf{1.89861\,70\,$\cdot\,\mathsf{10}^{\mathsf{-3}}$}  
\\
10
& \textsf{1.70465\,23\,$\cdot\,\mathsf{10}^{\mathsf{-2}}$}  
& \textsf{1.47285\,31\,$\cdot\,\mathsf{10}^{\mathsf{-6}}$}  
& \textsf{1.14711\,18\,$\cdot\,\mathsf{10}^{\mathsf{-5}}$}  
& \textsf{9.66439\,06\,$\cdot\,\mathsf{10}^{\mathsf{-5}}$}  
& \textsf{7.23897\,30\,$\cdot\,\mathsf{10}^{\mathsf{-4}}$}  
\\
        \hline
    \end{tabular}
    \caption{
    \label{table:complex:airy}
        Numerical values of \( \ai(\cdot, a_1') \)
        and \( \mathcal I(\cdot, \mu) \)
        for \( \mu \in \{-4,-3,-2,-1\} \).}
\end{table}

\begin{table}[hbt]
    \centering
    \begin{tabular}{c|ccccc}
        \hline
        \hline
        $n$
        & \( \mathcal I(n,0) \)
        & \( \mathcal I(n,1) \)
        & \( \mathcal I(n,2) \)
        & \( \mathcal I(n,3) \)
        & \( \mathcal I(n,4) \)
        \\
        \hline
0
& \textsf{1.71175\,17513}  
& \textsf{1.18977\,74561}  
& \textsf{0.24108\,61434}  
& \textsf{5.93679\,28\,$\cdot\,\mathsf{10}^{\mathsf{-3}}$}  
& \textsf{6.86168\,24\,$\cdot\,\mathsf{10}^{\mathsf{-6}}$}  
\\
1
& \textsf{1.43728\,66407}  
& \textsf{1.19994\,59529}  
& \textsf{0.27300\,07028}  
& \textsf{7.21364\,85\,$\cdot\,\mathsf{10}^{\mathsf{-3}}$}  
& \textsf{8.71568\,73\,$\cdot\,\mathsf{10}^{\mathsf{-6}}$}  
\\
2
& \textsf{1.00000\,00000}  
& \textsf{0.95171\,34841}  
& \textsf{0.23452\,58919}  
& \textsf{6.49354\,09\,$\cdot\,\mathsf{10}^{\mathsf{-3}}$}  
& \textsf{8.07052\,10\,$\cdot\,\mathsf{10}^{\mathsf{-6}}$}  
\\
3
& \textsf{0.61548\,95913}  
& \textsf{0.65239\,09401}  
& \textsf{0.17163\,35341}  
& \textsf{4.93680\,60\,$\cdot\,\mathsf{10}^{\mathsf{-3}}$}  
& \textsf{6.27592\,36\,$\cdot\,\mathsf{10}^{\mathsf{-6}}$}  
\\
4
& \textsf{0.34611\,79347}  
& \textsf{0.40319\,77130}  
& \textsf{0.11237\,21118}  
& \textsf{3.34254\,42\,$\cdot\,\mathsf{10}^{\mathsf{-3}}$}  
& \textsf{4.33391\,88\,$\cdot\,\mathsf{10}^{\mathsf{-6}}$}  
\\
5
& \textsf{0.18124\,38129}  
& \textsf{0.23001\,35204}  
& \textsf{6.75809\,67\,$\cdot\,\mathsf{10}^{\mathsf{-2}}$}  
& \textsf{2.07309\,81\,$\cdot\,\mathsf{10}^{\mathsf{-3}}$}  
& \textsf{2.73680\,17\,$\cdot\,\mathsf{10}^{\mathsf{-6}}$}  
\\
6
& \textsf{8.94935\,23\,$\cdot\,\mathsf{10}^{\mathsf{-2}}$}  
& \textsf{0.12295\,49383}  
& \textsf{3.79527\,70\,$\cdot\,\mathsf{10}^{\mathsf{-2}}$}  
& \textsf{1.19839\,99\,$\cdot\,\mathsf{10}^{\mathsf{-3}}$}  
& \textsf{1.60897\,45\,$\cdot\,\mathsf{10}^{\mathsf{-6}}$}  
\\
7
& \textsf{4.20355\,45\,$\cdot\,\mathsf{10}^{\mathsf{-2}}$}  
& \textsf{6.22242\,32\,$\cdot\,\mathsf{10}^{\mathsf{-2}}$}  
& \textsf{2.01266\,66\,$\cdot\,\mathsf{10}^{\mathsf{-2}}$}  
& \textsf{6.53278\,00\,$\cdot\,\mathsf{10}^{\mathsf{-4}}$}  
& \textsf{8.91294\,64\,$\cdot\,\mathsf{10}^{\mathsf{-7}}$}  
\\
8
& \textsf{1.89047\,15\,$\cdot\,\mathsf{10}^{\mathsf{-2}}$}  
& \textsf{3.00370\,59\,$\cdot\,\mathsf{10}^{\mathsf{-2}}$}  
& \textsf{1.01604\,43\,$\cdot\,\mathsf{10}^{\mathsf{-2}}$}  
& \textsf{3.38662\,88\,$\cdot\,\mathsf{10}^{\mathsf{-4}}$}  
& \textsf{4.69243\,93\,$\cdot\,\mathsf{10}^{\mathsf{-7}}$}  
\\
9
& \textsf{8.18150\,40\,$\cdot\,\mathsf{10}^{\mathsf{-3}}$}  
& \textsf{1.39096\,40\,$\cdot\,\mathsf{10}^{\mathsf{-2}}$}  
& \textsf{4.91269\,11\,$\cdot\,\mathsf{10}^{\mathsf{-3}}$}  
& \textsf{1.68015\,59\,$\cdot\,\mathsf{10}^{\mathsf{-4}}$}  
& \textsf{2.36312\,51\,$\cdot\,\mathsf{10}^{\mathsf{-7}}$}  
\\
10
& \textsf{3.42087\,90\,$\cdot\,\mathsf{10}^{\mathsf{-3}}$}  
& \textsf{6.20709\,94\,$\cdot\,\mathsf{10}^{\mathsf{-3}}$}  
& \textsf{2.28591\,84\,$\cdot\,\mathsf{10}^{\mathsf{-3}}$}  
& \textsf{8.01635\,84\,$\cdot\,\mathsf{10}^{\mathsf{-5}}$}  
& \textsf{1.14409\,39\,$\cdot\,\mathsf{10}^{\mathsf{-7}}$}  
\\
        \hline
    \end{tabular}
    \caption{
    \label{table:complex:airy:cont}
        Numerical values of
        \( \mathcal I(\cdot, \mu) \)
        for \( \mu \in \{0,1,2,3,4\} \).}
\end{table}

We provide the plots for \( \mu \in [-3,3] \) for the functions \( \mathcal
I(n,\mu) \)  with \( n \in \{1,2,3,4,5\} \) in~\cref{fig:I}.
Note that the value \( n = 2
\) corresponds to the plot in~\cref{fig:probability:bicyclic}, which provides
the remaining numbers for~\cref{theorem:probability:one:strong:component} when
\( 1 - 3r + d \in [-5, -1]\). All the figures represent bell-shaped curves
whose maxima lie in the interval \( (0, 1) \).
We rescale the plots by $2^{n/3}$ because this scaling factor also enters
the final probabilities in~\cref{theorem:probability:one:strong:component}.

\begin{figure}[ht]
  \includegraphics[width=.8\linewidth]{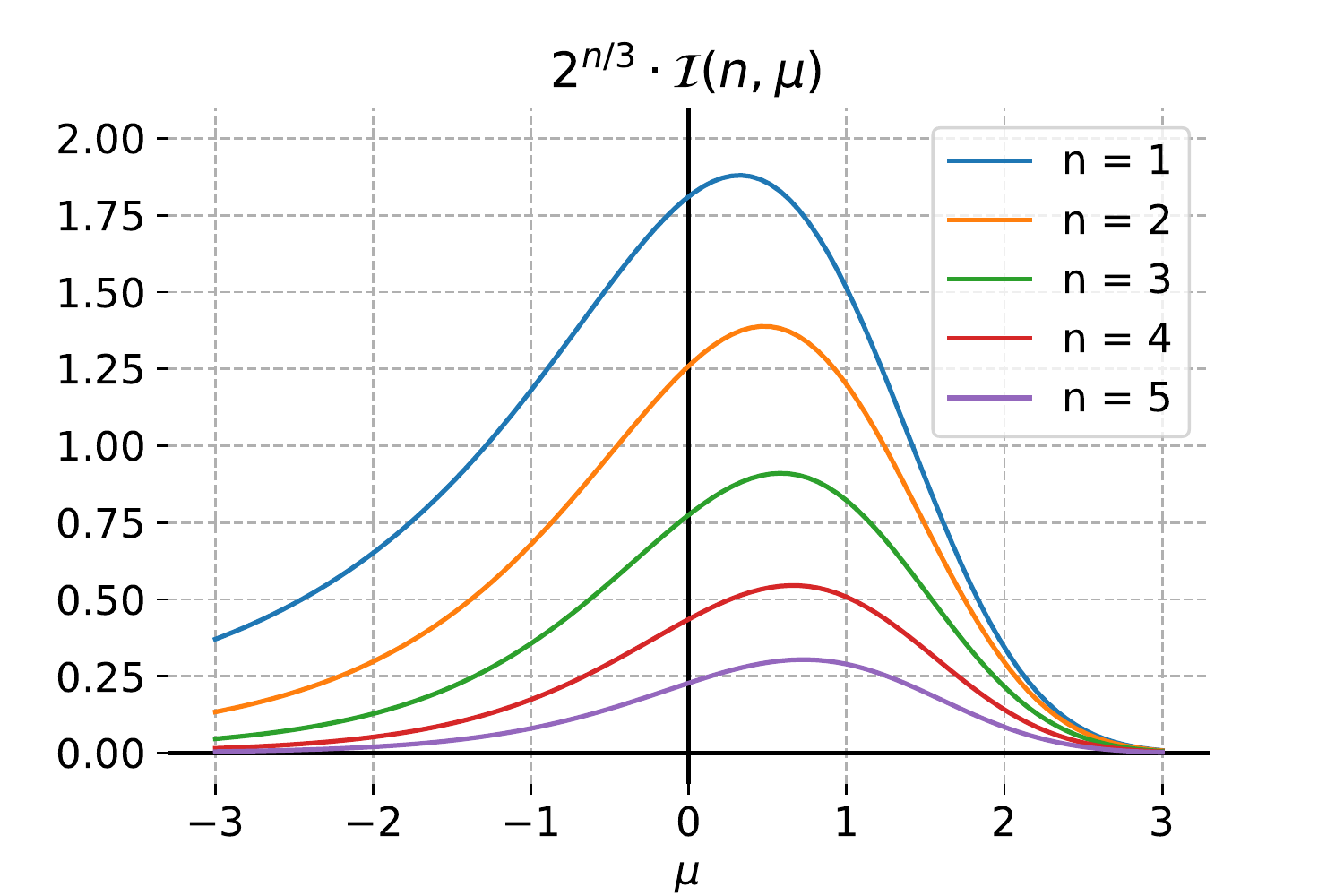}\\
  \caption{
  \label{fig:I}
  Numerical plots of the Airy integrals \( \mathcal I(n,\mu) \).
  }
\end{figure}

\subsection{Empirical probabilities within the critical window}
\label{section:empirical:probabilities:within}

As the theorems in~\cref{section:asymptotics:multidigraphs} provide only the
limiting values, with an error term of order \( O(n^{-1/3}) \) when \( \mu \) is
bounded, it is interesting to track down how quickly the empirical
probabilities for concrete values of \( n \) and \( p \) in the model \(
\DiGilbMulti(n,p) \) converge to their limiting values. These values (for small
values of \( n \) and for \( p = 1 + \mu n^{-1/3} \) inside the critical window)
are given in~\cref{table:mdag,table:elementary,table:bicyclic}.

It is also interesting to compare the analogous convergence for the case of
simple digraphs, depending on whether 2-cycles are allowed or not. The numerical
results for the probabilities that a digraph is elementary (inside the centre
of the critical window, \ie, when \( p = \tfrac{1}{n} \)) are given
in~\cref{table:simple} for small values of \( n \).

\begin{table}[hbt!]
    \centering
    \begin{tabular}{c|ccccccc}
        \hline
        \hline
        $n$
        & \( \mu = -3 \)
        & \( \mu = -2 \)
        & \( \mu = -1 \)
        & \( \mu =  0 \)
        & \( \mu =  1 \)
        & \( \mu =  2 \)
        & \( \mu =  3 \)
        \\
        \hline
        100
        & 3.00671 
        & 2.03423 
        & 1.14068 
        & 0.46304 
        & 0.11671 
        & 0.01642 
        & 0.00124 
        \\
        1000
        & 3.02522 
        & 2.06335 
        & 1.17410 
        & 0.47705 
        & 0.10793 
        & 0.01012 
        & 0.00030 
        \\
        3000
        & 3.03190 
        & 2.07219 
        & 1.18326 
        & 0.48068 
        & 0.10527 
        & 0.00859 
        & 0.00018 
        \\
        5000
        & 3.03442 
        & 2.07542 
        & 1.18652 
        & 0.48196 
        & 0.10430 
        & 0.00807 
        & 0.00014 
        \\
        10000
        & 3.03733 
        & 2.07906 
        & 1.19017 
        & 0.48337 
        & 0.10319 
        & 0.00751 
        & 0.00011 
        \\
        \hline
        $\infty$ &
        3.04943 &
        2.09362 &
        1.20431 &
        0.48873 &
        0.09876 &
        0.00550 &
        0.00004 \\
        \hline
    \end{tabular}
    \caption{
    \label{table:mdag}
        The rescaled probability \( n^{1/3} \mathbb P(n, \tfrac{1}{n}(1 +
        \mu n^{-1/3})) \)  that a random multidigraph is acyclic, for
        \( \mu \in \{-3,-2,-1,0,1,2,3\} \).}
\end{table}

\begin{table}[hbt!]
    \centering
    \begin{tabular}{c|ccccccc}
        \hline
        \hline
        $n$
        & \( \mu = -3 \)
        & \( \mu = -2 \)
        & \( \mu = -1 \)
        & \( \mu =  0 \)
        & \( \mu =  1 \)
        & \( \mu =  2 \)
        & \( \mu =  3 \)
        \\
        \hline
        100
        & 0.99782 
        & 0.98513 
        & 0.92149 
        & 0.71322 
        & 0.36692 
        & 0.10522 
        & 0.01574 
        \\
        1000
        & 0.99216 
        & 0.97411 
        & 0.90794 
        & 0.70645 
        & 0.34684 
        & 0.07442 
        & 0.00511 
        \\
        3000
        & 0.99020 
        & 0.97095 
        & 0.90435 
        & 0.70446 
        & 0.34077 
        & 0.06604 
        & 0.00333 
        \\
        5000
        & 0.98947 
        & 0.96981 
        & 0.90308 
        & 0.70374 
        & 0.33854 
        & 0.06307 
        & 0.00280 
        \\
        10000
        & 0.98863 
        & 0.96854 
        & 0.90167 
        & 0.70293 
        & 0.33601 
        & 0.05979 
        & 0.00229 
        \\
        \hline
        $\infty$ &
        0.98521 &
        0.96354 &
        0.89622 &
        0.69968 &
        0.32582 &
        0.04740 &
        0.00089 \\
        \hline
    \end{tabular}
    \caption{
    \label{table:elementary}
        The probability \( \mathbb P(n, \tfrac{1}{n}(1 +
        \mu n^{-1/3})) \)  that a random multidigraph is elementary, for
        \( \mu \in \{-3,-2,-1,0,1,2,3\} \).}
\end{table}

\begin{table}[hbt!]
    \centering
    \begin{tabular}{c|ccccccc}
        \hline
        \hline
        $n$
        & \( \mu = -3 \)
        & \( \mu = -2 \)
        & \( \mu = -1 \)
        & \( \mu =  0 \)
        & \( \mu =  1 \)
        & \( \mu =  2 \)
        & \( \mu =  3 \)
        \\
        \hline
        100
        & 0.00213 
        & 0.01360 
        & 0.05777 
        & 0.13535 
        & 0.14434 
        & 0.06652 
        & 0.01378 
        \\
        1000
        & 0.00741 
        & 0.02231 
        & 0.06352 
        & 0.12980 
        & 0.13073 
        & 0.04610 
        & 0.00448 
        \\
        3000
        & 0.00914 
        & 0.02459 
        & 0.06481 
        & 0.12833 
        & 0.12712 
        & 0.04081 
        & 0.00293 
        \\
        5000
        & 0.00978  
        & 0.02539  
        & 0.06525  
        & 0.12781  
        & 0.12585  
        & 0.03897  
        & 0.00248  
        \\
        10000
        & 0.01049 
        & 0.02627 
        & 0.06571 
        & 0.12723 
        & 0.12443 
        & 0.03693 
        & 0.00203 
        \\
        \hline
        $\infty$ &
        0.01333 &
        0.02958 &
        0.06737 &
        0.12500 &
        0.11896 &
        0.02931 &
        0.00081 \\
        \hline
    \end{tabular}
    \caption{
    \label{table:bicyclic}
        The probability \( \mathbb P(n, \tfrac{1}{n}(1 +
        \mu n^{-1/3})) \)  that a random multidigraph has one bicyclic complex
        component, for \( \mu \in \{-3,-2,-1,0,1,2,3\} \).}
\end{table}

\begin{table}[hbt!]
    \centering
    \begin{tabular}{c|ccc}
        \hline
        \hline
        $n$
        & \( \DiGilb(n,\tfrac{1}{n}) \)
        & \(\vphantom{\displaystyle\int} \DiGilbBoth(n,\tfrac{1}{n}) \)
        & \( \DiGilbMulti(n,\tfrac{1}{n}) \)
        \\
        \hline
        100
        & 0.777319 
        & 0.743468 
        & 0.713229 
        \\
        1000
        & 0.724241 
        & 0.714032 
        & 0.706451 
        \\
        3000
        & 0.713629 
        & 0.708269 
        & 0.704469 
        \\
        5000
        & 0.710426 
        & 0.706492 
        & 0.703748 
        \\
        10000
        & 0.707258 
        & 0.704693 
        & 0.702935 
        \\
        \hline
        $\infty$ &
        0.699687 &
        0.699687 &
        0.699687 \\
        \hline
    \end{tabular}
    \caption{
    \label{table:simple}
        The probability
        \( \mathbb P(n, \tfrac{1}{n}) \)  that a random
        (multi-)digraph is elementary.}
\end{table}

\subsection{Convergence outside the critical window}
\label{section:empirical:probabilities:outside}

In this section we analyse the ratio of the empirical probability and the theoretical prediction
when $p = \dfrac{\lambda}{n}$ when $\lambda$ are different fixed values separated from $1$,
and $n$ is large. We are dealing with all the three models
$\DiGilbMulti$, $\DiGilbBoth$ and $\DiGilb$.
For each of the models, we are dealing with three different families: directed acyclic digraphs,
elementary digraphs, and digraphs having one bicyclic complex component.
Furthermore, for each of the families, we consider the subcritical case $\lambda < 1$ and $\lambda > 1$.
In total, this gives $3 \times 3 \times 2 = 18$ cases.
For each of the 18 cases we compute the ratios of the empirical
and theoretical probabilities and demonstrate their convergence to 1 by using
log-log scale plots. These results are presented in~\cref{table:numerical}.

When $\lambda < 1$, we expect that
$\mathbb P_{\mathsf{empirical}}(n,p) / \mathbb P_{\mathsf{theoretical}}(n,p) - 1
= \mathcal O(n^{-1})$, while for $\lambda > 1$ we expect this difference to
converge as $\mathcal O(n^{-1/3})$. The constant in the error term is a function
of \( \lambda \). If this convergence takes place, then the values on the y-axis
\[
    \log \left[
    \left|
        \dfrac{\mathbb P_{\mathsf{empirical}}(n,p)}{\mathbb P_{\mathsf{theoretical}}(n,p)} - 1
        \right|
    \right]
\]
will be approximated by decreasing linear functions when \( n \) is large
enough, with a slope 
$\gamma = -1$ for $\lambda < 1$ and $\gamma = -\frac13$ for $\lambda > 1$.
The sign of the difference 
\(
    \dfrac
    {\mathbb P_{\mathsf{empirical}}(n,p)}
    {\mathbb P_{\mathsf{theoretical}}(n,p)} - 1
\) is empirically observed to be positive or negative, depending on the model,
which is reflected when labelling the y-axis.

Although we observe nearly parallel decreasing lines for all our plots, the
absolute value of the difference is higher for some models. This difference is
most clearly seen, for example, in the supercritical case for digraphs with one
bicyclic complex component, where in the multidigraph model we have the
log-error of order $-3.2$ for $\lambda = 10$ and $\log n \approx 7$, while
this error raises to $\approx 1$ in the model $\DiGilbBoth$ and further
to $\approx 3.5$ for the model $\DiGilb$.

\begin{table}[hbt!]
\centering
\begin{tabular}{ccc}
$\DiGilbMulti$ & $\DiGilbBoth$ & $\DiGilb$
\\
\hline
 \includegraphics[width=0.32\textwidth]{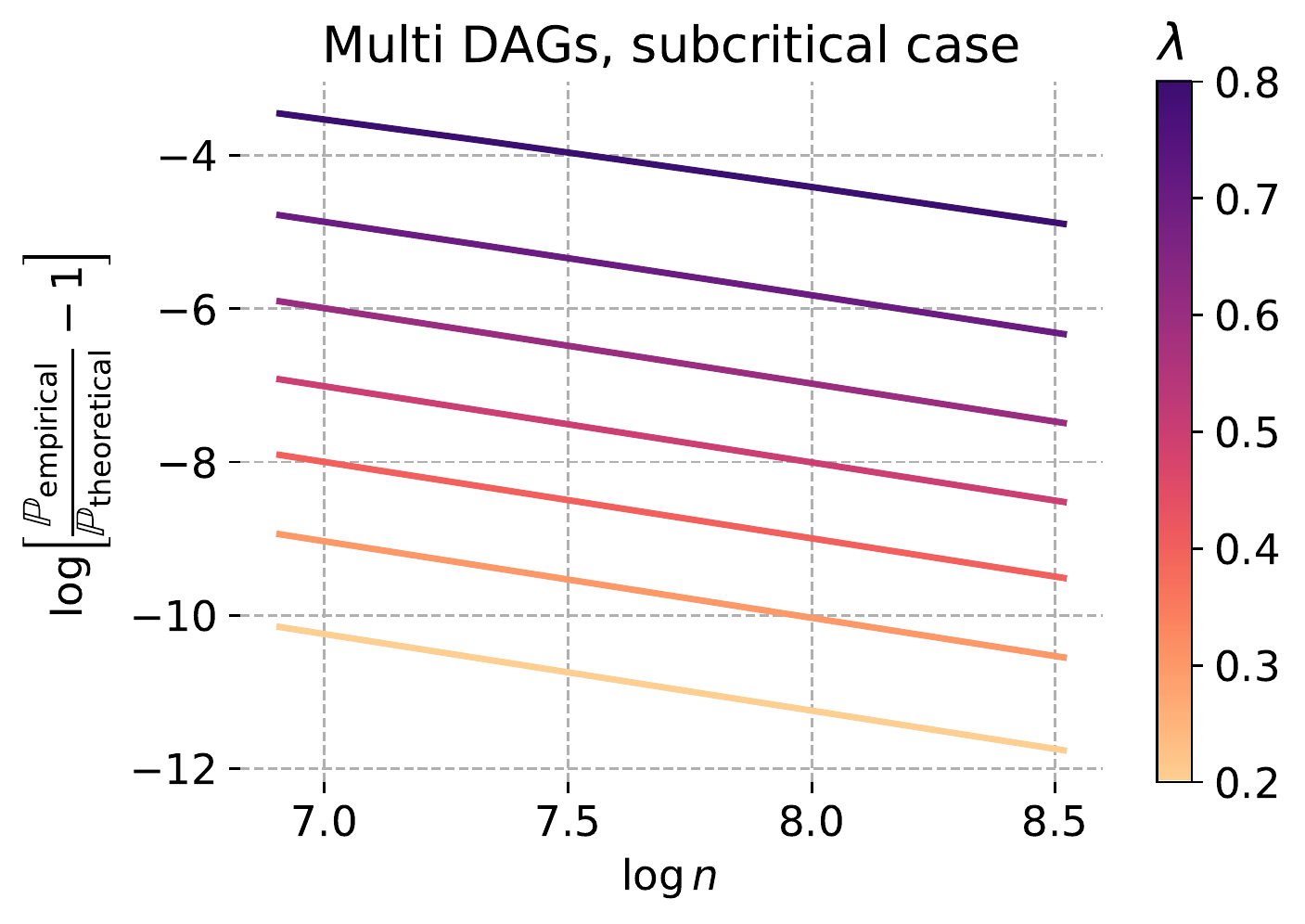}
&\includegraphics[width=0.32\textwidth]{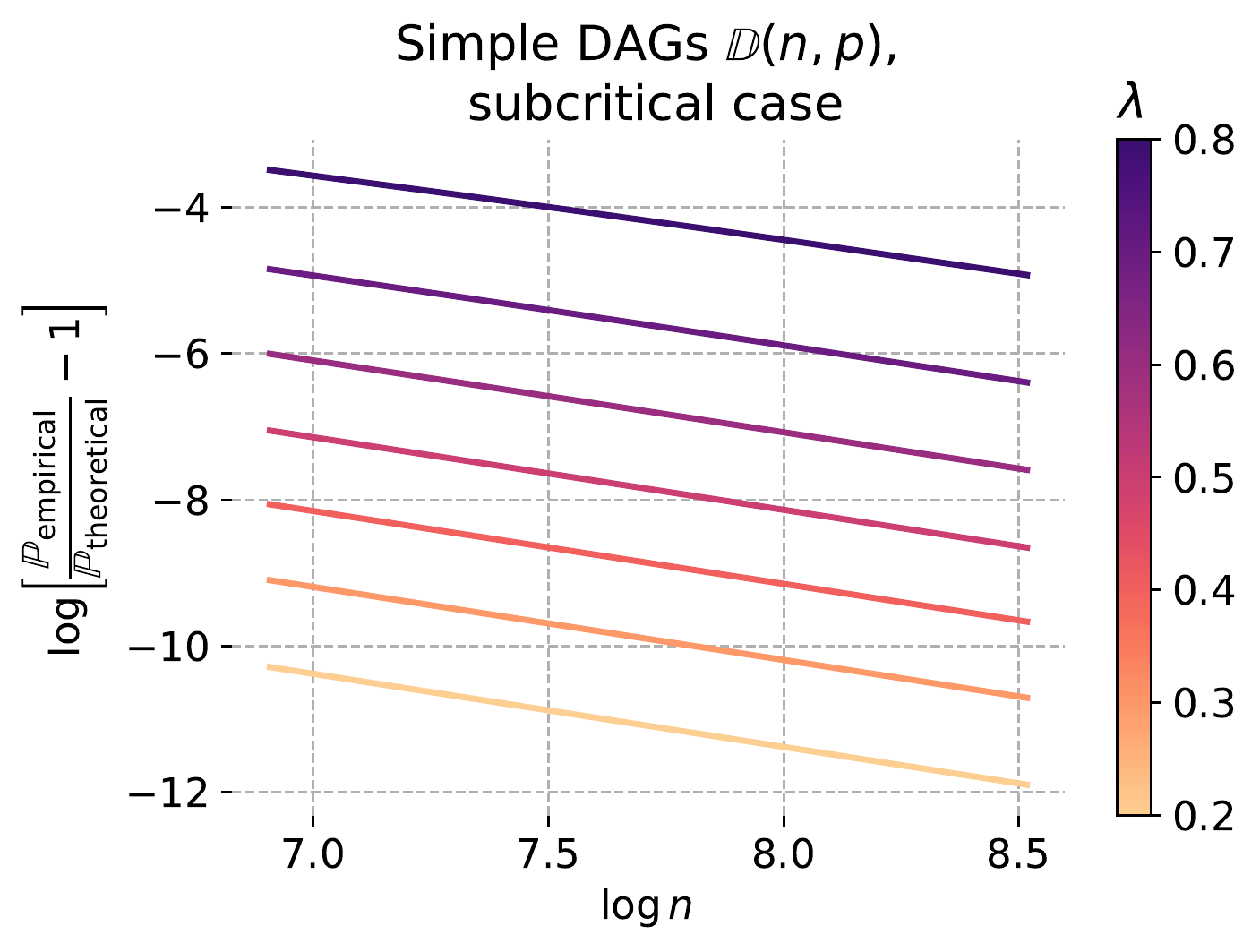}
&\includegraphics[width=0.32\textwidth]{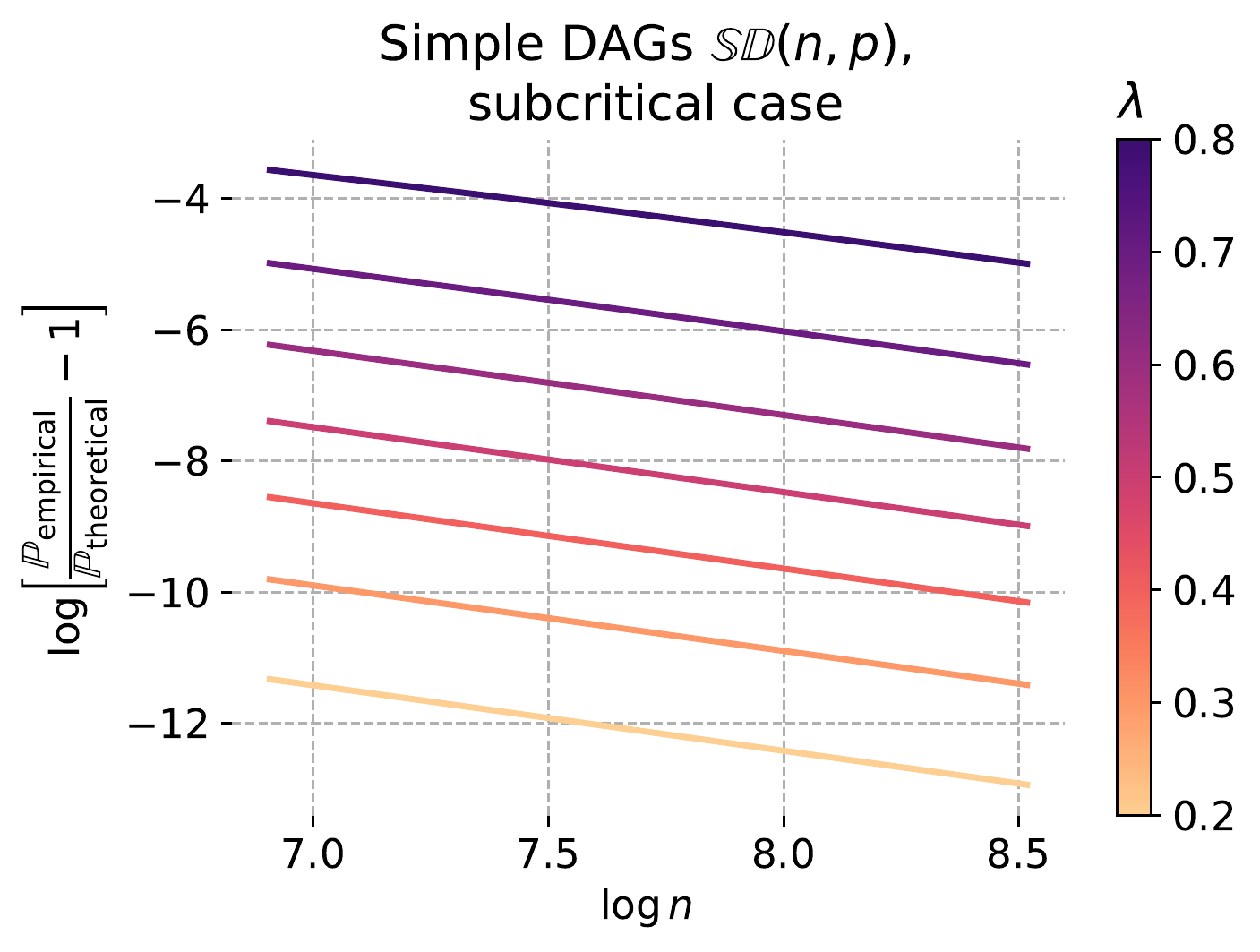}
\\
 \includegraphics[width=0.32\textwidth]{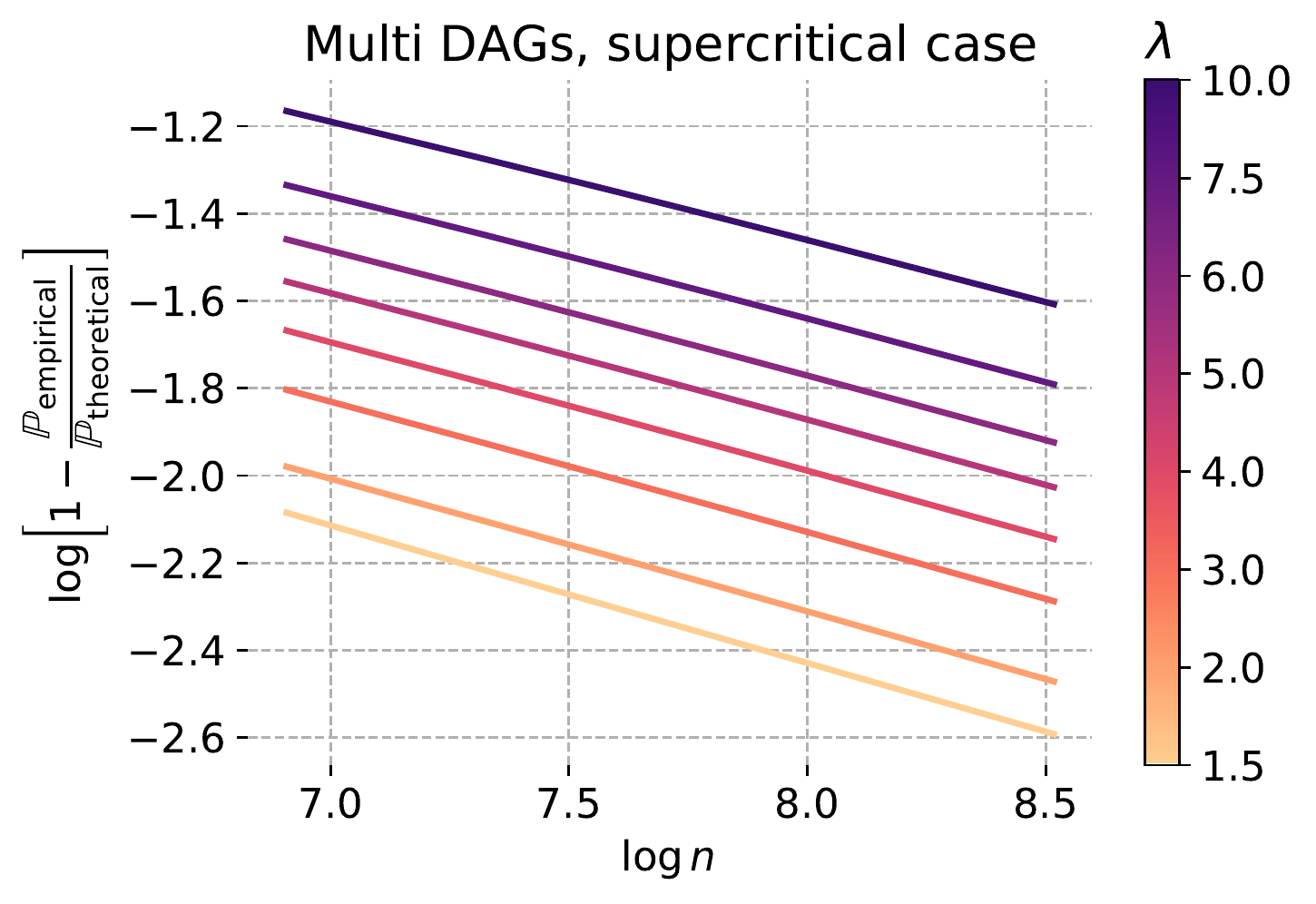}
&\includegraphics[width=0.32\textwidth]{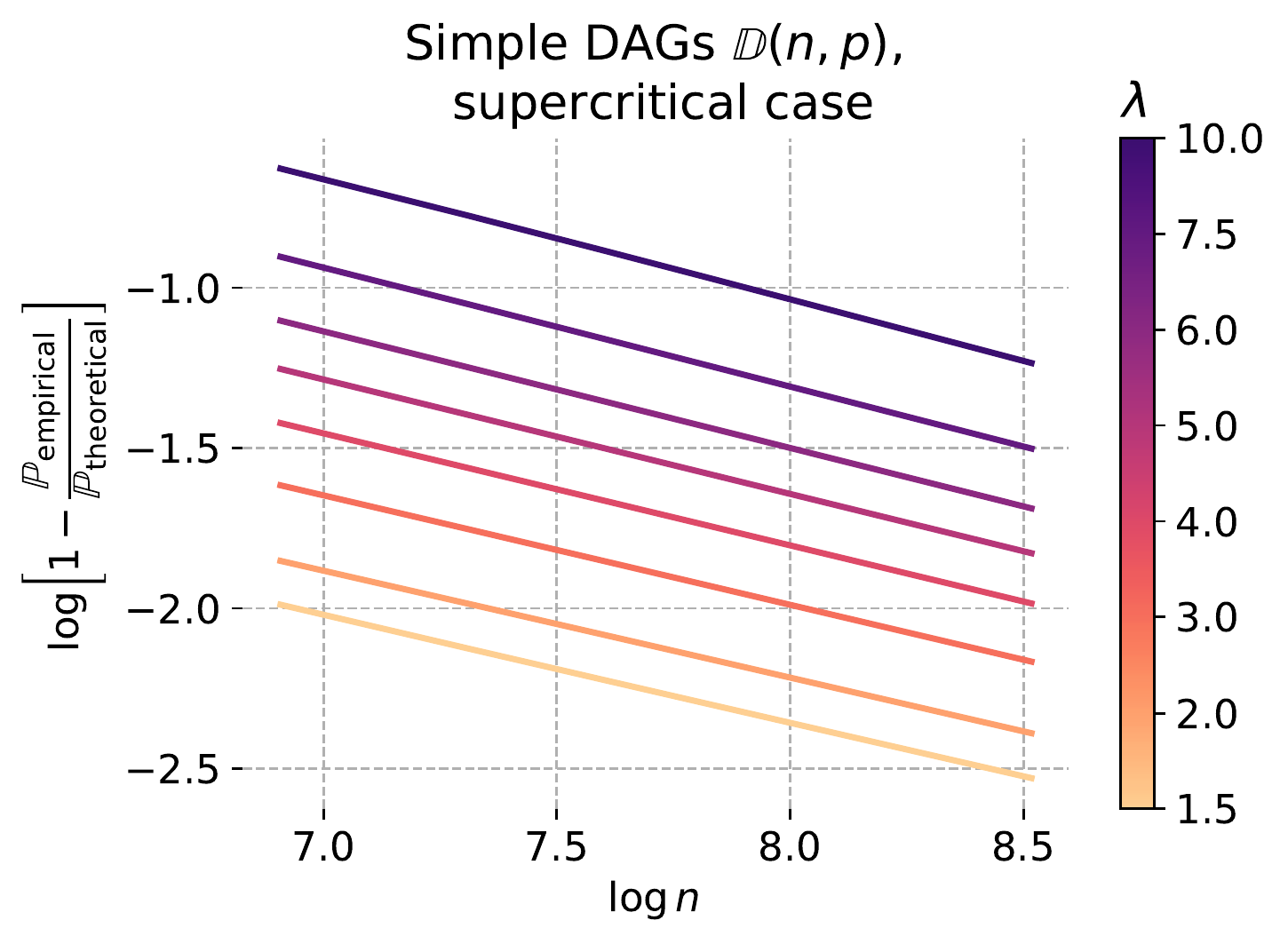}
&\includegraphics[width=0.32\textwidth]{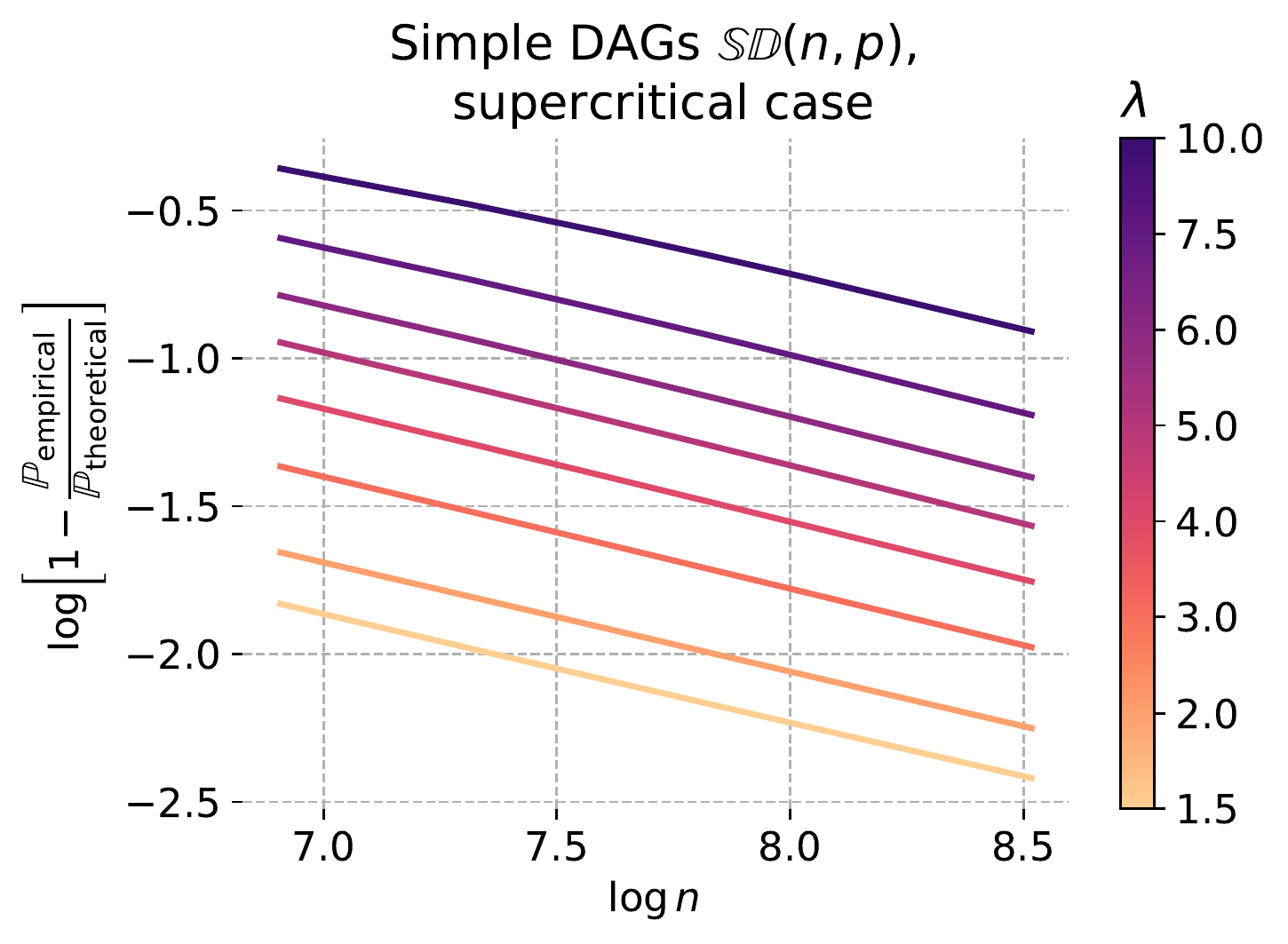}
\\
\hline
 \includegraphics[width=0.32\textwidth]{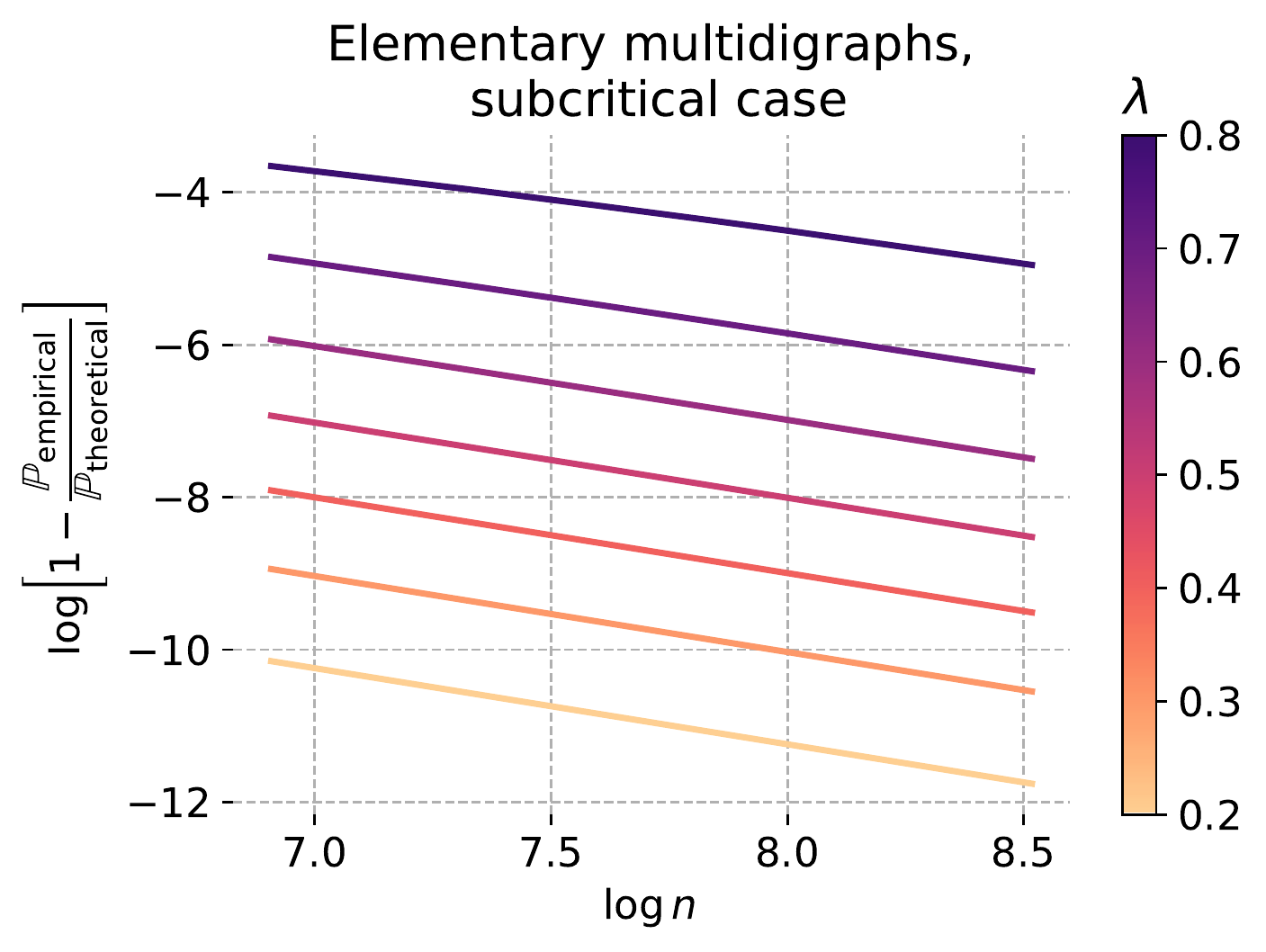}
&\includegraphics[width=0.32\textwidth]{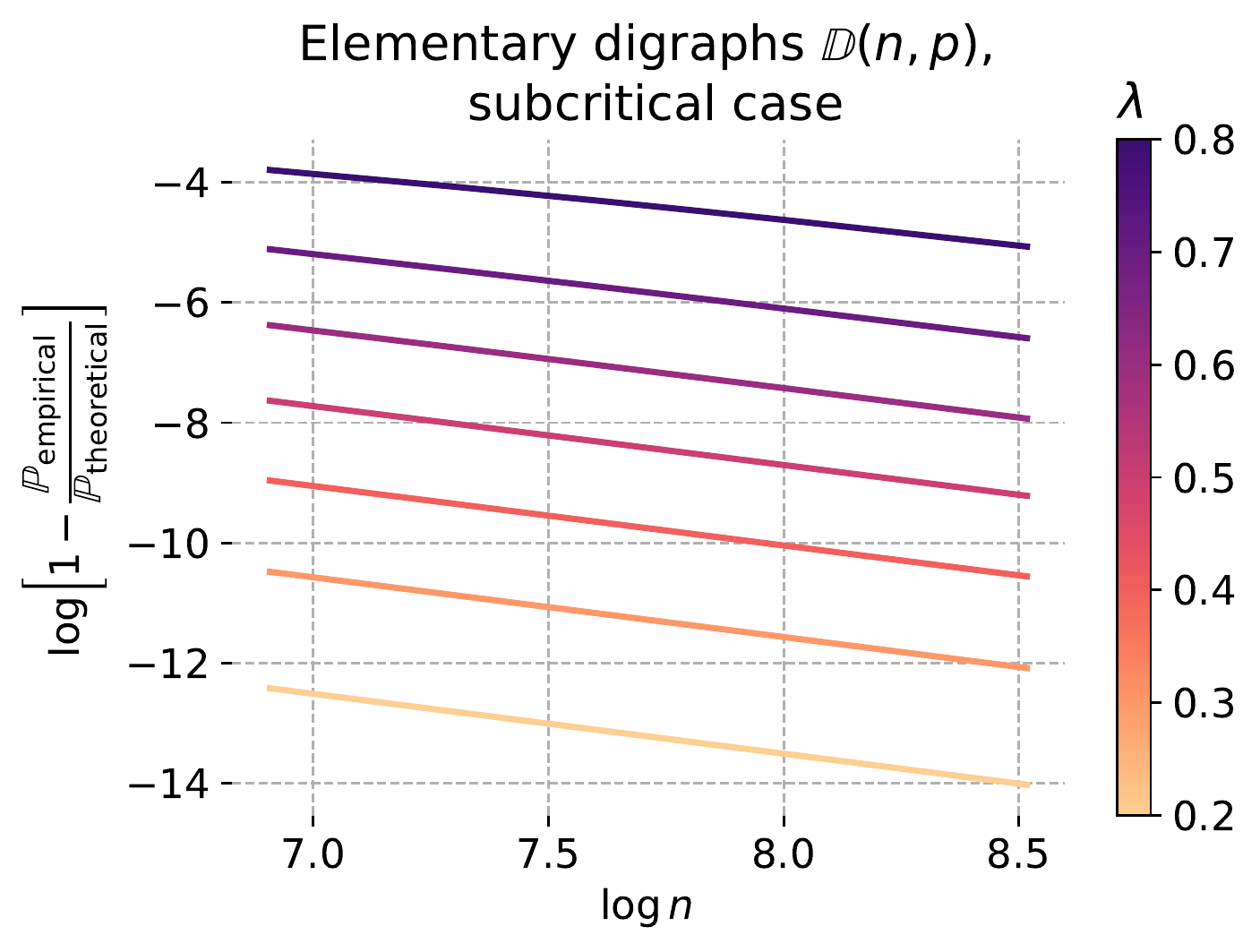}
&\includegraphics[width=0.32\textwidth]{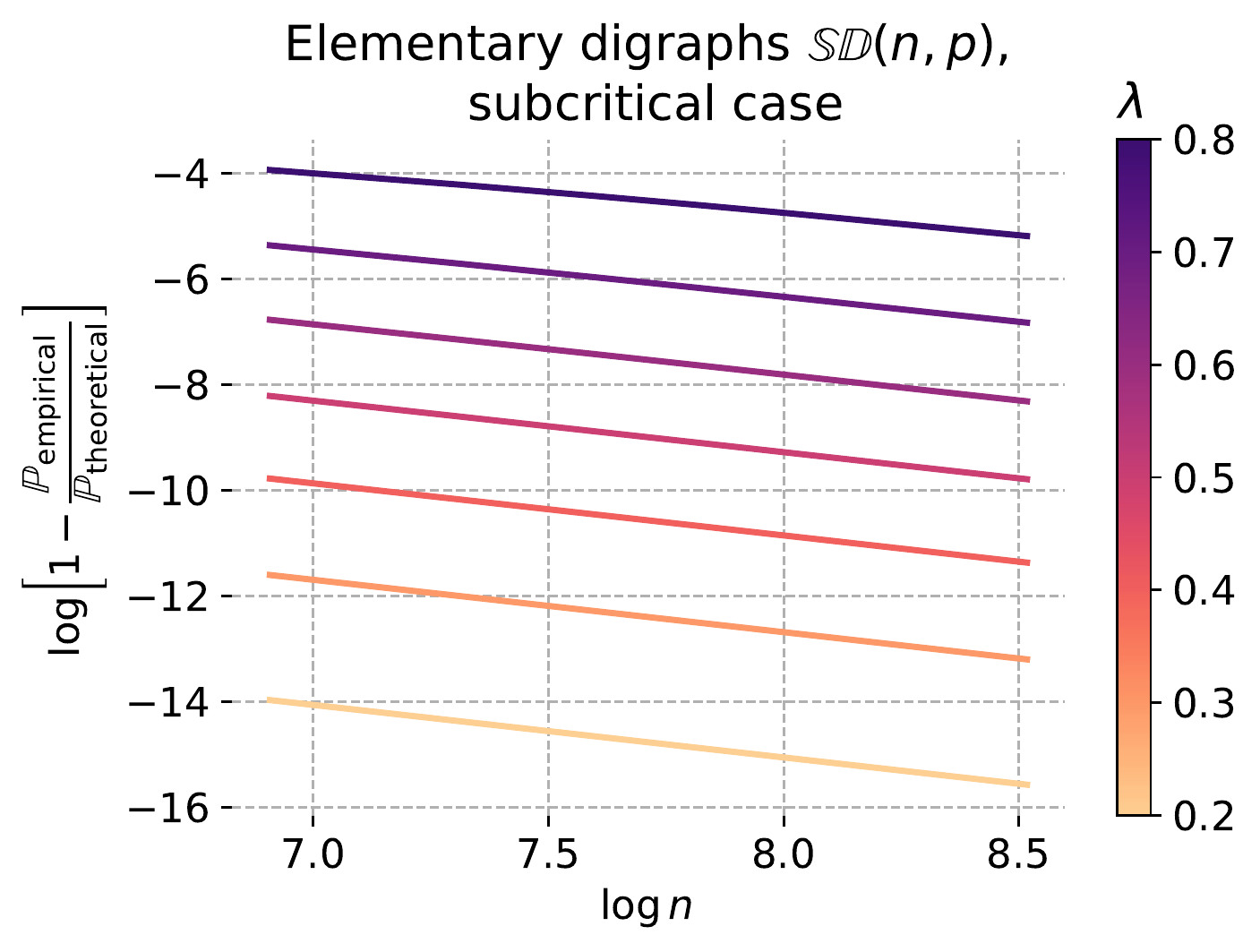}
\\
 \includegraphics[width=0.32\textwidth]{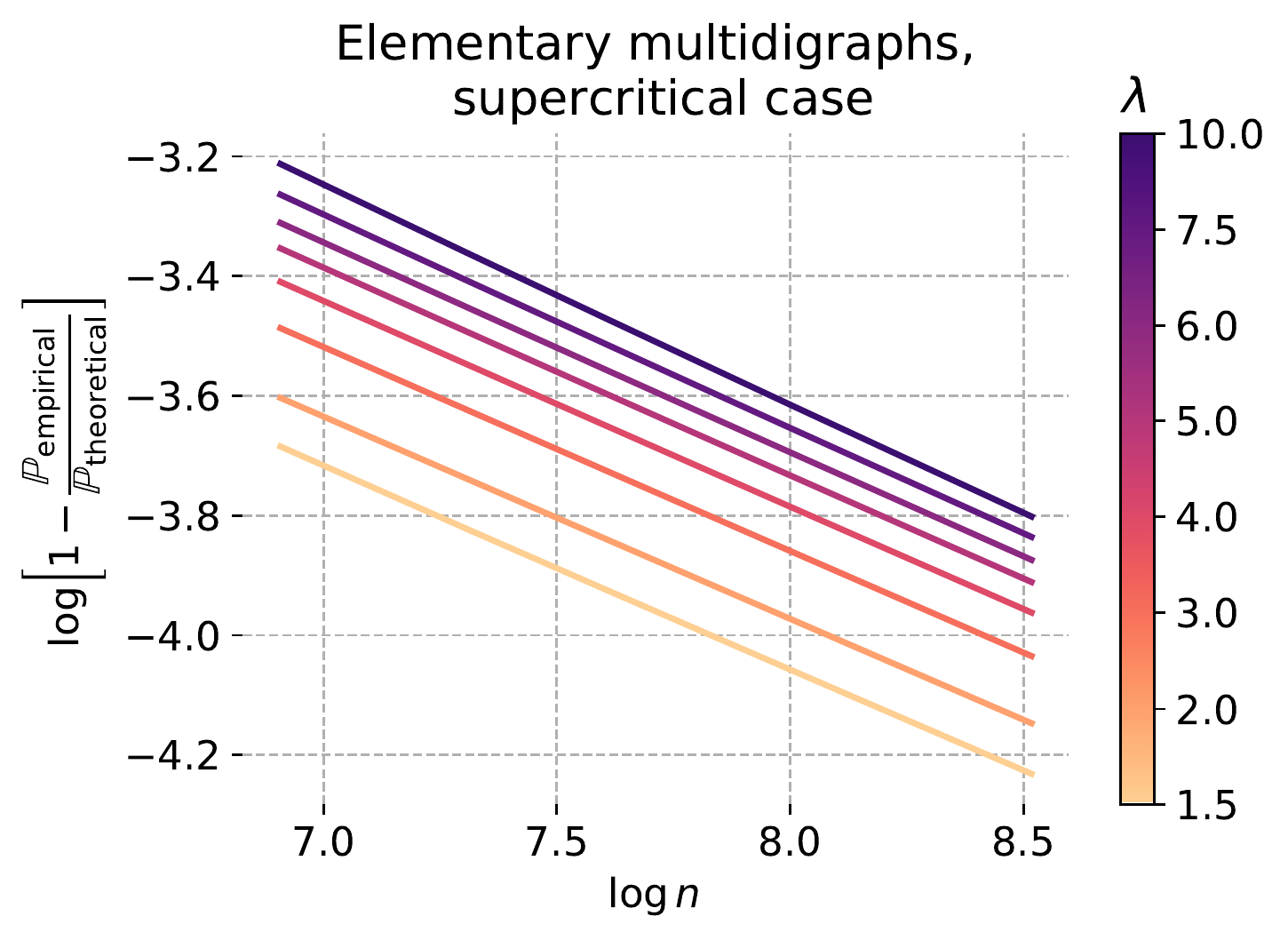}
&\includegraphics[width=0.32\textwidth]{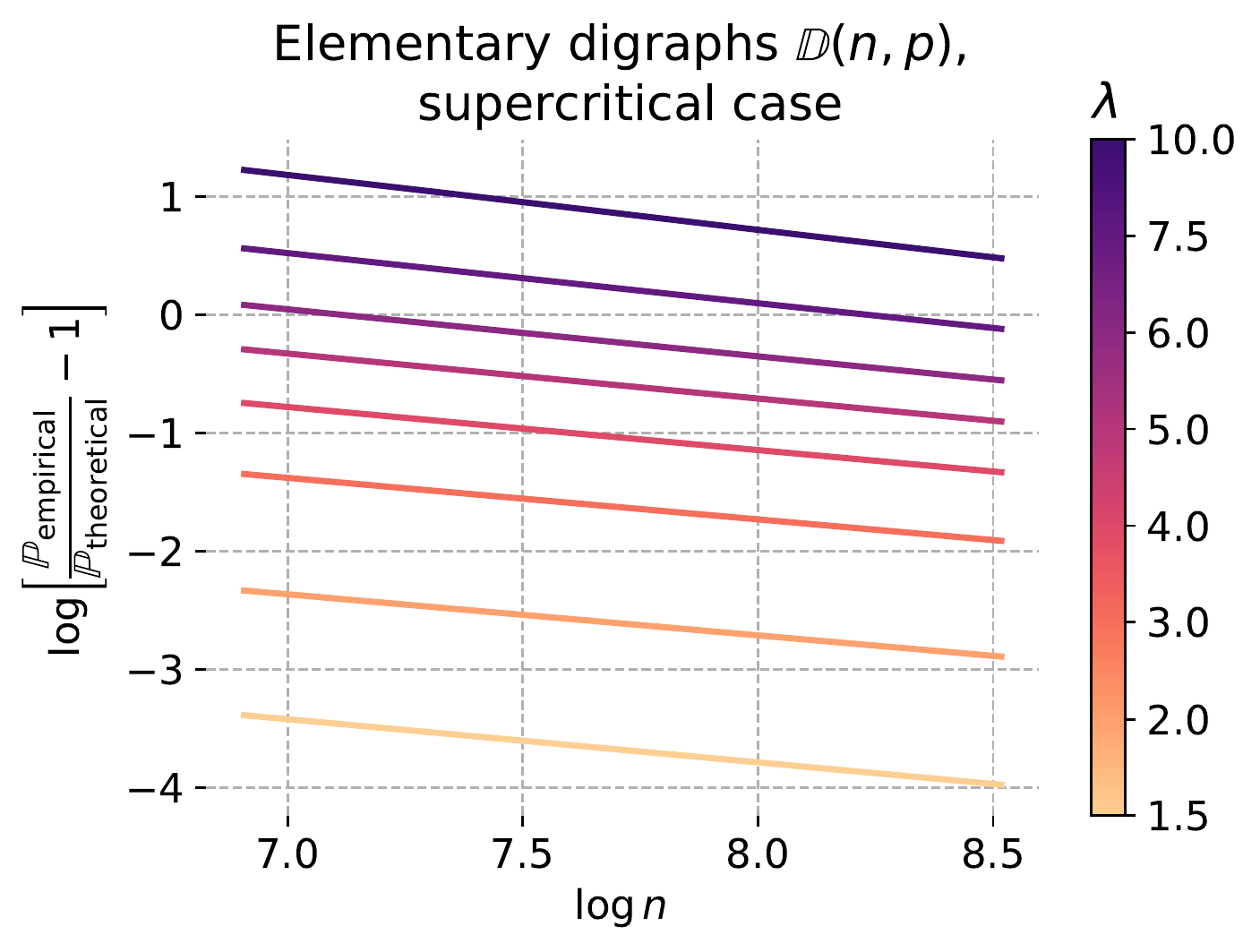}
&\includegraphics[width=0.32\textwidth]{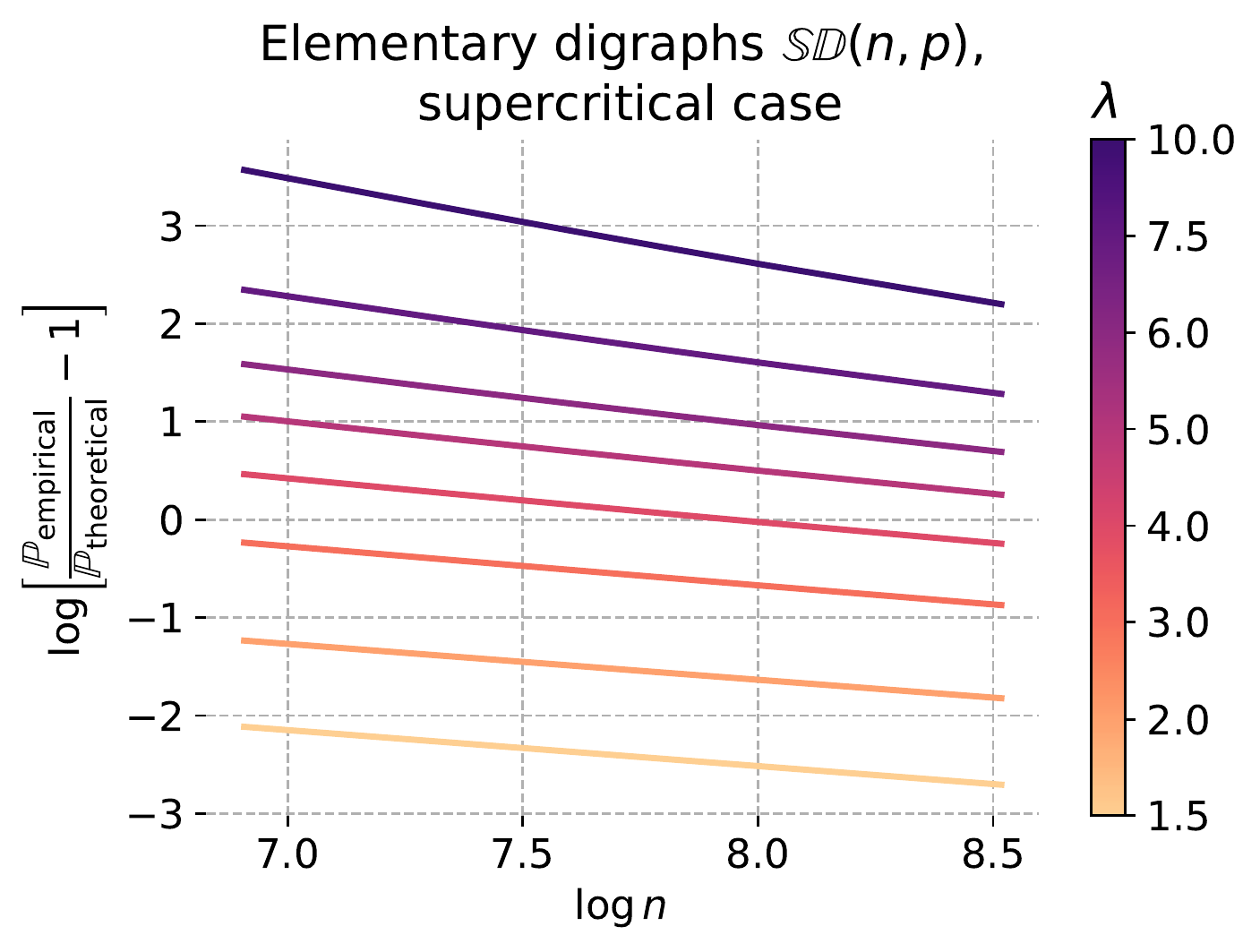}
\\
\hline
 \includegraphics[width=0.32\textwidth]{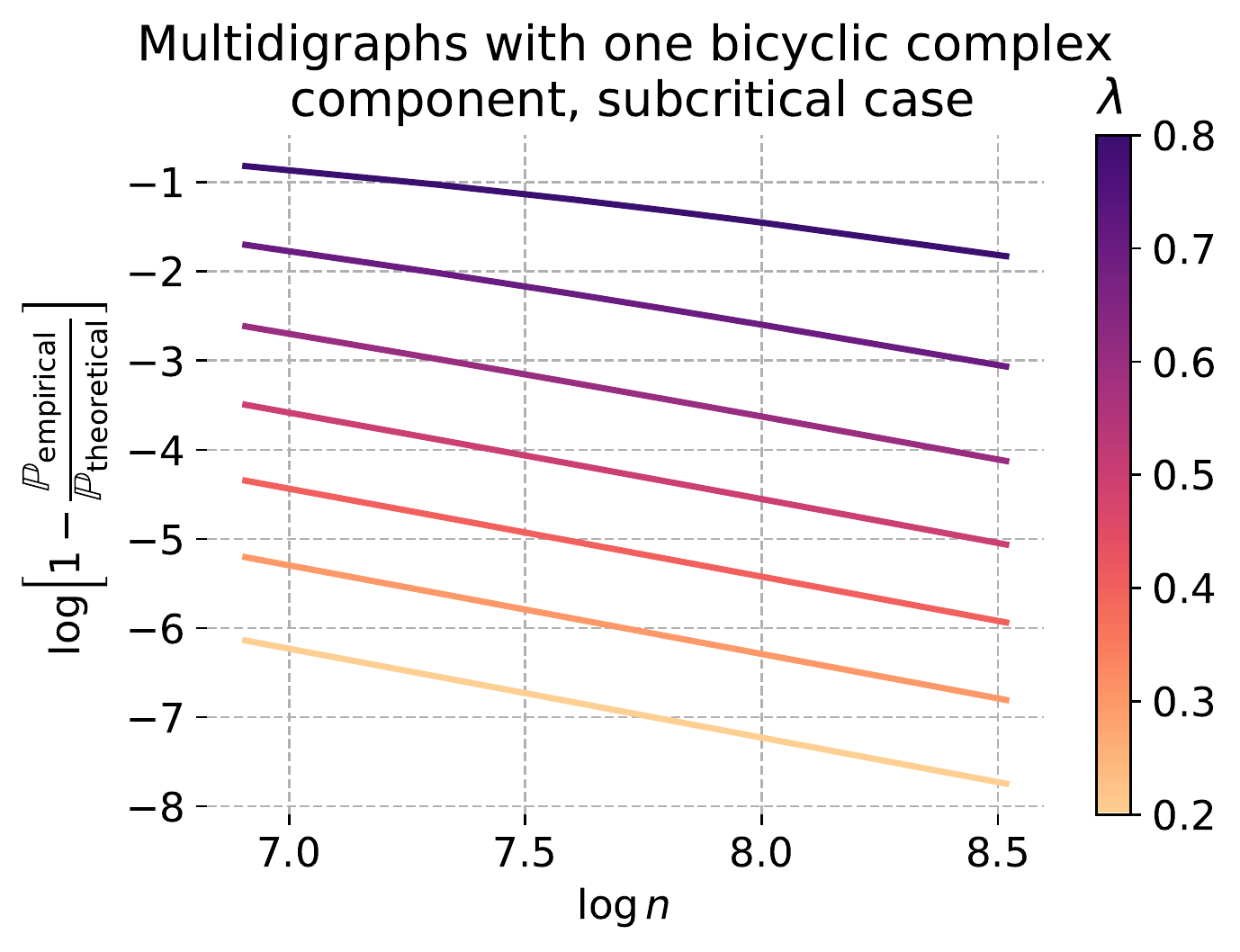}
&\includegraphics[width=0.32\textwidth]{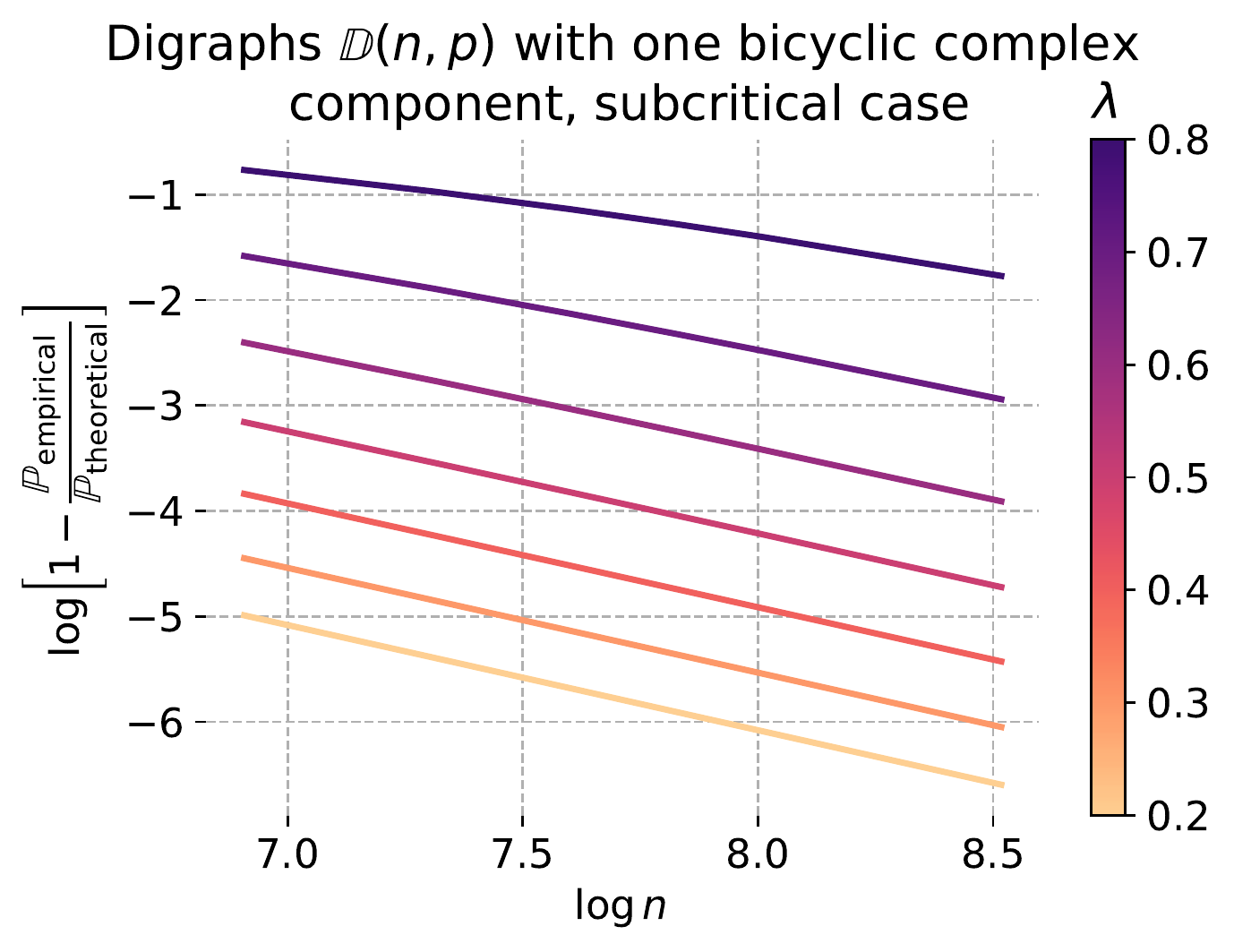}
&\includegraphics[width=0.32\textwidth]{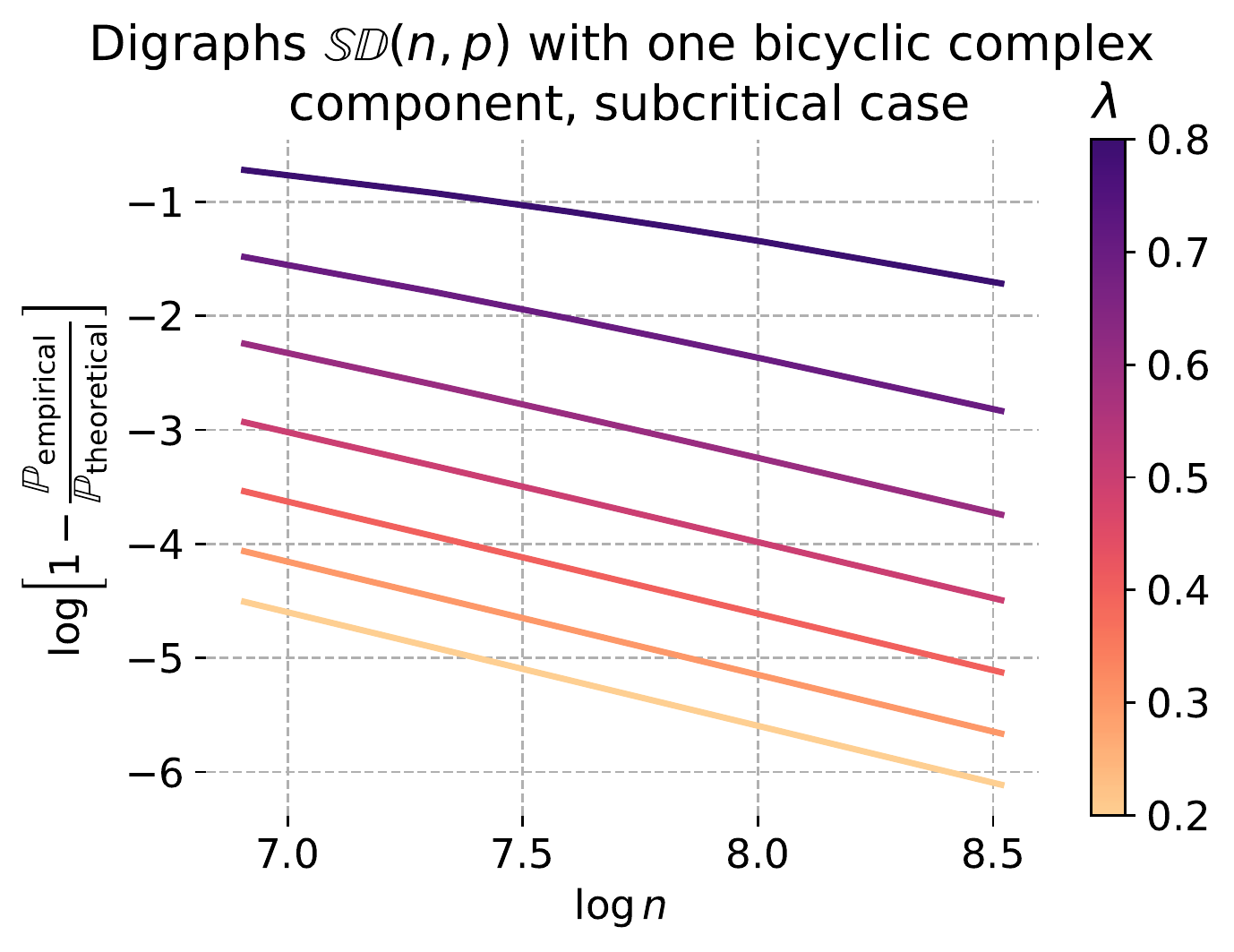}
\\
 \includegraphics[width=0.32\textwidth]{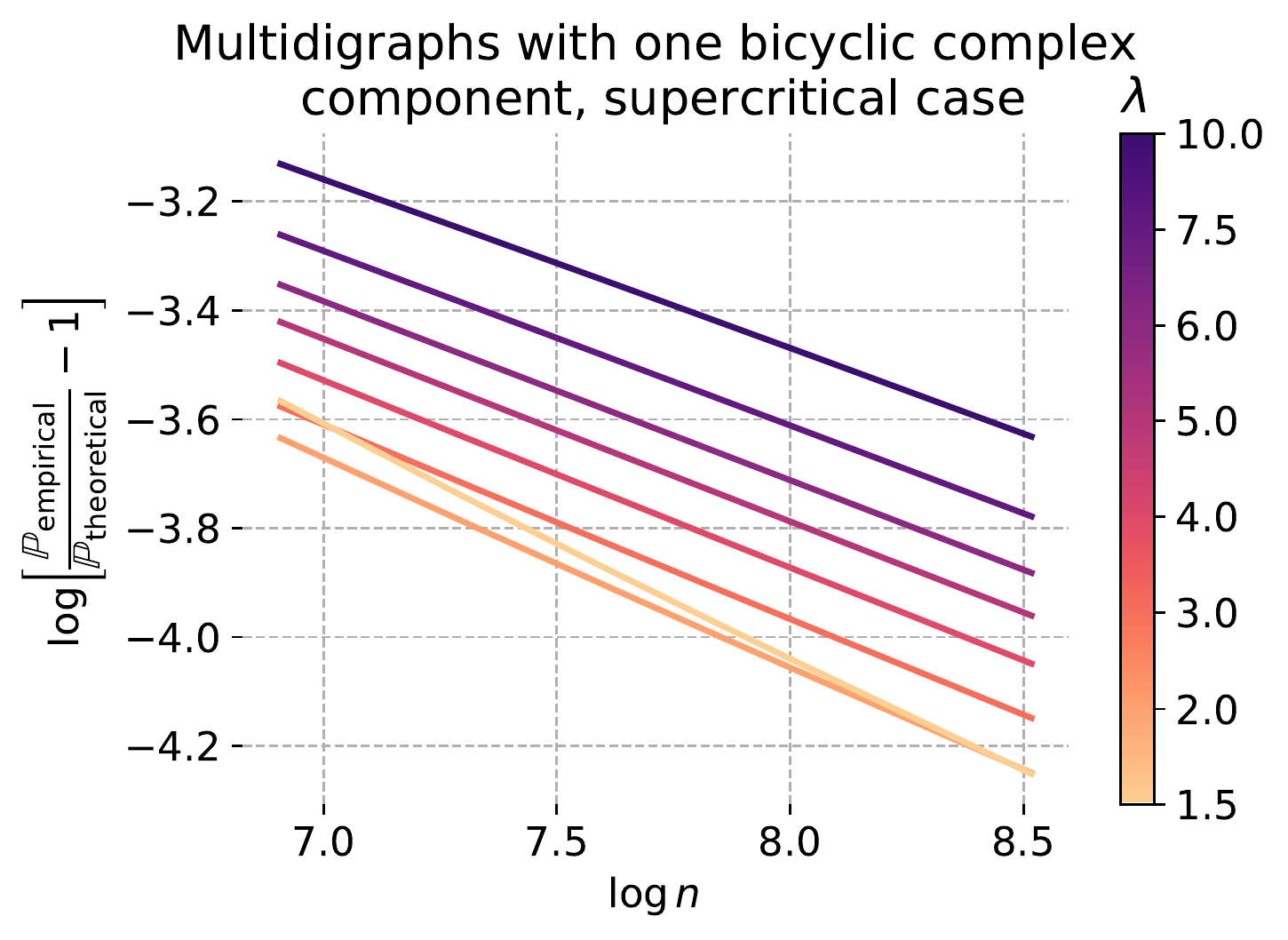}
&\includegraphics[width=0.32\textwidth]{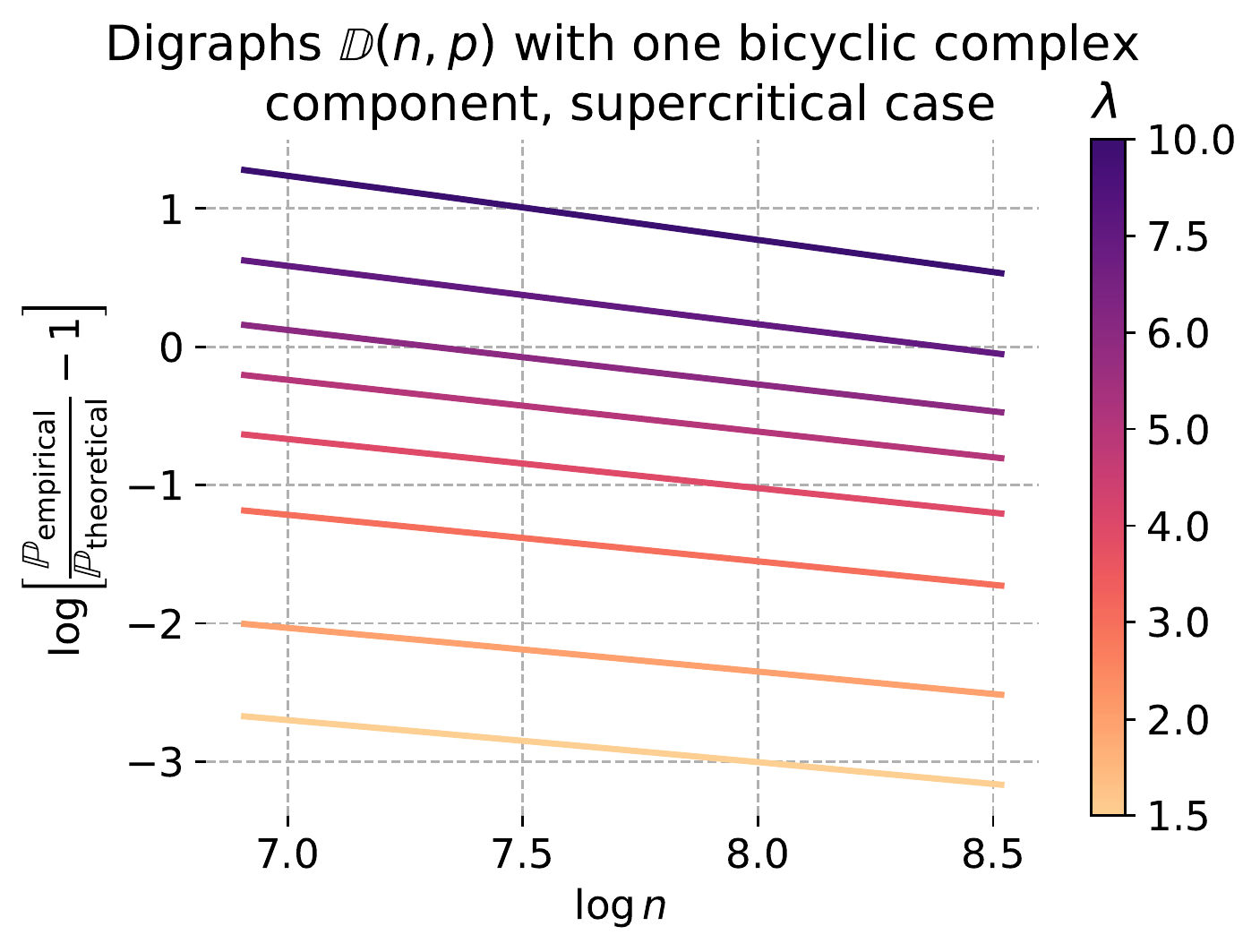}
&\includegraphics[width=0.32\textwidth]{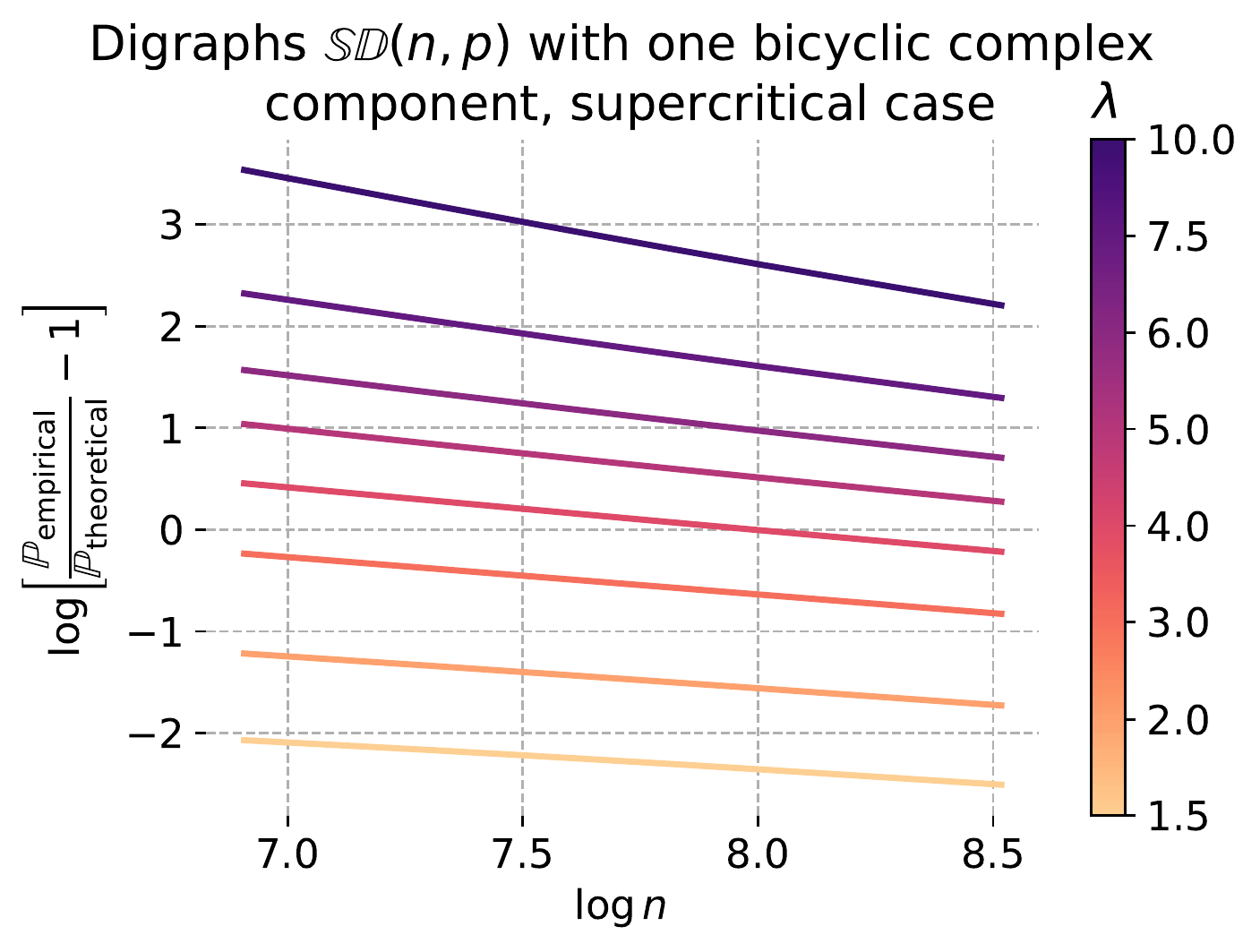}
\\
\hline
\end{tabular}
\caption{
\label{table:numerical}
Analysis of convergence of the empirical probabilities.}
\end{table}

\section{Discussion and open problems}
\label{section:discussion}

As the reader may have observed, the method that we develop in the current paper
opens lots of possibilities to study various families of digraph models. Note
that while all the analysis is carried out for \( \Di(n,p) \), \ie, when every
edge is chosen independently, a similar analysis for the model \( \Di(n,m) \)
(with a fixed number of edges) can be carried out as well. In the latter case,
the generating functions become bivariate, and one needs to extract the
coefficient in the form \( [z^n w^m] F(z, w) \). An example of such an analysis
has been performed in a previous paper of the first two
authors~\cite{Panafieu2020} in the case of simple digraphs where 2-cycles are
allowed. Their tools are equally applicable to the case of multidigraphs.
In that previous paper, a different tool was used, which
we refer to as \emph{graph decomposition}: the graphic generating functions
of DAGs and elementary digraphs could be expressed in terms of (derivatives) of
the generating function of graphs. Nevertheless, it is possible to avoid the
limitations of the graph decomposition method, and employ the double
saddle point method with respect to the two variables.

The asymptotic behaviour of the models $\DiGilbBoth(n,p)$ and $\Di(n,m)$
should essentially be the same if the average number of
edges in  $\DiGilbBoth(n,p)$ is  relatively close to the  number  of
edges $m$,  which means $pn(n-1)  \sim m$.
The same remak  holds for  $\DiGilb (n,p)$ and $\DiGilb (n,m)$,
$\DiGilbMulti(n,p)$ and  $\DiERMulti(n,m)$.
For the undirected version, such  relationships
are indicated in  \cite{Erdos1960} (see also
\cite[chap. II]{Bollobas2001}) about the binomial model $\GrER(n,p)$
and the uniform model $\GrER(n,m)$.

Many alternative graph and digraph models have recently gained more popularity,
including models
with degree constraints and with skewed degree distributions. The most typical type of
degree distribution that occurs in real-world applications is the power-law. It is therefore
tempting to extend the study of digraphs to the case where the in- and
out-degrees are constrained, and belong to two dedicated sets (possibly with
weights).

The classical \( \Di(n,p) \) models also have much more yet to study,
including the joint distribution of the parameters in the source-like and
sink-like components, the size of the giant strongly connected component, the
distribution of the excess of
the complex component, and so on. Using the symbolic method that we present
in~\cref{section:symbolic:method}, it is possible to tackle all these questions.
Note that, however, in some cases, new analytic tools still need to be obtained.
As a closely related question, since we expect that the excess of the strongly
connected component is growing when \( p \) is growing, should it mean that the
maximum point of the bell-shaped curve in~\cref{fig:I} moves towards the right
as the index of the Airy integral increases? Can it be shown that the
bell-shaped curve, when centred, is Gaussian in the limit?

Some of the questions that may arise when one thinks of the limiting
distributions in directed graphs, can be readily attacked with the tools that we
already have in~\cref{section:airy}. For example, if we are interested to know
the probability that there are two complex components with given
excesses (or three or more), we need to proceed with the coefficient extraction
from a generating function in which there is an allowed family of complex
components, each marked with a separate marking variable. The asymptotic
analysis of such expressions results in more complicated Airy integrals of the
form
\[
    C \cdot \dfrac{1}{2 \pi i} \int_{-i \infty}^{i \infty}
    \dfrac{\prod_{k = 1}^N \ai(r_k; \tau)^{q_k}}{\ai'(\tau)^M}
    e^{2^{-1/3} \mu \tau - \mu^3/6} \mathrm d \tau.
\]
Another model that also deserves to be studied is the model of directed
hypergraphs, for which a variant of the symbolic method has been recently obtained
in~\cite{Ravelomanana2020}.

As we observe empirically in~\cref{section:numerical:results}, in order to have
more precise computations for given \( n \) and \( p \) for the probabilities
that a digraph belongs to a given family, one needs more terms in the
asymptotic expansion, as the first error term is of order \( n^{-1/3} \).
Full asymptotic expansions (different for the models of
multidigraphs and simple digraphs) could be useful in this context.
For moderate values of \( n \) (\eg \( n \approx 100 \)),
the error of the approximation in the supercritical phase could reach \( 50\%
\), but it eventually converges when \( n \) grows. This means that there is a
need for several different kinds of full asymptotic expansions, each for a different
regime.

There is an intriguing similarity between the Airy integrals that appear in the
current paper, and Airy integrals that appear in the study of the probability
density function of a reflected Brownian motion with positive parabolic drift.
To give an example, in Knessl's paper~\cite[(3.8)]{knessl2000exact}, an
expression
\[
    Q(x,\mu)= e^{\mu x} 2^{2/3}
    \dfrac{1}{2 \pi i}
    \int_{-i \infty}^{i \infty}
    e^{\mu \tau - \mu^3/6}
    \dfrac{\ai(2^{1/3}(x + \tau))}{\ai(2^{1/3}\tau)^2}
    \mathrm d \tau
\]
arises as a solution to the system of partial differential equations
\[
    \begin{cases}
    Q_\mu = \tfrac12 Q_{xx} - \mu Q_x, & -\infty < \mu < \infty,\, x > 0;\\
    \tfrac12 Q_x(0, \mu) - \mu Q(0, \mu) = 0, & -\infty < \mu < \infty;\\
    Q \to \delta(x), & \mu \to -\infty,
    \end{cases}
\]
where \( Q(x, \mu) \) denotes for any given time \( \mu \) the probability
density function with respect to particle position \( x \geqslant 0 \) of a
positive reflected Brownian motion starting at \( \mu_0 = -\infty \) with drift
\( \mu \cdot \mathrm d\mu \) at time \( \mu \).
Remarkably, when \( x = 0 \), we obtain the rescaled asymptotic probability of
directed acyclic graphs times \( 2n^{1/3} \), so in fact,
\( Q(0, \mu) = 2 \varphi(\mu) \).

As we argue in~\cref{remark:knessl}, this similarity may potentially lead to another
link between the study of graphs and directed graphs and Brownian motions,
originating from the paper of Aldous~\cite{aldous1997brownian} and first applied
to the model of directed graphs in~\cite{Christina2019}. In~\cite{janson2012}, Janson also links various Airy integrals to the moments of the time when a Brownian motion with parabolic drift attains its maximum.

\paragraph*{Acknowledgements.}
Vlady Ravelomanana once asked the question why the convolution product of
generating functions was mostly concerned with \textit{disjoint} components, and
why there was seemingly no attempt to track allowed/forbidden edges
\textit{between} the distinct components.
This simple idea has grown into the discovery of the symbolic method for
directed graphs (which, as discovered later, has been independently established
around 50 years ago by Liskovets, Robinson and then continued by Stanley and
Gessel). Vlady also put in contact our two mini-groups who were working
independently on the problem of digraph asymptotics, which allowed us to join our
techniques and extend them. The authors are infinitely grateful to him for his
inspiration and passion with random graphs and their variations.

Many interesting remarks have arisen during the seminar talks and discussions on
various occasions, so we would like to thank
Cyril Banderier,
Christina Goldschmidt,
Valeriy Liskovets and
Andrea Sportiello,
for their interest and their witty observations.

Sergey Dovgal was partially supported by the HÉRA project, funded by The French National
Research Agency, grant no.: ANR-18-CE25-0002, and by the EIPHI Graduate School (contract
ANR-17-EURE-0002), FEDER and Région Bourgogne Franche-Comté.
\'Elie de Panafieu was supported by the Rise project RandNET, grant no.: H2020-EU.1.3.3 and the Lincs (\url{www.lincs.fr}). Dimbinaina Ralaivaosaona was supported by the National Research Foundation of South Africa, grant no.: 129413.
Stephan Wagner was supported by the Knut and Alice Wallenberg Foundation, grant number KAW 2017.0112.
\bibliographystyle{alpha}
\bibliography{ac-digraphs}

\end{document}